\documentclass[11pt,twoside,reqno]{amsart}
\usepackage{amsmath, amssymb, amsthm, graphicx,marvosym}
\usepackage{lipsum}
\usepackage{bbm}
\usepackage{enumitem}
\usepackage{mathtools}
\usepackage{xcolor}
\usepackage{graphicx}
\usepackage{tensor}
\usepackage{hyperref}
\usepackage{enumitem}
\usepackage{cancel}
\usepackage[nobysame]{amsrefs}[11pts]
\usepackage{epigraph}

\usepackage{color}
\usepackage{geometry}
\newtheorem{theorem}{Theorem}[section]
\newtheorem{definition}{Definition}[section]
\newtheorem{lemma}{Lemma}[section]
\newtheorem{remark}{Remark}[section]
\newtheorem{proposition}{Proposition}[section]
\newtheorem{corollary}{Corollary}[section]
\numberwithin{equation}{section}
\numberwithin{figure}{section}

\newtheoremstyle{case}{}{}{}{}{}{:}{ }{}
\theoremstyle{case}

\makeatletter \@addtoreset{equation}{section} \makeatother

\hypersetup{
	colorlinks,
	citecolor=black,
	filecolor=black,
	linkcolor=blue,
	urlcolor=blue
}

\def\tilde{\widetilde}
\renewcommand{\u}{{\boldsymbol{U}}}

\newcommand{\ve}{\varepsilon}
\newcommand{\x}{{\boldsymbol{x}}}
\newcommand{\y}{{\boldsymbol{y}}}

\newcommand{\ov}{\overline}
\newcommand{\ul}{\underline}
\newcommand{\R}{{\mathbb R}}
\newcommand{\dif}{{\mathrm d}}

\newcommand{\supp}{{\mathrm {supp}}}

\newcommand{\dd}{{\mathrm {div}}}
\newcommand{\D}{{\mathrm D}}
\newcommand{\snorm}[1]{\ensuremath{\lvert #1\rvert}}
\newcommand{\Bignorm}[1]{\ensuremath{\Big\lvert #1\Big\rvert}}
\newcommand{\norm}[1]{\ensuremath{\left\lvert #1\right\rvert}}
\newcommand{\sbnorm}[1]{\ensuremath{\lVert #1\rVert}}
\newcommand{\Bigbnorm}[1]{\ensuremath{\Big\lVert #1\Big\rVert}}
\newcommand{\bnorm}[1]{\ensuremath{\left\lVert #1\right\rVert}}

\DeclareMathOperator*{\esssup}{ess\,sup}

\begin{document}
\title[Compressible Navier-Stokes Equations]{Global Spherically Symmetric Solutions of the	Multidimensional Full Compressible
	Navier-Stokes Equations with  Large Data}

\author{Gui-Qiang G. Chen}
\address{Gui-Qiang G. Chen:\, Mathematical Institute,
	University of Oxford, Oxford,  OX2 6GG, UK.}
\email{\tt chengq@maths.ox.ac.uk}

\author{Yucong Huang}
\address{Yucong Huang:\, Mathematical Institute,
	University of Oxford, Oxford,  OX2 6GG, UK.}
\email{\tt  yucong.huang@maths.ox.ac.uk}

\author{Shengguo Zhu}
\address{Shengguo Zhu:\, School of Mathematical Sciences, CMA-Shanghai,   and MOE-LSC, Shanghai Jiao Tong University, Shanghai, 200240, China.}
\email{\tt zhushengguo@sjtu.edu.cn}


\date{\today}

\begin{abstract} We establish the global-in-time existence of
solutions of the Cauchy problem for
the full Navier-Stokes equations for compressible heat-conducting flow in
multidimensions with initial data that are large, discontinuous, spherically symmetric, and away from the vacuum.
The solutions obtained here are of global finite total relative-energy including the origin, while
cavitation may occur
as balls centred at the origin of symmetry for which the interfaces between
the fluid and the vacuum must be upper semi-continuous in space-time in the Eulerian coordinates. On any
region strictly away from the possible vacuum, 
the velocity and specific internal energy are H\"older continuous, and the density has a uniform upper bound.
To achieve these, our main strategy is to regard the Cauchy problem as the limit of a series of carefully
designed initial-boundary value problems that are formulated in finite annular regions. For such
approximation problems, we can derive uniform {\it a-priori} estimates  that are independent
of both the inner and outer radii of the annuli
considered in the spherically symmetric Lagrangian coordinates. The entropy inequality is recovered
after taking the limit of the outer radius to infinity by using Mazur's lemma and the convexity of
the entropy function, which is required for the limit of the inner radius tending to zero.
Then the global weak solutions of the original problem are attained via careful compactness arguments
applied to the approximate solutions in the Eulerian coordinates.
\end{abstract}
\maketitle

{\hypersetup{linkcolor=blue}
	\tableofcontents
}
\section{Introduction}\label{sec:intro}

The motion of a compressible viscous, heat-conductive, and Newtonian polytropic fluid occupying a spatial domain $\Omega \subset \mathbb{R}^n$ for $n\ge 2$
is governed by the following full compressible Navier-Stokes system (\textbf{CNS}) in the Eulerian coordinates:
\begin{equation}\label{eqs:FNS}
\begin{cases}
\partial_t \rho + \dd(\rho \u)=0,\\
\partial_t (\rho \u ) + \dd ( \rho \u \otimes \u ) +\nabla P = \dd\,\mathbb{S},\\
\partial_t(\rho E)+\dd ( \rho E\u  + P\u )  = \dd (\u\mathbb{S})  +\text{div}(\kappa\nabla  e).
\end{cases}
\end{equation}
Here and throughout the paper, $\rho\geq 0$ denotes the mass density, $\u=(U^1, \dotsc, U^n )^\top$ the fluid velocity,
$P$ the pressure of the fluid, $E=e+\frac{1}{2}\snorm{\u}^2$ the specific total energy,
$e$ the specific internal energy, $\x=(x^1,\dotsc, x^n)^\top\in \Omega$ the Eulerian spatial coordinates,
and $t\geq 0$ the time coordinate. The equations of state for polytropic fluids are
\begin{equation}\label{2}
    P=R\rho\theta=(\gamma-1)\rho e,\hspace{7mm} e=c_v\theta,\hspace{7mm}c_v=\frac{R}{\gamma-1},
\end{equation}
where $R$ is the gas constant,  $\theta$ the absolute temperature,  $c_v$  the specific heat at constant volume,
and  $\gamma>1$ the adiabatic exponent. The viscosity stress tensor $\mathbb{S}$ is given by
\begin{equation}\label{3}
\mathbb{S}=2\mu D(\u)+\lambda\,\dd\u\,\mathbb{I}_n,
\end{equation}
where $D(\u)=\frac{\nabla \u+(\nabla \u)^\top}{2}$ is the deformation tensor, $\mathbb{I}_n$ the $n\times n$ identity matrix,
$\mu$ the shear viscosity coefficient, and $\lambda+\frac{2}{n}\mu$ the bulk viscosity coefficient.
The constant, $\kappa$, in the energy equation is defined by $\kappa\vcentcolon= \frac{\kappa_{Q}}{c_{V}}$,
where $\kappa_{Q}\ge0$ is the coefficient of thermal conductivity. Furthermore, $(\mu,\lambda,\kappa)$ satisfies the following  physical condition:
\begin{equation}\label{eqs:viscoef}
\mu > 0, \qquad \lambda+\dfrac{2}{n}\mu\ge 0, \qquad  \kappa>0.
\end{equation}

In this paper,
we are concerned with global spherically symmetric
solutions of the form:
\begin{equation*}
(\rho,\u,e)(\x,t)
=(\rho(\snorm{\x},t),u(\snorm{\x},t)\dfrac{\x}{\snorm{\x}},
 e(\snorm{\x},t))
\end{equation*}
of system \eqref{eqs:FNS}--\eqref{eqs:viscoef} in domain $\Omega=\R^n$ with Cauchy data:
\begin{equation}\label{eqs:CauchyInit}
  (\rho,\u,e)(\x,0)=(\rho_0,\u_0,e_0)(\x)= ( \rho_0(\snorm{\x}),u_0(\norm{\x})\dfrac{\x}{\snorm{\x}},e_0(\snorm{\x})) \qquad\, \text{for $\x\in \R^n$.}
\end{equation}
Therefore, for the spherically symmetric case with the radial coordinate variable $r=|\x|$, problem \eqref{eqs:FNS}--\eqref{eqs:CauchyInit} can be reformulated as:
\begin{equation}\label{eqs:SFNS}
\begin{cases}
\partial_t\rho + \partial_r(\rho u) + m\dfrac{\rho u}{r} =0,\\[2pt]
\rho \partial_t u +\rho u \partial_r u +  \partial_r P(\rho,e) = (2\mu+\lambda) \partial_r\big(\partial_ru+m\dfrac{u}{r}\big),\\[2pt]
\rho \partial_t e +\rho u \partial_r e +P(\rho,e)\big(\partial_r u+m\dfrac{u}{r}\big)
= 2\mu \big(\snorm{\partial_ru}^2 + m \dfrac{u^2}{r^2}\big) + \lambda \big(\partial_ru+m\dfrac{u}{r}\big)^2
 + \kappa \big(\partial_r^2e+ m\dfrac{\partial_re}{r}\big)
\end{cases}
\end{equation}
in $\Omega_T \vcentcolon=[0,\infty)\times[0,T]$ with $m\equiv n-1$, and
\begin{align}
& u(0,t)=0, \qquad \partial_r e(0,t)=0  && \text{for} \ \ t\in[0,T], \label{BC}\\
&(\rho,u,e)(r,0)= (\rho_0,u_0,e_0)(r) && \text{for} \ \ r\in [0,\infty). \label{IC}
\end{align}
The boundary conditions \eqref{BC} is derived from the continuity of $\x\mapsto (\u,\nabla e)(\x,t)$ at $\snorm{\x}=0$.

There is a large literature regarding the full \textbf{CNS}. For the general three-dimensional (3-D) flow, in the absence
of vacuum ({\it i.e.}, $\inf_{\x}\rho_0(\x)>0$), the local well-posedness of classical solutions of the Cauchy problem
for \eqref{eqs:FNS} follows from the standard symmetric hyperbolic-parabolic structure satisfying the well-known Kawashima's
condition ({\it cf}. \cites{I1, I2, KA, Nash, serrin,Tani}), which was extended to the global one
by Matsumura-Nishida \cite{MN} for the initial data close to a non-vacuum equilibrium in some Sobolev space $H^s(\mathbb{R}^3)$ for $s>\frac{5}{2}$.
For both isentropic and non-isentropic flow,
we also refer the reader to
Danchin \cites{danchin,danchin2} for  strong solutions
with small data in Besov spaces of $\R^n$ for $n\ge 2$,
and Hoff \cites{H4,H3} for weak solutions with small and discontinuous data in some Sobolev spaces of $\R^n$ for $n=2$ or $3$.
Considering the global well-posedness of classical or weak solutions with arbitrarily large initial data,
a subtle issue is whether, in the absence of smallness conditions, cavitation may occur in solutions of \eqref{eqs:FNS}
for barotropic and/or full \text{CNS},
which has no clear answer yet so far, except the one-dimensional (1-D) case.
Specifically,  for the 1-D flow,
the unique global  classical solution was obtained by Kazhikhov-Shelukhin
\cite{KS} for large initial data in a bounded interval with $\inf_{\x} \rho_0(\x)>0$, and
this theory was then extended to the unbounded domains in Kawashima-Nishida in \cite{KN}.
The key estimate developed in \cites{KS,KN} is the pointwise upper and lower bounds of the density,
which have led to many notable results for the problems concerning 1-D flow ({\it e.g.},  \cites{H1,CHT}).
The global existence and stability of 1-D weak solutions  with large and  discontinuous initial data can be found in
Chen-Hoff-Trivisa \cite{CHT}, Jiang-Zlotnik \cite{JZlotnik}, Zlotnik-Amosov \cites{ZA1,ZA2}, and the references cited therein.
Furthermore, it was proved in Hoff-Smoller \cite{HS} that the weak solutions of the 1-D \textbf{CNS} do not develop
vacuum regions, as long as no
vacuum states are present initially.
While the multidimensional (M-D) case
is much more complicated: up to now, very few solid progress is known on the problems
of well-posedness or vacuum formation with large data.

It is worth pointing out that the approaches used in the references mentioned above do not work directly
when initial vacuum appears ({\it i.e.}, $\inf_\x {\rho_0(\x)}=0$), which occurs when some physical requirements are imposed,
such as finite total initial mass and energy in the whole space.
One of the main issues in the presence of vacuum is the degeneracy of the time evolution operator,
which makes it difficult to understand the behavior of the velocity field near the vacuum.
In the terms of structures, system \eqref{eqs:FNS} is a hyperbolic-parabolic coupled system in the fluids region,
but degenerates to a hyperbolic-elliptic one near the vacuum region.
By imposing initially some compatibility conditions,  Cho-Kim in \cite{CK} firstly established the local well-posedness
of 3-D strong solutions with vacuum in terms of $(\rho, u,\theta)$ in some homogeneous Sobolev spaces,
which has been extended recently to be global ones with small energy but large oscillations by Huang-Li \cite{HL}, Wen-Zhu \cite{WZ},
and the references cited therein.

The first global existence of M-D weak solutions of finite energy for the  isentropic \textbf{CNS}
with generic data and vacuum was established in Lions \cite{L2},
especially for $\gamma \ge\frac{9}{5}$ when $n=3$;
see also Feireisl-Novotn\'{y}-Petzeltov\'{a} \cite{fu1} for $\gamma>\frac{3}{2}$ when $n=3$.
Moreover, this theory has been partially extended to the non-isentropic compressible flow with
$(\mu,\,\lambda,\,\kappa)$ depending on the temperature by Feireisl in
\cites{fu2,fu3}, where the thermal energy equation holds as an inequality in the sense of distributions.
However, the uniqueness problem of these weak solutions is widely open due to their fairly low regularity.
Also see
\cites{L1,LX1,BJ,VY,BVY} and the references cited therein.

For the M-D spherically symmetric flow, some favourable regularity properties may be expected, since the equations exhibit largely 1-D behaviour when the flow is away from both the origin and the far-field. However, compared with the 1-D case,  in order  to establish the existences  of M-D spherically symmetric  classical or weak solutions with  arbitrarily large initial data of \eqref{eqs:SFNS}, some essential difficulties are encountered including the coordinate singularity at the origin, indicated by the singular factor $\frac{1}{r}$ in system \eqref{eqs:SFNS} and infinite initial total mass/energy from the far field;
also see \cites{JZ,JZ1,JZ2,H2,HJ,MRS}.

To circumvent these issues, it is customary to consider instead the initial-boundary value problem posed in some annular domains
\begin{equation*}
	\{ (\x,t)\in \R^n \times [0,\infty)\,\vcentcolon\,a \le \snorm{\x} \le b \}
\end{equation*}
with proper boundary conditions, where $a>0$ and $b>0$ are both finite constants. The global well-posedenss of 3-D
strong solutions with cylindrical symmetry and strictly positive initial density to the full \textbf{CNS}
has been given by Frid-Shelukhin \cite{FS} in an infinite cylinder of finite radius parallel to and centered
along the $x^3$--axis in $\mathbb{R}^3$. In the domain  between a static solid core and a free boundary
connected to a surrounding vacuum state, Yashima-Benabidallah in
\cite{YB} first established the global existence of spherically symmetric weak solutions with large and discontinuous initial data
to the full \textbf{CNS} with boundary data prescribed so that the uniform boundedness of the total energy of solutions is direct to achieve,
and later Chen-Kratka in \cite{CGQK} constructed the global large solutions in a more regular space with free normal stress on
the free boundary which is a natural physical situation for the fluids with the surrounding vacuum state.
Under the assumption that $0<\inf_\x\rho_0(\x)< \sup_\x\rho_0(\x) < \infty$, based on an elaborate truncation procedure on the initial data,
Jiang in
\cite{J} proved the global existence of spherically symmetric smooth solutions for (large) initial
data in the domain exterior to a ball
\begin{equation*}
	\{ (\x,t)\in \R^n \times [0,\infty)\,\vcentcolon\, r=\snorm{\x}\ge a  \} \qquad\,\,\,\mbox{in $\mathbb{R}^n$ for $n=2$ or $3$}
\end{equation*}
 with the following boundary condition and the far-field behavior:
\begin{equation*}
\begin{cases}
 u(a,t)=0, \ \ \partial_r e(a,t)=0 \qquad\,\,\,\,  \text{for $t\in[0,T]$},\\[2pt]
\lim\limits_{r\to\infty} ( \rho, u , e , \partial_r e )(r,t) = (1,0,1,0).
\end{cases}
\end{equation*}
A further natural question is whether the uniform estimates in $a>0$ could be obtained so that
the region containing the origin can be included as $a\searrow 0$. In the absence of the vacuum,
Hoff \cite{H2} first constructed the global spherically symmetric weak solutions of the Cauchy problem
for the isothermal flow with large data and strictly positive initial density via the limit
process as $a\searrow 0$ in $\mathbb{R}^n$ for $n \geq 2$;
see also Haspot \cite{Has} for the case of isentropic flow with degenerate viscosity.
For the initial data with finite total initial mass and energy in the whole space,
Jiang-Zhang \cites{JZ,JZ1,JZ2} proved the global existence of spherically symmetric weak solutions of the Cauchy problem
for the isentropic flow in $\mathbb{R}^n$ for $n = 2, 3$ when $\gamma>1$.
Also see Xin-Yuan \cite{XY} for some analysis on the vacuum formation problems of these weak solutions.
Later, when $\inf_\x \rho_0(\x)>0$,
Hoff-Jenssen in \cite{HJ} proved the global existence of spherically symmetric weak solutions of the initial-boundary
value problems of the non-isentropic flow in a finite ball in $\R^3$, which also works for the cylindrically symmetric case.
However, the carefully designed energy estimates and pointwise estimates applied in \cite{HJ} depends strictly on the boundedness
of the domain under the consideration, which seems difficult to adapt the arguments to the corresponding Cauchy problem.
In fact, this Cauchy problem is highly nontrivial due to the strong coupling between the velocity and temperature,
the coordinate singularity at the origin, and infinite initial total mass and energy from the far field.
To the best of our knowledge, in the existing literature,
there have been no such results in this direction for the non-isentropic  compressible flow with infinite energy.
The main aim of this paper is to establish the global-in-time existence of weak solutions of the Cauchy problem
for the full Navier-Stokes equations for M-D compressible heat-conducting flow
in $\mathbb{R}^n$ for $n\ge 2$,
with initial data that are large, discontinuous, spherically symmetric, and away from the vacuum, and to analyze the underlying properties of these weak solutions.

To achieve these, our main strategy is to regard the Cauchy problem as the limit of a series of initial-boundary value problems (IBVPs) that are formulated in finite annular regions and, as the key argument, to establish a set of {\it a-priori} estimates that are uniform with respect to both the inner and outer radii of the annuli considered in the spherically symmetric Lagrangian coordinates. One of the first obstacles in deriving such estimates is the possible degeneracy or singularity of the temporal derivative operators $\rho \partial_t u$ and $\rho \partial_t e$, which are influenced by the behavior of the density. This observation motivates the estimate of pointwise upper and lower bounds of the density; if it can be achieved, then the dissipation structure of the momentum and energy equations can be utilized to gain further information on the solutions.
Such
density bounds
are derived by
exploiting the underlying structure of
the continuity and momentum equations, and
the regional parameters in the derivation must be given careful consideration in order to obtain the density bounds
uniformly with respect to the size of domain.
For this, we face two major difficulties: The first is to ensure that such bounds are independent of
the inner radius $a\in(0,1)$ of the annular region.
To deal with this, we use the crucial observation that the density bounds
can actually be expressed explicitly in terms of the coordinate variables
so that, using this explicit formula, the dependence on $a\in(0,1)$ can be circumvented
by restricting to the spatial regions bounded away from the inner boundary.
The second difficulty is to show that the density bounds
are also uniform over the
outer radius of the annular region, which is resolved by taking the advantages of the specific form
of the entropy inequality for the solutions in the Lagrangian coordinates.
This is
the main motivation for our proof to start by constructing approximate solutions under the Lagrangian formulation of system $\eqref{eqs:SFNS}$,
rather than their Eulerian form. Once the uniform pointwise bounds of the density are established,
it can be used to derive further estimates stemming from the parabolic structure of the momentum and energy equations.
However, suitable spatial cut-off functions have to be introduced in their derivations,
due to the fact that the density bounds described above are restricted on the domain bounded away from the inner boundary.
This leads to certain problematic integrals near the boundary that cannot be estimated with the standard parabolic argument;
we resolve this issue by incorporating the dissipation terms in the entropy inequality.
By overcoming all these difficulties, we obtain the desired domain-independent {\it a-priori} estimates,
which then allow us to construct a global weak solution of the Cauchy problem via the compactness argument.

We now outline the organization of the rest of this paper.
In \S \ref{sec:main}, we first introduce the definition of weak solutions of the Cauchy problem \eqref{eqs:FNS}--\eqref{eqs:CauchyInit}
and then state our main results.
In \S \ref{sec:reform}, the main strategy and the overall scheme of the proof are described.
The Cauchy problem considered can be regarded as the limit of a series of initial-boundary value problems
that are formulated in finite annular regions.
For such approximation problems, in \S \ref{sec:rhobd}, we derive some {\it a-priori} estimates that are independent of
both the inner and outer radii of the annuli considered in the spherically symmetric Lagrangian coordinates.
In \S \ref{sec:WSCEx}, we establish the existence of weak solutions of some exterior problems in the Eulerian coordinates.
Finally, in \S \ref{sec:ato0}, the weak solutions of the original problem are attained via careful compactness arguments
applied to this set of approximate solutions in the Eulerian coordinates.

\section{Global Solutions and Main Theorem}\label{sec:main}
In this section, we present the main theorem of this paper.
For this purpose, we first introduce the following  definition
of global weak solutions of the Cauchy problem  \eqref{eqs:FNS}--\eqref{eqs:CauchyInit}.

\begin{definition}[Weak Solutions]\label{def:WFAV}
The vector function $(\rho,\u,e)(\x,t)$ is called to be a weak solution of
the Cauchy problem \eqref{eqs:FNS}--\eqref{eqs:CauchyInit} in $\R^n \times [0,T]$ if it satisfies the following{\rm :}
\begin{enumerate}[label=(\roman*),ref=(\roman*),font={\normalfont\rmfamily}]
\item There exists a upper semi-continuous function $\ul{r}(t) \vcentcolon [0,T] \to [0,\infty)$ and a constant $C_0=C_0(\rho_0,\u_0,e_0,\gamma,\mu,\lambda,\kappa,n)>0$ such that
$$
\lim\limits_{t\searrow 0}\ul{r}(t)=0, \qquad  0 \le \ul{r}(t)\le C_0 \ \ \text{for all $\, t\in[0,T]$}.
$$
With this, define
$\mathcal{F}\vcentcolon
=\{(\x,t)\in \R^n\times[0,T]\,\vcentcolon\, \snorm{\x}>\ul{r}(t)\,\,  \mbox{for $t\in [0,T]$} \}$ as the fluid region.

\smallskip
\item $\rho\in L_{\textnormal{loc}}^{\infty}(\mathcal{F})$, $\rho(\x,t)=0$
for $(\x,t)\in (\R^n \times [0,T])\backslash \mathcal{F}$ almost everywhere,
$(\u,e)(\x,t)$ is locally H\"older continuous in $\mathcal{F}$, and $\mathring{\mathcal{F}}\vcentcolon=\mathcal{F}\cap\{0<t<T\}$ is an open set.

\smallskip
\item For each $\Phi \in C^1\big([0,T]; C_{\rm c}^1 (\R^n) \big)$,
\begin{align*}
\int_{\R^n} \rho(\x,t) \Phi(\x,t)\, \dif \x - \int_{\R^n}\rho_0(\x) \Phi(\x,0)\,\dif \x
=\int_{0}^{t}\int_{\R^n} (\rho\partial_t\Phi+\rho\u\cdot\nabla\Phi)
\, \dif\x \dif s.
\end{align*}
		
\item $\nabla \u\in L^2_{\textnormal{loc}}(\mathcal{F})$.
Moreover, for any $\Psi\in C^{1}\big( [0,T] ; C^1_{\rm c}(\R^n) \big)$ with
$\supp(\Psi) \Subset \mathcal{F}$ $($i.e., $\supp(\Phi)$ is a compact subset of $\mathcal{F}${\rm )},
\begin{align*}
&\int_{\R^n} (\rho U^i) (\x,t) \Psi(\x,t)\,\dif \x
 - \int_{\R^n} (\rho_0 U_0^i) (\x) \Psi(\x,0)\,\dif\x \\
&\begin{aligned}
=& \int_{0}^{t}\int_{\R^n} \big( \rho U^{i} \partial_t \Psi + \rho U^{i} (\u \cdot \nabla) \Psi
  + (\gamma-1)\rho e \partial_{x^i} \Psi \big)
  \,\dif\x \dif s \\
&- \int_{0}^{t}\int_{\R^n}\big( \mu \nabla \Psi \cdot \nabla U^{i}
  + (\mu+\lambda) \partial_{x^i} \Psi\,\dd\,\u  \big)
  \,\dif \x \dif s \qquad \, \text{ for each $i=1,\dotsc,n$.}
\end{aligned}
\end{align*}
		
\item $\nabla e\in L^2_{\textnormal{loc}}(\mathcal{F})$.
Moreover, for any $\Phi\in C^{1}\big( [0,T] ; C^1_{\rm c}(\R^n) \big)$ with
$\supp(\Phi)  \Subset\mathcal{F}$,
\begin{align*}
&\int_{\R^n} (\rho E)(\x,t) \Phi (\x,t)\,\dif\x - \int_{\R^n} (\rho_0 E_0)(\x) \Phi(\x,0)\,\dif \x \\
&\begin{aligned}
=& \int_{0}^{t}\int_{\R^n}  \big(\rho E \partial_t \Phi
  +(\rho E + P ) (\u \cdot \nabla) \Phi \big)
  \,\dif \x \dif s \\
&-\int_{0}^{t} \int_{\R^n} \Big( \kappa \nabla e + \dfrac{\mu}{2} \nabla \snorm{\u}^2
 + \lambda \u \dd \u + \mu (\nabla \u )\cdot \u \Big) \cdot \nabla \Phi
 \,\dif\x \dif s.
\end{aligned}
\end{align*}
\end{enumerate}
\end{definition}

The main theorem provided below states the global-in-time existence and properties of weak solutions
in $\mathbb{R}^n$ for $n\ge 2$.

\begin{theorem}[Main Theorem]\label{thm:WSAV}
Let $(\rho_0,\u_0,e_0)(\x)=(\rho_0(r),u_0(r)\frac{\x}{r},e_0(r)), r\vcentcolon=|\x|$,
be spherically symmetric initial data defined in $\boldsymbol{x}\in \mathbb{R}^n$ satisfying
\begin{equation}\label{eqs:init}
\begin{split}
&e_0(r)\ge C_*^{-1}, \quad  C_*^{-1} \le \rho_0(r) \le C_* \qquad\, \text{ for all $r\in[0,\infty)$,}\\
&\int_{0}^{\infty}\Big(\rho_0\big(\dfrac{1}{2}\snorm{u_0}^2+ \psi(e_0)\big) + (\gamma-1)G(\rho_0)
+\snorm{(\rho_0-1, u_0^2, e_0-1)}^2\Big)(r)\,r^m \dif r \le C_*,
\end{split}
\end{equation}
for some given constant $C_*>0$, where $G(\zeta)=1-\zeta+\zeta\log \zeta$ and $\psi(\zeta) = \zeta-1-\log \zeta$.
Then, for each $T>0$, there exists a global spherically symmetric weak solution $(\rho,\u,e)(\x,t)$ of the
Cauchy problem \eqref{eqs:FNS}--\eqref{eqs:CauchyInit} in the sense of {\rm Definition \ref{def:WFAV}}
such that the following statements hold\textnormal{:}
\begin{enumerate}[label=(\roman*),ref=(\roman*),font={\normalfont\rmfamily}]
\item\label{item:WSAV1} There exists a continuous increasing function $g\vcentcolon [0,\infty) \to [0,\infty)$ with $g(y)\searrow 0$
as $y\searrow 0$ such that, for each bounded measurable set $E \subset \R^n$,
\begin{align*}
\qquad \esssup\limits_{t\in[0,T]}\int_{\R^n}\big(\frac{1}{2}\rho\snorm{\u}^2 + G(\rho)
\big)(\x,t)\,\dif \x \le C(T),
\quad \esssup\limits_{t\in[0,T]}\int_{E}(\rho e)(\x,t)\,\dif \x
\le C(T)\big(1 + g(|E|)\big),
\end{align*}
where $|E|$ denotes the Lebesgue measure of $E$ and $C(T)=C(T,C_*,\gamma,\mu,\lambda,\kappa,n)>1$ is a constant depending on $T>0$.
Moreover, for $\eta\in (0,1)$,
\begin{align*}
&\int_{0}^{T} \sup_{\snorm{\x}\ge\ul{r}(t)+\eta}\dfrac{\snorm{\u}}{\sqrt{e}}(\x,t)\,\dif t
\le C(T)\big(\eta^{\frac{2-n}{2}} + \eta^{2-n} \big) && \text{if } \ n=2,\, 3,\\
&\int_{0}^{T} \sup_{\snorm{\x}\ge\ul{r}(t)+\eta} \log \big( \max\big\{ 1, e(\x,t)\big\} \big)\, \dif t
\le C(T)\big(1+\sqrt{|\log \eta|}\big) && \text{if } \ n=2,\\
&\int_{0}^{T} \sup_{\snorm{\x}\ge\ul{r}(t)+\eta} \log \big( \max \big\{ 1, e^{\pm 1}(\x,t) \big\} \big)\,\dif t
\le C(T) \eta^{2-n} && \text{if } \ n=3.
\end{align*}
		
\smallskip		
\item\label{item:WSAV2} There exists a continuous function $(y,t)\mapsto \tilde{r}(y,t) \vcentcolon [0,\infty)\times[0,T] \to [0,\infty)$
such that $y\mapsto \tilde{r}(y,t)$ is strictly increasing, and $\ul{r}(t) = \lim_{y\searrow 0} \tilde{r}(y,t)$ for $t\in[0,T]$. Moreover,
there exists a constant $C_0=C_0(C_*,\gamma,\mu,\lambda,\kappa,n)>0$ such that
\begin{equation*}
\begin{dcases*}
\int_{\{\ul{r}(t) < \snorm{\x} < \tilde{r}(y,t) \}} \rho(\x,t)\,\dif \x = y  & for a.e. $(y,t)\in[0,\infty)\times[0,T]$,\\
\tilde{r}(y,t) = \tilde{r}_0(y) + \int_{0}^{t} u(\tilde{r}(y,s),s)\,\dif s & for a.e. $(y,t)\in (0,\infty)\times[0,T]$,\\
n y \psi^{-1}_{-}( \dfrac{C_0}{y}) \le \tilde{r}^{\,n}(y,t) \le C_0(1 +y) & for all $(y,t)\in(0,\infty)\times[0,T]$,
\end{dcases*}
\end{equation*}
where $u(\snorm{\x},t)\vcentcolon= \u(\x,t)\cdot\frac{\x}{\snorm{\x}}$, $\psi^{-1}_{-}$ is
the left inverse of $\psi(\zeta)=\zeta-1-\log \zeta$, and the function $y\mapsto\tilde{r}_0(y) \vcentcolon [0,\infty)\to [0,\infty)$ is implicitly defined as
\begin{equation*}
y= \int_{0}^{\tilde{r}_0(y)} \rho_0(r)\, r^m \dif r \qquad\text{for each $y\in[0,\infty)$.}
\end{equation*}

\item\label{item:WSAV3}
For $\ve\in(0,1]$, there exists a constant $C(\ve)=C(\ve,T,C_*,\gamma,\mu,\lambda,\kappa,n)>0$ such that
\begin{equation*}
\begin{cases*}
\snorm{\tilde{r}(y_1,t)-\tilde{r}(y_2,t)} \le C(\ve) \snorm{y_1-y_2} &\quad for all $(y_1,y_2,t)\in [\ve,\infty)^2\times[0,T]$,  \\
\norm{\tilde{r}(y,t_1)-\tilde{r}(y,t_2)} \le C(\ve) \snorm{t_1^{\frac{3}{4}}-t_2^{\frac{3}{4}}} &\quad for all $(y,t_1,t_2)\in [\ve,\infty)\times[0,T]^2$.
\end{cases*}
\end{equation*}
Moreover, let $\sigma(t)\vcentcolon=\min\{1,t\}$
and $B^{\rm c}(r)\vcentcolon=\{\x\in\R^n\vcentcolon \snorm{\x} \ge r\}$. Then
\begin{equation*}
C^{-1}(\ve)
\le \rho(\x,t) \le C(\ve), \,\,\, \snorm{\u(\x,t)} \le C(\ve)\sigma^{-\frac{1}{4}}(t),
\,\,\, 0<e(\x,t) \le  C(\ve)\sigma^{-\frac{1}{2}}(t)
\end{equation*}
for a.e. $(\x,t)$ such that $t\in[0,T]$ and $\x \in B^{\rm c}(\tilde{r}(\ve,t))$, and
\begin{align*}
&\sup_{ 0\le t \le  T} \int_{B^c(\tilde{r}(\ve,t))} \norm{(\rho-1,\snorm{\u}^2,e-1,\sqrt{\sigma}\nabla \u,\sigma \nabla e)}^2(\x,t)\, \dif \x\\
&\,\, +\int_{0}^{T}\int_{B^c(\tilde{r}(\ve,t))} \norm{(\nabla \u, (\u \cdot \nabla) \u, \nabla e, \sqrt{\sigma}\partial_t \u, \sigma \partial_t e)}^2(\x,t)\,\dif \x \dif t \le C(\ve).
\end{align*}
In addition, for all $t\in(0,T]$ and $\x,\y \in  B^{\rm c}(\tilde{r}(\ve,t))$,
\begin{equation*}
\sigma^{\frac{1}{2}}(t)\snorm{\u(\x,t)-\u(\y,t)} + \sigma(t)\snorm{e(\x,t)-e(\y,t)} \le C(\ve) \snorm{\x-\y}^{\frac{1}{2}},
\end{equation*}
and, for all $0<t_1<t_2\le T$ and $\snorm{\x} \ge \sup_{t_1\le t \le t_2}\tilde{r}(\ve,t)$,
\begin{equation*}
\sigma^{\frac{1}{2}}(t_1)\snorm{\u(\x,t_1)-\u(\x,t_2)} + \sigma(t_1)\snorm{e(\x,t_1)-e(\x,t_2)} \le C(\ve) \snorm{t_2-t_1}^{\frac{1}{4}}.
\end{equation*}
Furthermore, defining
$\rho^{(\ve)} (\x,t) \vcentcolon = \rho(\x,t) \chi_\ve(\x,t)$ with
\begin{equation*}
\chi_\ve(\x,t) \vcentcolon = \begin{dcases*}
1 & if $t\in[0,T]$ and $\x \in B^{c}(\tilde{r}(\ve,t))$,\\ 
0 & otherwise,
\end{dcases*}
\end{equation*}
then, for all $L\in\mathbb{N}$, $\rho^{(\ve)}(\x,t)\in C^0\big([0,T];H^{-1}(B_L)\big)$ and
\begin{equation*}
\big\|\rho^{(\ve)}(\cdot, t_1) - \rho^{(\ve)}(\cdot,t_2)\big\|_{ H^{-1}(B_L)}
\le C(\ve)\norm{t_1-t_2} \qquad \text{ for all $0\le t_1 \le t_2 \le T$,}
\end{equation*}
where $C(\ve)$ is independent of $L\in\mathbb{N}$ as before
and $B_L\vcentcolon=\{ \x\in\R^n \vcentcolon \snorm{\x}<L \}$.
\end{enumerate}
\end{theorem}

\begin{remark}\label{rem:moment}
The domain for which the momentum equations hold weakly can be extended to the entire space-time domain.
However, the viscosity term may be present in the weak form as a non-standard limit distribution
of approximate solutions.
More specifically, let $(\rho,\u,e)(\x,t)$ be a weak solution with initial data $(\rho_0,\u_0,e_0)(\x)$
obtained in {\rm Theorem} \textnormal{\ref{thm:WSAV}}.
Let $\xi \vcentcolon [0,\infty)\to [0,1]$ be a fixed increasing $C^1$--function with
$\xi \equiv0$ on $[0,1]$ and $\xi\equiv1$ on $[2,\infty)$,
and set $\xi^R(r)\vcentcolon= \xi(\frac{r}{R})$.
Then there exists a sequence of vector valued functions
$\{\u_{\alpha}(\x,t)\}_{\alpha\in\mathbb{N}}$ such that
$\u_{\alpha}^{i}\in L^2_{\textnormal{loc}}(\R^n\times[0,T])$
for each $\alpha\in\mathbb{N}$ and $i=1,\dotsc,n$.
Moreover, for each $i=1,\dotsc, n$ and each $\Psi\in C^2\big([0,T];C^2_{\rm c}(\R^n)\big)$
with $\Psi\vcentcolon \R^n\times[0,T] \to \mathbb{R}$, defining
the distribution $\mathbb{M}(\cdot,t) \in \big(C^2\big([0,T];C^2_{\rm c}(\R^n)\big)\big)^{\ast}$ by
\begin{align*}
\mathbb{M}^i(\Psi,t) := \lim\limits_{R\searrow 0} \lim\limits_{\alpha\to\infty} \int_{0}^{t}
\int_{\{\snorm{\x}\ge \alpha^{-1}\}} \big\{(\mu + \lambda) (\u_{\alpha} \cdot \nabla)
\partial_i \Psi^R + \mu \boldsymbol{U}_{\alpha}^i\Delta \Psi^R \big\}\,\dif \x \dif s,
\end{align*}
with $\Psi^R(\x,t)\vcentcolon= \xi^{R}(\snorm{\x})\Psi(\x,t)$,
then
\begin{equation*}
\begin{aligned}
&\int_{\R^n} (\rho U^i)(\x,t) \Psi(\x,t)\,\dif \x - \int_{\R^n} (\rho_0 U_0^i)(\x) \Psi(\x,0)\,\dif \x \\
&= \int_{0}^{t}\int_{\R^n} \big(\rho U^{i} \partial_t \Psi
 + \rho U^{i} (\u\cdot \nabla) \Psi + (\gamma-1)\rho e \partial_{\x^i} \Psi\big)(\x,s)\,\dif \x \dif s
  + \mathbb{M}^i(\Psi,t) \quad\mbox{for each $i=1,\dotsc,n$}.
\end{aligned}
\end{equation*}
 In fact, $\u_{\alpha}(\x,t)$ is the velocity of solution to the exterior sphere problem, which satisfies the auxiliary boundary condition\textnormal{:}
\begin{equation*}
\u_{\alpha}(\x,t) =0, \ \ \nabla e_{\alpha}(\x,t) =0
\qquad \text{for $\snorm{\x}=\alpha^{-1}$ and $t\in [0,T]$.}
\end{equation*}
The existence of such solution $(\rho_{\alpha},\u_{\alpha},e_{\alpha})$ is stated
in {\rm Theorem} \textnormal{\ref{thm:WSCEx}} in {\rm \S}\textnormal{\ref{subsec:eSFNS}}.

The derivation of the above result is obtained by using the uniform entropy estimates on
the approximate solutions $(\rho_\alpha,\u_\alpha,e_{\alpha})_{\alpha\in\mathbb{N}}$
\textnormal{(}\textit{cf}. {\rm Theorem} \textnormal{\ref{thm:WSCEx}\ref{item:WSCEx1};} also see \cite{HJ}\textnormal{)}.
An important followup
is
whether the following equality hold{\rm :}
\begin{equation*}
\mathbb{M}^i(\Psi,t) \overset{?}{=} \int_{0}^{t}
\int_{\R^n} \big\{(\mu + \lambda) \dd \u
\partial_i \Psi + \mu \nabla \u^i \nabla \Psi \big\}\,\dif \x \dif s,
\end{equation*}
which is beyond the scope of this paper and will further be addressed in a separate
paper.
\end{remark}

\section{Main Strategy and Reformulation}\label{sec:reform}
Throughout this paper, we use the following notation for five types of generic constants
that are all larger than $1$ and may be different at different occurrence (without loss of generality)
to indicate their specific
dependence on various parameters:
\begin{equation*}
\begin{aligned}
& C\equiv C(\gamma,\mu,\lambda,\kappa,n), \quad C_0\equiv C_0(C,C_*),
 \quad C(T)\equiv C(T,C_0),\\
& C(a)\equiv C(a,T,C_0), \quad C(\ve)\equiv C(\ve,T,C_0) \,
\qquad \text{ for each $a\in(0,1)$ and $\ve\in(0,1]$},
\end{aligned}
\end{equation*}
where $C_*=C_*(\rho_0,u_0,e_0)>1$ depends only on the initial data as
in \eqref{eqs:init} of Theorem \ref{thm:WSAV}.
We also denote $\beta\vcentcolon=2\mu + \lambda$.

\subsection{Main strategy}\label{subsec:MS}
Our approach for proving Theorem \ref{thm:WSAV} is to split the original problem \eqref{eqs:SFNS}  into the following  two parts:

\smallskip
\begin{enumerate}[label=(\Roman*), ref=(\Roman*)]
\item\label{item:1} For any fixed $a\in(0,1)$ and $T>0$,
we prove the global-in-time existence of a weak solution $(\rho_a,u_a,e_a)(r,t)$
of the initial-boundary value problem (IBVP) of system \eqref{eqs:SFNS}
in the exterior domain $(r,t)\in[a,\infty)\times[0,T]$ with the following initial-boundary conditions:
\begin{equation*}
\begin{cases*}
(\rho_a,u_a,e_a)(r,0) = (\rho_a^0,u_a^0,e_a^0)(r) &\,\, for $r\in[a,\infty)$,\\
u(a,t)=\partial_r e(a,t)=0 &\,\,  for $t\in[0,T]$,
\end{cases*}
\end{equation*}
where $(\rho_a^0,u_a^0,e_a^0)$ is the approximate initial data constructed from
the initial data $(\rho_0,u_0,e_0)$ given in Theorem \ref{thm:WSAV}.
The detailed construction can be found in \S \ref{subsec:mollify}, and $(\rho_a^0,u_a^0,e_a^0)$ is shown to satisfy the
initial condition \eqref{eqs:init} with a constant $C_0>0$ independent of $a\in(0,1)$ in Propositions \ref{prop:mollify}.
Notice that
the corresponding solutions $\{(\rho_a,u_a,e_a)(r,t)\}_{a\in(0,1)}$ satisfy certain high-order
estimates uniformly in $a\in(0,1)$ (see Theorem \ref{thm:WSCEx} in \S \ref{subsec:eSFNS}).

\smallskip
\item We employ the uniform  estimates, independent of $a$ obtained in \ref{item:1},
and make careful compactness arguments to prove that there is a subsequence $a_j \searrow 0$ as $j\to\infty$ such that
the corresponding limit triple $(\rho,u,e)$ is a global weak solution described
in Theorem \ref{thm:WSAV}. \label{item:2}
\end{enumerate}

\smallskip
Part \ref{item:1} of our approach will be achieved via the following six steps:

\begin{enumerate}[label=(I.\arabic*),ref=(I.\arabic*),font={\normalfont\rmfamily}]
\item\label{item:i} We modify the original initial data $(\rho_0,u_{0},e_0)$ into $(\rho_a^0,u_a^{0},e_{a}^0)(r)$ in the Eulerian coordinates, as constructed in \S \ref{subsec:mollify}.

\item\label{item:ii} The solution, $(\rho_a,u_a,e_a)$, in \ref{item:1} is obtained as the limit of solutions of the approximate problems
posed in bounded annular domains. For this purpose, we apply a truncation procedure on the mollified initial data $(\rho_a^0,u_a^0,e_a^0)$
in the Eulerian coordinates near the far field region ({\it i.e.}, $r\to \infty$), which is parameterized by integer $k\in\mathbb{N}$.

\item\label{item:iii} We transform the IBVP in the exterior domain $(r,t)\in[a,\infty)\times[0,T]$
stated in \ref{item:1} into a set of IBVPs shown in \eqref{eqs:LFNS-k}--\eqref{eqs:LFNS-kb} in the Lagrangian coordinates in the spatial interval $[0,k]$. Indeed, the main point of
the above reduction and modification of the initial data is to ensure that, for each $T>0$, there exists a unique global-in-time smooth solution of the approximate IBVP \eqref{eqs:LFNS-k}--\eqref{eqs:LFNS-kb} in
$(x,t)\in[0,k]\times[0,T]$.

\item\label{item:iv} We derive the entropy inequalities on these solutions. Consequently, we obtain the upper and lower bounds of the density.
These density bounds are independent of the approximation parameter $k\in \mathbb{N}$;
in addition, they are independent of $a\in(0,1)$ on the regions away from the origin in the Lagrangian coordinates.

\item\label{item:v} Based on the bounds of the density obtained in \ref{item:iv}, we seek for a set of high-order estimates
on the regions away from the origin, $x=0$, in the Lagrangian coordinates,
which is independent of the approximate parameter $a\in(0,1)$.
In fact, these estimates can be obtained by introducing a set of $C^1$ cut-off functions $g_\ve(x)$,
for each $\ve\in(0,1)$, away from the origin (see \eqref{eqs:g}--\eqref{eqs:grbd}).
However, the inclusion of these cut-off functions leads to the difficulties in estimating certain
integral terms near the boundary: $x=\ve$, which is described in the proof of Lemma \ref{lemma:eL2}.
In order to overcome these difficulties, we make full use of the dissipation terms in the entropy estimates from Step \ref{item:iv}, which can be found in Lemma \ref{lemma:euLinf}.
	
\item\label{item:vi} Finally, we denote $(\tilde{v}_{a,k},\tilde{u}_{a,k},\tilde{e}_{a,k})(x,s)$
(with the specific volume $\tilde{v}_{a,k} := (\tilde{\rho}_{a,k})^{-1}$) as the approximate solution of problem \eqref{eqs:LFNS-k}--\eqref{eqs:LFNS-kb}
in the Lagrangian domain $(x,s)\in[0,k]\times[0,T]$, and $$
\tilde{r}_{a,k}(x,s)=\Big(a^n + n \int_{0}^{x}\tilde{v}_{a,k}(y,s)\,\dif y\Big)^{\frac{1}{n}}
$$
as the corresponding particle path function.
For each $(a,k)\in(0,1)\times\mathbb{N}$, we formulate the coordinate transformation $(r,t)=\mathcal{T}_{a,k}(x,s)=(\mathcal{T}_{a,k}^{(1)},\mathcal{T}_{a,k}^{(2)})(x,s)$ as
\begin{equation*}
\mathcal{T}_{a,k}^{(1)}(x,s)=\tilde{r}_{a,k}(x,s), \quad
\mathcal{T}_{a,k}^{(2)}(x,s)=s \, \qquad \text{ for $(x,s)\in [0,k]\times[0,T]$.}
\end{equation*}
Owing to the density bounds from Step \ref{item:iv}, the inverse function $\mathcal{T}_{a,k}^{-1}$ exists.
By these coordinate transformations, we define
\begin{equation*}
(\overline{v}_{a,k},\overline{u}_{a,k},\overline{e}_{a,k})(r,t) \vcentcolon=( \tilde{v}_{a,k},\tilde{u}_{a,k},\tilde{e}_{a,k}) ( \mathcal{T}_{a,k}^{-1}(r,t)) \
\qquad \text{for $(r,t)\in R_{a,k}$},
\end{equation*}
where $R_{a,k} \vcentcolon= \mathcal{T}_{a,k}\big([0,k]\times [0,T] \big) = \big\{ (r,t)\,\vcentcolon\, t\in[0,T], r\in [a,\tilde{r}_{a,k}(k,t)] \big\}$.
We then extend these functions into the entire exterior Eulerian domain $(r,t)\in[a,\infty)\times[0,T]$ by introducing the following test functions:
Let $\xi \vcentcolon \R \to [0,1]$ be a function with $\xi\in C^{\infty}$ such that $\xi(\zeta)=1$ if $\zeta\le 0$, $\xi(\zeta)=0$
if $\zeta\ge 1$, and $\snorm{\xi^{\prime}(\zeta)} \le 2$ for all $\zeta\in\R$.
With this, we define the following cut-off functions:
\begin{equation*}
\varphi_{a,k}(r,t) = \xi( \dfrac{2r-\tilde{r}_{a,k}(k,t)}{\tilde{r}_{a,k}(k,t)}).
\end{equation*}
It is shown in \S \ref{subsec:Eulext} that $\supp(\varphi_{a,k})\subseteq R_{a,k}$
so that the approximate solutions in the Eulerian coordinates can be extended
to domain $[a,\infty)\times [0,T]$:
\begin{equation}\label{eqs:extfunc}
\qquad\,\, (\rho_{a,k},u_{a,k},e_{a,k})(r,t)=(1,0,1)+\begin{dcases}
\varphi_{a,k}({\overline{v}}^{-1}_{a,k} - 1,
\overline{u}_{a,k},
\overline{e}_{a,k}- 1)(r,t)\quad\, \text{for $(r,t)\in R_{a,k}$,}  \\
(0,0,0)\qquad\qquad\,\,\quad\text{for $(r,t)\in [a,\infty)\times[0,T]\backslash R_{a,k}$,}
\end{dcases}
\end{equation}

With
the uniform {\it a-priori} estimates derived in Steps \ref{item:iv}--\ref{item:v},
there exits
a subsequence $(\rho_{a,k_j},u_{a,k_j},e_{a,k_j})_{j\in\mathbb{N}}$ such that
its limit function $(\rho_a,u_a,e_a)$ is a weak solution of the exterior problem 
stated in \ref{item:1}. Furthermore, it can also be shown that such solutions continue to enjoy
the uniform regularity estimates in $a\in(0,1)$ attained in Steps \ref{item:iv}--\ref{item:v}.
The detailed analysis of this procedure can be found in \S \ref{sec:WSCEx}.
\end{enumerate}

\medskip
Part \ref{item:2} of our approach can be achieved via the following four steps:

\begin{enumerate}[label=(II.\arabic*),ref=(II.\arabic*),font={\normalfont\rmfamily}]
\item\label{item:two1} For given approximate initial data $(\rho_a^0,u_a^0,e_a^0)$ in Part \ref{item:1}, denote
\begin{equation*}
\big\{ (\rho_a, u_a, e_a)(r,t)\vcentcolon [a,\infty) \times [0,T] \to [0,\infty)\times\R\times[0,\infty) \big\}_{a\in(0,1)}
\end{equation*}
as the global weak solutions of the exterior problem
obtained in Part \ref{item:1}. Based on the entropy estimate, we prove the uniform integrability of the approximate
solutions $(\rho_a,u_a,e_a)_{a\in(0,1)}$ with respect to the approximate parameter $a\in(0,1)$ in \S \ref{subsec:meas},
which plays key roles in proving that the limit vector function of these approximate solutions
are indeed a weak solution of the original problem (see Lemmas \ref{lemma:alimRhoWF}--\ref{lemma:eWF}).

\item\label{item:two2} The particle path functions $\tilde{r}_a(x,t)\vcentcolon[0,\infty]\times[0,T]\to [a,\infty)$
associated with the approximate density $\rho_a$ for each $a\in(0,1)$ are defined to be the functions satisfying
\begin{equation*}
x = \int_{a}^{\tilde{r}_{a}(x,t)} \rho_{a}(r,t)\,r^m \dif r \qquad\,\, \text{for {\it a.e.} $(x,t)\in[0,\infty)\times[0,T]$,}
\end{equation*}
which can be found in Lemma \ref{lemma:rkcvg} in \S \ref{subsec:klimpath}.
Then, by the estimates obtained in Step \ref{item:two1} and compactness arguments,
we show in \S \ref{subsec:alimpath} that there exists a function $\tilde{r}(x,t)\vcentcolon [0,\infty)\times[0,T]\to [0,\infty)$ that
is the limit function of sequence $\{\tilde{r}_a(x,t)\}_{a\in(0,1)}$ when $a \searrow 0$.
Moreover, we show that, for $t\in[0,T]$, $\lim_{x\searrow0}\tilde{r}(x,t)$ does not necessarily equal to zero,
but a function of $t$ instead in general: $\ul{r}(t)\vcentcolon [0,T]\to [0,\infty)$.
In fact, $\tilde{r}(x,t)$ is the particle path of our original problem \eqref{eqs:SFNS},
and $\ul{r}(t)$ is the vacuum interface in \eqref{eqs:SFNS} that separates the fluid region
and the possible vacuum, which can be verified later.
Furthermore, it is shown in  Lemma \ref{lemma:alimpath}\ref{item:apath3} that,
for each $x>0$, there exists $\delta(x)>0$  depending only on $x>0$
such that
\begin{equation}\label{eqs:rlower}
\inf_{t\in[0,T]}\tilde{r}(x,t)\ge \delta(x)>0.
\end{equation}
This
is crucial for the next step where the limit, $a\searrow0$, is taken in $\{(\rho_a,u_a,e_a)\}_{a\in(0,1)}$.

\smallskip
\item\label{item:two3} Next, for each $\ve\in(0,1]$, we fix a domain away from the origin by
\begin{equation}\label{eqs:Fve}
\mathcal{F}_\ve\vcentcolon=\big\{(r,t)\in[0,\infty)\times[0,T]\,\vcentcolon\,r\in[\tilde{r}(\ve,t),\infty)\big\},
\end{equation}
where $\tilde{r}(x,t)$ is obtained in Step \ref{item:two2}.
Moreover, we set
$$
\mathcal{F}\vcentcolon=\{(r,t)\in[0,\infty)\times[0,T]\,\vcentcolon\, r> \ul{r}(t)\}.
$$
By the high-order estimates derived in Step \ref{item:v}, a limit function $(\rho,u,e)(r,t)\vcentcolon \mathcal{F}\to [0,\infty)\times\R \times[0,\infty)$ is constructed in \S \ref{subsec:alim}
as follows: By an application of the Arzel\'a-Ascoli theorem, we first obtain $(\rho_\ve,u_\ve,e_\ve)(r,t)$ defined in domain $\mathcal{F}_\ve$ for each $\ve\in(0,1]$,
then construct the vector function $(\rho,u,e)$ by extending $(\rho_\ve,u_\ve,e_\ve)$ with
\begin{equation*}
(\rho_\ve,u_\ve,e_\ve)(r,t) \vcentcolon
= \begin{dcases*}
(0,0,0) & if $(r,t)\in \mathcal{F}\backslash \mathcal{F}_\ve$,\\
(\rho_\ve,u_\ve,e_\ve)(r,t) & if $(r,t)\in\mathcal{F}_\ve$,
\end{dcases*}
\end{equation*}
and finally glue $\big\{(\rho_\ve,u_\ve,e_\ve)(r,t)\big\}_{\ve\in(0,1]}$ together
over parameter $\ve\in(0,1]$ by the telescoping sum:
\begin{equation*}
f = f_{\ve_1} + \sum_{i=1}^{\infty} (f_{\ve_{i+1}} - f_{\ve_i}),
\end{equation*}
for $f= \rho, u, e$, respectively, and $\ve_i\equiv \frac{1}{i}$.
The detail of this procedure is given in Propositions \ref{prop:aaExt}--\ref{prop:scHm1} in Appendix \ref{subsec:AAV}.
In addition, it is also shown in Lemma \ref{lemma:alimrho} that
\begin{equation*}
x = \int_{\ul{r}(t)}^{\tilde{r}(x,t)} \rho(r,t)\,r^m \dif r
\qquad  \text{for {\it a.e.} $(x,t)\in[0,\infty)\times[0,T]$,}
\end{equation*}
which justifies that $\ul{r}(t)$ obtained in Step \ref{item:two2} is the vacuum interface curve of our original problem \eqref{eqs:SFNS}.

\item\label{item:two4} Finally, by the estimates obtained in Step \ref{item:two1} and the convergences in Steps  \ref{item:two2}--\ref{item:two3},
we prove the weak forms of solutions listed in Theorem \ref{thm:WSAV}.
The weak forms of solutions in the Eulerian coordinates $(\x,t)\in \R^n\times[0,T]$ are shown in \S \ref{subsec:alimWF}.
\end{enumerate}

\begin{remark}
Before proceeding, we first give the following remarks{\rm :}
    \begin{enumerate}[label=(\alph*), font=\normalfont]
\item In Steps {\normalfont \ref{item:i}--\ref{item:ii}}, it is important to show that the
modified initial data sequence satisfies the entropy estimates uniformly with respect
to {\normalfont $(a,k)\in(0,1)\times\mathbb{N}$} and strongly
converges to the original initial data {\normalfont$(\rho_0,u_{0},e_0)(r)$}
in proper functional spaces. These results will be needed later in the compactness
arguments in Step {\normalfont \ref{item:vi}}, which is detailed in {\rm \S \ref{sec:WSCEx}}.

\item The main purpose of the reformulation from the Eulerian to the Lagrangian coordinates
in Step {\normalfont\ref{item:iii}} is to obtain the crucial upper and lower bounds of
the density $($see statement {\normalfont\ref{item:prio1}} of {\rm Theorem \ref{thm:priori}}
in  {\rm \S \ref{sec:rhobd})}, which are independent of $k\in\mathbb{N}$.
Note that these bounds of the density are an improvement on the result of Jiang \cite{J}.
		
\item As mentioned in Step {\normalfont \ref{item:v}}, the introduction of 
{\normalfont$C^1$} cut-off functions away from the origin leads to several hurdles in the estimates,
and these issues have also been encountered in Hoff-Jenssen \cite{HJ}.
In this paper, we resolve these issues by making use of the dissipation terms that are present
in the entropy estimate. This is detailed in {\rm Lemma \ref{lemma:euLinf}}
in {\rm \S \ref{subsec:L2total}}.

\item The purpose of the coordinate transformation back into the Eulerian domain,
introduced in Step {\normalfont\ref{item:vi}} is due to the following consideration{\rm :}
As {\normalfont $k\to\infty$}, a difficulty arises due to the nonlinear second-order spatial
differential operators in the Lagrangian forms of the momentum and internal energy equations.
Since these operators are linear in the Eulerian coordinates, this difficulty is resolved by
converting the approximate solutions, originally obtained in the Lagrangian coordinates,
into the functions
in the Eulerian domain. This allows us to take limit {\normalfont $k\to\infty$} and
obtain the desired weak forms in the Eulerian coordinates.

\item In Step {\normalfont\ref{item:vi}}, after taking limit {\normalfont $k\to\infty$},
it requires to show that the limit solution still satisfies the entropy inequality,
since it will be required in Step {\normalfont\ref{item:two1}}.
This is achieved by an application of Mazur's lemma $(${\rm Lemma \ref{lemma:Ent}}$)$
in {\rm \S \ref{subsec:klimEnt}} and {\rm Proposition \ref{prop:mazur}}.
		
\item In order to apply the compactness arguments from the high-order estimates obtained in {\normalfont\ref{item:v}},
it is crucial that $\mathcal{F}_{\ve}$ defined in \eqref{eqs:Fve} is a set strictly bounded away from the
origin, $r=0$, for each $\ve>0$.
We obtain this by assertion \eqref{eqs:rlower} in Step {\normalfont\ref{item:two2}}.
\end{enumerate}
\end{remark}

The rest of this section is organized as follows: In \S \ref{subsec:eSFNS}, we show our main existence
results for
weak solutions with large data of the exterior problem mentioned in Part \ref{item:1}.  In \S \ref{subsec:mollify}-- \S \ref{subsec:ffaprox},
we describe the detailed procedures for the modification and truncation of the initial data listed in Steps \ref{item:ii}--\ref{item:iii} above.
Finally, in \S \ref{subsec:lageqs}, we reformulate the exterior problem in the Eulerian coordinates into the approximate Lagrangian problems
in domain $(x,t)\in[0,k]\times[0,T]$ as described in Step \ref{item:iv},
and establish the corresponding global-in-time well-posedness of the unique strong solution with regular initial data.

\subsection{Mollification of the initial data in the Eulerian domain}\label{subsec:mollify}
Let $(\rho_0,u_0,e_0)(r)$, $r\in[0,\infty)$, be the initial data in Theorem \ref{thm:WSAV}.
First, we extend the initial data to
$\mathbb{R}$
by
\begin{equation}
(\hat{\rho}_0,\hat{u}_0,\hat{e}_0)(\zeta) \vcentcolon=\begin{dcases*}
(\rho_0(\zeta),u_0(\zeta),e_0(\zeta)) & if $\zeta\in[0,\infty)$,\\[1mm]
(\rho_0(-\zeta),-u_0(-\zeta),e_0(-\zeta)) & if $\zeta\in(-\infty,0)$.
\end{dcases*}
\end{equation}
Then, for each $a\in(0,1)$, let $j_{a}\in C_{\rm c}^{\infty}([-1,1])$ be
the standard 1-D mollifiers such that $\supp(j_a)\subseteq [-\frac{a}{2},\frac{a}{2}]$. We define
$$
(\hat{\rho}_{a}^0,\hat{u}_{a}^0,\hat{e}_{a}^0) \vcentcolon= ( \hat{\rho}_0 \ast j_{a}, \hat{u}_0\ast j_{a}, \hat{e}_0\ast j_{a}).
$$
It follows that, for some $C_*>0$, independent of $a\in (0,1)$,
\begin{equation}\label{prftemp0}
\begin{split}
&(\hat{\rho}_{a}^0,\hat{u}_{a}^0,\hat{e}_{a}^0 )\in C^{\infty}(\R),\\
&\hat{e}_{a}^0(\zeta)\ge C_*^{-1}, \,\,\,\,
  C_*^{-1}\le \hat{\rho}_{a}^0(\zeta) \le C_* \qquad\
   \text{for {\it a.e.} $\zeta\in\R$,}\\
&(\hat{\rho}_{a}^0,\hat{u}_{a}^0,\hat{e}_{a}^0)(\zeta) \to (\rho_0, u_0, e_0)(r)
\qquad\quad\ \text{as $a\searrow 0$ for {\it a.e.} $\zeta\in(0,\infty)$.}
\end{split}
\end{equation}

Next, let $\chi\vcentcolon[0,\infty)\to [0,1]$ be such that
$\chi\in C^{\infty}$, $\chi(\zeta)=0$ if $\zeta\le1$, and $\chi(\zeta)=1$ if $\zeta\ge 2$.
Set the cut-off functions:
$$
\chi_a(r)=\chi(\frac{r}{a}) \qquad\mbox{for each $a\in(0,1)$}.
$$
It follows that $\chi_a(r)=0$ for $r\le a$ and $\chi_a(r)=1$ for $r\ge 2a$.
With this, we define the approximate exterior initial data:
\begin{equation}\label{eqs:amolint}
(\rho_a^0, u_a^0, e_a^0)(r)
\vcentcolon= \big( (\hat{\rho}_a^0-1)\chi_a +1,\, \hat{u}_a^0\chi_a,\,
(\hat{e}_a^0-1)\chi_a+1 \big)(r) \qquad \text{for all $r\in (0,\infty)$.}
\end{equation}

\begin{proposition}\label{prop:mollify}
The approximate exterior initial data $(\rho_a^0, u_a^0, e_a^0)(r)$ satisfy
\begin{enumerate}[label=(\roman*), ref=(\roman*), font={\textnormal}]
\item\label{item:mollify1} $(\rho_{a}^0, u_{a}^0, e_{a}^0)\in C^{\infty}([0,\infty))$
and $u_{a}^0(a)= \partial_r e_{a}^0(a)=0$.

\item\label{item:mollify2} As $a\searrow 0$,
\begin{equation}\label{temp:mollify1}
\begin{split}
&\sbnorm{(\rho_{a}^0-1,\snorm{ u_{a}^0}^2, e_{a}^0-1) - (\rho_0-1,\snorm{u_0}^2, e_0-1) }_{L^2([0,\infty),r^m \dif r)} \to 0,\\
&(\rho_{a}^0, u_{a}^0, e_{a}^0)(r) \to ( \rho_0, u_0, e_0 )(r) \qquad\ \text{for a.e. $r\in[0,\infty)$.}
\end{split}
\end{equation}
\item\label{item:mollify3} There exists some constant $C_0>0$ independent of $a\in (0,1)$ such that
\begin{equation}\label{temp:mollify2}
\begin{split}
&\,\, e_a^0(r)\ge C_0^{-1}, \ \ \, C_0^{-1}\le \rho_{a}^0(r) \le C_0 \qquad\
\text{ for all $(a,r)\in(0,1)\times [0,\infty)$,}\\
&\sup\limits_{a\in(0,1)}\int_{a}^{\infty}\Big(\rho_{a}^0\big(\frac{1}{2}\snorm{u_{a}^0}^2 + \psi(e_{a}^0)\big)
   +G(\rho_a^0)  + \snorm{(\rho_a^0-1, |u_a^0|^2, e_a^0-1)}^2 \Big)(r)\,r^m \dif r \le C_0.
\end{split}
\end{equation}
\end{enumerate}
\end{proposition}

\begin{proof}
Statements \ref{item:mollify1} and
$C_*^{-1}\le \rho_{a}^0(r) \le C_*$
for all $r\in[0,\infty)$ follow immediately from construction \eqref{eqs:amolint},
while the almost everywhere convergence $\eqref{temp:mollify1}_2$ follows from \eqref{prftemp0}.

Next, by the Minkowski
inequality and \eqref{eqs:amolint}, we have
\begin{align*}
\sbnorm{u_a^0}_{L^4([0,\infty),r^m \dif r)}
\le \int_{-\frac{a}{2}}^{\frac{a}{2}} \Big( \int_{0}^{\infty} \snorm{ \chi_a(r) \hat{u}_0(r-\zeta) j_a(\zeta) }^4
  \,r^m \dif r \Big)^{\frac{1}{4}} \dif \zeta
  \le  \int_{-\frac{a}{2}}^{\frac{a}{2}} j_a(\zeta) \Big(\int_{a}^{\infty}
 \snorm{u_0(r-\zeta)}^4\,r^m \dif r\Big)^{\frac{1}{4}} \dif \zeta.
\end{align*}
Let $\xi\equiv r- \zeta$.
Then $0 \le  1 + \frac{\zeta}{\xi} \le 2$ for all $(\zeta,\xi)\in [-\frac{a}{2},\frac{a}{2}]\times [\frac{a}{2},\infty)$.
It follows that
\begin{align*}
\sbnorm{u_a^0}_{L^4([0,\infty),r^m \dif r)}
\le \int_{-\frac{a}{2}}^{\frac{a}{2}} j_a(\zeta)
 \Big( \int_{a-\zeta}^{\infty} \snorm{u_0(\xi)}^4 \xi^m ( 1+ \frac{\zeta}{\xi})^m \dif \xi \Big)^{\frac{1}{4}}
   \dif \zeta
\le 2^{\frac{m}{4}}C_* \int_{-\frac{a}{2}}^{\frac{a}{2}} j_a(\zeta) \dif \zeta = 2^{\frac{m}{4}}C_*.
\end{align*}
Applying the same argument to $(\rho_a^0-1,e_a^0-1)$ yields the third term of $\eqref{temp:mollify2}_2$.
	
Next, by the H\"older inequality,
\begin{align*}
\sbnorm{\snorm{u_a^0}^2-\snorm{u_0}^2}_{L^2([0,\infty),r^m\dif r)}\le (1 + 2^{\frac{m}{4}}) C_*
\sbnorm{u_a^0-u_0}_{L^4([0,\infty),r^m \dif r)}.
\end{align*}
From this, we now show $\sbnorm{u_a^0-u_0}_{L^4([0,\infty),r^m \dif r)}\to 0$ as $a\searrow 0$.
By construction \eqref{eqs:amolint},
\begin{align*}
\sbnorm{u_a^0-u_0}_{L^4([0,\infty),r^m \dif r)}
&\le \sbnorm{\chi_a(\hat{u}_a^0-u_0)}_{L^4([0,\infty),r^m \dif r)}
+ \Big( \int_{0}^{2a} (1-\chi_a)^4 \snorm{u_0}^4(r)\,r^m \dif r \Big)^{\frac{1}{4}}.
\end{align*}
Using the fact that $\sbnorm{u_0}_{L^4([0,\infty),r^m \dif r)}\le C_*$, $\chi_a \to 1$ as $a\searrow 0$ almost everywhere,
$0\le \chi_a\le 1$, and the dominated convergence theorem, we have
\begin{align*}
\lim\limits_{a\searrow 0}\int_{0}^{2a}
(1-\chi_a)^4 \snorm{u_0}^4(r)\,r^m \dif r =0.
\end{align*}
Thus,
we conclude that
\begin{align*}
\lim\limits_{a\searrow 0}\sbnorm{\snorm{u_a^0}^2-\snorm{u_0}^2}_{L^2([0,\infty),r^m\dif r)}
\le ( 1 + 2^{\frac{m}{4}} ) C_*
\lim\limits_{a\searrow 0}\sbnorm{\chi_a(\hat{u}_a^0-u_0)}_{L^4([0,\infty),r^m \dif r)}.
\end{align*}
Therefore, it is left to prove the limit:
$\sbnorm{\chi_a(\hat{u}_a^0-u_0)}_{L^4([0,\infty),r^m \dif r)} \to 0$ as $a\searrow 0$.
To show this, it follows from the Minkowski
inequality that
\begin{align*}
\sbnorm{\chi_a(\hat{u}_a^0-u_0)}_{L^4([0,\infty),r^m \dif r)}
&\le \int_{-\frac{a}{2}}^{\frac{a}{2}} j_a(\zeta)
\Big( \int_{a}^{\infty} \snorm{u_0(r-\zeta)-u_0(r)}^4\, r^m \dif r \Big)^{\frac{1}{4}} \dif \zeta\\
&\le  2^{\frac{m}{4}}\int_{-\frac{a}{2}}^{\frac{a}{2}} j_a(\zeta)
\Big( \int_{a-\zeta}^{\infty} \snorm{u_0(\xi)-u_0(\xi+\zeta)}^4\,\xi^m \dif \xi \Big)^{\frac{1}{4}}\dif\zeta,
\end{align*}
where we have used that $0 \le  1 + \frac{\zeta}{\xi}\le 2$
if $\zeta\in[-\frac{a}{2},\frac{a}{2}]$ and $\xi\ge a-\zeta$.
Now, by the separability of
space
$L^4( [0,\infty),r^m \dif r )$, for each $\delta>0$, there exists $h_\delta\in C_{\rm c}^{\infty}([0,\infty))$ such that
$\sbnorm{h_\delta-\hat{u}_0}_{L^4([0,\infty),r^m \dif r)}\le 3^{-\frac{m}{4}} \frac{\delta}{4}$. 
By the triangle inequality,
it follows that
\begin{align*}
2^{-\frac{m}{4}}\sbnorm{\chi_a(\hat{u}_a^0-u_0)}_{L^4([0,\infty),r^m \dif r)}
&\le \int_{-\frac{a}{2}}^{\frac{a}{2}}j_a(\zeta) \Big( \int_{\frac{a}{2}}^{\infty} \snorm{u_0(\xi)-h_\delta(\xi)}^4
 \,\xi^m \dif \xi \Big)^{\frac{1}{4}}  \dif \zeta\\
&\quad + \int_{-\frac{a}{2}}^{\frac{a}{2}}j_a(\zeta) \Big( \int_{\frac{a}{2}}^{\infty}
\snorm{h_\delta(\xi)-h_\delta(\xi+\zeta)}^4\,\xi^m \dif \xi \Big)^{\frac{1}{4}} \dif \zeta \\
&\quad + \int_{-\frac{a}{2}}^{\frac{a}{2}} j_a(\zeta) \Big( \int_{a-\zeta}^{\infty} \snorm{h_\delta(\xi+\zeta)-u_0(\xi+\zeta)}^4\,\xi^m \dif \xi \Big)^{\frac{1}{4}} \dif \zeta\\
&=\vcentcolon \sum_{i=1}^3 I_i.
\end{align*}
First, $I_1$ and $I_3$ are estimated as:
\begin{equation*}
\begin{aligned}
I_1\vcentcolon
=& \int_{-\frac{a}{2}}^{\frac{a}{2}} j_a(\zeta) \Big( \int_{\frac{a}{2}}^{\infty}
  \snorm{u_0(\xi)-h_\delta(\xi)}^4\,\xi^m \dif \xi \Big)^{\frac{1}{4}} \dif \zeta
  \le \dfrac{\delta}{4} \int_{-\frac{a}{2}}^{\frac{a}{2}} j_a(\zeta)\,\dif \zeta = \dfrac{\delta}{4},\\
I_3
\le&\, 3^{\frac{m}{4}} \int_{-\frac{a}{2}}^{\frac{a}{2}} j_a(\zeta)
\Big( \int_{a}^{\infty} \snorm{h_\delta(r)-u_0(r)}^4\,r^m  \dif r \Big)^{\frac{1}{4}} \dif \zeta
\le \dfrac{\delta}{4} \int_{-\frac{a}{2}}^{\frac{a}{2}} j_a(\zeta)\,\dif \zeta = \dfrac{\delta}{4},
\end{aligned}
\end{equation*}
where we have used the fact that $\frac{1}{2} \le 1 - \frac{\zeta}{r} \le \frac{3}{2}$
for all $(\zeta,r)\in [-\frac{a}{2},\frac{a}{2}]\times [a,\infty)$.
Since $h_\delta$ is compactly supported for each $\delta>0$,
there exists $a_{\delta}\in(0,1)$ such that
\begin{align*}
I_2 \vcentcolon= \int_{-\frac{a}{2}}^{\frac{a}{2}} j_a(\zeta)
\Big( \int_{\frac{a}{2}}^{\infty} \snorm{h_\delta(\xi)-h_\delta(\xi+\zeta)}^4\,\xi^m \dif \xi \Big)^{\frac{1}{4}} \dif \zeta
\le \dfrac{\delta}{2}\int_{-\frac{a}{2}}^{\frac{a}{2}} j_a\,\dif \zeta
= \dfrac{\delta}{2}
\qquad \text{ if $a\in(0,a_{\delta})$.}
\end{align*}
Thus, for given $\delta>0$, there exists $a_\delta\in(0,1)$ such that $2^{-\frac{m}{4}}\sbnorm{\chi_a(\hat{u}_a^0-u_0)}_{L^4([0,\infty),r^m \dif r)} \le\delta$ for all $a\in(0,a_\delta)$.
This implies that
$\lim_{a\searrow 0}\sbnorm{\snorm{u_a^0}^2-\snorm{u_0}^2}_{L^2([0,\infty),r^m\dif r)}=0$.
The proof for $\rho_a^0$ and $e_a^0$ follows from the same argument.

Next, by the Jensen inequality and the fact
that $G(\zeta), \psi(\zeta)$, and $\snorm{\zeta}^2$ are convex functions of $\zeta$,
it follows that
$G(\rho_{a}^0)\le G(\hat{\rho}_0) \ast j_{a}$, $\psi(e_{a}^0)\le \psi (\hat{e}_0) \ast j_{a}$,
and $\snorm{u_{a}^0} \le \snorm{\hat{u}_0}^2\ast j_{a}$ for $r\in(0,\infty)$.
Using Fubini-Tonelli's theorem, it follows that
\begin{align*}
\int_{a}^{\infty} G(\rho_a^0)(r)\,r^m \dif r
= \int_{-\frac{a}{2}}^{\frac{a}{2}} \Big(\int_{a-\zeta}^{\infty} G(\rho_0)(\xi) (\zeta+\xi)^m \dif \xi\Big)
  j_{a}(\zeta)\, \dif \zeta
\le 2^m \int_{-\frac{a}{2}}^{\frac{a}{2}} C_* j_{a}(\zeta)\, \dif \zeta = 2^m C_*,
\end{align*}
where we have used the initial condition $\eqref{eqs:init}_2$. In a similar way, we have
\begin{align*}
\int_{a}^{\infty} \rho_{a}^0 \psi(e_{a}^0) (r) r^m \dif r
\le 2^m C_* \int_{-\frac{a}{2}}^{\frac{a}{2}}
\Big(\int_{0}^{\infty} \rho_0\psi(e_0) (\xi) \xi^m  \dif \xi\Big) j_a(\zeta) \,\dif \zeta \le 2^m C_*^3.
\end{align*}
Similarly, we also obtain
$$
\sbnorm{\rho_a^0 \snorm{u_a^0}^2}_{L^1([0,\infty),r^m\dif r)}
\le 2^m C_*^3.
$$
Setting $C_0= 2^m C_*^3$,
we complete the proof.
\end{proof}

\subsection{The approximation problem in the exterior Eulerian domain}\label{subsec:eSFNS}
To resolve the coordinate singularity at the origin, which poses a significant difficulty
for the existence of solutions, we first consider the approximation problem
in the exterior domain: $\Omega_{a,T}\vcentcolon=[a,\infty)\times[0,T]$ for fixed $a\in(0,1)$,
and find a solution $(\rho_a,u_a,e_a)$ of the exterior problem for system \eqref{eqs:SFNS} over
domain $\Omega_{a,T}$ with the following initial-boundary conditions:
\begin{equation}\label{eqs:eSFNS}
\begin{cases}
u_a(a,t)=0, \quad \partial_r e_a(a,t)=0 \qquad & \text{for $t\in[0,T]$},\\[1mm]
(\rho_a,u_a,e_a)(r,0)=(\rho_a^0,u_a^0,e_a^0)(r) \qquad & \text{for $r\in[a,\infty)$}.
\end{cases}
\end{equation}

Now the boundary conditions in \eqref{eqs:eSFNS}
can be interpreted physically as the presence of an insulating ball at the origin,
with no slip-boundary condition for the velocity field and gradient field of the internal energy. Moreover,
the initial data
are given by $(\rho_a^0,u_a^{0},e_{a}^0)(r)$ constructed in \S \ref{subsec:mollify} above.

To prove Theorem \ref{thm:WSAV}, the existence of weak solutions of
the exterior problem \eqref{eqs:SFNS} and \eqref{eqs:eSFNS} for each $a\in(0,1)$ is first obtained
so that some compactness arguments
can be applied when $a\searrow 0$.

\begin{definition}[Spherically Symmetric Weak Solutions in the Exterior Domain]\label{def:SWeak}
Let
\begin{align*}
\mathcal{D}^a\vcentcolon
=&\big\{ \phi\in C^{\infty}\big([a,\infty)\times[0,T]\big)\,\vcentcolon\, \exists N>a\, \ \text{such that} \ \phi(r,t)=0 \
  \text{for $r\ge N$}\big\},\\[1mm]
\mathcal{D}_0^a \vcentcolon
=&\big\{ \phi\in \mathcal{D}^a\,\vcentcolon\, \phi(a,t)=0 \ \text{ for all $t\in[0,T]$} \big\}.
\end{align*}
Then $(\rho_a,u_a,e_a)(r,t)$ is a weak solution of problem \eqref{eqs:SFNS} and \eqref{eqs:eSFNS}, 
provided that
\begin{enumerate}[label=(\roman*),ref=(\roman*),font={\normalfont\rmfamily}]
\item\label{item:SWeak1} For each $\phi\in \mathcal{D}^a$,
\begin{align*}
&\int_{a}^{\infty} \rho_a(r,t) \phi(r,t)\,r^m \dif r
  - \int_{a}^{\infty} \rho_a^0(r) \phi(r,0)\,r^m \dif r
=\int_{0}^{t}\int_{a}^{\infty} \big(\rho_a \partial_t \phi + \rho_a u_a \partial_r \phi\big)\,r^m \dif r \dif s.
\end{align*}

\item\label{item:SWeak2} For each $\phi\in \mathcal{D}^a_0$,
\begin{align*}
&\int_{a}^{\infty} (\rho_a u_a)\phi(r,t)\,r^m \dif r
- \int_{a}^{\infty} (\rho_a^0 u_a^0)(r) \phi(r,0)\,r^m \dif r
- \int_{0}^{t}\int_{a}^{\infty} \rho_a u_a\big(\partial_t \phi + u_a \partial_r \phi\big)\,r^m \dif r \dif s \\
&= \int_{0}^{t}\int_{a}^{\infty} \Big( P_a - \beta \big( \partial_r u_a + m \dfrac{u_a}{r}\big)\Big) \big( \partial_r \phi + m\dfrac{\phi}{r}\big)\,r^m \dif r \dif s,
\end{align*}
where  $\beta\vcentcolon= 2\mu + \lambda$.

\smallskip
\item\label{item:SWeak3} For each $\phi\in \mathcal{D}^a$, 
\begin{align*}
&\int_{a}^{\infty} (\rho_a E_a)(r,t) \phi (r,t)\,r^m \dif r
 - \int_{a}^{\infty} (\rho_a^0 E_a^0)(r) \phi(r,0)\,r^m \dif r\\
&= \int_{0}^{t}\int_{a}^{\infty}  \big\{ \rho_a E_a \partial_t \phi +(\rho_a E_a + P_a ) u_a \partial_r \phi \big\}
\,r^m  \dif r \dif s\\
&\quad -\int_{0}^{t} \int_{a}^{\infty} \Big\{  2\mu u_a \partial_r u_a +\lambda u_a \big( \partial_r u_a + m \dfrac{u_a}{r} \big)
 + \kappa  \partial_r e_a \Big\} \partial_r\phi
 \,r^m \dif r \dif s,
\end{align*}
where $P_a\vcentcolon= (\gamma-1) \rho_a e_a$, $E_a\vcentcolon=\frac{1}{2}\snorm{u_a}^2 + e_a$, and $E_a^0\vcentcolon= \frac{1}{2}\snorm{u_a^0}^2 + e_a^0$.
\end{enumerate}
\end{definition}

From now on, we denote $\sigma=\sigma(t)\vcentcolon= \min\{1,t\}$.
Moreover, for any function $y(t)\vcentcolon [0,T] \to [0,\infty)$, we denote
\begin{equation*}
\begin{split}
\mathcal{E}[\rho,u,e;y](T)
\vcentcolon=&\sup_{ t\in [0,T]} \int_{y(t)}^{\infty}
 \snorm{( \rho-1, u^2 , e-1, \sqrt{\sigma} \partial_r u,\sigma \partial_r e)}^2(r,t)\,r^m \dif r\\
&+ \int_{0}^{T}\int_{y(t)}^{\infty}
\snorm{(\partial_r u, \partial_r e, u \partial_r u,\sqrt{\sigma} \partial_t u, \sigma\partial_t e)}^2
\,r^m\dif r \dif t.
\end{split}
\end{equation*}

\begin{theorem}\label{thm:WSCEx}
For fixed $a\in(0,1)$, let $(\rho_a^0,u_a^0,e_a^0)(r)$ be the initial data constructed in \textnormal{Proposition \ref{prop:mollify}}. Then, for each $T>0$, there exists a weak solution
$(\rho_a,u_a,e_a)(r,t)$ of the exterior problem \eqref{eqs:SFNS} and \eqref{eqs:eSFNS}.
Moreover, the weak solution $(\rho_a,u_a,e_a)(r,t)$ satisfies the following properties{\rm :}
\begin{enumerate}[label=(\roman*),ref=(\roman*),font={\normalfont\rmfamily}]
\item\label{item:WSCEx1} The entropy inequality holds\textnormal{:}
\begin{equation*}
\begin{aligned}
\quad \ &\esssup\limits_{t\in[0,T]}\int_{a}^{\infty}
\Big(\rho_a\big(\frac{1}{2}\snorm{u_a}^2+ \psi(e_a)\big)+ (\gamma-1)G(\rho_a) \Big)(r,t)\,r^m \dif r
+ \kappa\int_{0}^{T}\int_{a}^{\infty}\dfrac{\snorm{\partial_r e_a}^2}{ e_a^2}\,r^m \dif r\dif t \\
&\,\,+ \int_{0}^{T}\int_{a}^{\infty} \Big\{\big(\dfrac{2\mu}{n}+\lambda \big)\dfrac{\snorm{\partial_r u_a+m \frac{u_a}{r}}^2}{e_a} + \dfrac{2m\mu}{n} \dfrac{\snorm{\partial_r u_a-\frac{u_a}{r}}^2}{e_a}\Big\}\,r^m\dif r\dif t \le C(T).
\end{aligned}
\end{equation*}

\smallskip
\item\label{item:WSCEx2}
$\mathcal{E}[\rho_a,u_a,e_a;a](T)\le C(a)$ and, for a.e. $(r,t)\in [a,\infty)\times[0,T]$,
\begin{equation*}
\qquad\quad \ C^{-1}(a) \le \rho_a(r,t) \le C(a),
\ \snorm{u_a(r,t)} \le C(a)\sigma^{-\frac{1}{4}}(t), \ C^{-1}(a) \le e_a(r,t) \le C(a)\sigma^{-\frac{1}{2}}(t).
\end{equation*}

\item\label{item:WSCEx3} There exists a continuous function $(x,t)\mapsto \tilde{r}_a(x,t) \vcentcolon [0,\infty)\times[0,T] \to [a,\infty)$
such that $x\mapsto \tilde{r}_a(x,t)$ is strictly increasing for each $t\in[0,T]$ and
\begin{equation*}
\begin{dcases*}
\int_{a}^{\tilde{r}_a(x,t)} \rho_a(r,t)\,r^m \dif r =x  &\quad for a.e. $(x,t)\in[0,\infty)\times[0,T]$,\\
\tilde{r}_a(x,t) = \tilde{r}_a^0(x) + \int_{0}^{t} u_a(\tilde{r}_a(x,s),s)\,\dif s &\quad for all $(x,t)\in[0,\infty)\times[0,T]$,\\
n x \psi^{-1}_{-}(\dfrac{C_0}{x})
\le ({\tilde{r}}_a)^{n}(x,t) \le C_0( 1 +x ) &\quad for all $(x,t)\in[0,\infty)\times[0,T]$,
\end{dcases*}
\end{equation*}
where the function $x\mapsto\tilde{r}_a^0(x) \vcentcolon [0,\infty)\to [a,\infty)$ is implicitly defined as
\begin{equation*}
x= \int_{a}^{\tilde{r}_a^0(x)} \rho_a^0(r)\,r^m \dif r \qquad
\text{for each $x\in[0,\infty)$.}
\end{equation*}

\smallskip		
\item\label{item:WSCEx4} For each $\ve\in(0,1]$,
$\mathcal{E}[\rho_a,u_a,e_a;\tilde{r}_a(\ve,\cdot)](T) \le C(\ve)$ and
\begin{equation*}
\begin{dcases}
\snorm{\tilde{r}_a(x_1,t)-\tilde{r}_a(x_2,t)} \le C(\ve) \snorm{x_1-x_2} &\quad \text{for $(x_1,x_2,t)\in [\ve,\infty)^2\times[0,T]$},  \\
\snorm{\tilde{r}_a(x,t_1)-\tilde{r}_a(x,t_2)} \le C(\ve) \snorm{t_1^{\frac{3}{4}}-t_2^{\frac{3}{4}}}\,\, &\quad \text{for $(x,t_1,t_2)\in [\ve,\infty)\times[0,T]^2$}.
\end{dcases}
\end{equation*}
Furthermore, for a.e. $(r,t)$ such that $t\in[0,T]$ and $r\in [\tilde{r}_a(\ve,t),\infty)${\rm :}
\begin{equation*}
C^{-1}(\ve) \le \rho_a(r,t) \le C(\ve),\quad \snorm{u_a(r,t)} \le C(\ve)\sigma^{-\frac{1}{4}}(t),
\quad  e_a(r,t) \le C(\ve)\sigma^{-\frac{1}{2}}(t).
\end{equation*}
In addition,
for all $r_1,\,r_2\in [\tilde{r}_a(\ve,t),\infty)$ and $t\in[0,T]$,
\begin{equation*}
\sigma^{\frac{1}{2}}(t)\snorm{u_a(r_1,t)-u_a(r_2,t)} + \sigma(t)\snorm{e_a(r_1,t)-e_a(r_2,t)} \le C(\ve)\snorm{r_1-r_2}^{\frac{1}{2}},
\end{equation*}
and, for all $(r,t_1,t_2)$ such that $0<t_1<t_2\le T$
and $r\ge \sup_{t_1\le t \le t_2}\tilde{r}_a(\ve,t)$, 
\begin{equation*}
\sigma^{\frac{1}{2}}(t_1)\snorm{u_a(r,t_1)-u_a(r,t_2)} + \sigma(t_1) \snorm{e_a(r,t_1)-e_a(r,t_2)}\le C(\ve)\snorm{t_2-t_1}^{\frac{1}{4}}.
\end{equation*}
		
\item\label{item:WSCEx5} For each connected interval $I\subseteq [0,\infty)$, denote $\tilde{H}_0^1(I,r^m\dif r)$ as the closure of
\begin{align*}
\quad\,\, \, \mathcal{D}_0(I)\vcentcolon =\big\{ \phi\in C^{\infty}(I)\vcentcolon\,\exists N>0
\ \text{such that $[0,N]\subset I$ and $\phi(r)=0$ for $r\in I\cap[N,\infty)$}\big\}
\end{align*}
via the $H^1(I,r^m\dif r)$--norm, and denote its dual space as $\tilde{H}^{-1}( I, r^m \dif r)$.
For each $\ve>0$, define $\rho_a^{(\ve)}(r,t) \vcentcolon = \rho_a(r,t) \chi_{\ve}^a(r,t)$
with the indicator function $\chi_{\ve}^a(r,t)${\rm :}
\begin{equation*}
\chi_{\ve}^a(r,t) \vcentcolon = \begin{dcases*}
1 & if $t\in[0,T]$ and $r\in [\tilde{r}_a(\ve,t),\infty)$,\\
0 & otherwise.
\end{dcases*}
\end{equation*}
Then, for all
$L\in\mathbb{N}$, $\rho_{a}^{(\ve)}\in C^{0}\big([0,T];\tilde{H}^{-1}([0,L],r^m \dif r) \big)$ and
\begin{equation*}
\sbnorm{\rho_a^{(\ve)}(\cdot,t_1)-\rho_a^{(\ve)}(\cdot,t_2)}_{\tilde{H}^{-1}([0,L],r^m\dif r)} \le C(\ve)\snorm{t_1-t_2} \qquad\, \text{for all $t_1, t_2 \in [0,T]$.}
\end{equation*}

\smallskip
\item\label{item:WSCEx6} For each $\eta\in (a,1)$, the following estimates hold\textnormal{:}
\begin{equation*}
\qquad \ \begin{dcases*}
\int_{0}^{T} \sup_{r\ge\eta}\dfrac{\snorm{u_a}}{\sqrt{e_a}}(r,t)\,\dif t
  \le C(T)\big(\eta^{\frac{2-n}{2}}+ \eta^{2-n}\big) & if $n=2$, $3$,\\
\int_{0}^{T} \sup_{r\ge \eta} \log \big( \max\big\{ 1, e_a(r,t)\big\} \big)\, \dif t
  \le C(T)\big( 1 + \sqrt{|\log\eta|}
  \big) & if $n=2$,\\
\int_{0}^{T} \sup_{r\ge\eta} \log \big( \max \big\{ 1, e_a^{\pm 1}(r,t) \big\} \big)\,\dif t
  \le C(T) \eta^{2-n} & if $n=3$.
\end{dcases*}
\end{equation*}
\end{enumerate}
\end{theorem}

\begin{remark}
For the smooth initial data $(\rho_a^0,u_a^0,e_a^0)$ constructed
in {\rm \S \ref{subsec:mollify}}, it can be shown that there exists a unique global-in-time classical
solution of problem \eqref{eqs:SFNS} and \eqref{eqs:eSFNS} \textnormal{(}see \cite{J}\textnormal{)}. In this paper, since the higher regularity estimates for establishing  classical solutions will not be needed
for the proof of the main theorem, we omit them for simplicity.
\end{remark}

\subsection{Truncation of the initial data in the exterior Eulerian domain}\label{subsec:ffaprox}
Since $(\rho_a,u_a,e_a)$ in Theorem \ref{thm:WSCEx} is obtained as the limit of solutions
of the approximate problems posed in bounded annular domains, we need to construct the corresponding approximate initial data, which is the main purpose of this subsection.
In particular, we construct $(\rho_{a,k}^0,u_{a,k}^0,e_{a,k}^0)$ for $k\in\mathbb{N}$ by truncating $(\rho_a^{0},u_a^0,e_a^0)$ with suitable cut-off functions $\varphi_{a,k}^0$.

We first consider the function:
$r \mapsto\, \int_{a}^{r} \rho_a^0(\zeta)\,\zeta^m \dif \zeta$,
which is a strictly increasing continuous function since $C_0^{-1} \le \rho_a^0(r) \le C_0$.
Hence, the inverse function $\tilde{r}_a^0(x)\vcentcolon [0,\infty)\to [a,\infty)$ exists so that
\begin{equation}\label{eqs:intPath}
	x = \int_{a}^{\tilde{r}_a^0(x)} \rho_a^0(r)\,r^m \dif r\, \qquad \text{for each $x\in[0,\infty)$.}
\end{equation}
Moreover, due to the regularity of $\rho_a^0$, it follows from the inverse function theorem
that $\tilde{r}_a^0(x) \vcentcolon [0,\infty)\to [a,\infty)$ is a smooth, bijective, strictly increasing map.

From now on, we use $\D_x$ to represent the derivative with respect to $x\in[0,\infty)$.

\begin{proposition}\label{prop:r0diff}
For all $x\in [0,\infty)$,
\begin{equation}\label{3,2a}
\begin{aligned}
&\D_x \tilde{r}_a^0(x) = \dfrac{1}{(\tilde{r}_a^0)^m(x)\rho_a^0(\tilde{r}_a^0)(x)},\qquad
\ \tilde{r}_a^0(x) = \Big(a^n + n \int_{0}^{x} \dfrac{1}{\rho_a^0(\tilde{r}_a^0)(y)} \dif y\Big)^{\frac{1}{n}}.
\end{aligned}
\end{equation}
\end{proposition}

\begin{proof}
Since $\rho_a^0(r)\in C^{\infty}$ and $C_0^{-1}\le \rho_a^0(r)\le C_0$ by construction \eqref{eqs:amolint}, it follows that
\begin{align*}
\lim\limits_{r\to \tilde{r}_a^0(x)}\dfrac{1}{r-\tilde{r}_a^0(x)}\int_{\tilde{r}_a^0(x)}^{r}
\rho_a^0(\zeta)\,\zeta^m \dif \zeta
=(\tilde{r}_a^0)^m(x)\rho_a^0(\tilde{r}_a^0(x))
\qquad \text{for each $x\in[0,\infty)$.}
\end{align*}
Since
$y\mapsto \tilde{r}_a^0(y)$ is continuous, then, for
all $x\in[0,\infty)$,
\begin{align*}
\lim\limits_{y\to x}\dfrac{1}{\tilde{r}_a^0(y)-\tilde{r}_a^0(x)}\int_{\tilde{r}_a^0(x)}^{\tilde{r}_a^0(y)} \rho_a^0(\zeta)\,\zeta^m \dif \zeta
=(\tilde{r}_a^0)^m(x) \rho_a^0(\tilde{r}_a^0)(x).
\end{align*}
Notice that, by \eqref{eqs:intPath}, 
\begin{align*}
\dfrac{\tilde{r}_a^0(y)-\tilde{r}_a^0(x)}{y-x}
\Big(\dfrac{1}{\tilde{r}_a^0(y)-\tilde{r}_a^0(x)}
\int_{\tilde{r}_a^0(x)}^{\tilde{r}_a^0(y)}\rho_a^0(r)\,r^m \dif r\Big)
= \dfrac{1}{y-x} \int_{\tilde{r}_a^0(x)}^{\tilde{r}_a^0(y)}\rho_a^0(r)\,r^m \dif r = 1.
\end{align*}
Then, using the fact that
$0 < C_0^{-1} a^m \le  \rho_0(\tilde{r}_0)(y)\tilde{r}_0^{\,m}(y)
\le  C_0 \tilde{r}_0^{\,m}(y) < \infty$ for all $y\in[0,\infty)$, it follows that, for all $x\in[0,\infty)$,
\begin{align*}
\lim\limits_{y\to\ x} \dfrac{\tilde{r}_a^0(y)-\tilde{r}_a^0(x)}{y-x}
= \lim\limits_{y\to x} \Big( \dfrac{1}{\tilde{r}_a^0(y)
-\tilde{r}_a^0(x)} \int_{\tilde{r}_a^0(x)}^{\tilde{r}_a^0(y)}\rho_a^0(r)r^m \dif r\Big)^{-1}
= \dfrac{1}{(\tilde{r}_a^0)^m(x)\rho_a^0(\tilde{r}_a^0)(x)}.
\end{align*}
Thus, the derivative of $\tilde{r}_a^0(x)$ exists and $\eqref{3,2a}_1$ holds.

Next, it follows by the chain rule that
\begin{align*}
\dfrac{1}{n} \D_x ( \tilde{r}_a^0)^n(x)
=  (\tilde{r}_a^0)^m(x) \D_x \tilde{r}_a^0(x) = \frac{1}{\rho_a^0(\tilde{r}_a^0)(x)} \qquad \text{for $x\in[0,\infty)$.}
\end{align*}
Integrating in  $y\in[0,x]$ and using $\tilde{r}_a^0(0)=a$, we obtain $\eqref{3,2a}_2$.
The proof is complete.
\end{proof}

Now, let $\chi \vcentcolon \R \to [0,1]$ be such that $\chi\in C^{\infty}$, $\chi(\zeta)=1$ if $\zeta\le 0$, $\chi(\zeta)=0$ if $\zeta\ge 1$, and $\snorm{\chi^{\prime}(\zeta)}\le 2$ for all $\zeta\in\R$.
With this, the cut-off functions $\varphi_{a,k}^0$ are defined as
\begin{equation*}
\varphi_{a,k}^0(r) = \chi\big(\dfrac{2r-\tilde{r}_a^{0}(k)}{\tilde{r}_a^{0}(k)}\big) \qquad\,\, \text{for each $(a,k)\in(0,1)\times\mathbb{N}$.}
\end{equation*}
It follows that $\varphi_{a,k}^0(r)\in C^{\infty}([a,\infty))$ and
\begin{equation*}
\varphi_{a,k}^0(r)=1 \ \ \text{for $r\in[a,\frac{\tilde{r}_a^0(k)}{2}]$,} \qquad\,\,\,  \varphi_{a,k}^0(r)=0 \ \ \text{for $r\in[\tilde{r}_a^0(k),\infty)$.}
\end{equation*}
Using this, we define the truncated initial data:
\begin{equation}\label{def:appInitr1}
(\rho_{a,k}^0,u_{a,k}^0,e_{a,k}^0)(r)
\vcentcolon = ((\rho_a^0-1)\varphi_{a,k}^0+1,\, u_{a}^0\varphi_{a,k}^0, \,(e_a^0-1)\varphi_{a,k}^0+1)(r)
\quad \text{ for $r\in [a,\infty)$.}
\end{equation}
With above construction, it can be verified that
\begin{equation}\label{eqs:akInitProp}
\begin{split}
&( \rho_{a,k}^0, u_{a,k}^0, e_{a,k}^0 ) \in C^{\infty}([a,\infty)),
 \qquad\ u_{a,k}^0(r) = \partial_r e_{a,k}^0(r) = 0  \,\,\,\, \text{for $r\in [a, \tilde{r}_a^0(k)]$},\\[1mm]
&(\rho_{a,k}^0, u_{a,k}^0, e_{a,k}^0)(r) =
\begin{dcases*}
(\rho_a^0, u_a^0, e_a^0)(r) \quad & for all $r \in [a,\frac{\tilde{r}_a^0(k)}{2})$,\\
(1,0,1) & for all $r \in [\tilde{r}_a^0(k),\infty)$.
\end{dcases*}
\end{split}
\end{equation}

\begin{proposition}\label{prop:kapprox}
For $G(z)=1-z+z\log z$ and $\psi(z)=z-1-\log z$,
then
following hold:
\begin{equation*}
\begin{split}
& C_0^{-1} \le  \rho_{a,k}^0(r) \le C_0, \quad   e_{a,k}^0(r)\ge C_0^{-1}\qquad \ \text{ for all $(r,k)\in[a,\infty)\times\mathbb{N}$ and $a\in(0,1)$,}\\
&\sup_{\substack{a\in(0,1)\\ k\in\mathbb{N}}} \int_{a}^{\infty}\!\!
\Big\{\rho_{a,k}^0\big(\frac{1}{2}\snorm{u_{a,k}^0}^2 +\psi (e_{a,k}^0)\big) + G(\rho_{a,k}^0) + \big|(\rho_{a,k}^0-1, |u_{a,k}^0|^2, e_{a,k}^0-1)\big|^2\Big\}\,r^m \dif r \le C_0.
\end{split}
\end{equation*}
\end{proposition}

\begin{proof}
By Proposition \ref{prop:mollify}, we see that
\begin{align*}
\int_{a}^{\infty} \Big\{ \snorm{\rho_{a,k}^0-1}^2 + \snorm{u_{a,k}^0}^4
  +\snorm{e_{a,k}^0-1}^2 \Big\}\,r^m \dif r
\le \int_{a}^{\infty} \Big\{ \snorm{\rho_a^0-1}^2 + \snorm{u_a^0}^4 +\snorm{e_a^0-1}^2 \Big\}\,r^m\dif r\le C_0,
\end{align*}
and
$C_0^{-1}\le \rho_{a,k}^0(r) = \rho_a^0(r) \varphi_{a,k}^0(r) + 1-\varphi_{a,k}^0(r) \le C_0$
for all $r\in[a,\infty)$.
By the same argument, we can show that $e_{a,k}^0(r)\ge C_0^{-1}$ for all $r\in[a,\infty)$.
	
Since $G$ and $\psi$ are convex, $G(1)=\psi(1)=0$, and $0\le \varphi_{a,k}^0\le 1$,
it follows that
\begin{align*}
G(\rho_{a,k}^0) =&\, G(\rho_a^0 \varphi_{a,k}^0 + 1-\varphi_{a,k}^0)\le\varphi_{a,k}^0 G(\rho_a^0)+(1-\varphi_{a,k}^0)G(1)=\varphi_{a,k}^0 G(\rho_a^0),\\[1mm]
\psi(e_{a,k}^0) =&\, \psi( e_a^0\varphi_{a,k}^0 + 1-\varphi_{a,k}^0 ) \le \varphi_{a,k}^0 \psi(e_a^0) + (1-\varphi_{a,k}^0) \psi(1) = \varphi_{a,k}^0 \psi(e_a^0).
\end{align*}
Then, using Proposition \ref{prop:mollify}, we have
\begin{equation*}
\int_{a}^{\infty} \Big\{ \rho_{a,k}^0\big(\frac{1}{2}\snorm{u_{a,k}^0}^2+\psi(e_{a,k}^0)\big)+ G(\rho_{a,k}^0)\Big\}\,r^m \dif r
\le C_0\int_{a}^{\infty}\Big\{\rho_a^0\big(\dfrac{1}{2}\snorm{u_a^0}^2+ \psi(e_a^0)\big) +G(\rho_a^0)\Big\}\,r^m\dif r \le 	C_0.
\end{equation*}
This completes the proof.
\end{proof}

\subsection{Reformulation in the bounded Lagrangian domain}\label{subsec:lageqs}
For each $(a,k)\in(0,1)\times\mathbb{N}$, we transform $(\rho_{a,k}^0,u_{a,k}^0,e_{a,k}^0)(r)$ into the Lagrangian domain as
\begin{equation}\label{def:appInit}
(\tilde{v}_{a,k}^0,\tilde{u}_{a,k}^0,\tilde{e}_{a,k}^0)(x)\vcentcolon= \big((\rho_{a,k}^0)^{-1},u_{a,k}^0,e_{a,k}^0\big)(\tilde{r}_a^0(x))\qquad\, \text{for $x\in [0,\infty)$.}
\end{equation}
Since $\tilde{r}_a^0(x)\in C^{\infty}([0,\infty))$ due to $\rho_{a}^0 \in C^{\infty}([a,\infty))$, Proposition \ref{prop:r0diff} and \eqref{eqs:akInitProp} imply that
\begin{equation*}
(\tilde{v}_{a,k}^0,\tilde{u}_{a,k}^0,\tilde{e}_{a,k}^0)\in C^{\infty}([0,\infty)),
\, \,\, \tilde{u}_{a,k}^0 (x)\big\vert_{x=0,\,k} = \D_x \tilde{e}_{a,k}^0 (x)\big\vert_{x=0,\,k} = 0
\qquad \text{ for $(a,k) \in (0,1)\times\mathbb{N}$.}
\end{equation*}

\begin{proposition}\label{prop:kapproxL}
The following uniform estimate holds{\rm :}
\begin{equation*}
\sup_{a\in(0,1)}\sup_{k\in\mathbb{N}}\int_{0}^{k} \Big(\dfrac{1}{2}\snorm{\tilde{u}_{a,k}^0}^2
+ \psi(\tilde{e}_{a,k}^0) +  \psi(\tilde{v}_{a,k}^0)
+ \snorm{(\tilde{v}_{a,k}^0-1, |\tilde{u}_{a,k}^0|^2, \tilde{e}_{a,k}^0-1)}^2\Big)(x)\,\dif x \le C_0.
\end{equation*}
\end{proposition}

\begin{proof}
First, it follows from Proposition \ref{prop:r0diff} that the map: $x\mapsto \tilde{r}_a^0(x)$ satisfies
\begin{equation*}
0<C_0^{-1}\, (\tilde{r}_a^0)^{-m}(x)\le \snorm{\D_x \tilde{r}_a^0(x)}
= (\tilde{r}_a^0)^{-m}(x)\, (\rho_a^0)^{-1}(\tilde{r}_a^0(x)) \le C_0 a^{-m}<\infty \qquad
\text{for $x\in[0,\infty)$,}
\end{equation*}
so that the Jacobian of transformation $r=\tilde{r}_a^0(x)$ is bounded.
Then, for any function $f(r)\vcentcolon [a,\infty)\to \mathbb{R}$ such that $f\in L^1_{\text{loc}}$, the following coordinate transformation is satisfied:
\begin{align*}
\int_{a}^{\tilde{r}_a^0(k)} f(r) \rho_a^0(r)\,r^m \dif r
= \int_{0}^{k} f(\tilde{r}_a^0)(x)\,\dif x.
\end{align*}
Using this,  \eqref{eqs:akInitProp}, and Propositions \ref{prop:mollify}--\ref{prop:kapprox},
it follows that, for each $(a,k)\in(0,1)\times\mathbb{N}$,
\begin{align*}
\int_{0}^{k} \snorm{\tilde{v}_{a,k}^0-1}^2 (x)\,\dif x
= \int_{0}^{k} \norm{1-\rho_{a,k}^0\big(\tilde{r}_a^0(x)\big)}^2 \norm{\rho_{a,k}^0\big(\tilde{r}_a^0(x)\big)}^{-2}\dif x
\le C_0\int_{a}^{\infty}\norm{\rho_a^0(r)-1}^2 r^m \dif r \le C_0.
\end{align*}
By the same argument,
$\sbnorm{(\snorm{\tilde{u}_{a,k}^0}^2,\tilde{e}_{a,k}^0-1)}_{L^2([a,\infty),r^m\dif r)}\le C_0$.
	
Next, by Propositions \ref{prop:mollify}--\ref{prop:kapprox}
and the identity: $z\psi(z^{-1})= 1-z+z\log z = G(z)$, we have
\begin{align*}
&\int_{0}^{k} \Big(\frac{1}{2}\snorm{\tilde{u}_{a,k}^0}^2+\psi(\tilde{e}_{a,k}^0)+ (\gamma-1)\psi(\tilde{v}_{a,k}^0)\Big)(x)\,\dif x\\
&= \int_{a}^{\infty} \dfrac{\rho_a^0}{\rho_{a,k}^0}
\Big( \rho_{a,k}^0\big(\frac{1}{2}\snorm{u_{a,k}^0}^2+ \psi (e_{a,k}^0)\big)
 + (\gamma-1) G(\rho_{a,k}^0) \Big)(r)\,r^m \dif r \le C_0
\end{align*}
for all $(a,k)\in(0,1)\times\mathbb{N}$. This completes the proof.
\end{proof}

The vector function $(\tilde{v}_{a,k},\tilde{u}_{a,k},\tilde{e}_{a,k})$ is called a solution of the $(a,k)$-approximate IBVP in the Lagrangian coordinates,
provided that  $(\tilde{v}_{a,k},\tilde{u}_{a,k},\tilde{e}_{a,k})$ satisfies the equations:
\begin{equation}\label{eqs:LFNS-k}
\begin{dcases}
\D_t \tilde{v}_{a,k} - \D_x(\tilde{r}_{a,k}^m \tilde{u}_{a,k}) =0,\\
\D_t \tilde{u}_{a,k} + \tilde{r}_{a,k}^m \D_x p(\tilde{v}_{a,k},\tilde{e}_{a,k}) = \beta \tilde{r}_{a,k}^m \D_x \Big(\dfrac{\D_x (\tilde{r}_{a,k}^m \tilde{u}_{a,k})}{\tilde{v}_{a,k}} \Big),\\
\D_t \tilde{e}_{a,k} +p(\tilde{v}_{a,k},\tilde{e}_{a,k})\D_x(\tilde{r}_{a,k}^m \tilde{u}_{a,k}) = \mathcal{G}_{a,k}, \\
\tilde{r}_{a,k}(x,t) = \Big(a^n + n\displaystyle\int_{0}^x \tilde{v}_{a,k}(y,t)\dif y\Big)^{\frac{1}{n}}
\end{dcases}
\end{equation}
in $[0,k]\times[0,T]$ with $p(v,e)\vcentcolon=\frac{(\gamma-1)e}{v}$, $\beta\vcentcolon=2\mu + \lambda$, and
\begin{equation*}
    \mathcal{G}_{a,k} \vcentcolon= \beta \dfrac{\snorm{\D_x(\tilde{r}_{a,k}^m \tilde{u}_{a,k})}^2}{\tilde{v}_{a,k}} - 2m\mu \D_x (\tilde{r}_{a,k}^{m-1}\tilde{u}_{a,k}^2) + \kappa \D_x\Big( \dfrac{r_{a,k}^{2m} \D_x \tilde{e}_{a,k}}{\tilde{v}_{a,k}} \Big),
\end{equation*}
and the initial-boundary conditions:
\begin{equation}\label{eqs:LFNS-kb}
\begin{dcases}	
\tilde{u}_{a,k}(x,t)\big\vert_{x=0,\,k}=\D_x \tilde{e}_{a,k}(x,t)\big\vert_{x=0,\,k}=0 & \text{for $t\in[0,T]$},\\
(\tilde{v}_{a,k},\tilde{u}_{a,k},\tilde{e}_{a,k})(x,0)
= (\tilde{v}_{a,k}^0,\tilde{u}_{a,k}^0,\tilde{e}_{a,k}^0)(x)\qquad
&\text{for $x\in [0,k]$},
\end{dcases}
\end{equation}

\begin{remark}
The equations in \eqref{eqs:LFNS-k} are formally derived
from system \eqref{eqs:SFNS} and \eqref{eqs:eSFNS} as follows{\rm :} for any smooth solution $(\rho,u,e)$ of \eqref{eqs:SFNS} and \eqref{eqs:eSFNS}, the transformation
from the Eulerian coordinates $(r,t)$ to the Lagrangian coordinates $(x,t)$ is defined as
\begin{equation}
x(r,t)\vcentcolon = \int_{a}^{r} \rho(s,t)\,s^m \dif s, \quad t(r,t)\vcentcolon = t
\qquad \text{ for $(r,t)\in [a,\infty)\times[0,T]$}.
\end{equation}
The continuity equation $\eqref{eqs:SFNS}_1$ implies that
$\partial_t x(r,t) = -\rho u r^m$.
In addition, by the implicit function theorem,
on a neighbourhood where $\rho>0$, there is a function $r(x,t)$ such that
\begin{equation}\label{eqs:impl}
\int_a^{r(x,t)}\rho(\zeta,t)\,\zeta^m \dif \zeta = x.
\end{equation}
Denote $(v,u,e)(x,t):=(\rho^{-1}(r(x,t),t),u(r(x,t),t),e(r(x,t),t))$
and $(\D_t, \D_x)$ as the derivative with respect to the Lagrangian coordinates.
Then it follows by taking derivative $\D_t$ on $\eqref{eqs:impl}$ and the Leibniz theorem
that $\D_t r(x,t) = u(x,t)$. These relations can be summarized as
\begin{equation}\label{eqs:jacob}
\dfrac{\partial (x,t)}{\partial (r,t)} =
\begin{pmatrix}
			\rho r^m & -\rho u r^m \\[1mm]
			0        & 1
		\end{pmatrix}, \qquad
		\dfrac{\partial (r,t)}{\partial (x,t)} =
		\begin{pmatrix}
			v r^{-m} &  u  \\[1mm]
			0        & 1
\end{pmatrix}.
\end{equation}
Furthermore, from the continuity equation $\eqref{eqs:SFNS}_1$,
\begin{equation}\label{eqs:kinetic}
		r(x,t) = r_0(x) + \int_{0}^{t} u(x,\tau)\,\dif\tau,
		\qquad
		r^n(x,t) = a^n + n\int_{0}^x v(y,t)\,\dif y,
\end{equation}		
where $\displaystyle r_0(x) \vcentcolon=\big( a^n + n\int_{0}^{x} \rho_0^{-1}(r(y))\,\dif y\big)^{\frac{1}{n}}$. Equations $\eqref{eqs:LFNS-k}_1$--$\eqref{eqs:LFNS-k}_3$ are then obtained by translating \eqref{eqs:SFNS} using the coordinate transformation relations \eqref{eqs:jacob}.
\end{remark}

By virtue of the well-known global existence theorem for problem \eqref{eqs:LFNS-k}--\eqref{eqs:LFNS-kb}, the initial data constructed in \eqref{def:appInit} implies the following theorem (see \cites{J,YB,YB2}):

\begin{theorem}\label{thm:localt}
Given initial data $(\tilde{v}_{a,k}^0,\tilde{u}_{a,k}^0,\tilde{e}_{a,k}^0)(x)$ for $x\in[0,k]$,
then there exists a unique classical solution
$(v_{a,k},u_{a,k},e_{a,k})(x,t)$ of problem \eqref{eqs:LFNS-k}--\eqref{eqs:LFNS-kb} in domain $(x,t)\in[0,k]\times[0,\infty)$
satisfying
\begin{equation*}
v_{a,k}(x,t)\in C^{1+\alpha,1+\frac{\alpha}{2}}\big([0,k]\times[0,\infty)\big),
\quad (u_{a,k},e_{a,k})(x,t)\in C^{2+\alpha,1+\frac{\alpha}{2}}\big([0,k]\times[0,\infty)\big)
\end{equation*}
for some $\alpha\in(0,1)$.
\end{theorem}

By the above theorem, there exists a unique global classical solution of \eqref{eqs:LFNS-k}--\eqref{eqs:LFNS-kb} for each fixed $(a,k)\in(0,1)\times\mathbb{N}$. This allows us to derive the {\it a-priori} estimates in \S 4 below.

\section{The A-Priori Estimates}\label{sec:rhobd}

This section is devoted to the derivation of the {\it a-priori} estimates for solutions
of the approximation problem \eqref{eqs:LFNS-k}--\eqref{eqs:LFNS-kb}.
For simplicity, we suppress the approximation parameters $(a,k)\in(0,1)\times\mathbb{N}$
and denote
$(v,u,e,r)\equiv (\tilde{v}_{a,k},\tilde{u}_{a,k},\tilde{e}_{a,k},\tilde{r}_{a,k})$  as the solution obtained from Theorem \ref{thm:localt} in $[0,k]\times[0,T]$ for each $T>0$.
First, we state the main theorem of this section.

\begin{theorem}\label{thm:priori}
The following estimates are satisfied{\rm :}
\begin{enumerate}[label=(\roman*), ref=(\roman*),font={\normalfont\rmfamily}]
\item\label{item:prio0}
The entropy estimate holds with $\psi(\zeta):=\zeta-1-\log\zeta${\rm :}
\begin{align}\label{eqs:entropy}
&\sup\limits_{t\in[0,T]}\int_{0}^{k}
\Big(\dfrac{1}{2}\snorm{u}^2 + \psi(e)+(\gamma-1)\psi(v)\Big)(x,t)\,\dif x
  + \kappa\int_{0}^{T}\int_{0}^{k} \dfrac{r^{2m}\snorm{\D_x e}^2}{v e^2}\,\dif x\dif t\nonumber\\
&+ \int_{0}^{T}\int_{0}^{k} \Big\{\Big(\lambda + \dfrac{2\mu}{n}\Big)\dfrac{\snorm{\D_x (r^m u)}^2}{ve}+2m\mu \dfrac{v}{e}\Big(\dfrac{\D_x\big(r^m u\big)}{v\sqrt{n}}-\sqrt{n}\dfrac{u}{r}\Big)^2\Big\}\,\dif x \dif t
\le C_0.
\end{align}

\item\label{item:prio1}
For each given constant $\ve\ge0$,
\begin{equation}\label{eqs:vulb}
\ul{v}(\ve,t,a) \le v(x,t) \le \ov{v}(\ve,t,a)  \qquad
\text{for all $(x,t) \in [\ve,\infty)\times[0,T)$},
\end{equation}
with $\ov{v}=\ov{v}(z,t,a)$ and $\ul{v}=\ul{v}(z,t,a)$ explicitly given by
\begin{equation}\label{vbarname}
\begin{dcases*}
\ov{v}=C_0 (1+t)^2 f(z,a) \exp\big\{ C_0 t f(z,a) + C_0 (1+t) h(z,a) f(z,a)  \exp\{C_0 t f(z,a)\} \big\},\\
\ul{v}=C_0 (1+t)^{-1}f^{-1}(z,a)\exp\{ -t C_0 f^2(z,a) \}\Gamma(z,t,a),
\end{dcases*}
\end{equation}
where $h(z,a)$, $f(z,a)$, and $\Gamma(z,t,a)$ are defined by
\begin{equation}\label{fhF}
\begin{aligned}
& h(z,a)\vcentcolon=\big(a^n+nz\psi_{-}^{-1}(\frac{C_0}{z})\big)^{-\frac{2m}{n}},
\quad  f(z,a)\vcentcolon= \exp\big\{C_0h(z,a)\big\}, \\
& \Gamma(z,t,a)\vcentcolon = \exp\big\{-\dfrac{mC_0t}{\beta} h^{\frac{n}{2m}}(z,a)\big\},
\end{aligned}
\end{equation}
with $\psi_{-}^{-1}(y)\vcentcolon [0,\infty)\to (0,1]$ as
the left
inverse of
$\psi$.
In particular, for each $\ve>0$,
\begin{equation}\label{4.34.3}
\begin{dcases*}
C^{-1}(T) \le v(x,t) \le C(T) \quad &  for all $(x,t) \in [1,k]\times[0,T)$,\\
C^{-1}(\ve) \le v(x,t) \le C(\ve) & for all $(x,t) \in [\ve,k]\times[0,T]$,\\
C^{-1}(a) \le v(x,t) \le C(a) & for all $(x,t) \in [0,k]\times[0,T]$.
\end{dcases*}
\end{equation}

\item\label{item:prio2}
For each $a\in (0, 1),\, \mathcal{L}_0 [v,u,e](T)\le C(a)$ and, for $(x,t)\in [0,k]\times[0,T]$,
\begin{equation*}
\qquad \ \ C^{-1}(a) \le v(x,t) \le C(a),\, \,\, \snorm{u(x,t)}\le C(a)\sigma^{-\frac{1}{4}}(t),
\,\,\, C^{-1}(a) \le e(x,t) \le C(a)\sigma^{-\frac{1}{2}}(t);
\end{equation*}
and, for each $\ve\in(0,1]$, $\mathcal{L}_y [v,u,e](T)\le C(\ve)$ and
\begin{equation*}
\qquad C^{-1}(\ve) \le v(x,t) \le C(\ve), \,\,\,
\sigma(t)^{\frac{1}{4}}\snorm{u(x,t)}+\sigma^{\frac{1}{2}}(t) e(x,t) \le C(\ve)
\quad\mbox{for $(x,t)\in[\ve,k]\times[0,T]$},
\end{equation*}
where $\sigma(t)\vcentcolon= \min\{1,t\}$ and, for $y\in[0,1)$,
$\mathcal{L}_{y}[v,u,e](T)$ is defined by
\begin{equation*}
\begin{split}
\mathcal{L}_{y} [v,u,e](T)\vcentcolon
=&\sup_{ t\in [0,T]} \int_{y}^{k} \norm{\big( v-1, u^2 , e-1, \sqrt{\sigma}r^m \D_x u,\sigma r^m \D_x e\big)}^2(x,t)\,\dif x\\
&+ \int_{0}^{T}\int_{y}^{k} \norm{\big(r^m\D_x u, r^m \D_x e, r^m u \D_x u,\sqrt{\sigma} \D_t u, \sigma \D_t e\big)}^2
\,\dif x \dif t.
\end{split}
\end{equation*}
\end{enumerate}
\end{theorem}

\smallskip
\subsection{Entropy estimate}\label{subsec:entEst}
We now start the proof of Theorem \ref{thm:priori} with the derivation of entropy estimate stated in \ref{item:prio0}, which is motivated from the second law of thermodynamics. It encapsulates the dissipative effect of viscosity and thermal diffusion.

\begin{proof}[Proof of {\rm \ref{item:prio0}} in {\rm Theorem \ref{thm:priori}}]
	It follows from  $\eqref{eqs:LFNS-k}_1 $--$ \eqref{eqs:LFNS-k}_3$ and direct calculations that
\begin{align}\label{temp2}
&\dfrac{\dif}{\dif t}\int_{0}^{k}
\Big(\frac{1}{2}\snorm{u}^2+ \psi(e)+(\gamma-1)\psi(v)\Big)\,\dif x\nonumber\\
&= - \beta \int_{0}^{k} \dfrac{\snorm{\D_x(r^m u)}^2}{ev}\dif x
  + 2m\mu\int_{0}^{k} \dfrac{1}{e} \D_x(r^{m-1}u^2)\,\dif x
   - \kappa \int_{0}^{k} \dfrac{r^{2m}\snorm{\D_x e}^2}{v e^2}\dif x.
\end{align}
According to $\eqref{eqs:LFNS-k}_1$ and $\eqref{eqs:LFNS-k}_4$,
we can obtain
\begin{equation*}
2m\mu  \dfrac{1}{e} \D_x(r^{m-1}u^2) - \beta\dfrac{\snorm{\D_x(r^m u)}^2}{ev}
=  -\Big(\dfrac{2\mu}{n}+\lambda\Big)\dfrac{\snorm{\D_x (r^m u)}^2}{ve} - 2m\mu \dfrac{v}{e}\Big(\dfrac{\D_x(r^m u)}{v\sqrt{n}}-\dfrac{u\sqrt{n}}{r}\Big)^2,
\end{equation*}
which, along with \eqref{temp2}, yields \eqref{eqs:entropy}.
\end{proof}

\smallskip
\subsection{Upper and lower bounds of the density}\label{subsec:rhobd}
Now we aim to prove \eqref{eqs:vulb}--\eqref{4.34.3}.
By \eqref{eqs:entropy}, we first obtain some upper and lower bounds of the particle path function $r(x,t)$.
\begin{lemma}
\label{lemma:rbd} Let $\psi_{-}^{-1}(\cdot)$
be the left branch inverse of $\psi$. Then
\begin{equation}\label{eqs:rbd}
	a^n + nx\psi^{-1}_{-}(\dfrac{C_0}{x}) \le  r^n(x,t) \le  C_0(1+x)
	\qquad \text{ for $(x,t)\in [0,k]\times[0,\infty)$}.
	\end{equation}
\end{lemma}

\begin{proof}
Since $\psi(\zeta)\vcentcolon= \zeta-\log \zeta-1$ is strictly convex in $(0,\infty)$ and $\psi(1)=0$,
then there exists the left inverse function $\psi_{-}^{-1}\vcentcolon [0,\infty)\to (0,1]$
and the right inverse function $\psi_{+}^{-1}\vcentcolon [0,\infty)\to [1,\infty)$.
In particular, it is direct to see that $y\mapsto \psi_{-}^{-1}(y)$ is strictly decreasing
 ($\psi_{-}^{-1}(0)=1$ and $\psi_{-}^{-1}(y)\to 0$ as $y\to \infty$),
 while $y\mapsto \psi_{+}^{-1}(y)$ is strictly increasing ($\psi_{+}^{-1}(0)=1$ and $\psi_{+}^{-1}(y)\to \infty$ as $y\to \infty$).

Now, from relation \eqref{eqs:kinetic}, Theorem \ref{thm:priori}\ref{item:prio0}, and the Jensen inequality,
it follows that, for all $(x,t)\in [0,k]\times[0,\infty)$,
\begin{equation}\label{temp3}
\psi(\dfrac{r^n(x,t)-a^n}{nx})
= \psi(\dfrac{1}{x}\int_{0}^x v(y,t)\,\dif y)
\le \dfrac{1}{x}\int_{0}^x \psi(v)(y,t)\,\dif y \le \dfrac{C_0}{x}.
\end{equation}
	
Next, the proof is divided into two cases: $0<\frac{r^n-a^n}{nx}\le 1$ and $\frac{r^n-a^n}{nx}\ge 1$.
	
\smallskip	
Case 1: $0<\frac{r^n-a^n}{nx}\le 1$. Then  $\frac{r^n-a^n}{nx}=\psi_{-}^{-1}(\psi(\frac{r^n-a^n}{nx}))$.
Since $\psi_{-}^{-1}(\cdot)$ is monotone decreasing, it follows from \eqref{temp3} that
\begin{equation*}
\dfrac{r^n-a^n}{nx}=\psi_{-}^{-1}(\psi(\dfrac{r^n-a^n}{nx}))\ge \psi_{-}^{-1}(\dfrac{C_0}{x}).
\end{equation*}
On the other hand, since $\frac{r^n-a^n}{nx}\le 1$, it follows from $a<1$ that
\begin{equation*}
r^n(x,t)\le a^n + nx
\le n (1+x).
\end{equation*}
	
\smallskip
Case 2: $\frac{r^n-a^n}{nx}\ge 1$. Then it follows from $\psi_{-}^{-1}(y)\le 1$ for all $y\in[0,\infty)$ that
\begin{equation*}
r^n(x,t)\ge a^n+nx\ge a^n+nx\psi_{-}^{-1}(\dfrac{C_0}{x}).
\end{equation*}
Moreover, $\frac{r^n-a^n}{nx}=\psi_{+}^{-1}(\psi(\frac{r^n-a^n}{nx}))$. Since $\psi_{+}^{-1}(\cdot)$ is
monotone increasing, it follows from \eqref{temp3} that
\begin{equation*}
\dfrac{r^n-a^n}{nx}=\psi_{+}^{-1}(\psi(\dfrac{r^n-a^n}{nx}))\le\psi_{+}^{-1}(\dfrac{C_0}{x}).
\end{equation*}
It is verified in Proposition \ref{prop:omega}
that $x\mapsto x \psi_{+}^{-1}(\frac{C_0}{x})$ is monotone increasing.
If $x\le 1$, then $r^n(x,t)\le a^n+n x \psi_{+}^{-1}(\frac{C_0}{x}) \le 1 + n \psi_{+}^{-1}(C_0) \le C_0 + C_0 x$.
If $x> 1$, since $y\mapsto \psi_{+}^{-1}(y)$ is monotone increasing,
then $r^n(x,t)\le a^n+n x \psi_{+}^{-1}(\frac{C_0}{x}) \le 1 + n x \psi_{+}^{-1}(C_0) \le C_0 + C_0 x$.
This concludes the proof.
\end{proof}

\begin{lemma}\label{lemma:select} Fix $t\in[0,T]$. For each $i=1, \dotsc, k$,
there exist
$A_i(t)$, $B_i(t) \in(i-1,i)$ so that
\begin{align*}
C_0^{-1}\le v(A_i(t),t),\, e(B_i(t),t)\le C_0.
\end{align*}
\end{lemma}

\begin{proof}
We give our proof only for $A_i(t)$, since the proof for $B_i(t)$ is the same.
Fix an integer $i=1,\dotsc,k$. Since $\psi(\cdot)$
is convex, it follows from Jensen's inequality and Theorem \ref{thm:priori}\ref{item:prio0} that
\begin{align*}
\psi( \int_{i-1}^{i} v(x,t)\,\dif x)\le \int_{i-1}^{i}\psi(v(x,t))\,\dif x \le\dfrac{C_0}{\gamma-1}.
\end{align*}
As in Lemma \ref{lemma:rbd}, denote $z\mapsto\psi_{\pm}^{-1}(z)$ to be the left and
right branch of the convex function $\psi(\zeta)$. Since $\psi(1)=0$, we have
\begin{align*}
0<\psi_{-}^{-1}(\dfrac{C_0}{\gamma-1})
\le \int_{i-1}^{i} v(x,t)\,\dif x\le\psi_{+}^{-1}( \dfrac{C_0}{\gamma-1})<\infty.
\end{align*}
By the mean value theorem, there exists a point $A_i(t)\in(i-1,i)$ such that
\begin{align*}
\psi_{-}^{-1}(\dfrac{C_0}{\gamma-1})
\le v(A_i(t), t) \le \psi_{+}^{-1}(\dfrac{C_0}{\gamma-1}).
\end{align*}
Since $e(x,t)$ satisfies the same estimate that differs only by constant $\gamma-1$ in Theorem \ref{thm:priori}\ref{item:prio0},
by the same argument, there also exists $B_i(t)\in(i-1,i)$ such that $C_0^{-1}\le e(B_i(t),t)\le C_0$.
\end{proof}

Using Lemmas \ref{lemma:rbd}--\ref{lemma:select} and Theorem \ref{thm:priori}\ref{item:prio0},
we are ready to prove Theorem \ref{thm:priori}\ref{item:prio1}.

\begin{proof}[Proof of Theorem \textnormal{\ref{thm:priori}\ref{item:prio1}}]
Fix
$i=1,\dotsc,k-2$. Since we focus on obtaining an explicit bound for $v(x,t)$ near the origin $x=0$,
the case for $i\ge k-1$ is not considered in the proof (but can be handled by repeating the same argument).
We divide the proof into six steps.

\medskip	
1. Applying Lemma \ref{lemma:select}, it follows that, for each $(x,t)\in[i-1,i+1]\times[0,T)$,
there exist $A_i(t)\in (x, x+3)$ and $ B_i(t)\in (i,i+1)$ such that
\begin{equation}\label{eqs:crb}
C_0^{-1}\le v(A_i(t),t),\, e(B_i(t),t) \le C_0.
\end{equation}
Now, from $\eqref{eqs:LFNS-k}_1$--$\eqref{eqs:LFNS-k}_2$,
\begin{equation}\label{temp4}
\D_t u + r^m \D_x p = \beta r^m \D_x \D_t \log v.
\end{equation}
Dividing both sides of \eqref{temp4} by $\beta r^{m}$ and integrating over $[x,A_i(t)]\times[0,t]$ for some $(x,t)\in(i-1,i+1]\times[0,T)$, we have
{
\begin{equation*}
\log \dfrac{v(A_i(t),t)v_0(x)}{v_0(A_i(t))v(x,t)}
=\dfrac{1}{\beta}\int_{x}^{A_i(t)}\dfrac{u}{r^m}\bigg\vert_0^t\dif y
+\dfrac{m}{\beta}\int_{x}^{A_i(t)}\int_{0}^{t}\dfrac{u^2}{r^n}\,\dif s\dif y+\dfrac{1}{\beta}\int_{0}^{t}\big(p(A_i(t),s)-p(x,s)\big)\,\dif s,
\end{equation*}
}
where $p(x,s)\vcentcolon = p(v(x,s),e(x,s))$. Taking the exponential on both sides yields
\begin{equation}\label{temp5}
\dfrac{E(x,t)}{D(x,t)}Y(t)\dfrac{v_0(A_i(t))v(x,t)}{v(A_i(t),t)v_0(x)}
= \exp\Big\{\dfrac{1}{\beta}\int_{0}^{t}p(x,s)\,\dif s\Big\},
\end{equation}
where
{\small
\begin{equation*}
\begin{split}
&E(x,t)\vcentcolon = \exp\Big\{\dfrac{m}{\beta}\int_{0}^{t}\int_{x}^{A_i(t)}\dfrac{u^2}{r^n}\,\dif y \dif s\Big\},
\quad Y(t)\vcentcolon = \exp\Big\{\dfrac{1}{\beta}\int_{0}^{t}p(A_i(t),s)\,\dif s\Big\}, \\
&D(x,t)\vcentcolon = \exp\Big\{\dfrac{1}{\beta}\int_{x}^{A_i(t)} \Big(\dfrac{u_0(y)}{r_0^m(y)}
-\dfrac{u(y,t)}{r^m(y,t)}\Big)\,\dif y \Big\}.
\end{split}
\end{equation*}
}
Multiplying $\eqref{temp5}$ with $\beta^{-1}p(x,t)$ and integrating in time, we have
\begin{equation*}
1 + \dfrac{1}{\beta}\int_{0}^{t} \dfrac{E(x,s)}{D(x,s)}Y(s)p(x,s)
\dfrac{v_0(A(x,s))v(x,s)}{v(A(x,s),s)v_0(x)}\,\dif s
=\exp\Big\{\dfrac{1}{\beta}\int_{0}^{t}p(x,s)\,\dif s\Big\}.
\end{equation*}
Substituting this back into \eqref{temp5}, it follows that
\begin{equation}\label{eqs:repre}
\dfrac{E(x,t)}{D(x,t)}Y(t)\dfrac{v_0(A_i(t))v(x,t)}{v(A_i(t),t)v_0(x)}
=1 + \dfrac{1}{\beta}\int_{0}^{t} \dfrac{E(x,s)}{D(x,s)}Y(s)p(x,s)
\dfrac{v_0(A(x,s))v(x,s)}{v(A(x,s),s)v_0(x)}\,\dif s.
\end{equation}
	
\smallskip	
2. First, for $D(x,t)$, it follows from  Lemma \ref{lemma:rbd}--\ref{lemma:select}
and Theorem \ref{thm:priori}\ref{item:prio0} that
\begin{align*}
\Bignorm{\int_{x}^{A_i(t)}\dfrac{u}{r^m}(y,t)\,\dif y }
&\le \Big(\int_{x}^{A_i(t)}r^{-2m}\,\dif y\Big)^{\frac{1}{2}}
\Big(\int_{x}^{A_i(t)}\snorm{u}^2\,\dif y\Big)^{\frac{1}{2}}\\
&\le \Big(\int_{x}^{A_i(t)}\big(a^n+ny\psi_{-}^{-1}(\frac{C_0}{y})\big)^{-\frac{2m}{n}} \dif y\Big)^{\frac{1}{2}}C_0^{\frac{1}{2}}\\
&\le C_0 + C_0 \big(a^n + n x \psi_{-}^{-1}(\frac{C_0}{x})\big)^{-\frac{2m}{n}}.
\end{align*}
Then, using the definition of $f(x,a)$ in \eqref{fhF} and Lemma \ref{lemma:select}, we have
\begin{equation}\label{eqs:D}
(e^{C_0} f(x,a))^{-1} \le D(x,t) \le e^{C_0} f(x,a).
\end{equation}
	
\smallskip	
3. Now we estimate the terms: $Y(t)$ and $E(x,t)$.
In view of the definition of $\Gamma(z,t,a)$ in \eqref{fhF}, Theorem \ref{thm:priori}\ref{item:prio0},
and Lemma \ref{lemma:rbd}, we see that, for $t\ge s \ge 0$,
\begin{equation}\label{eqs:E}
1 \le E(x,t) \le \Gamma^{-1}(x,t,a),\,\,\,\,\,
\Gamma(x,t-s,a)\le\dfrac{E(x,s)}{E(x,t)}
=\exp\Big\{\dfrac{m}{\beta}\int_{s}^{t}\int_{i-1}^{i}\dfrac{u^2}{r^n}\,\dif y \dif \tau\Big\} \le 1.
\end{equation}
Combining \eqref{eqs:crb} and \eqref{eqs:D}--\eqref{eqs:E} with \eqref{temp5} yields that, for all $(x,t)\in(i-1,i+1]\times[0,T)$,
\begin{equation}\label{eqs:gron1}
v(x,t) Y(t)
\le  C_0 f^2(x,a)\Big\{1 + \int_{0}^{t}  e(x,s)Y(s)\,\dif s \Big\}.
\end{equation}
Since the above argument is true for any $(x,t)\in(i-1,i+1]\times[0,T)$, for a fixed $x\in(i-1,i]$,
we can integrate both sides of the above equation in $y\in[x,x+1]\subset(i-1,i+1]$. It then follows from the entropy inequality \eqref{eqs:entropy} and Jensen's inequality that
\begin{equation*}
Y(t) \le  C_0 f^2(x,a) \Big\{ 1 + \int_{0}^{t} Y(s)\,\dif s \Big\} \qquad \text{for all $t\in[0,T)$,}
\end{equation*}
where we have used that $x\mapsto f(x,a)$ is monotone decreasing. Thus, by Gr\"onwall's inequality,
\begin{equation}\label{eqs:Y}
1\le Y(t)
= C_0(1+t)f^3(x,a)\exp\big\{C_0 t f^2(x,a) \big\} \qquad \text{for each $t\in(0,T]$.}
\end{equation}
	
4. Since $B_i(t)\in(i,i+1)$ from \eqref{eqs:crb}, then $x<i<B_i(t)$ for all $t\in[0,T)$.
Using the Cauchy-Schwartz inequality, Lemma \ref{lemma:rbd}, estimate \eqref{eqs:entropy}, and Jensen's inequality, it follows that, for each $y\in[x,i+1]\subseteq (i-1,i+1]$ and $s\in[0,T)$,
\begin{align}\label{temp:esqrt}
&\Bignorm{\sqrt{e(B_i(s),s)}- \sqrt{e(y,s)}}
  = \dfrac{1}{2}\Bignorm{\int_{y}^{B_i(s)} \dfrac{\D_x e}{\sqrt{e}}(z,s)\,\dif z }\nonumber\\
&\le \dfrac{1}{2} \Big(\int_{0}^{k} \dfrac{r^{2m}\snorm{\D_x e}^2}{v e^2}(z,s)\,\dif z\Big)^{\frac{1}{2}}\Big(\sup\limits_{z\in[x,i+1]}\frac{v}{r^{2m}}(z,s) \int_{i-1}^{i+1}e(z,s)\,\dif z \Big)^{\frac{1}{2}}
  \nonumber\\
&\le C_0\Big(a^{n}+nx\psi_{-}^{-1}(\dfrac{C_0}{x}) \Big)^{-\frac{m}{n}} \big(Q(s)V(x,s)\big)^{\frac{1}{2}},
\end{align}
where
\begin{equation}\label{temp:VJ}
V(x,s)\vcentcolon= \sup\limits_{z\in[x,i+1]}v(z,s),
\qquad Q(s) \vcentcolon= \int_{0}^{k}\dfrac{r^{2m}\snorm{\D_x e}^2}{v e^2}(z,s)\,\dif z,
\end{equation}
and $Q(t)\in L^1(0,T)$ due to the entropy estimate \eqref{eqs:entropy}. According to the definition of $h(z,a)$ in \eqref{fhF}, the construction in \eqref{eqs:crb} implies that, for all $y\in[x,i+1]\subseteq(i-1,i+1]$,
\begin{equation}\label{eqs:eineq}
e(y,t) \le C_0 +  C_0 h(x,a) Q(s)V(x,s)\qquad  \text{for all $(y,s)\in[x,i+1]\times[0,T)$}.
\end{equation}
Substituting \eqref{eqs:eineq} and \eqref{eqs:Y} into \eqref{eqs:gron1}, it implies that, for all $y\in [x,i+1]\subseteq (i-1,i+1]$,
\begin{align*}
v(y,t)&\le v(y,t)Y(t) \\
&\le  C_0(1+t)f^5(x,a)\exp\big\{C_0 t f^2(x,a)\big\}\Big\{1+t+\int_{0}^{t}h(x,a)Q(s)V(x,s)\dif s\Big\},
\end{align*}
where we used that $x\mapsto f(x,a)$ is decreasing. 
Taking the supremum over $y\in[x , i+1]$, we have 
\begin{align}\label{eqs:gronv}
V(x,t) \le C_0(1+t)f^5(x,a)\exp\big\{C_0 t f^2(x,a)\big\}\Big\{1+t+\int_{0}^{t}h(x,a)Q(s)V(x,s)\,\dif s\Big\}.
\end{align}
Applying Gr\"onwall's inequality to \eqref{eqs:gronv}, we obtain that, for all $(x,t)\in[i-1,i]\times[0,T)$,
\begin{align*}
v(x,t) \le V(x,t) \le C_0(1+t)^2 f^{5}(x,a)\exp\Big\{C_0t f^2(x,a) + C_0(1+t)(f^5h)(x,a)\exp\{C_0 t f^2(x,a)\} \Big\}.
\end{align*}
	
In addition, for the lower bound of $v(x,t)$, it follows from \eqref{temp5} that
\begin{equation*}
v(x,t) = \dfrac{D(x,t)}{Y(t)E(x,t)} \dfrac{v_0(x)v(A_i(t),t)}{v_0(A_i(t))} \exp\Big\{\int_{0}^{t} p(x,s)\,\dif s \Big\}
\ge C_0^{-1}\dfrac{D(x,t)}{Y(t)E(x,t)} .
\end{equation*}
By \eqref{eqs:crb}, \eqref{eqs:D}--\eqref{eqs:E}, and estimates \eqref{eqs:Y}, we find that, for all $(x,t)\in [i-1,i]\times [0,T]$,
\begin{align*}
v(x,t)\ge C_0 f^{-1} Y^{-1} E^{-1} \ge C_0 (1+t)^{-1} f^{-4}(x,a)\exp\{-C_0 t f^2(x,a)\} \Gamma(x,t,a).
\end{align*}
Since this is true for each $i=1,\dotsc, k-2$, it follows that, for all $(x,t)\in[0,k-2]\times[0,T]$,
\begin{equation*}
\begin{dcases}
v(x,t)\le C_0 (1+t)^2 f^{5}(x,a) \exp\Big\{C_0 t f^2(x,a) + C_0 (1+t) (f^5 h)(x,a)  \exp\{C_0 t f^2(x,a)\} \Big\} ,\\
v(x,t) \ge C_0 f^{-4}\exp\{-C_0t f^2\} \Gamma(x,t,a).
\end{dcases}
\end{equation*}
	
\smallskip	
5. From definition \eqref{fhF}, it follows that, for all $x\in[1,k-2]$,
\begin{align*}
f(x,a)\le\exp\big\{C_0(n\psi_{-}^{-1}(C_0))^{-\frac{2m}{n}}\big\},
\,\,\, h(x,a) \le ( n\psi_{-}^{-1}(C_0))^{-\frac{2m}{n}},
\,\,\, \Gamma(x,t,a)\ge\exp\big\{-\dfrac{mC_0t}{n\beta \psi_{-}^{-1}(C_0)} \big\},
\end{align*}
so that $C^{-1}(T)\le v(x,t)\le C(T)$ for all $(x,t)\in[1,k-2]\times[0,T]$. Since $x\mapsto x \psi_{-}^{-1}(\frac{C_0}{x})$ is positive and monotone increasing, then, for all $x\in[\ve,k-2]$,
\begin{align*}
f(x,a)\le \exp\big\{C_0(n\ve\psi_{-}^{-1}(\frac{C_0}{\ve}))^{-\frac{m}{n}}\big\},\,\,\,
h(x,a)\le (n\ve\psi_{-}^{-1}(\frac{C_0}{\ve}))^{-\frac{m}{n}},\,\,\,
\Gamma(x,t,a) \ge \exp\big\{-\dfrac{mC_0t}{n\beta\ve\psi_{-}^{-1}(\frac{C_0}{\ve})} \big\},
\end{align*}
which yield that, for each $\ve\in(0,1]$, $C^{-1}(\ve)\le v(x,t) \le C(\ve)$ for all $(x,t)\in[\ve,k-2]\times[0,T]$.
	
\smallskip
6. Finally, we remark that, for the case when $x\in(k-2,k]$, we can choose $k\ge 3$ and $A_k(t), B_k(t)\in(k-1,k)$ by Lemma \ref{lemma:select}, and
then repeat the same argument presented above. From this, we conclude that $C^{-1}(T)\le v(x,t)\le C(T)$ for all $(x,t)\in(k-2,k]\times[0,T)$.
\end{proof}

\subsection{Exterior \texorpdfstring{$L^{\infty}\big(0,T;L^2\big)$}{LTL2}--estimates}\label{subsec:L2total}
To obtain the uniform estimates in $(a,k)\in(0,1)\times\mathbb{N}$, the derivations are restricted to domain $[\ve,k]\times[0,T]$.
Thus, given fixed $\ve>0$, we introduce the spatial cut-off function $g_\ve(x)\in C^1$:
{\small
\begin{equation}\label{eqs:g}
g_{\ve}(x)\vcentcolon=
\begin{dcases*}
0 & if $x\le \dfrac{\ve}{2}$,\\
\dfrac{8}{\ve^2}\big(x-\frac{\ve}{2}\big)^2 & if $\dfrac{\ve}{2} \le x\le \dfrac{3\ve}{4}$,\\
1-\dfrac{8}{\ve^2}(x-\ve)^2 & if $\dfrac{3\ve}{4} \le x\le \ve$,\\
1 & if $x\ge \ve$.
\end{dcases*}
\end{equation}
}
It can directly be verified that $g_\ve(x)$ satisfies the following properties:
\begin{equation}\label{eqs:gprop}
\supp (g_{\ve}^{\prime}) = [\frac{\ve}{2},\ve], \quad\,
0\le g_\ve(x) \le 1, \quad g_\ve^{\prime}(x)\ge 0, \quad
\snorm{g_{\ve}^{\prime}(x)}^2\le \dfrac{32}{\ve^2} g_{\ve}(x) \qquad\,\mbox{for all $x\in[0,\infty)$}.
\end{equation}

Moreover, it is crucial that the estimates are independent of the approximation parameters $(a,k)\in(0,1)\times\mathbb{N}$. In order to achieve this,
the upper and lower bounds of $r(x,t)$ must be given careful consideration. Recall that, from Lemma \ref{lemma:rbd}, $r(x,t)$ satisfies
\begin{equation*}
n x\psi_{-}^{-1}(\frac{C_0}{x}) \le r^n(x,t) \le C_0(1 + x)
\qquad \text{for $(x,t)\in[0,k]\times[0,T]$}.
\end{equation*}
It follows that, for each $\ve\in(0,1]$,
\begin{equation}\label{eqs:grbd}
\sup_{x\in[\frac{\ve}{2},\ve]}r(x,t)\le C(\ve), \qquad \,\, \sup_{x\in[\ve,k]}r^{-1}(x,t) \le C(\ve).
\end{equation}

Using $g_\ve(x)$ and Theorem \ref{thm:priori}\ref{item:prio0}--\ref{item:prio1},
our aim is to derive the $L^{\infty}(0,T;L^2)$--estimates listed in  Theorem \ref{thm:priori}\ref{item:prio2}.
In order to do this, we first prove the following lemma, which is necessary for resolving the problematic boundary integrals for $x\in[\frac{\ve}{2},\ve]$.

\begin{lemma}[Exterior $L^{1}(0,T;L^{\infty})$--Estimates of $e$ and $u^2$]\label{lemma:euLinf}
For any fixed $\ve>0$,
\begin{equation}\label{temp:euLinf}
\int_{0}^{T} \sup_{y\in [\ve,k]} \big(u^2+e\big)(y,s)\, \dif s \le C(\ve).
\end{equation}
\end{lemma}

\begin{proof}
By Lemma \ref{lemma:select}, for each $(y,s)\in[\ve,k]\times[0,T]$, there exists $B_i(s)\in[0,k]$
such that $\norm{y-B_i(s)}\le 1$ and $C_0^{-1} \le e(B_i(s),s) \le C_0$. By the same calculation in \eqref{temp:esqrt}, we have
\begin{align}
\snorm{\sqrt{e(y,s)} - \sqrt{e(B_i(s),s)}}^{2}
\le&\, \frac{1}{4}\sup_{z\in[\ve,k]} \dfrac{v}{r^{2m}} (z,s) \Bignorm{\int_{B_i(s)}^{y} e(x,s)\,\dif x }
 \Bignorm{ \int_{B_i(s)}^{y}  \dfrac{r^{2m}\,\snorm{\D_x e}^2}{ve^2}(x,s)\,\dif x }\nonumber\\
\le &\, C(\ve)\int_{0}^{k} \dfrac{r^{2m}\,\snorm{\D_x e}^2}{v e^2}(x,s)\,\dif x, \label{4.25a}
\end{align}
where
we have used Theorem \ref{thm:priori}\ref{item:prio1}, \eqref{eqs:g}--\eqref{eqs:grbd},
and $\sbnorm{\psi(e)}_{L^{1}(0,k)}\le C_0$ from Theorem \ref{thm:priori}\ref{item:prio0} and Jensen's inequality.
Taking the supremum over $y\in[\ve,k]$ on the left-hand side of \eqref{4.25a} and integrating in $t\in[0,T]$,
it follows from
Lemma \ref{lemma:select} and Theorem \ref{thm:priori}\ref{item:prio0} that
\begin{equation}\label{temp:infE}
\int_{0}^{T} \sup_{y\in [\ve,k]} e(y,s)\,\dif s
\le C_0 T + C(\ve)\int_{0}^{T}\int_{0}^{k}\dfrac{r^{2m}\,\snorm{\D_x e}^2}{ve^2}
\,\dif x \dif s
\le C(\ve).
\end{equation}
	
Next, notice that
\begin{align*}
\snorm{u\D_x u}=\snorm{r^{-m}u \D_x(r^m u)-m r^{-n}v u^2}
\le 2r^{-2m} \snorm{\D_x(r^m u)}^2 +2 m^2 r^{-2n} v^2 \snorm{u}^2.
\end{align*}
Then, by the Sobolev embedding theorem: $W^{1,1}(\ve,k)\xhookrightarrow[]{}C^0(\ve,k)$, and Theorem \ref{thm:priori}\ref{item:prio0}--\ref{item:prio1},
\begin{align}\label{eqs:uSob}
\sup_{y\in[\ve,k]}u^2(y,t)
&\le C\int_{\ve}^{k} u^2\,\dif x + C\int_{\ve}^{k}\snorm{u\D_x u}\,\dif x
   \nonumber\\
&\le C_0 + C \sup_{y\in[\ve,k]}\Big(\dfrac{v}{r^{2m}}e
 +\dfrac{v}{r^n}\Big)(y,s)\int_{\ve}^{k} u^2(x,t)\,\dif x
  + C\int_{\ve}^{k} \dfrac{\snorm{\D_x(r^m u)}^2}{ve}\,\dif x\nonumber\\
&\le C(\ve) + C(\ve) \sup_{y\in[\ve,k]}e(y,t) + C\int_{0}^{k} \dfrac{\snorm{\D_x (r^mu)}^2}{ve}\,\dif x.
\end{align}
$\eqref{temp:euLinf}_2$ follows from integrating the above in $[0,T]$ and then using \eqref{temp:infE} and Theorem \ref{thm:priori}\ref{item:prio0}.
\end{proof}

\begin{lemma}[Exterior $L^{\infty}(0,T;L^2)$--Estimate of the Total Energy]\label{lemma:eL2}
For any fixed $\ve>0$,
    \begin{align}\label{eqs:eL2}
	        \sup\limits_{t\in[0,T]}\int_{\ve}^{k}\big(\snorm{e-1}^2 + u^4\big)(x,t)\,\dif x
	        + \int_{0}^{T}\sup\limits_{x\in[\ve,k]}e^2(x,s)\,\dif s
+\int_{0}^{T}\int_{\ve}^{k}\dfrac{r^{2m}}{v}
	          \big(\snorm{\D_x e}^2 + \snorm{u\D_x u}^2\big)\,\dif x \dif s  \le  C(\ve).
	    \end{align}
\end{lemma}

\begin{proof} We divide the proof into four steps.

\smallskip
1. Denote the total energy density function $w\vcentcolon= \frac{1}{2}u^2+e-1$.
Then multiplying $\eqref{eqs:LFNS-k}_2$ with $u$ and adding it to $\eqref{eqs:LFNS-k}_3$ yield
\begin{equation}\label{TEtemp12}
\D_t w = \D_x\big(r^m u \tilde{F}\big) -2m\mu \D_x(r^{m-1}u^2)
+ \kappa \D_x\big(\dfrac{r^{2m}}{v}\D_x e\big),
\end{equation}
where $\tilde{F}\vcentcolon=\beta v^{-1}\D_x( r^m u)-p(v,e)$.
Multiplying \eqref{TEtemp12} with $g_\ve w$ and then integrating in $x$, we have
\begin{align*}
\dfrac{1}{2}\dfrac{\dif}{\dif t}\int_{0}^k g_\ve w^2\,\dif x
&= -\int_{0}^{k} g_\ve r^m u\tilde{F} \D_x w\,\dif x
  + 2m\mu \int_{0}^{k} g_\ve r^{m-1}u^2 \D_x w\,\dif x
  - \kappa \int_{0}^{k} g_\ve \dfrac{r^{2m}}{v} \D_x e \D_x w\,\dif x \\
&\quad\, +\int_{0}^{k} g_\ve^{\prime} \Big\{-r^m u w \tilde{F}
  +2m\mu r^{m-1}u^2 w
  -\kappa  \dfrac{r^{2m}}{v} w \D_x e\Big\}\,\dif x \\
&=\vcentcolon \sum_{i=1}^{4} I_i.
\end{align*}

2. The terms $I_j, j=1,2,3$, can be estimated as follows:
\begin{align*}
I_1
=& -\beta \int_{0}^{k} g_\ve \big\{\dfrac{r^{2m}}{v} u\D_x u \D_x e + m  r^{m-1} u^3 \D_x u
   + m r^{m-1} u^2 \D_x e \big\}\,\dif x  \\
&+(\gamma-1)\int_{0}^{k}g_\ve\dfrac{e}{v}r^m\big\{u^2\D_x u+u\D_xe\big\}\,\dif x
  -\beta\int_{0}^{k} g_\ve\dfrac{r^{2m}}{v}\snorm{u \D_x u}^2\,\dif x \\
\le &\, \big(\dfrac{5\beta^2}{2\kappa}-\dfrac{2\beta}{3}\big) \int_{0}^{k}g_\ve\dfrac{\snorm{r^m u\D_x u}^2}{v}\,\dif x +\dfrac{3\kappa}{10} \int_{0}^{k} g_\ve \dfrac{\snorm{r^m\D_x e}^2}{v}\,\dif x
 +C\int_{0}^k g_\ve \Big( \dfrac{v u^4}{r^2} + \dfrac{e^2u^2}{v} \Big)\,\dif x,\\
I_2
   =&\, 2m\mu \int_{0}^{k} g_\ve r^{m-1} (u^3 \D_x u + u^2 \D_x e)\,\dif x \\
\le& \,\dfrac{\beta}{6} \int_{0}^{k} g_\ve\dfrac{r^{2m}}{v}\snorm{u \D_x u}^2\,\dif x
  +\dfrac{\kappa}{10} \int_{0}^{k} g_\ve\dfrac{r^{2m}}{v}\snorm{\D_x e}^2\,\dif x
  + C \int_{0}^{k} g_\ve \dfrac{v}{r^2} u^4\,\dif x,\\
I_3
 =&  - \kappa \int_{0}^{k} g_\ve \dfrac{r^{2m}}{v} \D_x e\{\D_x e + u \D_x u\}\,\dif x\\
\le & -\dfrac{9}{10}\kappa\int_{0}^{k}g_\ve\dfrac{r^{2m}}{v}\snorm{\D_x e}^2\,\dif x
  +\dfrac{5\kappa}{2}\int_{0}^{k} g_\ve\dfrac{r^{2m}}{v}\snorm{u \D_x u}^2\,\dif x.
\end{align*}
Using \eqref{eqs:g}--\eqref{eqs:grbd} and Theorem \ref{thm:priori}\ref{item:prio1},
we have
\begin{align*}
I_4
=& \int_{\frac{\ve}{2}}^{\ve} g_\ve^{\prime} \Big\{-r^m u w \Big(\beta\dfrac{ r^m}{v} \D_x u + \dfrac{m\beta}{r} u - (\gamma-1)\dfrac{e}{v}\Big) + 2m\mu r^{m-1}u^2 w -\kappa \dfrac{r^{2m}}{v} w \D_x e \Big\}\,\dif x\\
\le& \int_{\frac{\ve}{2}}^{\ve} \dfrac{\ve^2\snorm{g_{\ve}^{\prime}}^2}{32}\dfrac{r^{2m}}{v}
  \Big\{ \dfrac{\beta}{2}\snorm{u \D_xu}^2+\dfrac{\kappa}{4}\snorm{\D_x e}^2 \Big\}\,\dif x
 + C\int_{\frac{\ve}{2}}^{\ve}\Big\{ \dfrac{\ve^2\snorm{g_{\ve}^{\prime}}^2}{32}\big(\dfrac{vu^4}{r^2}+\dfrac{e^2 u^2}{v}\big)
 + \dfrac{32}{\ve^2}\dfrac{r^{2m}}{v}w^2 \Big\}\,\dif x\\
\le& \int_{0}^{k}g_\ve \dfrac{r^{2m}}{v}
\Big\{ \dfrac{\beta}{2}\norm{u \D_xu}^2 + \dfrac{\kappa}{4}\norm{\D_x e}^2 \Big\}\,\dif x
 + C\int_{0}^{k}g_\ve\Big\{\dfrac{v}{r^2}u^4 +\dfrac{e^2 u^2}{v}\Big\}\,\dif x
 +  C(\ve)\int_{\frac{\ve}{2}}^{\ve} w^2\,\dif x.
\end{align*}
Combining all the above estimates, it follows that
\begin{equation}\label{TEtemp13}
\begin{aligned}
\dfrac{1}{2}\dfrac{\dif}{\dif t}\int_{0}^k g_\ve w^2\,\dif x
\le& -\dfrac{\kappa}{4}\int_{0}^{k} g_\ve \dfrac{r^{2m}}{v} \norm{\D_x e}^2 \dif x + \dfrac{5(\beta^2+\kappa^2)}{2\kappa}\int_{0}^{k} g_\ve \dfrac{r^{2m}}{v} \norm{u \D_x u}^2 \dif x \\
&+ C \int_{0}^{k} g_\ve \dfrac{v}{r^2} u^4\,\dif x
  + C \int_{0}^{k} g_\ve \dfrac{e^2 u^2}{v}\,\dif x
  + C(\ve)\mathcal{R}_{\ve}(t),
\end{aligned}
\end{equation}
where $\mathcal{R}_{\ve}(t)$ is the problematic boundary integral near the origin, given by
\begin{equation}\label{TEtemp20}
\mathcal{R}_\ve(t)
\vcentcolon = \int_{\frac{\ve}{2}}^{\ve}\big((e-1)^2 + u^4\big)(x,t)\,\dif x.
\end{equation}
	
3. Multiplying $g_\ve u^3$ to the momentum equation $\eqref{eqs:LFNS-k}_2$ and integrating in $x\in[0,k]$, we obtain the following $L^4$-estimate for $u$:
\begin{align*}
\dfrac{1}{4}\dfrac{\dif}{\dif t}\int_{0}^{k}g_\ve \norm{u}^4 \dif x
&= -\beta\int_{0}^{k} g_\ve\dfrac{1}{v} \D_x(r^m u)\D_x(r^m u^3)\,\dif x
 +(\gamma-1) \int_{0}^{k} g_\ve \dfrac{e}{v} \D_x(r^m u^3)\,\dif x   \\
&\quad +   \int_{0}^{k} g_\ve^{\prime} \Big\{\dfrac{(\gamma-1)e}{v}  r^m u^3
  -\dfrac{\beta}{v} \D_x(r^m u) r^m u^3\Big\}\,\dif x\\
&=\vcentcolon \,\sum_{i=1}^3 J_i.
\end{align*}
The terms $J_1$ and $J_2$ can be estimated as follows:
\begin{align*}
J_1
=& -\beta \int_{0}^{k} g_\ve \Big( 3\dfrac{\snorm{r^{m}u\D_x u}^2}{v}
   + 4m r^{m-1}u^3 \D_x u + m^2 \dfrac{u^4v}{r^2} \Big)\,\dif x\\
\le&\, -\beta \int_{0}^{k} g_\ve \dfrac{r^{2m}}{v}\snorm{u \D_x u}^2\,\dif x
     + m^2 \beta \int_{0}^{k} g_\ve \dfrac{v}{r^2} u^4 \,\dif x,\\
J_2
\le&\, \dfrac{\beta}{4}\int_{0}^{k}g_\ve \dfrac{\snorm{r^mu\D_x u}^2}{v}\,\dif x + C \int_{0}^{k} g_\ve \Big( \dfrac{e^2u^2}{v} + \dfrac{vu^4}{r^2} \Big)\,\dif x.
\end{align*}
Using \eqref{eqs:g}--\eqref{eqs:grbd}
and Theorem \ref{thm:priori}\ref{item:prio1}, we have
\begin{align*}
J_3
\le &\, \dfrac{\gamma-1}{2} \int_{\frac{\ve}{2}}^{\ve} \dfrac{\ve^2}{32}\snorm{g_\ve^{\prime}}^2 \dfrac{e^2 u^2}{v}\,\dif x
 +\dfrac{\gamma-1}{2}\int_{\frac{\ve}{2}}^{\ve}\dfrac{32}{\ve^2}\dfrac{r^{2m}}{v}u^4\,\dif x
 + C \int_{\frac{\ve}{2}}^{\ve} \dfrac{1}{\ve} r^{m-1} u^4\,\dif x\\
&\, + \dfrac{\beta}{4} \int_{\frac{\ve}{2}}^{\ve} \dfrac{\ve^2}{32}
  \snorm{g_\ve^{\prime}}^2 \dfrac{r^{2m}}{v} \snorm{u \D_x u}^2\,\dif x
 +\beta\int_{\frac{\ve}{2}}^{\ve}\dfrac{32}{\ve^2} \dfrac{r^{2m}}{v} u^4\,\dif x \\
\le &\, \dfrac{\beta}{4}\int_{0}^{k}g_\ve\dfrac{r^{2m}}{v}\snorm{u \D_x u}^2\,\dif x
 +\dfrac{\gamma-1}{2}\int_{0}^{k} g_\ve \dfrac{e^2 u^2 }{v}\,\dif x
 + C(\ve) \int_{\frac{\ve}{2}}^{\ve} u^4\,\dif x.
\end{align*}
Combining all the above estimates, it follows that
\begin{equation}\label{TEtemp14}
\begin{aligned}
\dfrac{\dif}{\dif t}\int_{0}^{k}g_\ve \dfrac{\snorm{u}^4}{4} \dif x
\le &\, C\int_{0}^kg_\ve\Big(\dfrac{e^2 u^2}{v}+\dfrac{vu^4}{r^2}\Big)\,\dif x
 -\dfrac{\beta}{2}\int_{0}^{k} g_\ve \dfrac{\snorm{r^{m}u\D_x u}^2}{v}\,\dif x
  + C(\ve) \mathcal{R}_\ve(t).
\end{aligned}
\end{equation}
	
4. Multiplying \eqref{TEtemp14} by $4\tilde{C}\equiv \frac{10(\beta^2+\kappa^2)}{\kappa\beta}$ and adding it to \eqref{TEtemp13},
it follows from \eqref{eqs:entropy} that
\begin{align}\label{TEtemp15}
&\tilde{C}\dfrac{\dif}{\dif t}\int_{0}^{k} g_\ve u^4\,\dif x
+ \dfrac{1}{2}\dfrac{\dif}{\dif t} \int_{0}^k g_\ve w^2\,\dif x
+ \int_{0}^{k} g_\ve \dfrac{r^{2m}}{v} \Big\{ \dfrac{\kappa}{4}\snorm{\D_x e}^2 + \beta\tilde{C} \snorm{u \D_x u}^2 \Big\}\,\dif x
  \nonumber\\
&\le  C(\ve) \mathcal{R}_\ve(t) + C \int_{0}^{k} g_\ve \Big(\dfrac{v}{r^2}u^4 +\dfrac{\norm{ue}^2}{v}\Big)\,\dif x.
\end{align}
Using Theorem \ref{thm:priori}\ref{item:prio0}--\ref{item:prio1} and \eqref{eqs:g}--\eqref{eqs:grbd},
we have
\begin{align}\label{TEtemp16}
\int_{0}^{k} g_\ve \Big(\dfrac{v}{r^2}u^4 +\dfrac{\norm{ue}^2}{v}\Big)\,\dif x
\le  C(\ve) \int_{0}^{k} g_\ve u^4\,\dif x + C(\ve) \sup_{y\in[\frac{\ve}{2},k]}(g_\ve\snorm{e}^2)(y,t).
\end{align}
By Lemma \ref{lemma:select}, for each $(x,t)\in[0,k]\times[0,T]$,
there exist both an integer $i\in\{0,\dotsc, k-1\}$ and a point $B_i(t)\in[0,k]$
such that $x, B_i(t) \in [i,i+1]$ and $C_0^{-1}\le e(B_i(t),t) \le C_0$.
Using Theorem \ref{thm:priori}\ref{item:prio0}, $\snorm{x-B_i(t)}\le 1$, and \eqref{eqs:g}--\eqref{eqs:grbd},
we have
\begin{align*}
&\snorm{(g_\ve e)(x,t) - (g_\ve e)(B_i(t),t)}^2
  =\Big(\int_{B_i(t)}^{x} (g_\ve \D_x e + g_\ve^{\prime} e)(y,t)\,\dif y \Big)^2  \\
&\le 2\Big(\int_{B_i(t)}^{x}g_\ve\dfrac{r^{2m}}{v}\dfrac{\snorm{\D_x e}^2}{e^2}
  \,\dif y \Big) \Big( \int_{B_i(t)}^{x}g_\ve\dfrac{v}{r^{2m}}e^2\,\dif y \Big) + C(\ve)\Big(\int_{B_i(t)}^{x}e(y,t)\,\dif y\Big)^2\\
&\le 4\Big(\int_{0}^{k} \dfrac{r^{2m}}{v}\dfrac{\norm{\D_x e}^2}{e^2}\,\dif y\Big) \Big\{ \Bignorm{\int_{B_i(t)}^{x}g_\ve\dfrac{v}{r^{2m}}\snorm{e-1}^2\,\dif y} + \Bignorm{\int_{B_i(t)}^{x}\dfrac{g_\ve v}{r^{2m}}\,\dif y}\Big\} + C(\ve)\Big(\int_{i}^{i+1}e\,\dif y\Big)^2\\
&= C(\ve)\Big\{1+ Q(t)\Big(1+\Bignorm{\int_{B_i(t)}^{x}g_\ve\snorm{e-1}^2\,\dif y}\Big)\Big\},
\end{align*}
where $Q(t)$ is defined in \eqref{temp:VJ}. 
Taking the supremum over $x\in[\frac{\ve}{2},k]$, it follows from Theorem \ref{thm:priori}\ref{item:prio1} and \eqref{eqs:g}--\eqref{eqs:grbd}
that
\begin{equation}\label{TEtemp17}
\sup_{x\in[\frac{\ve}{2},k]} (g_\ve\snorm{e}^2)(x,t)
\le  C(\ve)\Big\{1 +  Q(t) \Big(1 + \int_{0}^{k} g_\ve\snorm{e-1}^2\,\dif x\Big)\Big\}.
\end{equation}
Substituting \eqref{TEtemp17} into \eqref{TEtemp16} yields
\begin{equation}\label{TEtemp21}
\begin{aligned}
\int_{0}^{k} g_\ve \Big(\dfrac{v}{r^2}u^4 +\dfrac{\snorm{ue}^2}{v}\Big)\,\dif x
\le C(\ve) \big( 1+ Q(t)\big) \Big( 1+ \int_{0}^{k} g_\ve ( u^4 + w^2 )\,\dif x \Big),
\end{aligned}
\end{equation}
where we have used:
$\snorm{u}^4 + w^2
\ge \dfrac{3}{4}\snorm{u}^4 + \dfrac{1}{2}(e-1)^2$.
Substituting \eqref{TEtemp21} into \eqref{TEtemp15}, it follows that
\begin{align*}
&\dfrac{\dif}{\dif t}\int_{0}^{k} g_\ve (u^4 + w^2)\,\dif x
 + \int_{0}^{k} g_\ve \dfrac{r^{2m}}{v}\big(\snorm{\D_x e}^2
 \snorm{u \D_x u}^2\big)\,\dif x \\
&\le C(\ve) \big( 1+ Q(t)\big) \Big( 1+ \int_{0}^{k} g_\ve ( u^4 + w^2)\,\dif x \Big) + C(\ve) \mathcal{R}_{\ve}(t) .
\end{align*}
Therefore, applying Gr\"onwall's inequality and using $Q(t)\in L^1(0,T)$,
we have
\begin{equation}\label{TEtemp18}
\begin{aligned}
\int_{0}^{k} \big(\snorm{u}^4 + \snorm{e-1}^2\big)(x,t)\,\dif x
\le  C(\ve)\Big(1+ \int_{0}^{k} (u_0^4 + w_0^2)\,\dif x
         + \int_{0}^{t} \mathcal{R}_\ve\,\dif s\Big).
\end{aligned}
\end{equation}
	
5. Using Lemma \ref{lemma:euLinf}, we claim that $\mathcal{R}_\ve(t)$ defined in \eqref{TEtemp20} satisfies:
\begin{equation}\label{TEtemp19}
\int_{0}^t \mathcal{R}_\ve(s)\,\dif s
= \int_{0}^t \int_{\frac{\ve}{2}}^{\ve}\big(\snorm{e-1}^2 + \snorm{u}^4\big)\,\dif x \dif s \le C(\ve).
\end{equation}

To show this claim, we first use Lemma \ref{lemma:euLinf} and Theorem \ref{thm:priori}\ref{item:prio0} to obtain	
\begin{align*}
\int_{0}^{t} \mathcal{R}_\ve(s)\,\dif s
&\le  C\int_{0}^t \sup_{y\in[\frac{\ve}{2},\ve]}\snorm{e-1}(y,s)
\big(\int_{\frac{\ve}{2}}^{\ve}(e+1)(x,s)\,\dif x\big)\, \dif s
+C\int_{0}^t\sup_{y\in[\frac{\ve}{2},\ve]}u^2(y,s)\big(\int_{\frac{\ve}{2}}^{\ve}u^2(x,s)\,\dif x\big)\,\dif s\\
&\le C(\ve)\Big(1+
\int_{0}^t \sup_{y\in[\frac{\ve}{2},\ve]} e(y,s) \big(\int_{\frac{\ve}{2}}^{\ve} e(x,s)\,\dif x\big)\,\dif s\Big).
\end{align*}
It follows by Jensen's inequality and Theorem \ref{thm:priori}\ref{item:prio0} that
\begin{align*}
\psi( \dfrac{2}{\ve} \int_{\frac{\ve}{2}}^{\ve} e(x,s)\,\dif x )
\le \dfrac{2}{\ve} \int_{\frac{\ve}{2}}^{\ve} \psi(e)(x,s)\,\dif x
\le \dfrac{2C_0}{\ve},
\end{align*}
so that
$$
\int_{\frac{\ve}{2}}^{\ve} e(x,s)\, \dif x \le \dfrac{\ve}{2} \psi^{-1}_{+} (\dfrac{2C_0}{\ve}) \le C(\ve).
$$
Substituting this back into the estimate of $\mathcal{R}_\ve$
and using Lemma \ref{lemma:euLinf}, it follows that
\begin{align*}
\int_{0}^{t} \mathcal{R}_\ve(s)\,\dif s
\le C(\ve) \Big(1+ \int_{0}^t \sup_{y\in[\frac{\ve}{2},\ve]} e(y,s)\,\dif s\Big)
\le C(\ve).
\end{align*}
This prove the claim \eqref{TEtemp19}.

\medskip
6. Using Gr\"onwall's inequality and \eqref{TEtemp18}--\eqref{TEtemp19},
we have
\begin{align*}
&\int_{0}^{k} g_\ve(x)(u^4 + w^2)(x,t)\,\dif x
+ \int_{0}^{t}\int_{0}^{k} g_\ve \dfrac{r^{2m}}{v} \snorm{\D_x e}^2
\snorm{u \D_x u}^2\,\dif x\dif s
\le \int_{0}^{k} ( u_0^4 + w_0^2 )\, \dif x + C(\ve).
\end{align*}
By the inequality:
$\snorm{u}^4 + w^2
\ge \frac{3}{4}\snorm{u}^4 + \frac{1}{2}(e-1)^2$,
we see that $\sbnorm{(e-1,u^2)}_{L^{\infty}(0,T;L^2(\ve,k))}\le C(\ve)$.
		
Finally, substituting the above estimate into \eqref{TEtemp17}, we obtain
\begin{align*}
\int_{0}^{T}\sup\limits_{x\in[\ve,k]} e^2(x,t)\,\dif t
&\le \int_{0}^{T}\sup_{x\in[\frac{\ve}{2},k]} (g_\ve e^2)(x,t)\,\dif t
\le C(\ve).
\end{align*}
This completes the proof.
\end{proof}

\begin{lemma}[Exterior $L^{\infty}(0,T;L^2)$--Estimate of the Specific Volume]\label{lemma:vL2}
For each $\ve>0$,
\begin{align*}
\sup\limits_{t\in[0,T]}\int_{\ve}^{k} \snorm{v-1}^2(x,t)\,\dif x \le C(\ve).
\end{align*}
\end{lemma}

\begin{proof}
Multiplying $\eqref{eqs:LFNS-k}_1$ with $g_\ve(v-1)$ and integrating in $[0,k]$, then, by Theorem \ref{thm:priori}\ref{item:prio1},
\begin{align*}
\dfrac{1}{2} \D_t \int_{0}^k g_\ve (v-1)^2\,\dif x
= \int_{0}^k g_\ve (v-1) \D_x ( r^m u )\,\dif x
\le C(\ve) \int_{0}^{k} g_\ve (v-1)^2\,\dif x
 + \dfrac{1}{2} \int_{0}^k g_\ve \dfrac{\snorm{\D_x(r^m u)}^2}{v}\,\dif x.
\end{align*}
Applying Lemma \ref{lemma:eL2} and Gr\"onwall's inequality, the lemma is proved.
\end{proof}

The following lemma is necessary for the high-order estimates of $(u,e)$
in \S \ref{subsec:HRue}
below.

\begin{lemma}[Exterior $L^{2}$--estimate of $\D_x u$] \label{lemma:L2Dxu}
For each $\ve>0$,
\begin{equation}\label{eqs:dissip}
\int_{0}^{T}\int_{\ve}^{k} \Big(\dfrac{\snorm{\D_x (r^m u)}^2}{v} +\dfrac{r^{2m}\snorm{\D_x u}^2}{v} \Big)\,\dif x \dif s \le C(\ve).
\end{equation}
\end{lemma}

\begin{proof}
Multiplying $\eqref{eqs:LFNS-k}_1$ by $g_\ve$ and integrating over $[0,k]$,
we have
\begin{equation}\label{temp17}
\int_{0}^{k} g_\ve(x) \D_t (v(x,t)-1)\,\dif x
= - \int_{0}^{k} g_{\ve}^{\prime}(x) (r^m u)(x,t)\,\dif x.
\end{equation}
Multiplying $\eqref{eqs:LFNS-k}_2$ by $g_\ve u$ and integrating over $x\in[0,k]$,
using Theorem \ref{thm:priori}\ref{item:prio0}--\ref{item:prio1}, Lemma  \ref{lemma:eL2},
and \eqref{temp17}, it follows that
\begin{align}\label{461-di1}
&\frac{1}{2}\dfrac{\dif }{\dif t}\int_{0}^{k} g_{\ve}\snorm{u}^2\,\dif x
 +\beta\int_{0}^{k} g_{\ve}\dfrac{\snorm{\D_x(r^m u)}^2}{v} \dif x \nonumber\\
&= (\gamma-1)\int_{0}^{k} g_{\ve} \dfrac{ \D_x(r^m u)}{v} (e-1)\,\dif x + (\gamma-1)\int_{0}^{k} g_{\ve} \dfrac{ \D_t v}{v} \dif x
+ \int_{0}^{k}   g_{\ve}^{\prime}(x) \dfrac{r^mu}{v}\big((\gamma-1)e
  -\beta \D_x(r^m u)\big)\,\dif x \nonumber\\
&=\vcentcolon \sum_{i=1}^3 I_i.
\end{align}
The term $I_1$ is estimated by using Theorem \ref{thm:priori}\ref{item:prio1}
and Lemma \ref{lemma:eL2}:
\begin{equation}\label{461-di2}
\begin{split}
&I_1 \vcentcolon
= (\gamma-1)\int_{0}^{k} g_{\ve} \dfrac{ \D_x( r^m u)}{v} (e-1)\,\dif x
\le \dfrac{\beta}{4} \int_{0}^k g_\ve \dfrac{\snorm{\D_x( r^m u)}^2}{v} \dif x+ C(\ve).
\end{split}
\end{equation}
$I_2$ is estimated by using $\eqref{eqs:LFNS-k}_1$, Theorem \ref{thm:priori}\ref{item:prio0},
and \eqref{eqs:g}--\eqref{eqs:grbd}:
\begin{align}\label{461-di3}
\dfrac{I_2}{\gamma-1} &\vcentcolon = \int_{0}^{k} g_{\ve} \dfrac{ \D_t v}{v} \dif x
= \int_{0}^{k} g_{\ve} \D_t (\log v-v+1)\,\dif x
  - \int_{0}^{k} g_{\ve}^{\prime} r^m u\,\dif x\nonumber\\
&\le \int_{0}^{k} g_{\ve} \D_t (\log v -v +1)\,\dif x
 + \Big(\int_{\frac{\ve}{2}}^{\ve} r^{2m} \dif x\Big)^{\frac{1}{2}}
   \Big( \int_{\frac{\ve}{2}}^{\ve} u^2 \dif x \Big)^{\frac{1}{2}}\nonumber\\
&\le \int_{0}^{k} g_{\ve} \D_t ( \log v -v +1 )\,\dif x + C(\ve).
\end{align}
$I_3$ is estimated by using Lemma \ref{lemma:eL2} and \eqref{eqs:g}--\eqref{eqs:grbd}:
\begin{align}\label{461-di4}
I_3
\le &\, \dfrac{(\gamma-1)^2}{2}\sup\limits_{y\in[\frac{\ve}{2},\ve]} (v^{-2}r^{2m})(y,t)
\int_{\frac{\ve}{2}}^{\ve} u^2 \dif x
+ (\gamma-1) \Big( \int_{\frac{\ve}{2}}^{\ve} v^{-2} r^{2m} \dif x \Big)^{\frac{1}{2}}
\Big( \int_{\frac{\ve}{2}}^{\ve} u^2 \dif x \Big)^{\frac{1}{2}}
    \nonumber\\
& + \dfrac{1}{2}\int_{\frac{\ve}{2}}^{\ve} e^2 \dif x
+ \dfrac{\beta}{4} \int_{\frac{\ve}{2}}^{\ve}\dfrac{\ve^2}{32}\snorm{g_\ve^{\prime}(x)}^2
  \dfrac{\snorm{\D_x(r^m u)}^2}{v} \dif x  +  \dfrac{32\beta}{\ve^2}\sup\limits_{y\in[\frac{\ve}{2},\ve]}\dfrac{r^{2m}}{v}(y,t)
    \int_{\frac{\ve}{2}}^{\ve} u^2\,\dif x\nonumber\\
\le&\, \dfrac{\beta}{4} \int_{0}^{k}g_\ve \dfrac{\snorm{\D_x(r^m u)}^2}{v} \dif x + C(\ve).
\end{align}
According to \eqref{461-di1}--\eqref{461-di4}, we have
\begin{align*}
\frac{1}{2}&\int_{0}^{k}g_\ve(x)u^2(x,t)\,\dif x
+\dfrac{\beta}{2}\int_{0}^{t}\int_{0}^{k}g_{\ve}\dfrac{\snorm{\D_x (r^m u)}^2}{v} \dif x \dif s\\
&\le C(\ve) + (\gamma-1) \int_{0}^{k} g_\ve (\log v - v +1)\,\dif x
  + (\gamma-1) \int_{0}^{k} g_\ve (v_0 - \log v_0 - 1)\,\dif x.
\end{align*}
Since $\log v - v +1 \le 0$ and $C_0 \le v_0(x) \le C_0^{-1}$,
it implies from Taylor's theorem that
\begin{equation}\label{eqs:disp}
\dfrac{\beta}{2}\int_{0}^{t}\int_{0}^{k}g_{\ve}\dfrac{\snorm{\D_x (r^m u)}^2}{v} \dif x \dif s \le C(\ve) + (\gamma-1) C_0 \int_{0}^{k} (v_0-1)^2 \dif x \le C(\ve).
\end{equation}
The other term in \eqref{eqs:dissip} is estimated by rewriting \eqref{eqs:disp}
and using Theorem \ref{thm:priori}\ref{item:prio0}--\ref{item:prio1}.
\end{proof}

\subsection{Exterior \texorpdfstring{$L^{\infty}\big(0,T;H^1\big)$}{LTH1}--estimates of the velocity and internal energy}\label{subsec:HRue}
We now derive the high-regularity estimates for $(u,e)$ in Theorem \ref{thm:priori}\ref{item:prio2}. For simplicity, we define
\begin{equation*}
    \Lambda(v,e) \equiv (v-1)^2 + (e-1)^2, \quad \,
    \sigma(t) \vcentcolon = \min\{1,t\}.
\end{equation*}
It follows from Lemma \ref{lemma:eL2}-\ref{lemma:vL2} that $\bnorm{\Lambda(v,e)}_{L^{\infty}(0,T;L^1(\ve,k))}\le C(\ve)$. Furthermore, we introduce the effective viscosity flux $F$, which is defined by
\begin{align*}
	F\vcentcolon= \beta\dfrac{\D_x(r^m u)}{v} - p(v,e) +(\gamma-1) = \beta\dfrac{\D_x(r^m u)}{v} - (\gamma-1)\dfrac{e-v}{v} .
\end{align*}
Using
Lemmas \ref{lemma:eL2}--\ref{lemma:L2Dxu},
we have the following observation on $F$:

\begin{lemma}[Estimates on Effective Viscosity Flux]\label{lemma:estF}
\begin{equation}\label{eqs:SL2F}
\begin{split}
&\int_{0}^{k} g_\ve(x) F^2(x,t)\,\dif x
\le C(\ve)\Big(1+\int_{0}^{k} g_\ve \dfrac{\snorm{\D_x(r^m u)}^2}{v}\dif x\Big),\\
&\int_{0}^{T}\int_{0}^{k} g_\ve(x) F^2(x,t)\,\dif x \dif t \le C(\ve),\\
&\int_{0}^T \int_{0}^{k} g_\ve(x) (r^{2m}\snorm{\D_x F}^2)(x,t)\, \dif x\dif t
\le \int_{0}^{T}\int_{0}^{k} g_\ve(x) \snorm{\D_t u}^2 (x,t)\,\dif x \dif t.
\end{split}
\end{equation}
\end{lemma}

\begin{proof}
It follows from the definition of  $F$ and  the momentum equation $\eqref{eqs:LFNS-k}_2$ that
\begin{align*}
&F
= \beta \dfrac{\D_x(r^m u)}{v} - (\gamma-1)\dfrac{e-1}{v} + (\gamma-1) \dfrac{v-1}{v},\\
&\D_t u = \beta r^m \D_x\Big(\dfrac{\D_x(r^m u)}{v}\Big) - r^m \D_x( p - (\gamma-1) ) = r^m \D_x F,
\end{align*}
which, along with Theorem \ref{thm:priori}\ref{item:prio1} and Lemmas \ref{lemma:eL2}--\ref{lemma:L2Dxu}, yields that
\begin{align*}
\int_{0}^{k}g_\ve(x) F^2 (x,t)\,\dif x
\le& \, 2\beta^2 \int_{0}^{k}g_\ve \dfrac{\snorm{\D_x(r^m u)}^2}{v^2} \dif x + 4(\gamma-1)^2 \int_{0}^{k}g_\ve \dfrac{\Lambda(v,e)}{v^2} \dif x \\
\le& \,C(\ve) \int_{0}^{k}g_\ve \dfrac{\snorm{\D_x(r^m u)}^2}{v} \dif x + C(\ve),
\end{align*}
and $\eqref{eqs:SL2F}_2$--$\eqref{eqs:SL2F}_3$ obviously.
\end{proof}

\begin{lemma}[Intermediate Step]\label{lemma:uH1int}
\begin{equation}\label{yinlis.9}
\int_{0}^{T}\int_{0}^{k} \sigma g_\ve \snorm{\D_t u}^2\,\dif x \dif s
+\sup_{0\le s\le T}\int_{0}^k\Big(\sigma g_\ve
  \dfrac{\snorm{\D_x(r^m u)}^2}{v}\Big)(x,s)\,\dif x
  \le  C(\ve)\Big\{1
+ \Big(\int_{0}^T\int_{0}^k \sigma^2 g_\ve^2 \snorm{\D_t e}^2\,\dif x \dif s\Big)^{\frac{1}{2}}\Big\}.
\end{equation}
\end{lemma}

\begin{proof} We divide the proof into six steps.

\smallskip
1. Multiplying equations $\eqref{eqs:LFNS-k}_2$ with $g_\ve \sigma \D_t u$ and
then integrating yield
\begin{align*}
&\int_{0}^{t}\int_{0}^{k} \sigma g_\ve \snorm{\D_t u}^2\,\dif x \dif s
+ \dfrac{\beta}{2}\int_{0}^{k} \Big(\sigma g_\ve\dfrac{\snorm{\D_x(r^m u)}^2}{v}\Big)(x,t)\,\dif x \\
&= \dfrac{\beta}{2}\int_{0}^{t}\int_{0}^k \sigma^{\prime} g_\ve \dfrac{\snorm{\D_x(r^m u)}^2}{v} \dif x \dif s + m\beta \int_{0}^{t} \int_{0}^{k} \sigma g_\ve \D_x ( r^{m-1} u^2 ) \dfrac{\D_x(r^m u)}{v} \dif x \dif s \\
&\quad -\int_{0}^t \int_{0}^{k} \sigma g_\ve r^m \D_t u \D_x p\,\dif x \dif s
 + \dfrac{\beta}{2} \int_{0}^{t} \int_{0}^{k} \sigma g_\ve \D_t (v^{-1}) \snorm{\D_x (r^m u)}^2\,\dif x \dif s\\
&\quad-\beta \int_{0}^t \int_{0}^{k} \sigma g_\ve^{\prime} r^m \D_t u \dfrac{\D_x(r^m u)}{v} \dif x \dif s \\
&=\vcentcolon \sum_{i=1}^5 I_i.
\end{align*}
$I_1$ is estimated by using $\sigma(t)=\min\{ 1,t \}$
and Lemma \ref{lemma:L2Dxu} as
\begin{align*}
I_1 \vcentcolon = \dfrac{\beta}{2}\int_{0}^{t}\int_{0}^k \sigma^{\prime} g_\ve \dfrac{\snorm{\D_x(r^m u)}^2}{v} \dif x \dif s \le \dfrac{\beta}{2}\int_{0}^{1}\int_{0}^k g_\ve \dfrac{\snorm{\D_x(r^m u)}^2}{v} \dif x \dif s \le C(\ve).
\end{align*}
$I_2$ is estimated by using Lemmas \ref{lemma:eL2}--\ref{lemma:L2Dxu}, Theorem \ref{thm:priori}\ref{item:prio1}, and \eqref{eqs:g}--\eqref{eqs:grbd} as
\begin{align*}
I_2
&= 2m\beta \int_{0}^{t} \int_{0}^{k} \sigma g_\ve r^{m-1} u \D_x u \dfrac{\D_x(r^m u)}{v} \dif x \dif s + m(m-1)\beta\int_{0}^{t}\int_{0}^{k} \sigma g_\ve \dfrac{v}{r^2} u^2 \dfrac{\D_x(r^m u)}{v} \dif x \dif s \\
&\le  m\beta \int_{0}^{t} \sup_{y\in[\frac{\ve}{2},k]} r^{-2}(y,s)
\big(\int_{0}^{k} g_\ve \dfrac{r^{2m}}{v}\snorm{ u \D_x u}^2\,\dif x\big)\,\dif s
+  m\beta \int_{0}^{t} \int_{0}^{k}g_\ve\dfrac{\snorm{\D_x(r^m u)}^2}{v} \dif x \dif s \\
&\quad + \dfrac{m(m-1)\beta}{2}\Big\{
\int_{0}^{t} \sup_{y\in[\frac{\ve}{2},k]}\dfrac{v}{r^4}(y,t)\big(\int_{0}^{k} g_\ve  u^4\,\dif x\big)\,\dif s
+ \int_{0}^{t}\int_{0}^{k} g_\ve \dfrac{\snorm{\D_x (r^m u)}^2}{v} \dif x \dif s
\Big\}\le C(\ve).
\end{align*}

2. Using $p-\gamma+1 = \frac{(\gamma-1)(e-v)}{v}$, the term $I_3$ is separated into six parts:
\begin{align*}
\dfrac{I_3}{\gamma-1}
&= \int_{0}^{k} \Big(\sigma g_\ve  \D_x(r^m  u) \dfrac{e-v}{v}\Big)(x,t)\,\dif x
  - \int_{0}^t \int_{0}^{k} \sigma^{\prime} g_\ve  \D_x(r^m  u) \dfrac{e-v}{v}
   \dif x \dif s \\
&\quad - \int_{0}^t\int_{0}^{k} \sigma g_\ve  \D_x(r^m  u) \D_t \Big(\dfrac{e-v}{v}\Big)\,\dif x \dif s -\int_{0}^t \int_{0}^{k} \sigma g_\ve \D_x (r^{m-1}  u^2) \dfrac{e-v}{v}\dif x \dif s\\
&\quad + \int_{0}^t \int_{0}^{k} \sigma g_\ve^{\prime} \D_t (r^m  u) \dfrac{e-v}{v}\dif x \dif s
- \int_{0}^t \int_{0}^{k} \sigma g_\ve^{\prime} r^{m-1}  u^2 \dfrac{e-v}{v}\dif x \dif s\\
&= \vcentcolon \sum_{i=1}^6 I_3^{(i)}.
\end{align*}
$I_3^{(1)}$ is estimated by using Theorem \ref{thm:priori}\ref{item:prio1} and Lemma \ref{lemma:eL2}--\ref{lemma:L2Dxu}: \begin{align*}
I_3^{(1)}&\vcentcolon = \int_{0}^{k} \sigma g_\ve  \D_x(r^m u) \dfrac{e-v}{v} \dif x
= \int_{0}^{k} \sigma g_\ve  \D_x(r^m u) \big(\dfrac{e-1}{v} + \dfrac{1-v}{v}\big) \,\dif x\\
&\le \dfrac{1}{\beta} \int_{0}^{k} \sigma g_\ve \dfrac{\Lambda(v,e)}{v} \dif x + \dfrac{\beta}{4}\int_{0}^{k} \sigma g_\ve \dfrac{\snorm{\D_x(r^m u)}^2}{v} \dif x \\
&\le  C(\ve)   + \dfrac{\beta}{4}\int_{0}^{k} \Big(\sigma g_\ve \dfrac{\snorm{\D_x(r^m u)}^2}{v}\Big)(x,t)\,\dif x.
\end{align*}
$I_3^{(2)}$ is estimated by using the definition of $\sigma$, Theorem \ref{thm:priori}\ref{item:prio1}, and Lemma \ref{lemma:eL2}--\ref{lemma:L2Dxu}:
\begin{align*}
I_3^{(2)}
&\vcentcolon= -\int_{0}^t \int_{0}^{k} \sigma^{\prime} g_\ve  \D_x(r^m  u) \dfrac{e-v}{v} \dif x \dif s
= -\int_{0}^t\int_{0}^{k}\sigma^{\prime} g_\ve \D_x(r^m  u)
  \Big( \dfrac{e-1}{v} + \dfrac{1-v}{v} \Big)  \dif x \dif s\\
&\le \int_{0}^{1} \int_{0}^k g_\ve \dfrac{\snorm{\D_x(r^m u)}^2}{v} \dif x \dif s + \dfrac{1}{4}\int_{0}^{1}\int_{0}^{k} g_\ve \dfrac{\Lambda(v,e)}{v} \dif x \dif s \le C(\ve).
\end{align*}
$I_3^{(3)}$ is estimated by using Theorem \ref{thm:priori}\ref{item:prio1} and Lemma \ref{lemma:L2Dxu}:
\begin{align*}
I_3^{(3)}
&= -\int_{0}^t\int_{0}^{k} \sigma g_\ve  \D_x(r^m  u) \Big\{\dfrac{\D_t e}{v} - \dfrac{e}{v^2}\D_x(r^m u)\Big\}\,\dif x \dif s\\
&\le  C(\ve)\Big(\int_{0}^t \int_{\frac{\ve}{2}}^k \dfrac{\snorm{\D_x( r^m u)}^2}{v}\dif x \dif s\Big)^{\frac{1}{2}}\Big(\int_{0}^{t}\int_{0}^k \sigma^2 g_\ve^2 \snorm{\D_t e}^2\,\dif x \dif s\Big)^{\frac{1}{2}}\\
&\quad + C(\ve)\int_{0}^t \sup_{y\in[\frac{\ve}{2},k]} e(y,s)
  \big(\int_{0}^{k}\sigma g_\ve \dfrac{\snorm{\D_x(r^m u)}^2}{v} \dif x\big)\,\dif s \\
&\le  C(\ve)\Big(\int_{0}^{t}\int_{0}^k \sigma^2 g_\ve^2 \norm{\D_t e}^2 \dif x \dif s\Big)^{\frac{1}{2}}
+ C(\ve)\int_{0}^t\Big(\sup_{y\in[\frac{\ve}{2},k]} e(y,s)
\int_{0}^{k}\sigma g_\ve \dfrac{\snorm{\D_x(r^m u)}^2}{v} \dif x\Big)\dif s.
\end{align*}
$I_3^{(4)}$ is estimated by using Theorem \ref{thm:priori}\ref{item:prio1} and Lemmas \ref{lemma:eL2}--\ref{lemma:vL2}:
\begin{align*}
I_3^{(4)}
&=-\int_{0}^t\int_{0}^{k}\sigma g_\ve \Big\{ 2  r^{m-1} u \D_x u  + (n-2) \dfrac{v}{r^2} u^2\Big\}\Big(\dfrac{e-1}{v}+\dfrac{1-v}{v}\Big)\,\dif x \dif s\\
&\le \int_{0}^{t}\int_{0}^{k} \sigma g_\ve \dfrac{r^{2m}}{v}\snorm{u\D_x u}^2\,\dif x \dif s
  + C(\ve)\int_{0}^{t}\int_{0}^{k} g_\ve
 \big\{ u^4 +  \Lambda(v,e) \big\}\,\dif x \dif s \le C(\ve).
\end{align*}
For $I_3^{(5)}$,
we use Theorem \ref{thm:priori}\ref{item:prio1},
Lemmas \ref{lemma:eL2}--\ref{lemma:vL2},
and \eqref{eqs:g}--\eqref{eqs:grbd}
to obtain
\begin{align*}
I_3^{(5)}
&=\int_{0}^t \int_{0}^{k} \sigma g_\ve^{\prime}\big(m r^{m-1} u^2 + r^m \D_t u\big)
  \Big(\dfrac{e-1}{v}+\dfrac{1-v}{v} \Big)\,\dif x \dif s\\
&\le C\Big( \int_{0}^t\int_{\frac{\ve}{2}}^{\ve} \sigma \dfrac{\ve^2}{32}\snorm{g_\ve^{\prime}}^2 \snorm{\D_t u}^2
\,\dif x \dif s \Big)^{\frac{1}{2}}
\Big( \int_{0}^t \sup_{y\in[\frac{\ve}{2},\ve]} \dfrac{r^{2m}}{v^2}(y,s)
\big(\int_{\frac{\ve}{2}}^{\ve} \dfrac{32}{\ve^2}
    \Lambda(v,e)\,\dif x\big)\,\dif s\Big)^{\frac{1}{2}}\\
&\quad + C\Big( \int_{0}^t \sup_{y\in[\frac{\ve}{2},\ve]} r^{-2}(y,s)
 \big(\int_{\frac{\ve}{2}}^{\ve} u^4 \dif x\big)\,\dif s \Big)^{\frac{1}{2}}\Big(\int_{0}^t \sup_{y\in[\frac{\ve}{2},\ve]}\dfrac{r^{2m}}{v^2}(y,s)\big(\int_{\frac{\ve}{2}}^{\ve}\Lambda(v,e)
 \,\dif x\big)\,\dif s \Big)^{\frac{1}{2}}\\
&\le \dfrac{1}{8}\int_{0}^t\int_{0}^k\sigma g_\ve\snorm{\D_t u}^2\,\dif x\dif s +C(\ve),
\end{align*}
and $I_3^{(5)}$
is estimated by using Theorem \ref{thm:priori}\ref{item:prio1}, Lemma \ref{lemma:eL2}--\ref{lemma:vL2}, and \eqref{eqs:g}--\eqref{eqs:grbd}.
\begin{align*}
I_3^{(5)}
&\le \dfrac{1}{2}\int_{0}^t \sup_{y\in[\frac{\ve}{2},\ve]} \dfrac{r^{2(m-1)}}{v^2}(y,s)
  \big(\int_{\frac{\ve}{2}}^{\ve} \snorm{g_\ve^{\prime}} u^4\,\dif x\big)\,\dif s
    + \dfrac{1}{2} \int_{0}^t \int_{\frac{\ve}{2}}^{\ve} \snorm{g_\ve^{\prime}} \Lambda(v,e)
\,\dif x \dif s \le C(\ve).
\end{align*}
Combining the estimates of $I_3^{(1)}$--$I_3^{(6)}$,
it follows that
\begin{align*}
		I_3
\le &\,  \dfrac{\beta}{4} \int_{0}^k\Big(\sigma g_\ve \dfrac{\snorm{\D_x(r^m u)}^2}{v}\Big)(x,t)\,\dif x 
+\dfrac{1}{8}\int_{0}^{t}\int_{0}^{k}\sigma g_\ve\snorm{\D_t u}^2 \dif x \dif s
+ C(\ve)\\
&+ C(\ve)\Big(\int_{0}^{t}\int_{0}^k \sigma g_\ve^2 \snorm{\D_t e}^2
		  \dif x \dif s\Big)^{\frac{1}{2}} 
  + C(\ve)\int_{0}^t \sup_{y\in[\frac{\ve}{2},k]} e(y,s)
  \big(\int_{0}^{k}\Big(\sigma g_\ve \dfrac{\snorm{\D_x(r^m u)}^2}{v}\Big)(x,s)\, \dif x\big) \dif s.
\end{align*}

4. $I_4$ is rewritten into three parts by using the continuity equation $\eqref{eqs:LFNS-k}_1$ and the definition of $F$:
	\begin{align*}
		I_4
		= & -\dfrac{1}{2\beta} \int_{0}^t \int_{0}^k \sigma g_\ve F^2  \D_x(r^m u) \,\dif x \dif s
		-\dfrac{\gamma-1}{\beta} \int_{0}^t \int_{0}^k \sigma g_\ve  F \dfrac{e-v}{v} \D_x(r^m u)\,\dif x \dif s \\
		& -\dfrac{(\gamma-1)^2}{2\beta} \int_{0}^t \int_{0}^k \sigma g_\ve
		  \Big(\dfrac{e-v}{v}\Big)^2 \D_x(r^m u)\,\dif x \dif s\\
		\vcentcolon=&\, \sum_{i=1}^3 I_4^{(i)}.
	\end{align*}
$I_4^{(1)}$ is estimated by using integration by parts, Lemma \ref{lemma:estF},
and \eqref{eqs:g}--\eqref{eqs:grbd}.
\begin{align*}
\beta I_4^{(1)}
&= \int_{0}^t \int_{0}^k \sigma g_\ve u F  r^m \D_x F\,\dif x \dif s
+\dfrac{1}{2}\int_{0}^t\int_{0}^k\sigma g_\ve^{\prime} F^2 r^m u\,\dif x\dif s\\
&\le  2 \int_{0}^{t}\int_{0}^{k} \sigma g_\ve u^2 F^2\,\dif x \dif s
 + \dfrac{1}{8} \int_{0}^{t}\int_{0}^{k} \sigma g_\ve r^{2m} \snorm{\D_x F}^2\,\dif x \dif s\\
&\quad +\dfrac{1}{4}\int_{0}^t \int_{\frac{\ve}{2}}^{\ve}\sigma \dfrac{\ve^2}{32}\snorm{g_\ve^{\prime}}^2
 \snorm{u}^2 F^2\,\dif x \dif s
 + \dfrac{1}{4}\int_{0}^t \int_{\frac{\ve}{2}}^{\ve} \dfrac{32}{\ve^2} r^{2m} F^2\,\dif x \dif s \\
&\le   \dfrac{1}{8} \int_{0}^{t}\int_{0}^{k} \sigma g_\ve \snorm{\D_t u}^2\,\dif x \dif s
  +C(\ve) + C(\ve)\int_{0}^{t} \sup_{y\in[\frac{\ve}{2},k]}\snorm{u}^2(y,s)
   \big(\int_{0}^k \sigma g_\ve \dfrac{\snorm{\D_x(r^m u)}^2}{v} \dif x\big)\dif s.
\end{align*}
$I_4^{(2)}$ is estimated by using Theorem \ref{thm:priori}\ref{item:prio1}
and
Lemmas \ref{lemma:eL2}--\ref{lemma:estF}:
\begin{align*}
\dfrac{\beta}{\gamma-1}I_4^{(2)}
&\le  \int_{0}^t \int_{0}^k \sigma g_\ve F^2\,\dif x \dif s
+ \int_{0}^{t} \sup_{y\in[\frac{\ve}{2},k]} \dfrac{(e-v)^2}{v}
  \big(\int_{0}^{k} \sigma g_\ve \dfrac{\snorm{\D_x(r^m u)}^2}{v} \dif x\big)\,\dif s\\
&\le  C(\ve) + C(\ve)\int_{0}^{t}\sup_{y\in[\frac{\ve}{2},k]} \snorm{e(y,s)}^2 \big(\int_{0}^{k} \sigma g_\ve
 \dfrac{\snorm{\D_x(r^m u)}^2}{v} \dif x\big)\,\dif s.
	\end{align*}
$I_4^{(3)}$ is estimated by using Theorem \ref{thm:priori}\ref{item:prio1}
and Lemma \ref{lemma:eL2}--\ref{lemma:L2Dxu}:
\begin{align*}
\dfrac{2\beta}{(\gamma-1)^2}I_4^{(3)}
&\le 2 \int_{0}^t \int_{0}^k \sigma g_\ve \dfrac{\Lambda(v,e)}{v^2} \norm{\D_x(r^m u)}\, \dif x \dif s \\
&\le \dfrac{(\gamma-1)^2}{2\beta} \int_{0}^t \int_{0}^k \sigma g_\ve \dfrac{\Lambda(v,e)^2}{v^3} \dif x \dif s + \dfrac{(\gamma-1)^2}{2\beta} \int_{0}^t\int_{0}^k \sigma g_\ve \dfrac{\snorm{\D_x(r^m u)}^2}{v} \dif x \dif s \\
&\le  C(\ve) + C(\ve) \int_{0}^{t}
\big(\sup_{y\in[\frac{\ve}{2},k]} \snorm{e-1}^2(y,s)+1\big)
\big(\int_{0}^{k} \sigma g_\ve \snorm{e-1}^2\,\dif x\big)\,\dif s \\
&\le  C(\ve) + C(\ve) \int_{0}^{t} \sup_{y\in[\frac{\ve}{2},k]} e^2(y,s)\,\dif s\le C(\ve).
\end{align*}
Combining the estimates for $I_4^{(1)}$--$I_4^{(3)}$, it follows that
\begin{align*}
I_4 \le  \dfrac{1}{8} \int_{0}^{t}\int_{0}^{k} \sigma g_\ve \snorm{\D_t u}^2\dif x \dif s
+ C(\ve)
+C(\ve) \int_{0}^t\sup_{y\in[\ve/2,k]} ( e^2 + u^2)(y,s)
  \big(\int_{0}^{k} \sigma g_\ve \dfrac{\snorm{\D_x(r^m u)}^2}{v} \dif x\big)\,\dif s.
\end{align*}

5. $I_5$
is estimated by using Theorem \ref{thm:priori}\ref{item:prio1}, Lemmas \ref{lemma:eL2}--\ref{lemma:L2Dxu},
and \eqref{eqs:g}--\eqref{eqs:grbd}: 
\begin{align*}
I_5
&\le \beta^2 \int_{0}^{t} \sup_{y\in[\frac{\ve}{2},\ve]} \dfrac{r^{2m}}{v}(y,s)
  \big(\int_{\frac{\ve}{2}}^{\ve} \dfrac{32}{\ve^2} \dfrac{\snorm{\D_x(r^m u)}^2}{v}\dif x\big)\,\dif s
  + \dfrac{1}{4}\int_{0}^t \int_{\frac{\ve}{2}}^{\ve}\sigma\dfrac{\ve^2}{32}\snorm{g_\ve^{\prime}}^2
  \snorm{\D_t u}^2\,\dif x\dif t \\
&\le \dfrac{1}{4}\int_{0}^t \int_{\frac{\ve}{2}}^{\ve} \sigma g_\ve \snorm{\D_t u}^2\,\dif x \dif t + C(\ve).
\end{align*}

\smallskip
6. Combining the estimates of $I_1$--$I_5$
together, it follows that
\begin{align*}
&\dfrac{1}{2}\int_{0}^{t}\int_{0}^{k} \sigma g_\ve \snorm{\D_t u}^2\,\dif x \dif s
+ \dfrac{\beta}{4} \int_{0}^k \Big(\sigma g_\ve\dfrac{\snorm{\D_x(r^m u)}^2}{v}\Big)(x,t)\,\dif x\\
&\le C(\ve) + C(\ve)\Big(\int_{0}^{t}\int_{0}^k \sigma^2 g_\ve^2 \snorm{\D_t e}^2\,\dif x \dif s \Big)^{\frac{1}{2}}\\
&\quad
+ C(\ve) \int_{0}^t \sup_{y\in[\frac{\ve}{2},k]}(e+e^2+u^2)(y,s)
  \Big(\int_{0}^{k}\big(\sigma g_\ve \dfrac{\snorm{\D_x(r^m u)}^2}{v}\big)(x,s)\,\dif x\Big)\,\dif s.
\end{align*}
Applying Gr\"onwall's inequality, we have
\begin{align*}
&\dfrac{1}{2}\int_{0}^{t}\int_{0}^{k} \sigma g_\ve \snorm{\D_t u}^2\,\dif x \dif s
  + \dfrac{\beta}{4} \int_{0}^k
  \Big(\sigma g_\ve \dfrac{\snorm{\D_x(r^m u)}^2}{v}\Big)(x,t)\,\dif x\\
&\le C(\ve)\Big(1+ \Big(\int_{0}^t\int_{0}^k \sigma^2 g_\ve^2
\snorm{\D_t e}^2\,\dif x \dif s\Big)^{\frac{1}{2}}\Big)\Big(1+ \big(\int_{0}^t h_\ve(s)\,\dif s\big)
\exp\big(\int_{0}^t h_{\ve}(s)\,\dif s\big)\Big),
\end{align*}
where $h_\ve(s)\vcentcolon=C(\ve)\sup_{y\in[\frac{\ve}{2},k]}(e+e^2+u^2)(y,s)$.
By Lemmas \ref{lemma:euLinf}--\ref{lemma:eL2}, we obtain \eqref{yinlis.9}.
\end{proof}

\begin{lemma}[$L^{\infty}(0,T;L^2(\ve,k))$--Estimates of $\D_x u$ and $\D_x e$]\label{lemma:euH1}
For each $\ve>0$,
\begin{align}
&\int_{0}^{t}\int_{\ve}^{k} \big(\sigma^2 \snorm{\D_t e}^2+\sigma g_\ve \snorm{\D_t u}^2\big)
\,\dif x \dif s
 + \sup_{ 0\le s\le t}\int_{\ve}^k \Big(\sigma^2 \dfrac{r^{2m}}{v} \snorm{\D_x e}^2+\sigma \dfrac{\snorm{\D_x(r^m u)}^2}{v}\Big)(x,s)\,\dif x
 \le  C(\ve).\label{eqs:euH1}
\end{align}
\end{lemma}

\begin{proof}  We divide the proof into six steps.

\smallskip
1. Rewriting the energy equation $\eqref{eqs:LFNS-k}_3$ in terms of the effective
viscosity flux $F$, then multiplying both sides by $\sigma^2(t) g_\ve^2 \D_t e$, and integrating by parts, it follows that
\begin{align*}
&\int_{0}^{t}\int_{0}^{k} \sigma^2 g_\ve^2 \snorm{\D_t e}^2\,\dif x \dif s
+\dfrac{\kappa}{2}\int_{0}^k\Big(\sigma^2 g_\ve^2\dfrac{r^{2m}}{v}\snorm{\D_x e}^2\Big)(x,t)\,\dif x\\
&= \kappa\int_{0}^{t} \int_{0}^k \sigma\sigma^{\prime}(s) g_\ve^2 \dfrac{r^{2m}}{v}
 \snorm{\D_x e}^2\,\dif x \dif s
 + m\kappa\int_{0}^{t} \int_{0}^k \sigma^2 g_\ve^2 \dfrac{r^{2m-1}u}{v}
   \snorm{\D_x e}^2\,\dif x\dif s \\
&\quad + \int_{0}^{t}\int_{0}^k \sigma^2 g_\ve^2 \mathcal{A} \D_t e\,\dif x \dif s -\dfrac{\kappa}{2}\int_{0}^{t} \int_{0}^k \sigma^2 g_\ve^2 \dfrac{r^{2m}}{v^2} \D_x(r^m u) \snorm{\D_x e}^2\,\dif x \dif s \\
&\quad +\int_{0}^{t}\int_{0}^{k} \sigma^2 g_\ve^2 F \D_x(r^m u) \D_t e\,\dif x \dif s
 - \kappa \int_{0}^t \int_{0}^{k} 2\sigma^2 g_\ve g_\ve^{\prime} \dfrac{r^{2m}}{v}
 \D_t e \D_x e\,\dif x \dif s\\
&\vcentcolon= \sum_{i=1}^6 I_i,
\end{align*}
where $\mathcal{A}\vcentcolon=-(\gamma-1) \D_x(r^m u)- 4m\mu\dfrac{u}{r}\D_x(r^m u)
     + 2m n\mu \dfrac{v}{r^2} u^2$.
By Lemma \ref{lemma:eL2}, we have
\begin{equation*}
I_1 \vcentcolon = \kappa \int_{0}^{t} \int_{0}^k \sigma\sigma^{\prime}(s) g_\ve^2
\dfrac{r^{2m}}{v} \snorm{\D_x e}^2\,\dif x \dif s
\le \kappa \int_{0}^{1} \int_{\frac{\ve}{2}}^k \dfrac{r^{2m}}{v} \snorm{\D_x e}^2\,\dif x \dif s \le C(\ve).
\end{equation*}
$I_2$ is estimated by using \eqref{eqs:g}--\eqref{eqs:grbd}
and Lemma \ref{lemma:eL2}:
\begin{align*}
I_2
&\le 2m\kappa\int_{0}^{t} \sup_{y\in[\frac{\ve}{2},k]}\dfrac{u}{r}(y,s)
\big(\int_{0}^k \sigma^2 g_\ve^2 \dfrac{r^{2m}}{v}\,\snorm{\D_x e}^2\,\dif x\big) \dif s\\
&\le C(\ve) + C(\ve) \int_{0}^{t} \sup_{y\in[\frac{\ve}{2},k]}u^2(y,s)
  \big(\int_{0}^k \sigma^2 g_\ve^2 \dfrac{r^{2m}}{v} \snorm{\D_x e}^2\,\dif x\big)\,\dif s.
\end{align*}

\smallskip
2. $I_3$ is estimated by using Theorem \ref{thm:priori}\ref{item:prio1} and
Lemmas \ref{lemma:eL2}--\ref{lemma:L2Dxu}:
\begin{equation}\label{HEtemp4}
\begin{aligned}
I_3
&=\int_{0}^{t}\int_{0}^k \sigma^2 g_\ve^2
 \Big\{-(\gamma-1) \D_x(r^m u) - 4m\mu\dfrac{u}{r}\D_x(r^m u) + 2m n \mu \dfrac{v}{r^2} u^2\Big\} \D_t e\,\dif x \dif s \\
&\le \dfrac{1}{16} \int_{0}^{t}\int_{0}^{k}\sigma^2 g_\ve^2 \snorm{\D_t e}^2\,\dif x \dif s
  + C(\ve) \int_{0}^t \sup_{y\in[\frac{\ve}{2},k]}(1+u^2)(y,s)
\int_{0}^{k}\sigma^2 g_\ve^2 \dfrac{\snorm{\D_x(r^m u)}^2}{v} \dif x \dif s \\
&\quad + C(\ve)\int_{0}^{T}\int_{0}^{k}\sigma^2 g_\ve^2 u^4\,\dif x \dif t\\
&\le \dfrac{1}{16}\int_{0}^t \int_{0}^k \sigma^2 g_\ve^2 \snorm{\D_t e}^2\, \dif x \dif t
+C(\ve)\Big(1+\int_{0}^{t}\sup_{y\in[\frac{\ve}{2},k]}u^2(y,s)
 \big(\int_{0}^k \sigma^2 g_\ve \dfrac{\snorm{\D_x(r^m u)}^2}{v}\dif x\big)\,\dif s\Big).
\end{aligned}
\end{equation}
In fact, by Lemmas \ref{lemma:euLinf} and \ref{lemma:uH1int}, it follows that
\begin{align*}
\int_{0}^{t} \sup_{y\in[\frac{\ve}{2},k]} u^2(y,s)
\big(\int_{0}^k \sigma^2 g_\ve \dfrac{\snorm{\D_x(r^m u)}^2}{v}\dif x\big)\,\dif s
&\le \sup_{ 0\le s \le  t}\int_{0}^k\Big(\sigma g_\ve \dfrac{\snorm{\D_x(r^m u)}^2}{v}\Big)(x,s)\,
 \dif x\int_{0}^{t} \sup_{y\in[\frac{\ve}{2},k]} u^2(y,\tilde{t}\,)\,\dif \tilde{t}\\
&\le C(\ve) + C(\ve)\Big(\int_{0}^t\int_{0}^k\sigma^2 g_\ve^2\snorm{\D_t e}^2\,\dif x \dif s\Big)^{\frac{1}{2}}\\
&\le  \dfrac{1}{16} \int_{0}^t \int_{0}^k \sigma^2 g_\ve^2 \snorm{\D_t e}^2\, \dif x \dif s + C(\ve).
\end{align*}
Substituting the above estimate into \eqref{HEtemp4} leads to
\begin{align*}
I_3 \le \dfrac{1}{8}\int_{0}^t\int_{0}^k\sigma^2 g_\ve^2 \snorm{\D_t e}^2\,\dif x \dif s+C(\ve).
\end{align*}

\smallskip
3.
$I_4$ is first rewritten in terms of $F$ and then estimated by using Theorem \ref{thm:priori}\ref{item:prio1}:
\begin{equation*}
\begin{aligned}
I_4 \vcentcolon
&=-\dfrac{\kappa}{2\beta}\int_{0}^t \int_{0}^k \sigma^2 g_\ve^2
 \Big(F + (\gamma-1)\dfrac{e-v}{v} \Big)\dfrac{\snorm{r^m \D_x e}^2}{v}\,\dif x \dif s \\
&\le  \dfrac{\kappa}{2\beta}\int_{0}^t
 \sup_{y\in[\frac{\ve}{2},k]} \snorm{\sqrt{g_\ve} F}(y,s)
\big(\int_{0}^k \sigma^2 g_\ve^{\frac{3}{2}} \dfrac{r^{2m}}{v} \snorm{\D_x e}^2\,\dif x)\,\dif s + C(\ve) \\
&\quad + C(\ve)\int_{0}^t \sup_{y\in[\frac{\ve}{2},k]}e(y,s)
 \big(\int_{0}^k \sigma^2 g_\ve^2 \dfrac{r^{2m}}{v} \snorm{\D_x e}^2\,\dif x\big)\,\dif s.
\end{aligned}
\end{equation*}
By Lemma \ref{lemma:estF} and the Sobolev embedding theorem: $W^{1,1}(\ve,k) \xhookrightarrow[]{} C^0(\ve,k)$, it follows that
\begin{equation}\label{HEtemp2}
\begin{aligned}
\sup_{y\in[\frac{\ve}{2},k]} \snorm{\sqrt{g_\ve} F}^2(y,s)
&\le C \int_{\frac{\ve}{2}}^k\big( g_\ve F^2 + 2 g_\ve \snorm{F \D_x F} + 2 g_\ve^{\prime} F^2 \big) (x,s)\,\dif x \\
&\le \int_{\frac{\ve}{2}}^k \big( C(\ve) F^2 + g_\ve r^{2m}  \snorm{\D_x F}^2 \big)\,\dif x \\
&\le  \int_{\frac{\ve}{2}}^{k} \Big(C(\ve)\dfrac{\snorm{\D_x(r^m u)}^2}{v}
  + g_\ve \snorm{\D_t u}^2 \Big)\,\dif x + C(\ve),
\end{aligned}
\end{equation}
so that, using Lemma \ref{lemma:L2Dxu}
and the Cauchy-Schwartz inequality, we have
\begin{align*}
&\dfrac{\kappa}{\beta}\int_{0}^t \sup_{y\in[\frac{\ve}{2},k]} \snorm{g_\ve^{\frac{1}{2}} F}(y,s)
  \big(\int_{0}^k \sigma^2 g_\ve^{\frac{3}{2}} \dfrac{r^{2m}}{v} \snorm{\D_x e}^2 \dif x\big)\dif s\\
&\le C(\ve)\int_{0}^t \Big\{ \sigma^2  +\int_{\frac{\ve}{2}}^{k}
 \sigma^2 \dfrac{\snorm{\D_x(r^m u)}^2}{v}\,\dif x
 + \int_{\frac{\ve}{2}}^{k} \sigma^2 g_\ve \snorm{\D_t u}^2\,\dif x \Big\} \dif s \\
&\quad + C(\ve) \int_{0}^{t}\sigma^2
\Big( \int_{0}^k g_\ve^{\frac{3}{2}} \dfrac{r^{2m}}{v}\snorm{\D_x e}^2\,\dif x \Big)^2 \dif s\\
&\le  \int_{0}^t \int_{0}^{k} \sigma g_\ve \snorm{\D_t u}^2\,\dif x \dif s + C(\ve) 
+C(\ve)\int_{0}^t\Big(\int_{0}^k g_\ve\dfrac{r^{2m}}{v}\snorm{\D_x e}^2\,\dif x \Big)
 \Big(\int_{0}^k\sigma^2g_\ve^2 \dfrac{r^{2m}}{v}\snorm{\D_x e}^2\,\dif x\Big)\dif s.
\end{align*}
Using the above estimate, Lemma \ref{lemma:uH1int}, and the Cauchy-Schwartz inequality,
one has
\begin{align*}
I_4
&\le C(\ve) + \int_{0}^t \int_{0}^{k} \sigma g_\ve \snorm{\D_t u}^2\,\dif x \dif s
 + C(\ve)\int_{0}^t \sup_{y\in[\frac{\ve}{2},k]}e(y,s)
 \big(\int_{0}^k \sigma^2 g_\ve^2 \dfrac{r^{2m}}{v} \snorm{\D_x e}^2\,\dif x\big) \dif s\\
&\quad +C(\ve) \int_{0}^t \Big(\int_{0}^k g_\ve\dfrac{r^{2m}}{v}\snorm{\D_x e}^2\,\dif x \Big) \Big(\int_{0}^k \sigma^2 g_\ve^2 \dfrac{r^{2m}}{v}\snorm{\D_x e}^2\,\dif x\Big) \dif s,\\
&\le   \dfrac{1}{8}\int_{0}^t \int_{0}^{k} \sigma^2 g_\ve^2 \snorm{\D_t e}^2\,\dif x \dif s + C(\ve) 
+ C(\ve)\int_{0}^t \sup_{y\in[\frac{\ve}{2},k]}e(y,s)
 \big(\int_{0}^k \sigma^2 g_\ve^2 \dfrac{r^{2m}}{v} \snorm{\D_x e}^2\, \dif x\big)\dif s\\
&\quad +C(\ve) \int_{0}^t \Big(\int_{0}^k g_\ve\dfrac{r^{2m}}{v}\snorm{\D_x e}^2\,\dif x \Big)\Big(\int_{0}^k\sigma^2 g_\ve^2\dfrac{r^{2m}}{v}\snorm{\D_x e}^2\,\dif x\Big)\dif s.
\end{align*}

\smallskip
4. $I_5$ is first estimated as follows:
\begin{align}\label{HEtemp3}
\begin{aligned}
I_5
&\le 4\int_{0}^t \sup_{y\in[\frac{\ve}{2},k]} (g_\ve v \snorm{F}^2)(y,s)
 \big(\int_{0}^k \sigma^2 g_\ve \dfrac{\snorm{\D_x(r^m u)}^2}{v} \dif x\big)\dif s
  + \dfrac{1}{16}\int_{0}^t \int_{0}^k \sigma^2 g_\ve^2 \snorm{\D_t e}^2\,\dif x \dif s.
\end{aligned}
\end{align}
By the same argument as \eqref{HEtemp2}, we have
\begin{align*}
\sup_{y\in[\frac{\ve}{2},k]}(g_\ve \snorm{F}^2)(y,s)
\le C(\ve) + C(\ve)\int_{\frac{\ve}{2}}^{k} \dfrac{\snorm{\D_x(r^m u)}^2}{v} \dif x
+ C(\ve)\Big( \int_{0}^{k}g_\ve F^2\,\dif x \Big)^{\frac{1}{2}}
\Big(\int_{0}^k g_\ve \snorm{\D_t u}^2\,\dif x\Big)^{\frac{1}{2}},
\end{align*}
which, along with Theorem \ref{thm:priori}\ref{item:prio1} and Lemma \ref{lemma:L2Dxu}, implies that
\begin{align*}
&4\int_{0}^t \sup_{y\in[\frac{\ve}{2},k]}(g_\ve v \snorm{F}^2)(y,s)
\big(\int_{0}^k \sigma^2 g_\ve \dfrac{\snorm{\D_x(r^m u)}^2}{v} \dif x\big)\dif s\\
&\le  C(\ve)
 + C(\ve)\sup_{0\le s \le t}\int_{0}^k
 \Big(\sigma g_\ve \dfrac{\snorm{\D_x(r^m u)}^2}{v}\Big)(x,s)\,\dif x\\
&\quad + C(\ve)\Big( \int_{0}^t\int_{0}^{k} \sigma g_\ve F^2 \Big)^{\frac{1}{2}}
 \Big(\int_{0}^t\int_{0}^k \sigma g_\ve \snorm{\D_t u}^2\Big)^{\frac{1}{2}}
 \sup_{ 0\le s \le  t}\int_{0}^k \Big(\sigma g_\ve \dfrac{\snorm{\D_x(r^m u)}^2}{v}\Big)(x,s)\,\dif x.
\end{align*}
Then applying Lemmas \ref{lemma:estF}--\ref{lemma:uH1int} on the above
and using the Cauchy-Schwartz inequality yield
\begin{align*}
&4\int_{0}^t \sup_{y\in[\frac{\ve}{2},k]} (g_\ve v \snorm{F}^2)(y,s)
 \big(\int_{0}^k \sigma^2 g_\ve \dfrac{\snorm{\D_x(r^m u)}^2}{v} \dif x\big) \dif s\\
&\le  C(\ve)
+C(\ve)\Big\{1+\Big(\int_{0}^t\int_{0}^k\sigma^2g_\ve^2\snorm{\D_t e}^2\,\dif x\dif s\Big)^{\frac{1}{2}}\Big\}^{\frac{1}{2}}
\Big\{ 1+\Big(\int_{0}^t\int_{0}^k \sigma^2 g_\ve^2 \snorm{\D_t e}^2\,
  \dif x \dif s\Big)^{\frac{1}{2}} \Big\}\\
&\le  \dfrac{1}{16}\int_{0}^t\int_{0}^k \sigma^2 g_\ve^2 \snorm{\D_t e}^2\,\dif x \dif s + C(\ve).
\end{align*}
Substituting the above into \eqref{HEtemp3}, we obtain
\begin{align*}
I_5 \le \dfrac{1}{8}\int_{0}^t\int_{0}^k \sigma^2 g_\ve^2 \snorm{\D_t e}^2\,\dif x \dif s + C(\ve).
\end{align*}

\smallskip
5. $I_6$
is estimated by using Theorem \ref{thm:priori}\ref{item:prio1}, Lemma \ref{lemma:eL2}, and \eqref{eqs:g}--\eqref{eqs:grbd}:
\begin{align*}
I_6
&\le \dfrac{1}{8} \int_{0}^t \int_{\frac{\ve}{2}}^{\ve} \sigma^2 g_\ve^2 \snorm{\D_t e}^2\,\dif x \dif s
 + 8\kappa^2 \int_{0}^t \sup_{y\in[\frac{\ve}{2},\ve]}\dfrac{r^{2m}}{v}(y,s)
   \big(\int_{\frac{\ve}{2}}^{\ve} \sigma^2 \snorm{g_\ve^{\prime}}^2 \dfrac{r^{2m}}{v}
 \snorm{\D_x e}^2\,\dif x\big)\,\dif s\\
&\le   \dfrac{1}{8} \int_{0}^t \int_{0}^{k} \sigma^2 g_\ve^2 \snorm{\D_t e}^2\,\dif x \dif s + C(\ve).
\end{align*}

\smallskip
6. Combining
the estimates for $I_1$--$I_6$ together, we have
\begin{align}\label{HEtemp5}
&\dfrac{1}{2}\int_{0}^{t}\int_{0}^{k} \sigma^2 g_\ve^2 \snorm{\D_t e}^2\,\dif x \dif s
+\dfrac{\kappa}{2}\int_{0}^k\Big(\sigma^2 g_\ve^2\dfrac{r^{2m}}{v}\snorm{\D_x e}^2\Big)(x,t)\,\dif x \nonumber\\
&\le  C(\ve) + C(\ve) \int_{0}^t \tilde{h}_\ve(s)
 \big(\int_{0}^k\sigma^2 g_\ve^2 \dfrac{r^{2m}}{v}\snorm{\D_x e}^2\,\dif x\big)\,\dif s,
\end{align}
where $\tilde{h}_\ve(s) \vcentcolon
= \int_{0}^k \big(g_\ve\dfrac{r^{2m}}{v}\norm{\D_x e}^2\big)(x,s)\,\dif x
   +\sup_{y\in[\frac{\ve}{2},k]}(u^2+e)(y,s)$.
Then it follows from  Lemmas \ref{lemma:euLinf}--\ref{lemma:eL2}
and the Gr\"onwall inequality that
\begin{align*}
\dfrac{1}{2}\int_{0}^{t}\int_{0}^{k} \sigma^2 g_\ve^2 \snorm{\D_t e}^2\,\dif x \dif s
+\dfrac{\kappa}{2}\int_{0}^k\Big(\sigma^2 g_\ve^2\dfrac{r^{2m}}{v}\snorm{\D_x e}^2\Big)(x,t)\,
\dif x\le C(\ve),
\end{align*}
which, along  with Lemma \ref{lemma:uH1int}, yields \eqref{eqs:euH1}.
This completes the proof.
\end{proof}

\begin{corollary}\label{coro:eubd}
For all $(x,t) \in [\ve,k]\times[0,T]$,
\begin{equation}\label{eqs:ubd}
\sigma^{\frac{1}{4}}(t)\snorm{u(x,t)}+\sigma^{\frac{1}{2}}(t) e(x,t)\le C(\ve).
\end{equation}
\end{corollary}

\begin{proof}
First, repeating the same calculation \eqref{eqs:uSob} in the proof of Lemma \ref{lemma:euLinf},  one has
\begin{align*}
\sigma^{\frac{1}{2}}(t)\sup_{y\in[\ve,k]}u^2(y,t)
\le C(\ve) \int_{\ve}^{k} u^2(x,t)\,\dif x + C \int_{\ve}^{k}\sigma(t) \dfrac{\snorm{\D_x(r^mu)}^2}{v}(x,t) \dif x \le C(\ve),
\end{align*}
where we have used Lemma \ref{lemma:euH1}.
Then, by the Sobolev embedding theorem, Theorem \ref{thm:priori}\ref{item:prio1},
Lemmas \ref{lemma:eL2} and \ref{lemma:euH1}, and \eqref{eqs:g}--\eqref{eqs:grbd},
we have
\begin{align*}
\sigma(t) \sup_{y\in[\ve,k]}(e-1)^2(y,t)
&\le C \int_{\ve}^k \sigma(t) (e-1)^2(x,t)\,\dif x
+ C \int_{\ve}^{k} \sigma(t) (\snorm{e-1} \snorm{\D_x e}) (x,t)\,\dif x\\
&\le C \sup_{y\in[\ve,k]}\Big( \sigma(t) + \dfrac{v}{r^{2m}}(y,t)\Big)
\int_{\ve}^{k}  (e-1)^2 (x,t)\,\dif x  \\
&\quad\, + \int_{\ve}^k \sigma^2(t)
\Big(\dfrac{r^{2m}}{v} \snorm{\D_x e}^2\Big)(x,t)\,\dif x \le C(\ve),
\end{align*}
which implies $\eqref{eqs:ubd}_2$.
\end{proof}

\subsection{Lower bound of the internal energy}\label{subsec:elbd}
We now estimate the lower bound of
$e$.

\begin{lemma}[Lower Bound of the Internal Energy $e$] \label{lemma:ebd}
There exists $C(a)>0$ such that
\begin{equation} \label{eqs:ebd}
e(x,t)\ge C(a)^{-1} \qquad \text{for all $(x,t)\in [0,\infty)\times[0,T]$}. 	\end{equation}
\end{lemma}

\begin{proof}
Multiplying equation $\eqref{eqs:LFNS-k}_3$ with $-e^{-2}$ to obtain
\begin{equation}\label{eqs:e-1}
\D_t e^{-1}
= (\gamma-1) \dfrac{\D_x(r^m u)}{e v}
  -\kappa e^{-2} \D_x\Big(\dfrac{r^{2m}\D_x e}{v}\Big)
   +\dfrac{2\mu m}{e^2} \D_x(r^{m-1} u^2)
   - \beta \dfrac{\snorm{\D_x(r^m u)}^2}{e^2v},
\end{equation}
where $\beta:= 2\mu + \lambda$.
To further reduce the equation, we write
\begin{equation*}
- \kappa e^{-2} \D_x\Big( \dfrac{r^{2m}}{v}\D_x e \Big) = \kappa \D_x\Big( \dfrac{r^{2m}}{v}\D_x e^{-1} \Big) - 2\kappa \dfrac{r^{2m}\snorm{\D_x e}^2}{v e^3} .
\end{equation*}
Moreover, using the relation $\D_x(r^{m-1}u^2) = \frac{2u}{r}\D_x(r^m u) - n\frac{vu^2}{r^2}$,
we have
\begin{align*}
-\beta\dfrac{\snorm{\D_x(r^m u)}^2}{e^2v} + \dfrac{2\mu m}{e^2} \D_x(r^{m-1} u^2)
&= -\Big(\dfrac{2\mu}{n} + \lambda\Big)\dfrac{\snorm{\D_x(r^m u)}^2}{e^2 v} - 2 m \mu \dfrac{v}{e^2} \Big(\dfrac{\D_x(r^m u)}{\sqrt{n} v} - \sqrt{n}\dfrac{u}{r} \Big)^2\\
&\le -\Big(\dfrac{2\mu}{n} + \lambda\Big)\dfrac{\snorm{\D_x(r^m u)}^2}{e^2 v}.
\end{align*}
Therefore, using the above two inequalities, \eqref{eqs:e-1} can be reduced to the inequality:
\begin{equation}\label{eqs:eIneq}
\D_t e^{-1} + 2\kappa \dfrac{r^{2m}\snorm{\D_x e}^2}{v e^3}+\Big(\dfrac{2\mu}{n} + \lambda\Big)\dfrac{\snorm{\D_x(r^m u)}^2}{e^2 v} \le \dfrac{(\gamma-1)\D_x(r^m u)}{e v} + \kappa \D_x\Big( \dfrac{r^{2m}}{v}\D_x e^{-1} \Big) .
\end{equation}
Multiplying \eqref{eqs:eIneq} by $j e^{-j+1}$ for $j\ge 2$ integers and integrating in $x\in[0,k]$ yield
\begin{align*}
&\dfrac{\dif }{\dif t} \int_{0}^{k} e^{-j} \dif x + j\Big(\dfrac{2\mu}{n}+\lambda\Big)\int_{0}^{k}\dfrac{\snorm{\D_x(r^m u)}^2}{e^{j+1}v} \dif x + 2j\kappa\int_{0}^{k}\dfrac{r^{2m}\snorm{\D_x e}^2}{ve^{j+2}} \dif x \\
&\le j (\gamma-1) \int_{0}^{k}  e^{-j} \dfrac{\D_x(r^m u)}{v} \dif x - j(j-1)\kappa \int_{0}^{k}  \dfrac{r^{2m}\snorm{\D_x e^{-1}}^2}{ve^{j-2}} \dif x \\
&\le C \int_{0}^k v^{-1} e^{1-j} \dif x + \dfrac{j}{2}\Big( \dfrac{2\mu}{n} + \lambda \Big)
\int_{0}^k \dfrac{\snorm{\D_x(r^m u)}^2}{e^{j+1}v} \dif x - j(j-1)\kappa \int_{0}^{k}  \dfrac{r^{2m}\snorm{\D_x e^{-1}}^2}{ve^{j-2}} \dif x.
\end{align*}
Rearranging the above and using Lemma \ref{coro:eubd} and Theorem \ref{thm:priori}\ref{item:prio1}, we have
\begin{align*}
\dfrac{\dif }{\dif t} \int_{0}^{k} e^{-j} \dif x
\le C\int_{0}^k e^{-j} \dfrac{e}{v} \dif x \le  C(a) \sigma^{-\frac{1}{2}}(t) \int_{0}^k e^{-j} \dif x.
	\end{align*}
Integrating the above inequality in time, it follows that
\begin{equation*}
\int_{0}^{k} e^{-j}(x,t)\,\dif x
\le \sbnorm{e_0^{-1}}_{L^j(0,k)}^j
+ C(a) \int_{0}^{t} \sigma^{-\frac{1}{2}}(\tau) \big(\int_{0}^{k} e^{-j} \dif x\big)\,\dif \tau.
\end{equation*}
Thus, by the Gr\"onwall inequality, we obtain that, for each integer $j\ge 2$,
\begin{align*}
\sbnorm{e^{-1}(\cdot,t)}_{L^j(0,k)}^j &\le C(a) \sbnorm{e_0^{-1}}_{L^j(0,k)}^j
+ C(a) \Big(\int_{0}^{t} \sigma^{-\frac{1}{2}}(\tau)\dif \tau \Big) \exp\Big\{C(a)\int_{0}^{t} \sigma^{-\frac{1}{2}}(\tau)\dif\tau\Big\}\le C(a).
\end{align*}
Letting $j\to \infty$ in the above inequality, we conclude that
$\sbnorm{e^{-1}(\cdot,t)}_{L^\infty} \le C(a) \sbnorm{e_0^{-1}}_{L^\infty}$.
\end{proof}

\section{Global Weak Solutions of the Exterior Problem in the Eulerian Coordinates}
\label{sec:WSCEx}
In this section, the approximate Lagrangian solutions are converted into
the approximate solutions in the Eulerian coordinates
by constructing a set of coordinate transformations.
Then these functions are extended into the entire domain
$(r,t)\in[a,\infty)\times[0,T]$ by using a set of smooth cut-off functions.
Finally, a global weak solution of problem \eqref{eqs:SFNS} and \eqref{eqs:eSFNS} is obtained
via the limit: $k\to\infty$.

\subsection{Coordinate transformation and the Jacobian}\label{subsec:coord}
Let $\{(\tilde{v}_{a,k},\tilde{u}_{a,k},\tilde{e}_{a,k},\tilde{r}_{a,k})(x,s)\}_{k\in\mathbb{N}}$
be the solutions obtained in the bounded Lagrangian domain for $(x,s)\in[0,k]\times[0,T]$.
For each fixed $(a,k)\in(0,1)\times\mathbb{N}$, the coordinate transformation $(r,t)=\mathcal{T}_{a,k}(x,s)$ is defined as
\begin{equation}\label{eqs:coordtrans}
r=\mathcal{T}_{a,k}^{(1)}(x,s)=\tilde{r}_{a,k}(x,s),\quad t=\mathcal{T}_{a,k}^{(2)}(x,s)=s \qquad \,\,\,\text{for $(x,s)\in [0,k]\times[0,T]$}.
\end{equation}
Then we see that $\mathcal{T}_{a,k}$ has the image:
$$
\mathcal{T}_{a,k}([0,k]\times [0,T] ) = R_{a,k}
\vcentcolon= \big\{(r,t)\,\vcentcolon\,t\in[0,T], r\in [a,\tilde{r}_{a,k}(k,t)] \big\}.
$$
Moreover, it follows from a direct calculation  that
\begin{equation*}
J\vcentcolon=
{\rm det}(
		\D_x\mathcal{T}_{a,k}\,\, \,\D_s\mathcal{T}_{a,k})
	= \tilde{v}_{a,k} (\tilde{r}_{a,k})^{-m},
\end{equation*}
which, along with Theorem \ref{thm:priori} and Lemma \ref{lemma:rbd}, yields that
\begin{equation}\label{eqs:Jac}
	C^{-1}(a)( a^n + C(a) x )^{-\frac{m}{n}} \le J(x,t) \le a^{-m}C(a)
	\qquad \text{ for all $(x,t)\in[0,k]\times[0,T]$}.
\end{equation}
Thus, the map $\mathcal{T}_{a,k}\vcentcolon [0,k]\times[0,T]\to R_{a,k}$ is a diffeomorphism so that the inverse map $\mathcal{T}_{a,k}^{-1}$ exists.
Define $(x_{a,k},s_{a,k})(r,t)\vcentcolon= \mathcal{T}_{a,k}^{-1}(r,t)$. Then
\begin{equation*}
 \tilde{r}_{a,k}( x_{a,k}(r,t), t) = r, \quad s_{a,k}(r,t)=t
\qquad\,\, \text{for all $(r,t) \in R_{a,k}$}.
\end{equation*}
Let $\tilde{\theta}(x,t)\vcentcolon [0,k]\times[0,T]\to \R$ be such that $\tilde{\theta}\in C^1$.
Define $\theta(r,t)\vcentcolon=\tilde{\theta}(x_{a,k}(r,t),t)$ for each $(r,t)\in R_{a,k}$.
Then, by the inverse function theorem,
its Eulerian derivative can be computed:
\begin{equation}\label{eqs:coorderiv}
	\begin{dcases}
\partial_r \theta(r,t) =  \dfrac{\big(\tilde{r}_{a,k}\big)^m(x_{a,k}(r,t),t)}{\tilde{v}_{a,k}(x_{a,k}(r,t),t)} \D_x \tilde{\theta}(x_{a,k}(r,t),t) = r^m  \dfrac{\D_x \tilde{\theta}}{\tilde{v}_{a,k}}(x_{a,k}(r,t),t), \\
\partial_t \theta(r,t) = \D_t \tilde{\theta}(x_{a,k}(r,t),t)
- r^m \big(\dfrac{ \tilde{u}_{a,k}}{\tilde{v}_{a,k}} \D_x \tilde{\theta}\big)( x_{a,k}(r,t),t ),
	\end{dcases}
\end{equation}
for each $(r,t)\in R_{a,k}$.
Now, for each $(a,k)$, solution $(\tilde{v}_{a,k}, \tilde{u}_{a,k}, \tilde{e}_{a,k})(x,s)$ in $(x,t)\in[0,k]\times[0,T]$ is pulled back to $R_{a,k}$ as:
\begin{equation}\label{eqs:transfunc}
(\overline{v}_{a,k},\overline{u}_{a,k},\overline{e}_{a,k})(r,t) \vcentcolon= ( \tilde{v}_{a,k}, \tilde{u}_{a,k}, \tilde{e}_{a,k} ) ( x_{a,k}(r,t),t )
\qquad \ \text{for $(r,t)\in R_{a,k}$}.
\end{equation}

\subsection{Extension in the Eulerian domain} \label{subsec:Eulext}
We now extend $(\overline{v}_{a,k},\overline{u}_{a,k},\overline{e}_{a,k})$  into the whole Eulerian domain $[a,\infty)\times[0,T]$.
First, let $\chi\in C^{\infty} \vcentcolon \R \to [0,1]$ be such that
$\chi(\zeta)=1$ if $\zeta\le 0$, $\chi(\zeta)=0$ if $\zeta\ge 1$, and $\snorm{\chi^{\prime}} \le 2$ in $\R$.
With this, the following cut-off functions are defined:
\begin{equation}\label{eqs:COfunc}
	\varphi_{a,k}(r,t) \vcentcolon= \chi(\dfrac{2r-\tilde{r}_{a,k}(k,t)}{\tilde{r}_{a,k}(k,t)}).
\end{equation}
Since $\tilde{u}_{a,k}(k,t)=0$ for all $t\in[0,T]$, it follows that, for all $(r,t)\in[a,\infty)\times[0,T]$, 
\begin{equation}\label{eqs:varphiDeriv}
\partial_r \varphi_{a,k}(r,t)
=\dfrac{2}{\tilde{r}_{a,k}(k,t)} \chi^{\prime}(\dfrac{2r-\tilde{r}_{a,k}(k,t)}{\tilde{r}_{a,k}(k,t)}),
\qquad \partial_t \varphi_{a,k}(r,t) = 0.
\end{equation}
Then
$\varphi_{a,k}\in C^{\infty}([0,\infty)\times[0,T])$.
Moreover, by construction,
\begin{equation*}
\varphi_{a,k}(r,t)=1 \ \text{ if $r\in[0,\frac{\tilde{r}_{a,k}(k,t)}{2}]$},
\qquad \varphi_{a,k}(r,t)=0 \ \text{ if $r\in[\tilde{r}_{a,k}(k,t),\infty)$}.
\end{equation*}
Now, by Theorem \ref{thm:priori}\ref{item:prio1},
\begin{align*}
\tilde{r}_{a,k}(k,t)&= \Big( a^n + n \int_{0}^k v_{a,k}(y,t)\,\dif y \Big)^{\frac{1}{n}}
\ge \Big( n \int_{\frac{1}{2}}^k v_{a,k}(y,t)\,\dif y \Big)^{\frac{1}{n}}
\ge \Big(\frac{n}{C(T)}\int_{\frac{1}{2}}^k \dif y \Big)^{\frac{1}{n}}
\ge \frac{1}{C(T)},
\end{align*}
where the universal constant $C(T)>0$ depends only on $T$ and $n$.
It also follows from Theorem \ref{thm:priori} that
\begin{align*}
	\tilde{r}_{a,k}(k,t) = \Big( a^n + n \int_{0}^k v_{a,k}(y,t)\,\dif y \Big)^{\frac{1}{n}}
	\ge \frac{k^{\frac{1}{n}}}{C(a)},
\end{align*}
so that, by \eqref{eqs:varphiDeriv},
\begin{equation}\label{eqs:extcut}
\snorm{\partial_r\varphi_{a,k}(r,t)} \le  \min\{C(T),C(a)k^{-\frac{1}{n}}\}, \ \
\partial_t \varphi_{a,k}(r,t)=0
\qquad \mbox{for all $(r,t)\in[a,\infty)\times[0,T]$}.
\end{equation}
Using \eqref{eqs:coordtrans}--\eqref{eqs:COfunc}, the extended approximate functions $(\rho_{a,k},u_{a,k},e_{a,k})(r,t)$ in the entire Eulerian domain $(r,t)\in[a,\infty)\times[0,T]$
are defined by \eqref{eqs:extfunc} in \S \ref{subsec:MS}.

\smallskip
The following lemmas show that the extended functions inherit the {\it a-priori} estimates derived
in \S \ref{sec:rhobd}. As before, we still denote $\sigma\vcentcolon=\min\{1,t\}$.
\begin{lemma} \label{lemma:entEstext}
There exists an integer $N(a)=N(a,T,C_0)\in\mathbb{N}$ such that, for all $k\ge N(a)$,
\begin{align}\label{eqs:entEul}
&\sup\limits_{t\in[0,T]}\int_{a}^{\infty}
 \Big\{\rho_{a,k}\big(\frac{1}{2}\snorm{u_{a,k}}^2+\psi(e_{a,k})\big)
 +G(\rho_{a,k}) \Big\}\,r^m \dif r
 + \kappa\int_{0}^{T}\int_{a}^{\infty}
   \dfrac{\snorm{\partial_r e_{a,k}}^2}{e_{a,k}^2}\,r^m\dif r\dif t\nonumber \\
&+\int_{0}^{T}\int_{a}^{\infty} \Big\{\big(\dfrac{2\mu}{n}+\lambda\big) \dfrac{\snorm{\partial_r u_{a,k}+m \frac{u_{a,k}}{r}}^2}{e_{a,k}}
+ \dfrac{2m\mu}{n} \dfrac{\snorm{\partial_r u_{a,k}-\frac{u_{a,k}}{r}}^2}{e_{a,k}}\Big\}
\,r^m\dif r\dif t \le C(T).
\end{align}
\end{lemma}

\begin{proof}
    Throughout the proof, we suppress $(a,k)$ for simplicity and denote:
    \begin{equation}\label{aksupp}
        \begin{aligned}
            &(\rho,u,e,\varphi)(r,t) \equiv (\rho_{a,k},u_{a,k},e_{a,k},\varphi_{a,k})(r,t) && \text{for $(r,t)\in[a,\infty)\times[0,T]$}, \\[1mm]
            &(\ov{v},\ov{u},\ov{e})(r,t) \equiv (\ov{v}_{a,k},\ov{u}_{a,k},\ov{e}_{a,k})(r,t)
              && \text{for $(r,t)\in R_{a,k}$}, \\[1mm]
            &(\tilde{v},\tilde{u},\tilde{e},\tilde{r})(x,t) \equiv (\tilde{v}_{a,k},\tilde{u}_{a,k},\tilde{e}_{a,k},\tilde{r}_{a,k})(x,t)&& \text{for $(x,t)\in[0,k]\times[0,T]$}.
        \end{aligned}
    \end{equation}

Since $G(\cdot)$ and $\psi(\cdot)$ are convex, and $0\le \varphi \le 1$, it follows
that $G(\rho)\le \varphi G(\ov{v}^{-1})$ and $\psi(e)\le \varphi\,\psi(\ov{e})$.
Using these and the fact that the Jacobian is bounded from \eqref{eqs:Jac},
one can convert the integral in the Eulerian coordinates into the Lagrangian coordinates:
\begin{align*}
&\int_{a}^{\infty} \Big(\rho\big(\frac{1}{2}\snorm{u}^2+\psi(e)\big) + G(\rho) \Big)(r,t)
 \,r^m \dif r\\
&\le  \int_{a}^{\tilde{r}(k,t)}
\Big(\frac{1}{2}\snorm{\ov{u}}^2+ \psi(\ov{e}) + \ov{v} G(\ov{v}^{-1}) \Big) \ov{v}^{-1}\,r^m \dif r
+ \int_{a}^{\tilde{r}(k,t)}(1-\varphi)\Big(\frac{1}{2}\snorm{\ov{u}}^2+ \psi(\ov{e}) \Big)\,r^m\dif r\\
&= \int_{0}^{k} \Big(\frac{1}{2}\snorm{\tilde{u}}^2  + \psi(\tilde{e})  + \psi(\tilde{v})
  \Big)(x,t)\,\dif x + \int_{0}^{k} \big(1-\varphi(\tilde{r}(x,t),t)\big)
\Big(\frac{1}{2}\tilde{v}\snorm{\tilde{u}}^2+ \tilde{v} \psi(\tilde{e})  \Big) (x,t)\,\dif x,
\end{align*}
where
$zG(z^{-1})=\psi(z)$ has been used.
By Theorem \ref{thm:priori}\ref{item:prio0}, the first term of the right-hand side is bounded by $C_0>0$.
By the definition of $\tilde{r}(x,t)$,
\begin{equation*}
\Big( a^n + \frac{n}{C(a)} x \Big)^{\frac{1}{n}} \le \tilde{r}(x,t)
= \Big( a^n + n \int_{0}^x \tilde{v}(y,t)\,\dif y \Big)^{\frac{1}{n}}
\le \big( a^n + n C(a) x\big)^{\frac{1}{n}}.
\end{equation*}
Thus, for fixed $a>0$, if $k\in\mathbb{N}$ is such that
$k\ge M(a)\vcentcolon= \dfrac{2^n -1}{n}a^n C(a)+ 2^n C(a)^2$, then
\begin{equation}\label{temp:kMa}
   \tilde{r}(k,t)\ge 2 \tilde{r}(1,t).
\end{equation}
Using this and the support of the cut-off function $\varphi$, it follows that
\begin{align*}
&\int_{0}^k \big(1-\varphi(\tilde{r}(x,t),t)\big)
  \Big(\frac{1}{2}\tilde{v}\snorm{\tilde{u}}^2+\tilde{v}\psi(\tilde{e}) \Big)(x,t)\,\dif x\\
&= \int_{a}^{\tilde{r}(k,t)} \big(1-\varphi(r,t)\big)
 \Big(\frac{1}{2}\snorm{\overline{u}}^2 +\psi(\overline{e})\Big)(r,t)\,r^m \dif r
\le \int_{\tilde{r}(1,t)}^{\tilde{r}(k,t)}
 \Big(\frac{1}{2}\snorm{\overline{u}}^2+\psi(\overline{e})\Big)(r,t)\,r^m \dif r\\
&= \int_{1}^{k} \tilde{v}(x,t)\Big(\frac{1}{2}\snorm{\tilde{u}}^2 +\psi(\tilde{e})\Big) (x,t)
\,\dif x \le \sup_{y\in[1,k]} \tilde{v}(y,t) \int_{1}^{k}
\Big(\frac{1}{2}\snorm{\tilde{u}}^2 +\psi(\tilde{e})\Big) (x,t)\,\dif x \le C(T) C_0,
\end{align*}
where Theorem \ref{thm:priori}\ref{item:prio0}--\ref{item:prio1} has been used.

Next, the dissipation terms in $\eqref{eqs:entEul}$ are considered.
Using \eqref{eqs:extfunc} and the fact that $z\mapsto z^{-2}$ is convex in
$z\in (0,\infty)$ yield
\begin{align*}
\dfrac{\snorm{\partial_r e}^2}{2\snorm{e}^2}
&= \dfrac{\snorm{\varphi\partial_r\overline{e} + (\overline{e}-1) \partial_r \varphi }^2}{2\snorm{\varphi\overline{e}+(1-\varphi)}^2}\le 2^{-1}\big( \varphi \overline{e}^{-2} + (1-\varphi) \big) \snorm{\varphi\partial_r\overline{e} + (\overline{e}-1) \partial_r \varphi }^2 \\
&\le \varphi^2\dfrac{\snorm{\partial_r \overline{e}}^2}{\overline{e}^2} + \dfrac{\snorm{\overline{e}-1}^2}{\overline{e}^2} \varphi \snorm{\partial_r\varphi}^2 + (1-\varphi) \varphi^2 \snorm{\partial_r \overline{e}}^2 + (1-\varphi)\snorm{\partial_r\varphi}^2 \snorm{\overline{e}-1}^2.
\end{align*}
Multiplying both sides by $r^m$ and integrating in $[a,\infty)\times[0,T]$, it follows from \eqref{eqs:extcut} and \eqref{temp:kMa} that, if $k\ge M(a)$, then
\begin{align*}
&\dfrac{1}{2}\int_{0}^T\int_{a}^{\infty} \dfrac{\snorm{\partial_re}^2}{e^2}\,r^m \dif r \dif t\\
&\le \int_{0}^T \int_{a}^{\infty} \Big\{ \varphi^2\dfrac{\snorm{\partial_r \overline{e}}^2}{\overline{e}^2} + \dfrac{\snorm{\overline{e}-1}^2}{\overline{e}^2} \varphi \snorm{\partial_r\varphi}^2 + (1-\varphi) \Big(  \varphi^2 \snorm{\partial_r \overline{e}}^2 + \snorm{\partial_r\varphi}^2 \snorm{\overline{e}-1}^2 \Big)   \Big\}\,r^m \dif r \dif t \\
&\le  \int_{0}^T \int_{a}^{\tilde{r}(k,t)}  \dfrac{\snorm{\partial_r \overline{e}}^2}{\overline{e}^2} r^m \dif r \dif t + \int_{0}^T\int_{\tilde{r}(1,t)}^{\tilde{r}(k,t)} \Big\{ \snorm{\partial_r\overline{e}}^2 + C(T) \snorm{\overline{e}-1}^2 + C(a)k^{-\frac{2}{n}} \dfrac{\snorm{\overline{e}-1}^2}{\overline{e}^2}  \Big\}\,r^m \dif r \dif t.
\end{align*}
If $k\ge \max\{C^{\frac{n}{2}}(a)C^{\frac{n}{2}}(T),M(a)\}$, then translating the above integral into
the Lagrangian coordinates and using \eqref{eqs:coorderiv} and Theorem \ref{thm:priori} yield	
\begin{align*}
		&\int_{0}^T\int_{a}^{\infty} \dfrac{\snorm{\partial_re}^2}{e^2}\,r^m \dif r \dif t\\
		&\le \int_{0}^T \int_{0}^{k} \dfrac{\tilde{r}^{2m}}{\tilde{v}} \dfrac{\snorm{\D_x \tilde{e}}^2}{\tilde{e}^2} \dif x \dif t + \int_{0}^T\int_{1}^{k} \Big\{ \dfrac{\tilde{r}^{2m}}{\tilde{v}}\snorm{\D_x \tilde{e}}^2 + C(T)\tilde{v}\snorm{\tilde{e}-1}^2 + \dfrac{C(a)}{k^{\frac{2}{n}}} \dfrac{\tilde{v}\snorm{\tilde{e}-1}^2}{\tilde{e}^2} \Big\}\dif x \dif t\\
		&\le  C(T)\Big(1+ \sup \big\{ (\tilde{v}+k^{-\frac{2}{n}}
		  \dfrac{\tilde{v}}{\tilde{e}^{2}})(y,s)\,\vcentcolon\,(y,s)\in[1,k]\times[0,T] \big\}\Big)\le C(T),
	\end{align*}

Since $z\to z^{-1}$ is convex in $z\in(0,\infty)$ and $\partial_r u + m \frac{u}{r} = r^{-m}\partial_r(r^m u)$,
then
\begin{align*}
\dfrac{\snorm{\partial_r u + m \frac{u}{r}}^2}{2e}
&\le \frac{1}{2}\big(\varphi\overline{e}^{-1} + (1-\varphi)\big)
  \big|\varphi\partial_r \overline{u} + \overline{u} \partial_r\varphi + m \varphi \frac{\overline{u}}{r}\big|^2\\
&\le  \varphi^3 \ov{e}^{-1} \snorm{ r^{-m} \partial_r(r^m \ov{u})}^2  + \varphi \snorm{\partial_r \varphi}^2 \overline{e}^{-1}\snorm{\overline{u}}^2
+ (1-\varphi)\big(\varphi^2 \snorm{r^{-m} \partial_r(r^m \ov{u})}^2
+ \snorm{\ov{u}\partial_r \varphi}^2\big).
\end{align*}

By similar arguments as above, we see that, for all $k\ge \max\{ C^{\frac{n}{2}}(a)C^{\frac{n}{2}}(T), M(a) \}$,
\begin{align*}
&\dfrac{1}{2}\int_{0}^T\int_{a}^{\infty} \dfrac{\snorm{r^{-m}\partial_r(r^m \ov{u})}^2}{e} r^m \dif r \dif t\\
&\le  \int_{0}^T\int_{a}^{\tilde{r}(k,t)}
  \dfrac{\snorm{r^{-m}\partial_r( r^m \overline{u})}^2}{\overline{e}} r^m \dif r \dif t +\dfrac{C(a)}{k^{\frac{2}{n}}}\int_{0}^T\int_{a}^{\tilde{r}(k,t)}
    \dfrac{\snorm{\overline{u}}^2}{\overline{e}}\,r^m \dif r \dif t \\
&\quad + \int_{0}^T\int_{\tilde{r}(1,t)}^{\tilde{r}(k,t)} \big\{ \snorm{r^{-m}\partial_r(r^m \ov{u})}^2
  + C(T)\overline{u}^2 \big\}\,r^m \dif r \dif t \\
&= \int_{0}^T \int_{0}^{k} \Big\{ \dfrac{\snorm{\D_x(\tilde{r}^m \tilde{u})}^2}{\tilde{v}\tilde{e}} + \dfrac{C(a)}{k^{\frac{2}{n}}}\dfrac{\tilde{v}}{\tilde{e}} \snorm{\tilde{u}}^2  \Big\}\,\dif x \dif t
+ \int_{0}^T \int_{1}^{k}\Big\{ \dfrac{\snorm{\D_x(\tilde{r}^m \tilde{u})}^2}{\tilde{v}}
  + C(T) \tilde{v} \snorm{\tilde{u}}^2 \Big\}\,\dif x \dif t\le C(T),
\end{align*}
By the same analysis, we also see that, for all $k\ge \max\{ C^{\frac{n}{2}}(a)C^{\frac{n}{2}}(T), M(a) \}$,
\begin{align*}
\int_{0}^T\int_{a}^{\infty} \dfrac{1}{e}\Bignorm{\partial_r u -  \dfrac{u}{r}}^2 r^m \dif r \dif t
\le C(T).
\end{align*}
Set $N(a)\vcentcolon=
[(\max\{ C^{\frac{n}{2}}(a)C^{\frac{n}{2}}(T), M(a) \}]\in \mathbb{N}$, where $[\cdot]$ is the ceiling function.
Then the proof is completed.
\end{proof}

\begin{lemma}\label{lemma:HREext}
For each $a\in(0,1)$,
\begin{equation}\label{temp:aHREext1}
\begin{split}
&\sup_{k\in\mathbb{N}}\sup_{ 0\le t \le  T} \int_{a}^{\infty} \Big\{ \big|(\rho_{a,k}-1,u_{a,k}^2, e_{a,k}-1,\sqrt{\sigma}\partial_r u_{a,k}, \sigma\partial_r e_{a,k})\big|^2 \Big\} (r,t)\, r^m \dif r \\
&\, + \sup_{k\in\mathbb{N}}\int_{0}^{T} \int_{a}^{\infty} \Big\{ \big|( \partial_r u_{a,k}, \partial_r e_{a,k}, u_{a,k} \partial_r u_{a,k}, \sqrt{\sigma} \partial_t u_{a,k}, \sigma \partial_t e_{a,k} )\big|^2 \Big\}\,
r^m \dif r\dif t \le C(a),
\end{split}
\end{equation}
and, for all $(k,r,t)\in\mathbb{N}\times[a,\infty)\times[0,T]$,
	\begin{equation}\label{temp:aHREext2}
	C^{-1}(a) \le \rho_{a,k}(r,t) \le C(a), \,\,  e_{a,k}(r,t) \ge C^{-1}(a),\,\,
		\sigma^{\frac{1}{4}}(t)\snorm{u_{a,k}(r,t)}+\sigma^{\frac{1}{2}}(t)e_{a,k}(r,t) \le C(a).
\end{equation}
Moreover, for fixed $\ve\in(0,1]$,
\begin{align}
\sup_{\substack{k\in\mathbb{N}\\ a\in(0,1)}}
\bigg\{&
\sup_{ 0\le t \le  T} \int_{\tilde{r}_{a,k}(\ve,t)}^{\infty} \Big\{ \big|(\rho_{a,k}-1,u_{a,k}^2, e_{a,k}-1,\sqrt{\sigma}\partial_r u_{a,k}, \sigma\partial_r e_{a,k})\big|^2 \Big\} (r,t)\,r^m \dif r\bigg.
 \nonumber\\
&+ \int_{0}^{T} \int_{\tilde{r}_{a,k}(\ve,t)}^{\infty}
 \Big\{ \big|(\partial_r u_{a,k}, \partial_r e_{a,k},
 u_{a,k} \partial_r u_{a,k},\sqrt{\sigma} \partial_t u_{a,k},
 \sigma \partial_t e_{a,k} )\big|^2\Big\}\,
r^m \dif r\dif t\bigg\} \le C(\ve),\quad\label{temp:veHREext1}
\end{align}
and, for any $(r,t)\in [\tilde{r}(\ve,t),\infty)\times [0,T]$ and $(a,k)\in(0,1)\times\mathbb{N}$,
\begin{equation}\label{temp:veHREext2}
\begin{split}
& C^{-1}(\ve) \le \rho_{a,k}(r,t) \le C(\ve), \quad \ \sigma^{\frac{1}{4}}(t)\snorm{u_{a,k}(r,t)}
+\sigma^{\frac{1}{2}}(t)e_{a,k}(r,t) \le C(\ve).
\end{split}
\end{equation}
\end{lemma}

\begin{proof}
As before, we suppress $(a,k)$ for simplicity and use the notation in \eqref{aksupp}.
Only \eqref{temp:veHREext1}--\eqref{temp:veHREext2} are shown, since \eqref{temp:aHREext1}--\eqref{temp:aHREext2}
follow in the same manner.

First, using the fact that the Jacobian is bounded from \eqref{eqs:Jac}
and construction \eqref{eqs:COfunc}, the integration in the Eulerian coordinates
is converted to the following in the Lagrangian coordinates:
\begin{align*}
\int_{\tilde{r}(\ve,t)}^{\infty} \snorm{\rho-1}^2\,r^m \dif r
\le   \int_{\tilde{r}(\ve,t)}^{\tilde{r}(k,t)} \snorm{\overline{v}^{-1}-1}^2\,r^m \dif r
= \int_{\ve}^{k} \frac{\snorm{\tilde{v}-1}^2}{\tilde{v}} \dif x \le C(\ve),
\end{align*}
where  Theorem \ref{thm:priori} has been used in the last inequality.
By the same method, one can also obtain the corresponding estimates for $\snorm{e-1}^2$ and $\snorm{u}^4$ in $\eqref{temp:veHREext1}_1$.

Now, using \eqref{eqs:coorderiv}, \eqref{eqs:extcut}, and Theorem \ref{thm:priori}, it follows that
\begin{align*}
\int_{\tilde{r}(\ve,t)}^{\infty} \sigma(t) \snorm{\partial_r u}^2
\,r^m \dif r
&= \int_{\tilde{r}(\ve,t)}^{\infty}  \sigma(t) \snorm{\varphi\partial_r \overline{u}
+ \overline{u}\partial_r\varphi}^2
\,r^m \dif r\\
&\le  \int_{\tilde{r}(\ve,t)}^{\tilde{r}(k,t)}  \sigma(t) \snorm{\partial_r \overline{u}}\,r^m \dif r
+ C\int_{\tilde{r}(\ve,t)}^{\tilde{r}(k,t)}  \snorm{\overline{u}}^2\,r^m \dif r\\
&= \int_{\ve}^{k} \sigma \dfrac{\tilde{r}^{2m}}{\tilde{v}} \snorm{\D_x \tilde{u}}^2\,\dif x
+ \int_{\ve}^{k} \tilde{v} \snorm{\tilde{u}}^2\,\dif x \le C(\ve).
\end{align*}
In the same way, we  can also obtain the estimates for $\sigma^2 \snorm{\partial_r e}^2$
and $\snorm{(\partial_r u,\partial_r e, u \partial_r u)}^2$ in \eqref{temp:veHREext1}.
	
By \eqref{eqs:coorderiv} and Theorem \ref{thm:priori}, it follows that
\begin{align*}
&\int_{0}^{T}\int_{\tilde{r}(\ve,t)}^{\infty}
\sigma^2(t)\snorm{  \partial_t e}^2(r, t)\,r^m \dif r \dif t
= \int_{0}^{T}\int_{\tilde{r}(\ve,t)}^{\infty}\sigma^2(t)
\varphi^2(r,t)\snorm{\partial_t \overline{e}}^2(r, t)\,r^m \dif r \dif t\\
&\le \int_{0}^{T}\int_{\tilde{r}(\ve,t)}^{\tilde{r}(k,t)}\sigma^2(t)
\snorm{ \partial_t \overline{e}}^2(r, t)\,r^m \dif r \dif t
= \int_{0}^{T}\int_{\ve}^{k}\sigma^2(t) \tilde{v}(x,t)
\Bignorm{ \D_t \tilde{e} - \tilde{r}^m \dfrac{\tilde{u}}{\tilde{v}}\D_x \tilde{e} }^2(x, t)\,\dif x \dif t\\
&\le  \int_{0}^{T}\int_{\ve}^{k}\sigma^2 \tilde{v} \snorm{ \D_t \tilde{e} }^2 \, \dif x \dif t
+ \int_{0}^{T}\int_{\ve}^{k}\sigma^{\frac{3}{2}} \snorm{\sigma^{\frac{1}{4}} \tilde{u}}^2
\dfrac{\tilde{r}^{2m}}{\tilde{v}} \snorm{\D_x\tilde{e}}^2\, \dif x \dif t \le C(\ve).
\end{align*}
By the same argument, one can also obtain the estimate for $\sigma\snorm{\partial_t u}^2$ in \eqref{temp:veHREext1}.
	
Next, for $(r,t)$ such that $t\in[0,T]$ and $r\in [\tilde{r}(\ve,t),\infty)\cap R_{a,k}$, there exists $(x,s)\in[\ve,k]\times[0,T]$ such that $(r,t)=(\tilde{r}(x,s), s)$
and the following estimates hold for $(r,t)\in R_{a,k}$:
\begin{align*}
\rho(r,t)
&= \varphi(\tilde{r}(x,t),t)\tilde{v}^{-1}(x,t) + 1 -\varphi(\tilde{r}(x,t),t)
\le C(\ve)\varphi(\tilde{r}(x,t),t) +1 -\varphi(\tilde{r}(x,t),t) \le C(\ve),\\
\rho(r,t)
	&\ge C(\ve)^{-1}\varphi(\tilde{r}(x,t),t) +1 -\varphi(\tilde{r}(x,t),t) \ge C(\ve)^{-1}.
\end{align*}
For $(r,t)$ such that $t\in[0,T]$ and $r\in [\tilde{r}(\ve,t),\infty)\times [0,T] \backslash R_{a,k}$,
then $r> \tilde{r}(k,t)$ so that $\varphi(r,t)=0$ and
$\rho(r,t)=\overline{v}^{-1}(r,t)\varphi(r,t) - \varphi(r,t)+1 = 1$,
which leads to the upper and lower bounds of the extended density.
In a similar manner, we can also obtain the bounds for $(u,e)$
listed in \eqref{temp:veHREext2}.
\end{proof}

Throughout \S\ref{subsec:klimpath}--\S \ref{subsec:indepPath},
we will suppress parameter $a\in(0,1)$ and denote
\begin{equation}\label{asuppress}
    \begin{aligned}
        &(\rho_k,u_k,e_k,\varphi_k)\equiv (\rho_{a,k},u_{a,k},e_{a,k},\varphi_{a,k}), \ \  (\ov{v}_k,\ov{u}_k,\ov{e}_k) \equiv (\ov{v}_{a,k},\ov{u}_{a,k},\ov{e}_{a,k}),\\[1mm]
        &(\tilde{v}_{k},\tilde{u}_{k},\tilde{e}_{k},\tilde{r}_{k}) \equiv (\tilde{v}_{a,k},\tilde{u}_{a,k},\tilde{e}_{a,k},\tilde{r}_{a,k})
    \end{aligned}
\end{equation}
Note that the main aim of these subsections is to prove Theorem \ref{thm:WSCEx} via the limit: $k\to\infty$.

\subsection{Construction of the particle path function}\label{subsec:klimpath}
In this section, we construct the particle path function
$\tilde{r}(x,t)$.

\begin{lemma}\label{lemma:rkcvg}
There exist both a continuous function $\tilde{r}(x,t) \vcentcolon [0,\infty)\times[0,T] \to [a,\infty)$
and a subsequence $k\to\infty$ such that
\begin{enumerate}[label=(\roman*),ref=(\roman*),font={\normalfont\rmfamily}]
\item\label{item:rkcvg1} for any compact subset $K\Subset [0,\infty)\times[0,T]$,
\begin{align*}
\lim\limits_{k\to \infty} \sup_{(x,t)\in K} \snorm{\ov{r}_{k}(x,t)-\tilde{r}(x,t)} = 0,
\end{align*}
where $\ov{r}_k(x,t)$ is defined by
\begin{equation}\label{rrtemp0}
\ov{r}_k(x,t) =\begin{dcases*}
				\tilde{r}_k(x,t) & if $(x,t)\in[0,k]\times[0,T]$,\\
				x-k + \tilde{r}_k(k,t) & if $(x,t)\in[k,\infty)\times[0,T]$.
			\end{dcases*}
\end{equation}
\item\label{item:rkcvg2} for each $\ve>0$,
\begin{equation*}
\begin{aligned}
&\snorm{\tilde{r}(x_1,t)-\tilde{r}(x_2,t)} \le C(\ve) \snorm{x_1-x_2}
&&\quad \text{for all} \  (x_1,x_2,t)\in [\ve,\infty)^2\times[0,T],\\
&\snorm{\tilde{r}(x,t_1)-\tilde{r}(x,t_2)}
\le C(\ve) \snorm{t_1^{\frac{3}{4}}-t_2^{\frac{3}{4}}}
&&\quad \text{for all} \ (x,t_1,t_2)\in [\ve,\infty)\times[0,T]^2.
\end{aligned}
\end{equation*}
		
\smallskip
\item\label{item:rkcvg3} for each $t\in[0,T]$, the function $x\mapsto\tilde{r}(x,t)$ is strictly increasing and
\begin{equation*}
			n x \psi^{-1}_{-}(\dfrac{C_0}{x}) \le \tilde{r}^{\,n}(x,t) \le C_0(1 +x)
			\qquad\text{for all} \ \  (x,t)\in[0,\infty)\times[0,T].
		\end{equation*}
	\end{enumerate}
\end{lemma}

\begin{proof}
For each integer $l\in\mathbb{N}$,
let $S_l\vcentcolon= [0,l]\times[0,T]\subset [0,\infty)\times[0,T]$.
Set $x_1,\, x_2,\, x \in [0,l]$ and $t_1,\,t_2,\,t \in[0,T]$. We consider two cases: $k\ge l$ or $k<l$.
	
\smallskip
Case 1: $k\ge l$. Then $\ov{r}_k(x_i,t)=\tilde{r}_k(x_i,t)$ for $i=1,\,2$.
It follows from Theorem \ref{thm:priori} and the mean value theorem that
\begin{equation*}
\snorm{\ov{r}_k(x_1,t)- \ov{r}_k(x_2,t)} \le \snorm{x_1-x_2}  \sup_{x\in[0,k]}|\D_x \tilde{r}_k(x,t)|
\le \snorm{x_1-x_2}  \sup_{x\in[0,k]}\big|\dfrac{\tilde{v}_k}{\tilde{r}_k^{\,m}}\big|
\le C(a) \snorm{x_1-x_2}.
\end{equation*}

Case 2: $k<l$. For this case, we consider three sub-cases:

If $x_1,\, x_2 \in[0,k]$, the same inequality as Case 1 can be obtained by repeating
 the same argument.

If $x_1, x_2 \in [k,l]$, then, by \eqref{rrtemp0},
$\snorm{\ov{r}_k(x_1,t)-\ov{r}_k(x_2,t)}= \snorm{x_1-x_2}$.

If $x_1\in [0,k]$ and $x_2\in[k,l]$, then, by \eqref{rrtemp0},
\begin{equation*}
\snorm{\ov{r}_k(x_1,t)-\ov{r}_{k}(x_2,t)} \le \snorm{\tilde{r}_k(x_1,t)-\tilde{r}_k(k,t)} + \snorm{x_2-k} \le C(a)(k-x_1)+x_2-k\le C(a) \snorm{x_1-x_2}.
\end{equation*}

\smallskip
In summary, for both cases,
\begin{equation}\label{rrtemp1a}
\sup_{k\in\mathbb{N}} \snorm{\ov{r}_k(x_1,t)-\ov{r}_k(x_2,t)} \le C^{\ast}(a) \snorm{x_1-x_2}.
\end{equation}
Moreover, for both cases $k\ge l$ and $k<l$, it can be shown that, if $t_*=\min\{1,t_1,t_2\}$, then
	\begin{equation}\label{rrtemp1b}
		\snorm{\ov{r}_k(x,t_1)-\ov{r}_k(x,t_2)} = \Bignorm{\int_{t_1}^{t_2} u_k(x,s)\,\dif s }
		\le C(a) \Bignorm{\int_{t_1}^{t_2} s^{-\frac{1}{4}}\,\dif s } \le C(a) t_*^{-\frac{1}{4}} \snorm{t_1-t_2}.
	\end{equation}
Similarly, by \eqref{rrtemp0}, Theorem \ref{thm:priori}, and Lemma \ref{lemma:rbd},
it can be shown that, for each $\ve\in(0,1]$,
\begin{equation}\label{rrtemp2}
\begin{aligned}
	&\sup_{k \in \mathbb{N}}\snorm{\ov{r}_k(x_1,t)- \ov{r}_k(x_2,t)} \le C(\ve) \snorm{x_1-x_2}
	 && \ \text{ for all $(x_1,x_2,t)\in [\ve,l]^2\times[0,T]$,} \\
	&\sup_{k\in \mathbb{N}}\snorm{\ov{r}_k(x,t_1)- \ov{r}_k(x,t_2)} \le C(\ve) \big|t_1^{\frac{3}{4}}-t_2^{\frac{3}{4}}\big|
	 && \ \text{ for all $(x,t_1,t_2)\in [\ve,l]\times[0,T]^2$.}
	\end{aligned}
\end{equation}
Furthermore, the following bound holds for all $k\in\mathbb{N}$:
\begin{equation}\label{rrtemp3}
\sup_{(x,t)\in S_l}\ov{r}_k(x,t)
= \sup_{(x,t)\in S_l}\Big( a^n + n\int_{0}^{x} \tilde{v}_k(y,t)\,\dif y \Big)^{\frac{1}{n}}
\le \big(a^n + n C(a) l\big)^{\frac{1}{n}}.
\end{equation}
By \eqref{rrtemp1a}--\eqref{rrtemp1b} and \eqref{rrtemp3}, one can apply Proposition \ref{prop:aaExt}
to obtain a subsequence (still denoted) $k \to\infty$
and a continuous function $\tilde{r}(x,t)\vcentcolon [0,\infty)\times[0,T]\to [a,\infty)$ such that
\begin{equation}\label{rtemp4}
\lim\limits_{k\to\infty} \sup_{(x,t)\in K} \snorm{\tilde{r}(x,t)-\ov{r}_{k}(x,t)} =0
\qquad \text{ for all}
\ K\Subset [0,\infty)\times[0,T].
\end{equation}
Applying the strong convergence \eqref{rtemp4} for \eqref{rrtemp2}, one obtains Lemma \ref{lemma:rkcvg}\ref{item:rkcvg2}.
	
\smallskip	
Next, let $0<x_1<x_2<\infty$. For all $k \ge 2^n x_2 C^2(a) + \frac{2^n-1}{n}C(a)$, one can verify $2\tilde{r}_k(x_1,t)\le2\tilde{r}_k(x_2,t)\le \tilde{r}_k(k,t)$. It follows that
	\begin{equation*}
		\int_{\tilde{r}_k(x_1,t)}^{\tilde{r}_k(x_2,t)} \rho_k(r,t)\,r^m \dif r
		= \int_{\tilde{r}_k(x_1,t)}^{\tilde{r}_k(x_2,t)} \overline{v}_k^{-1}(r,t)\,r^m \dif r
       = x_2 - x_1.
	\end{equation*}
	Since $\rho_k(r,t)\le C(x_1) $ for $r\in[\tilde{r}_k(x_1,t),\infty)$ from Lemma \ref{lemma:HREext},
	\begin{align*}
		x_2-x_1 = \int_{\tilde{r}_k(x_1,t)}^{\tilde{r}_k(x_2,t)} \rho_k(r,t)\,r^m \dif r 
		\le \dfrac{C(x_1)}{n} \big(\tilde{r}_k^{\, n}(x_1,t) -\tilde{r}_k^{\,n}(x_2,t)\big).
	\end{align*}
	Taking $k\to\infty$ on the above and using \eqref{rtemp4} yields
	\begin{equation*}
		0 < x_2 - x_1 < \dfrac{C(x_1)}{n} \big(\tilde{r}^{\,n}(x_1,t) -\tilde{r}^{\,n}(x_2,t)\big).
	\end{equation*}
	This shows that $x\mapsto \tilde{r}(x,t)$ is strictly increasing for each $t\in[0,T]$.
	
	For a fixed point $(x,t)\in[0,\infty)\times[0,T]$, it follows from Lemma \ref{lemma:rbd} that
	\begin{align*}
		nx \psi_{-}^{-1}(\frac{C_0}{x}) \le \tilde{r}_k^{\,n}(x,t) \le C_0(1 + x)
		\qquad \text{for all $k\ge x$}.
	\end{align*}
	Taking $k\to\infty$ and using \eqref{rtemp4},
	we conclude Lemma \ref{lemma:rkcvg}\ref{item:rkcvg3}.
\end{proof}

\smallskip
\subsection{Compactness results}\label{subsec:compact}
\begin{lemma}\label{lemma:klimholder}
There exist both a continuous function $(u,e)(r,t)\vcentcolon [a,\infty)\times (0,T]\to \R\times [0,\infty)$
and a subsequence $($still denoted$)$ $\{(u_k,e_k)\}_{k\in\mathbb{N}}$ such that,
	for each compact subset $K\Subset [a,\infty)\times[0,T]$,
	\begin{equation}\label{temp:klimue0}
		\lim\limits_{k\to\infty} \sup\limits_{(r,t)\in K} \snorm{(u_k-u,e_k-e)}(r,t) = 0.
	\end{equation}
	Moreover, for each fixed $\ve>0$, if $t\in[0,T]$ and $r_1,\, r_2 \ge \tilde{r}(\ve,t)$,
	then
	\begin{equation}\label{temp:klimue1}
		\sigma^{\frac{1}{2}}(t)\snorm{u(r_1,t)-u(r_2,t)} + \sigma(t)\snorm{e(r_1,t)-e(r_2,t)} \le C(\ve)  \snorm{r_1-r_2}^{\frac{1}{2}},
	\end{equation}
	and, if $0<t_1<t_2 \le T$ and $r \ge \sup_{t_1\le t \le t_2}\tilde{r}(\ve,t)$, then
	\begin{equation}\label{temp:klimue2}
		\sigma^{\frac{1}{2}}(t_1)\snorm{u(r,t_1)-u(r,t_2)}+ \sigma(t_1)\norm{e(r,t_1)-e(r,t_2)} \le C(\ve)  \snorm{t_2-t_1}^{\frac{1}{4}}.
	\end{equation}
	In addition,
	\begin{equation}\label{temp:klimue3}
		\begin{dcases*}
			\snorm{u(r,t)} \le C(a) \sigma^{-\frac{1}{4}}(t), \ C^{-1}(a) \le e(r,t) \le C(a) \sigma^{-\frac{1}{2}}(t) & if $(r,t)\in[a,\infty)\times[0,T]$,\\
			\snorm{u(r,t)} \le C(\ve) \sigma^{-\frac{1}{4}}(t),
			\ \ e(r,t) \le C(\ve) \sigma^{-\frac{1}{2}}(t) & if $t\in[0,T]$, $r\ge \tilde{r}(\ve,t)$.
		\end{dcases*}
	\end{equation}
\end{lemma}

\begin{proof}
We give our proof only for $u(r,t)$, since the proof is the same for $e(r,t)$.

Let $(r_1,r_2,t)\in [a,\infty)^2\times(0,T]$. By
Lemma \ref{lemma:HREext}, for all $k\in\mathbb{N}$,
\begin{align*}
\norm{u_k(r_2,t)-u_k(r_1,t)}
&\le (r_2 - r_1)^{\frac{1}{2}} r_1^{-\frac{m}{2}}
 \Big( \int_{r_1}^{r_2} \snorm{\partial_r u_k}^2 (r,t)\,r^m \dif r \Big)^{\frac{1}{2}}\\
&\le \sigma^{-\frac{1}{2}}(t) a^{-\frac{m}{2}} \Big( \int_{a}^{\infty} \sigma(t) \norm{\partial_r u_k}^2 (r,t)\,r^m
  \dif r \Big)^{\frac{1}{2}} \snorm{r_2-r_1}^{\frac{1}{2}} \\
& \le  C(a) \sigma^{-\frac{1}{2}}(t) \snorm{r_2-r_1}^{\frac{1}{2}}.
\end{align*}
Let $r\in [a,\infty)$ and $0<t_1<t_2\le T$, and set $h\vcentcolon= \sqrt{t_2-t_1}$.
By the mean value theorem,
\begin{equation}\label{aholtemp1}
\dfrac{1}{h}\int_{r}^{r+h} \snorm{ u_k(\zeta,t_1) - u_k(\zeta,t_2) }\,\dif \zeta
= \snorm{ u_k(\xi,t_1) - u_k(\xi,t_2) }
\qquad \ \text{for some $\xi\in (r,r+h)$}.
\end{equation}
By the triangle inequality and the fundamental theorem of calculus, for all $k\in\mathbb{N}$,
\begin{align*}
&\snorm{u_k(r,t_1) - u_k(r,t_2)}-\snorm{ u_k(\xi,t_1) - u_k(\xi,t_2) }
\le   \Bignorm{\int_{r}^{\xi} \big(\partial_r u_k (\zeta,t_1) - \partial_r u_k (\zeta,t_2) \big)\,\dif \zeta} \\
&\le \norm{\xi-r}^{\frac{1}{2}} r^{-\frac{m}{2}}
  \sum_{i=1}^2 \Big( \int_{r}^{\xi}\snorm{\partial_r u_k(\zeta,t_i)}^2\,\zeta^m \dif \zeta \Big)^{\frac{1}{2}} \\
&\le  a^{-\frac{m}{2}} \sigma^{-\frac{1}{2}}(t_1) \sqrt{h}
\sum_{i=1}^2 \Big(\int_{a}^{\infty}
\sigma(t_i)\snorm{\partial_r u_k(\zeta,t_i)}^2\,\zeta^m \dif \zeta \Big)^{\frac{1}{2}}
\le C(a) \sigma^{-\frac{1}{2}}(t_1) \sqrt{h}.
\end{align*}
Thus, it follows from \eqref{aholtemp1} that, for all $k\in\mathbb{N}$,
\begin{align*}
&\snorm{u_k(r,t_1) - u_k(r,t_2)}\\
&\le  C(a) \sigma^{-\frac{1}{2}}(t_1) \sqrt{h}
  + \dfrac{1}{h}\int_{r}^{r+h}  \int_{t_1}^{t_2} \snorm{ \partial_t u_k (\zeta,t)}\,\dif t \dif \zeta \\
&\le   C(a) \sigma^{-\frac{1}{2}}(t_1)\sqrt{h}
 + \dfrac{\sigma^{-\frac{1}{2}}(t_1) a^{-\frac{m}{2}}}{h}
  \Big(\int_{r}^{r+h}  \int_{t_1}^{t_2} \dif t  \dif \zeta\Big)^{\frac{1}{2}}
   \Big(\int_{t_1}^{t_2}\int_{a}^{\infty}\sigma(t)\snorm{ \partial_t u_k}^2\,\zeta^m \dif\zeta\dif t \Big)^{\frac{1}{2}}\\
&\le  C(a) \sigma^{-\frac{1}{2}}(t_1) \sqrt{h} + C(a)\sigma^{-\frac{1}{2}}(t_1) h^{-\frac{1}{2}}
  \snorm{t_2-t_1}^{\frac{1}{2}}
  \le C(a)\sigma^{-\frac{1}{2}}(t_1) \snorm{t_2-t_1}^{\frac{1}{4}}.
\end{align*}
Therefore, we have
\begin{equation}\label{aholtemp2}
		\begin{dcases*}
			\sup\limits_{k\in \mathbb{N}}\,\snorm{u_k(r_2,t)-u_k(r_1,t)} \le C(a) \sigma^{-\frac{1}{2}}(t) \snorm{r_2-r_1}^{\frac{1}{2}} &\quad if $(r_1,r_2,t)\in[a,\infty)^2\times(0,T]$,\\
			\sup\limits_{k\in \mathbb{N}}\,\snorm{u_k(r,t_2)-u_k(r,t_1)}
			\le C(a) \sigma^{-\frac{1}{2}}(t_*) \snorm{t_2 - t_1}^{\frac{1}{4}}
			 &\quad if $(r,t_1,t_2)\in[a,\infty)\times(0,T]^2$,
		\end{dcases*}
\end{equation}
where $t_*\equiv \min\{t_1,t_2\}$. Let $S_l=[0,\infty)\times[l^{-1},T]$ for each $l\in\mathbb{N}$.
Then we can apply Proposition \ref{prop:aaExt} with \eqref{aholtemp2} and $S_l$
to obtain a continuous function $u(r,t)\vcentcolon [a,\infty)\times(0,T]\to \R$ and a subsequence (still denoted) $\{u_k\}_{k\in\mathbb{N}}$ such that \eqref{temp:klimue0} holds.

\smallskip
Next, let $\ve\in(0,1]$. Since $x\to \tilde{r}(x,t)$ is strictly increasing and $(x,t)\mapsto\tilde{r}(x,t)$ is continuous,
	it follows that $ d  \vcentcolon= \inf_{t\in[0,T]}\{\tilde{r}(\ve,t) - \tilde{r}(\frac{\ve}{2},t)\}>0$.
	By Lemma \ref{lemma:rkcvg}, there exists $N_\ve\in\mathbb{N}$ such that
	$\sup_{0\le t \le T}\snorm{\tilde{r}_k(\frac{\ve}{2},t)-\tilde{r}(\frac{\ve}{2},t)} \le \frac{d}{2}$
	if $k\ge N_\ve$, which implies that
\begin{equation}\label{veholtemp1}
\tilde{r}_k(\frac{\ve}{2},t) < \tilde{r}(\ve,t) \qquad \text{for each $t\in[0,T]$ and $k\ge N_\ve$}.
\end{equation}
Let $(r_1,r_2,t)$ be a triplet such that $t\in(0,T]$ and $r_1,\, r_2\in [\tilde{r}(\ve,t),\infty)$.
Then \eqref{veholtemp1} implies that $r_1, r_2 > \tilde{r}_k(\frac{\ve}{2},t)$ if $k\ge N_\ve$.
By
Lemma \ref{lemma:HREext}, it follows that, for all $k\ge N_\ve$,
\begin{align*}
\snorm{u_k(r_2,t)-u_k(r_1,t)}
\le \dfrac{\sigma^{-\frac{1}{2}}(t)}{\tilde{r}_k^{\,\frac{m}{2}}(\frac{\ve}{2},t)}
 \Big(\int_{\tilde{r}_k(\frac{\ve}{2},t)}^{\infty} \sigma(t) \snorm{\partial_r u_k}^2(r,t)\,r^m \dif r \Big)^{\frac{1}{2}} \snorm{r_2-r_1}^{\frac{1}{2}}
 \le C(\ve) \sigma^{-\frac{1}{2}}(t) \snorm{r_2-r_1}^{\frac{1}{2}}.
\end{align*}
Moreover, let $(r,t_1,t_2)\in [a,\infty)\times(0,T]^2$ be a triplet such that
$0<t_1 < t_2\le T$ and $r \ge \sup_{t_1\le t\le t_2}\tilde{r}(\ve,t)$.
Then, by \eqref{veholtemp1}, $r > \tilde{r}_k(\frac{\ve}{2},t)\ge C^{-1}(\ve)$ for all $t\in[0,T]$ if $k\ge N_\ve$.
Set $h\vcentcolon= \sqrt{t_2-t_1}$. Then, by the mean value theorem,
\begin{equation*}
\dfrac{1}{h}\int_{r}^{r+h} \snorm{ u_k(\zeta,t_1) - u_k(\zeta,t_2) }\,\dif \zeta
= \snorm{ u_k(\xi,t_1) - u_k(\xi,t_2) } \qquad \text{for some $\xi\in (r,r+h)$}.
\end{equation*}
By the triangle inequality,
for all $k\ge N_\ve$,
\begin{align*}
&\snorm{u_k(r,t_1) - u_k(r,t_2)}-\snorm{ u_k(\xi,t_1) - u_k(\xi,t_2) }\\
&\le \sqrt{\snorm{\xi-r}}
r^{-\frac{m}{2}}
 \sum_{i=1}^2\Big(\int_{r}^{\xi}\snorm{\partial_r u_k(\zeta,t_i)}^2\,\zeta^m\dif\zeta\Big)^{\frac{1}{2}} \\
&\le  C(\ve) \sqrt{h}
\sigma^{-\frac{1}{2}}(t_1)
  \sum_{i=1}^2 \Big( \int_{\tilde{r}_k(\frac{\ve}{2},t_i)}^{\infty}
  \sigma(t_i)\snorm{\partial_r u_k(\zeta,t_i)}^2 \zeta^m \dif \zeta \Big)^{\frac{1}{2}}
  \le C(\ve) \sigma^{-\frac{1}{2}}(t_1)\sqrt{h}
\end{align*}
It follows that, for all $k\ge N_\ve$,
\begin{align*}
&\snorm{u_k(r,t_1) - u_k(r,t_2)} \\
&\le  C(\ve) \sqrt{h}
\sigma^{-\frac{1}{2}}(t_1)
+ \dfrac{1}{h}\int_{r}^{r+h}  \int_{t_1}^{t_2} \snorm{ \partial_t u_k (\zeta,t)}\,\dif t  \dif \zeta \\
&\le   C(\ve) \sqrt{h}
\sigma^{-\frac{1}{2}}(t_1) + \dfrac{\sigma^{-\frac{1}{2}}(t_1) r^{-\frac{m}{2}}}{h}\Big(\int_{r}^{r+h}  \int_{t_1}^{t_2} \dif t  \dif \zeta\Big)^{\frac{1}{2}}
\Big( \int_{t_1}^{t_2}\int_{\tilde{r}_k(\frac{\ve}{2},t)}^{\infty} \sigma \snorm{ \partial_t u_k }^2\,\zeta^m
  \dif\zeta \dif t \Big)^{\frac{1}{2}}\\
&\le  C(\ve) \sqrt{h}
\sigma^{-\frac{1}{2}}(t_1) + C(\ve) h^{-\frac{1}{2}} \snorm{t_2-t_1}^{\frac{1}{2}} \sigma^{-\frac{1}{2}}(t_1)
\le C(\ve) \sigma^{-\frac{1}{2}}(t_1) \snorm{t_2-t_1}^{\frac{1}{4}}.
\end{align*}
Then
\begin{equation*}
\begin{dcases*}
\sup\limits_{k\ge N_\ve}\snorm{u_k(r_2,t)-u_k(r_1,t)} \le C(\ve) t^{-\frac{1}{2}} \snorm{r_2-r_1}^{\frac{1}{2}} &if $t\in(0,T]$, $r_1, r_2 \ge \tilde{r}(\ve,t)$,\\
\sup\limits_{k\ge N_\ve}\snorm{u_k(r,t_2)-u_k(r,t_1)} \le C(\ve) \sigma^{-\frac{1}{2}}(t_*)
\snorm{t_2 - t_1}^{\frac{1}{4}} &if $0<t_1<t_2\le T$, $r\ge R(\ve,t_1,t_2)$,
\end{dcases*}
\end{equation*}
where $R(\ve,t_1,t_2)\vcentcolon= \sup_{t_1\le t\le t_2}\tilde{r}(\ve,t)$.
Taking $k\to\infty$ and using \eqref{temp:klimue0}, we obtain the continuity estimates for $u$ in \eqref{temp:klimue1}--\eqref{temp:klimue2}.
	
Finally, Lemma \ref{lemma:HREext} and \eqref{temp:klimue0} lead to $\eqref{temp:klimue3}_1$.
In addition, for each $\ve\in(0,1]$, \eqref{veholtemp1} implies that $r\ge \tilde{r}_k(\frac{\ve}{2},t)$
if $r\ge \tilde{r}(\ve,t)$ and $k\ge N_\ve$.
Hence, $\norm{u_k(r,t)}\le C(\ve)\sigma^{-\frac{1}{4}}(t)$ for all $(r,t,k)$ such that $k\ge N_\ve$, $t\in[0,T]$,
and $r \ge \tilde{r}(\ve,t)$.
Letting $k\to\infty$ and using \eqref{temp:klimue0}, we obtain $\eqref{temp:klimue3}_2$.
\end{proof}

\begin{proposition}\label{prop:veDomain}
Let $\ve\in(0,1]$, and let $(r,t)\in[a,\infty)\times(0,T]$ be such that $r> \tilde{r}(\ve,t)$.
Denote $\D_{\delta}(r,t)\vcentcolon= \{ (\zeta,s) \vcentcolon \sqrt{(\zeta-r)^2 + (s-t)^2} < \delta \}$.
Then there exist $N(\ve,r,t)\in\mathbb{N}$ and $\delta(\ve,r,t)>0$ such that,
for all $\delta\in (0,\delta(\ve,r,t))$ and $ k\ge N(\ve,r,t)$,
\begin{equation*}
\D_\delta(r,t) \subset \{ (\zeta,s)\in[a,\infty)\times[0,T]\,\vcentcolon\,\zeta> \tilde{r}_k(\ve,s) \}.
\end{equation*}
\end{proposition}

\begin{proof}
Let $\eta\vcentcolon=r-\tilde{r}(\ve,t)>0$.
Since $(\zeta,s) \mapsto \zeta-\tilde{r}(\ve,s)$ is continuous,
there exists $\delta(\ve,r,t)>0$ such that, for all $\delta\le \delta(\ve,r,t)$,
if $(\zeta,s)\in \D_\delta(r,t)$, then
$\snorm{\zeta-\tilde{r}(\ve,s) - r + \tilde{r}(\ve,t)} < \frac{\eta}{2}$ which implies that
$\zeta> \frac{\eta}{2} + \tilde{r}(\ve,s)$.
Therefore,
\begin{equation}\label{temp:veDomain}
\D_\delta(r,t) \subset \{ (\zeta,s)\in[a,\infty)\times[0,T]\,\vcentcolon\,\zeta>\frac{\eta}{2} + \tilde{r}(\ve,s) \}.
\end{equation}
By Lemma \ref{lemma:rkcvg}, $\lim_{k\to\infty} \sup_{s\in[0,T]}\snorm{\tilde{r}_k(\ve,s)-\tilde{r}(\ve,s)}=0$.
Thus, there exists $N(\ve,r,t)\in\mathbb{N}$ such that
$\tilde{r}(\ve,s)+\frac{\eta}{4}> \tilde{r}_k(\ve,s)$ for all $s\in[0,T]$ and $k\ge N(\ve,r,t)$.
If $(\zeta,s)$ is such that $\zeta> \frac{\eta}{2} + \tilde{r}(\ve,s)$,
then $\zeta > \frac{\eta}{4} + \tilde{r}(\ve,s) > \tilde{r}_k(\ve,s)$ for all $k\ge N(\ve,r,t)$.
We combining this with \eqref{temp:veDomain} to conclude the proof.
\end{proof}

\begin{lemma}\label{lemma:rhoWeak}
There exist $\rho\ge 0$
and a subsequence $($still denoted$)$ $\{\rho_k\}_{k\in\mathbb{N}}$ such that
\begin{equation}\label{temp:rhoWeak0}
(\rho_k-1) \overset{\ast}{\rightharpoonup}  (\rho-1)
\qquad \text{in $L^{\infty}\big( 0,T ; L^2( [a,\infty), r^m \dif r ) \big)$}.
\end{equation}
Moreover, for each $\ve>0$,
\begin{equation}\label{temp:rhoWeak}
\begin{dcases*}
C^{-1}(\ve)\le \rho(r,t) \le C(\ve)
  &\,\, for a.e. $(r,t)\in\{(\zeta,s)\in[a,\infty)\times[0,T]\,\vcentcolon\,\zeta\ge \tilde{r}(\ve,t)\}$,\\
C^{-1}(a)\le \rho(r,t) \le C(a) &\,\, for a.e. $(r,t)\in[a,\infty)\times[0,T]$.
\end{dcases*}
\end{equation}
\end{lemma}

\begin{proof}
From Lemma \ref{lemma:HREext},
\begin{equation}\label{temp:rho1}
\sup_{k\in\mathbb{N}}\sup_{ 0\le t \le  T}\int_{a}^{\infty}\snorm{\rho_{k}-1}^2(r,t)\,\dif r \dif t
\le a^{-m}\sup_{k\in\mathbb{N}}\sup_{ 0\le t \le  T}\int_{a}^{\infty}\snorm{\rho_{k}-1}^2(r,t)\,r^m\dif r \dif t\le C(a).
\end{equation}
Thus there exist $(f,\rho-1) \in L^{\infty}\big(0,T;L^2([a,\infty),\dif r)\big)$ and a subsequence (still denoted) $\{\rho_k\}_{k\in\mathbb{N}}$ such that, as $k\to\infty$,
\begin{equation}\label{temp:rhowc}
\rho_k-1\overset{\ast}{\rightharpoonup} (\rho-1), \quad
r^{\frac{m}{2}} (\rho_k-1) \overset{\ast}{\rightharpoonup} f
\qquad\,\text{ in $L^{\infty}\big(0,T;L^2( [a,\infty), \dif r)\big)$}.
\end{equation}
Let $\phi\in C^{\infty}_{\rm c}([a,\infty)\times[0,T])$.
It follows from \eqref{temp:rhowc} that, as $k\to\infty$,
\begin{align*}
\int_{0}^{T} \int_{a}^{\infty} f \phi\,\dif r \dif t
\leftarrow \int_{0}^{T} \int_{a}^{\infty} r^{\frac{m}{2}}(\rho_k -1 ) \phi\,\dif r \dif t
\to \int_{0}^{T} \int_{a}^{\infty}  (\rho -1 ) r^{\frac{m}{2}} \phi\, \dif r \dif t.
\end{align*}
By the fundamental lemma of calculus,
$f(r,t)= r^{\frac{m}{2}} (\rho(r,t)-1)$ for {\it a.e.} $(r,t)\in[a,\infty)\times[0,T]$.
By \eqref{temp:rhowc} again, the weak-star convergence \eqref{temp:rhoWeak0} is obtained.
	
Next, fix $\ve\in(0,1]$ and a point $(r,t)\in[a,\infty)\times(0,T]$ such that $r>\tilde{r}(\ve,t)$.
Let $\delta(\ve,r,t)>0$ and $N(\ve,r,t)\in\mathbb{N}$ be obtained in Proposition \ref{prop:veDomain}.
For $0<\delta \le \delta(\ve,r,t)$,  set
\begin{equation}\label{rhobdtemp3}
\phi_{\delta}^{(r,t)}(\zeta,s)\vcentcolon= \dfrac{\mathbbm{1}_{\D_{\delta}(r,t)}(\zeta,s)}{\snorm{\D_\delta(r,t)}} = \begin{dcases*}
			\snorm{\D_\delta(r,t)}^{-1} & if $(\zeta,s)\in \D_\delta(r,t)$,\\
	0 & if $(\zeta,s)\in [a,\infty)\times[0,T] \backslash \D_\delta(r,t)$,
\end{dcases*}
\end{equation}
	where $\D_\delta(r,t)$ is defined in Proposition \ref{prop:veDomain}. It follows that 
	$\phi_\delta^{(r,t)}\in L^1\big(0,T;L^2( [a,\infty),\dif r )\big)$.
	Moreover, by Proposition \ref{prop:veDomain},  for all $k\ge N(\ve,r,t)$ and $\delta\le \delta(\ve,r,t)$,
	\begin{equation}\label{rhobdtemp1}
		\supp(\phi_\delta^{(r,t)}) \subset \{ (\zeta,s)\in[a,\infty)\times(0,T]\,\vcentcolon\,\zeta > \tilde{r}_k(\ve,s) \}.
	\end{equation}
	For $0<\delta\le \delta(\ve,r,t)$ and $\eta>0$, 
	there exists $N_{\eta,\delta}\in\mathbb{N}$ by \eqref{temp:rhowc} such that, for all $k\ge N_{\eta,\delta}$,
	\begin{equation}\label{rhobdtemp2}
		-\eta + \int_{0}^{T}\int_{a}^{\infty} \rho_{k} \phi_\delta^{(r,t)} \dif \zeta \dif s \le \int_{0}^{T}\int_{a}^{\infty}  \rho\phi_\delta^{(r,t)} \dif \zeta \dif s \le \eta + \int_{0}^{T}\int_{a}^{\infty} \rho_{k}\phi_\delta^{(r,t)} \dif \zeta \dif t.
	\end{equation}
Let $N_1\vcentcolon= \max\{ N_{\eta,\delta}, N(\ve,r,t) \}$. Then, by \eqref{rhobdtemp1} and Lemma \ref{lemma:HREext},
it follows that, for all $k\ge N_1$, $C^{-1}(\ve) \le  \rho_k(\zeta,s)\le C(\ve)$ if $(\zeta,s)\in \supp(\phi_\delta^{(r,t)})$.
Using this in \eqref{rhobdtemp2} yields
\begin{align*}
C^{-1}(\ve) - \eta \le \dfrac{1}{\snorm{\D_{\delta}(r,t)}} \iint_{\D_{\delta}(r,t)} \rho(\zeta,s)\,\dif \zeta \dif s
\le C(\ve) + \eta
\qquad \text{for $\eta>0$ and $\delta\in (0,\delta(\ve,r,t)]$}.
\end{align*}
By the Lebesgue differentiation theorem over $\delta$ and the fact that $\eta>0$ is arbitrarily chosen,
$\eqref{temp:rhoWeak}_1$ is proved.
	
Finally, we consider the proof for $\eqref{temp:rhoWeak}_2$. For $(r,t)\in[a,\infty)\times[0,T]$ and $\delta>0$,
we use $\phi_\delta^{(r,t)}$ defined in \eqref{rhobdtemp3}.
Then $\eqref{temp:rhoWeak}_2$ can be proved by repeating the same argument as before
and using $C^{-1}(a) \le \rho_k(r,t) \le C(a)$ for all $(k,r,t)\in\mathbb{N}\times[a,\infty)\times[0,T]$
from Lemma \ref{lemma:HREext}.
\end{proof}

We now introduce the following function spaces: for each connected interval $I\subseteq [0,\infty)$, define the space $\tilde{H}_0^1(I,r^m\dif r)$ to be the closure of
\begin{equation}\label{eqs:tildeH1}
	\mathcal{D}_0(I)\vcentcolon=\{ \phi\in C^{\infty}(I)\,\vcentcolon\,\exists N>0 \ \text{such that} \ [0,N]\subset I \ \text{and} \ \phi(r)=0 \ \text{for} \ r\in I\cap[N,\infty)  \}
\end{equation}
via the $H^1(I,r^m\dif r)$--norm. We also denote $\tilde{H}^{-1}( I, r^m \dif r)$ as its dual space.

\begin{lemma}\label{lemma:rhoScvg}
Let $(\rho,u,e)$ be the limit function obtained in {\rm Lemma \ref{lemma:klimholder}}--{\rm \ref{lemma:rhoWeak}}.
Then
\begin{enumerate}[label=(\roman*),ref=(\roman*),font={\normalfont\rmfamily}]
\item\label{item:rhoScvg1}
For all $L\in\mathbb{N}$, $\rho\in C^0\big( [0,T]; \tilde{H}^{-1}([a,L],r^m \dif r) \big)$ and
 there exists a further subsequence $($still denoted$)$ $\rho_{k}(r,t)$  such that
\begin{equation*}
\lim\limits_{k\to\infty} \sup_{t\in[0,T]}\sbnorm{\rho_{k}(\cdot,t)- \rho (\cdot,t)}_{\tilde{H}^{-1}([a,L], r^m \dif r)} = 0.
\end{equation*}

 \item\label{item:rhoScvg2}
 For each $\ve\in(0,1]$, define $\rho^{(\ve)} (r,t) \vcentcolon = \rho(r,t) \chi_\ve(r,t)$ with $\chi_\ve(r,t)$
 given by
 \begin{equation}\label{temp:chive}
 \chi_\ve(r,t) \vcentcolon = \begin{dcases*}
    			1 & if $t\in[0,T]$ and $r\in [\tilde{r}(\ve,t),\infty)$,\\
    			0 & otherwise.
 \end{dcases*}
 \end{equation}
 Then, for all $L\in\mathbb{N}$,
\begin{align*}
\begin{dcases}
 \rho^{(\ve)} \in C^0\big( [0,T] ; \tilde{H}^{-1}([0,L], r^m \dif r) \big) ,\\
\sbnorm{\rho^{(\ve)} (\cdot,t_1)- \rho^{(\ve)}(\cdot,t_2)}_{\tilde{H}^{-1}([0,L],r^m\dif r)} \le C(\ve)\snorm{t_1-t_2}
 \qquad \text{for $t_1, t_2 \in [0,T]$}.
\end{dcases}
\end{align*}
\end{enumerate}
\end{lemma}

\begin{proof}
Let $\phi\in \tilde{H}^1_0([a,\infty),r^m \dif r)$.
Then, by  \eqref{eqs:extfunc} and \eqref{eqs:coorderiv}--\eqref{eqs:extcut},
\begin{align*}
&\int_{a}^{\infty} \big\{\rho_k(r,t_2) - \rho_k(r,t_1)\big\} \phi(r)\,r^m \dif r
= \int_{t_1}^{t_2} \int_{a}^{\infty} \partial_t \rho_k (r,t)  \phi(r)\,r^m \dif r \dif t\\
&= \int_{t_1}^{t_2} \int_{a}^{\infty} (\varphi_k \partial_t \overline{v}_k^{-1})(r,t)  \phi(r)\,r^m \dif r \dif t = \int_{t_1}^{t_2} \int_{a}^{\tilde{r}_k(k,t)} (\varphi_k \partial_t \overline{v}_k^{-1})(r,t)  \phi(r)\,r^m \dif r \dif t\\
&= - \int_{t_1}^{t_2} \int_{0}^{k} (\varphi_k \,\phi)(\tilde{r}_k(x,t),t) \big(\tilde{v}_k^{-1} \D_t \tilde{v}_k
 +\tilde{r}_k^m \tilde{u}_k\D_x\tilde{v}_k^{-1}\big)(x,t)\, \dif x \dif t.
\end{align*}
Integrating by parts and applying the continuity equation $\eqref{eqs:LFNS-k}_1$, the boundary condition
$\tilde{u}_k(0,t)=\tilde{u}_k(k,t)$, and Theorem \ref{thm:priori}\ref{item:prio0}--\ref{item:prio1}, we have
	\begin{align*}
		&\int_{a}^{\infty} \big(\rho_k(r,t_2) - \rho_k(r,t_1)\big) \phi(r)\, r^m \dif r\\
	&= \int_{t_1}^{t_2} \int_{0}^{k} \tilde{u}_k (x,t) \big(\partial_r\varphi_k\, \phi
	+ \partial_r\phi\,\varphi_k \big)(\tilde{r}_k(x,t),t)\,\dif x \dif t \\
	&\le  C\int_{t_1}^{t_2} \Big( \int_{0}^{k} \dfrac{\snorm{\tilde{u}_k}^2}{\tilde{v}_k}\dif x\Big)^{\frac{1}{2}}
	\Big(\int_{0}^{k}\tilde{v}_k (x,t)\big(\snorm{\phi}^2 +\snorm{\partial_r\phi}^2\big)(\tilde{r}_k(x,t))\,\dif x  \Big)^{\frac{1}{2}} \dif t\\
		&\le C(a) \bnorm{\phi(\cdot)}_{H^1([a,\infty),r^m\dif r)} \snorm{t_2-t_1}.
	\end{align*}
	Thus, the following equicontinuity inequality holds:
	\begin{equation*}
		\sup_{k\in\mathbb{N}} \sbnorm{\rho_k(\cdot,t_2)-\rho_k(\cdot,t_1)}_{\tilde{H}^{-1}([a,\infty),r^m \dif r)} \le C(a) \snorm{t_2-t_1}.
	\end{equation*}
	By the same calculation, $\rho_k \in L^{\infty}(0,T;\tilde{H}^{-1}([a,L],r^m \dif r))$ for all $(k,L)\in\mathbb{N}^2$.
	In addition, it follows from Lemma \ref{lemma:HREext} that
	\begin{equation*}
		\sup_{k\in\mathbb{N}} \sup_{t\in[0,T]} \int_{a}^{\infty} \snorm{\rho_k -1 }^2 (r,t)\, r^m \dif r\le C(a).
	\end{equation*}
Then we can apply Proposition \ref{prop:scHm1} 
to obtain $f$ such that, for all $L\in \mathbb{N}$,
	\begin{equation*}
		\begin{dcases}
		f\in C^0\big( [0,T] ; \tilde{H}^{-1}([a,L], r^m \dif r) \big), \ \ f-1\in L^{\infty}( 0,T; L^2( [a,\infty), r^m \dif r ) ),\\
		\lim\limits_{k\to \infty}\sup\limits_{t\in[0,T]}\sbnorm{\rho_{k}(\cdot,t)-f(\cdot,t)}_{\tilde{H}^{-1}( [a,L], r^m \dif r)} = 0,\\
		\sbnorm{f(\cdot,t_1)-f(\cdot,t_2)}_{\tilde{H}^{-1}( [a,L], r^m \dif r)} \le C(a) \snorm{t_1-t_2}
		 \qquad \text{for all $t_1,\,t_2 \in[0,T]$.}
		\end{dcases}
	\end{equation*}
	By Lemma \ref{lemma:rhoWeak} and the fundamental lemma of calculus, we see that $\rho=f$ for {\it a.e.} $(r,t)\in[a,\infty)\times[0,T]$.
This completes the proof of \ref{item:rhoScvg1}.

For each $\ve>0$ and $k\in\mathbb{N}$, define $\rho_k^{(\ve)} (r,t) \vcentcolon = \rho_k(r,t) \chi_\ve^k(r,t)$
with
\begin{equation}\label{temp:chivek}
		\chi_\ve^k(r,t) \vcentcolon = \begin{dcases*}
			1 & if $t\in[0,T]$ and $r\in [\tilde{r}_k(\ve,t),\infty)$,\\
			0 & otherwise.
		\end{dcases*}
\end{equation}
If $k\ge M(\ve,a)\vcentcolon= \ve C^2(a) + \frac{2^n -1}{n}a^n C(a)>0$,
then it follows from construction \eqref{eqs:COfunc} that $\varphi_k(\tilde{r}_k(\ve,t),t) = 1$ for all $t\in[0,T]$.
Thus, by \eqref{eqs:extfunc} and \eqref{eqs:transfunc},
\begin{equation}\label{temp:vekMa}
(\rho_k,u_k,e_k)(\tilde{r}_k(\ve,t),t) = (\overline{v}_k^{-1},\overline{u}_k,\overline{e}_k)(\tilde{r}_k(\ve,t),t) = (\tilde{v}_k^{-1},\tilde{u}_k,\tilde{e}_k)(\ve,t).
\end{equation}
Let $\phi\in \tilde{H}^1_0([0,\infty),r^m \dif r)$. Then, by \eqref{temp:vekMa} and the Leibniz rule,
for all $k\ge M(\ve,a)$,
\begin{align}\label{veSctemp0}
&\int_{0}^{\infty} \big(\rho_k^{(\ve)}(r,t_2) - \rho_k^{(\ve)}(r,t_1)\big) \phi(r)\,r^m \dif r \nonumber\\
&= \int_{t_1}^{t_2}\int_{\tilde{r}_k(\ve,t)}^{\infty} \phi \partial_t \rho_k\,r^m \dif r \dif t
- \int_{t_1}^{t_2} \tilde{r}_k^{\,m}(\ve,t) \phi(\tilde{r}_k(\ve,t)) \tilde{u}_k(\ve,t) \tilde{v}_k^{-1}(\ve,t)\,\dif t\nonumber\\
&=\vcentcolon I_1 + I_2.
\end{align}
Using \eqref{eqs:extfunc} and \eqref{eqs:coorderiv}--\eqref{eqs:transfunc}, $I_1$ can be written as
\begin{align*}
I_1
&= \int_{t_1}^{t_2} \int_{\tilde{r}_k(\ve,t)}^{\tilde{r}_k(k,t)} \varphi_k \partial_t \overline{v}_k^{-1} (r,t)  \phi(r)\, r^m \dif r \dif t\\
&= \int_{t_1}^{t_2} \int_{\ve}^{k} \big(\varphi_k \, \phi\big)(\tilde{r}_k(x,t),t) \big\{ \tilde{v}_k \D_t \tilde{v}_k^{-1} - \tilde{r}_k^m \tilde{u}_k \D_x \tilde{v}_k^{-1}  \big\}(x,t)\,\dif x \dif t.
\end{align*}
Integrating by parts and using the boundary condition $\tilde{u}_k(k,t)=0$, the continuity equation $\eqref{eqs:LFNS-k}_1$, and Theorem \ref{thm:priori}\ref{item:prio0}--\ref{item:prio1},
it follows that
\begin{align*}
I_1 &= \int_{t_1}^{t_2} \int_{\ve}^{k} (\varphi_k\cdot \phi)(\tilde{r}_k(x,t),t) \big\{ - \tilde{v}_k^{-1}\D_x(\tilde{r}_k^m\tilde{u}_k)-\tilde{r}_k^m\tilde{u}_k\D_x\tilde{v}_k^{-1} \big\}\,\dif x \dif t\\
&=\int_{t_1}^{t_2}  \phi(\tilde{r}_k(\ve,t)) (\tilde{r}_k^m\tilde{u}_k \tilde{v}_k^{-1})(\ve,t)\,\dif t
		+ \int_{t_1}^{t_2} \int_{\ve}^{k} \tilde{u}_k (x,t)
		\big(\partial_r\varphi_k\, \phi + \partial_r\phi\,\varphi_k \big)(\tilde{r}_k(x,t),t)
		\,\dif x \dif t\\
&\le -I_2 + C\int_{t_1}^{t_2}\Big(\int_{\ve}^{k}\dfrac{\snorm{\tilde{u}_k}^2}{\tilde{v}_k}\dif x\Big)^{\frac{1}{2}} \Big(\int_{\ve}^{k} \tilde{v}_k (x,t) \big(\snorm{\phi}^2 + \snorm{\partial\phi}^2 \big)(\tilde{r}_k(x,t))\,
		 \dif x  \Big)^{\frac{1}{2}} \dif t\\
		&\le - I_2 + C(\ve) \sbnorm{\phi}_{H^1([0,\infty),r^m\dif r)} \snorm{t_2-t_1}.
	\end{align*}
	Substituting this back into \eqref{veSctemp0}, it follows that
	the following equicontinuity inequality holds:
	\begin{equation}\label{veSctemp1}
		\sup_{k\in\mathbb{N}} \sbnorm{\rho_k^{(\ve)}(\cdot,t_2)-\rho_k^{(\ve)}(\cdot,t_1)}_{\tilde{H}^{-1}([0,\infty),r^m \dif r)} \le C(\ve) \snorm{t_2-t_1}.
	\end{equation}
Let $L\in\mathbb{N}$ be such that $\tilde{r}_k(1,t)< L$ for all $(k,t)\in\mathbb{N}\times[0,T]$.
By the same derivation, $\rho_k^{(\ve)} \in L^{\infty}( 0,T; \tilde{H}^{-1}([0,L],r^m \dif r) )$ for all $(k,L)\in\mathbb{N}$. Moreover, set $r_{\ve}=\frac{1}{2}\inf_{t\in[0,T]}\tilde{r}(\ve,t)>0$.
Then, by Lemma \ref{lemma:rkcvg}, there exists $N_\ve\in\mathbb{N}$ such that $	r_{\ve}<\tilde{r}_k(\ve,t)$
for all $k\ge N_\ve$ and $t\in[0,T]$. From \eqref{veSctemp1} and Lemma \ref{lemma:HREext}, it follows that
\begin{equation*}
\begin{dcases}
\sup_{k\ge N_\ve}\sup\limits_{t\in[0,T]} \int_{r_\ve}^{\infty} \snorm{\rho^{(\ve)}_k -1}^2(r,t)\,r^m \dif r \le C(\ve),\\
\sup_{k\ge N_\ve} \sbnorm{\rho_k^{(\ve)}(\cdot,t_2)-\rho_k^{(\ve)}(\cdot,t_1)}_{\tilde{H}^{-1}([r_\ve,L],r^m \dif r)}
\le C(\ve) \snorm{t_2-t_1}.
\end{dcases}
\end{equation*}
Thus, we can apply Proposition \ref{prop:scHm1}
to obtain $f_\ve$
such that, for all $L\in \mathbb{N}$, 	
\begin{equation}\label{veSctemp2}
\begin{dcases}
f_\ve\in C^0\big( [0,T] ; \tilde{H}^{-1}( [r_\ve,L], r^m \dif r) \big), \ \ f_\ve-1\in L^{\infty}\big( 0,T; L^2( [r_\ve,\infty), r^m \dif r) \big),\\
\lim\limits_{k\to \infty}\sup\limits_{t\in[0,T]}\sbnorm{\rho_{k}^{(\ve)}(\cdot,t)-f_\ve(\cdot,t)}_{\tilde{H}^{-1}( [r_\ve,L], r^m \dif r)} = 0,\\
\sbnorm{f_\ve(\cdot,t_1)-f_\ve(\cdot,t_2)}_{\tilde{H}^{-1}( [r_\ve,L], r^m \dif r )} \le C(\ve) \snorm{t_1-t_2}
\qquad \text{for all $t_1,\,t_2 \in[0,T]$.}
\end{dcases}
\end{equation}
Let $\phi\in C^{\infty}_c([a,\infty)\times[0,T])$. Then
\begin{equation*}
\begin{aligned}
&\int_{0}^{T}\int_{r_\ve}^{\infty} \phi\big( \rho^{(\ve)} - f_\ve \big)\,\dif r \dif t\\
&= \int_{0}^{T}\int_{r_\ve}^{\infty} \phi \chi_\ve ( \rho  - \rho_k )\,\dif r \dif t
+ \int_{0}^{T}\int_{r_\ve}^{\infty} \phi \rho_k ( \chi_\ve - \chi_\ve^k )\,\dif r \dif t
+ \int_{0}^{T}\int_{r_\ve}^{\infty} \phi ( \rho_k^{(\ve)} - f_\ve )\,\dif r \dif t \\
&=: \sum_{i=1}^{3} \text{II}_i^k.
\end{aligned}
\end{equation*}

By Lemma \ref{lemma:rhoWeak}, $\text{II}_1^k \to 0$ as $k\to\infty$.
By \eqref{temp:chive}--\eqref{temp:chivek} and Lemma \ref{lemma:rkcvg},
$\chi_{\ve}^{k}\to\chi_\ve$ {\it a.e.} as $k\to\infty$.
Since $\rho_k \le C(a)$ from Lemma \ref{lemma:HREext},
by the dominated convergence theorem,
$\text{II}_2^k \to 0$ as $k\to\infty$.
Finally, $\text{II}_3^k \to 0$ as $k\to\infty$ by \eqref{veSctemp2}.
Therefore,
$\rho^{(\ve)}=f_\ve$ for {\it a.e.} $(r,t)\in[r_\ve,\infty)\times[0,T]$.
Since $\rho^{(\ve)}=\rho^{(\ve)}_k=0$ for $r\in[0,r_\ve]$ and $k\ge N_\ve$,
Lemma \ref{lemma:rhoScvg}\ref{item:rhoScvg2} is proved.
\end{proof}

\begin{lemma}\label{lemma:Relation}
The following relations hold{\rm :}
\begin{equation}\label{temp:relation}
\begin{aligned}
&x = \int_{a}^{\tilde{r}(x,t)} \rho(r,t)\,r^m \dif r &&  \text{ for a.e. $t\in [0,T]$ and for each $x>0$,}\\
&\tilde{r}(x,t) = \tilde{r}_a^0(x) + \int_{0}^{t} u( \tilde{r}(x,s), s)\,\dif s
 && \text{ for all $(x,t)\in[0,\infty)\times[0,T]$,}
\end{aligned}
\end{equation}
where the initial data $\tilde{r}_a^0(x)$ is the function constructed in {\rm \S \ref{subsec:ffaprox}},
which satisfies
\begin{equation*}
x = \int_{a}^{\tilde{r}_a^0(x)} \rho_a^0(r)\,r^m \dif r \qquad \text{for all $x\in[0,\infty)$}.
\end{equation*}
\end{lemma}

\begin{proof}
Let $(x,t)\in [0,\infty)\times[0,T]$. Then, for $k \ge M(a,x) \vcentcolon=\frac{2^n-1}{n}C(a)a^n + 2^n C(a)^2 x$,
$2\tilde{r}_k(x,t) \le \tilde{r}_k(k,t)$.
Thus, by \eqref{eqs:extfunc}, if $k\ge M(a,x)$ and $r\in[a,\tilde{r}_k(x,t)]$, then $\rho_k (r,t) = \varphi_k \overline{v}_k^{-1} (r,t) + ( 1- \varphi_k(r,t) ) = \overline{v}_k^{-1} (r,t)$.
Using this and \eqref{eqs:Jac}--\eqref{eqs:transfunc}, we have
\begin{equation*}
\int_{a}^{\tilde{r}_k(x,t)} \rho_k(r,t)\,r^m \dif r
= \int_{a}^{\tilde{r}_k(x,t)} \overline{v}_k^{-1} (r,t)\,r^m \dif r
= \int_{0}^{x} \tilde{v}_k^{-1}(y,t) \tilde{v}_k (y,t)\,\dif y = x.
\end{equation*}
Using the above identity leads to
\begin{align*}
x-\int_{a}^{\tilde{r}(x,t)} \rho(r,t)\,r^m \dif r
= \int_{\tilde{r}(x,t)}^{\tilde{r}_k(x,t)} \rho_k (r,t)\,r^m \dif r
 +\int_{a}^{\tilde{r}(x,t)} ( \rho_k - \rho)(r,t)\, r^m \dif r.
\end{align*}
Let $\chi(t)\in C^{\infty}_c([0,T])$. Multiplying it to the above and integrating in time,
it follows that
\begin{align*}
&\int_{0}^{T} \Big( x-\int_{a}^{\tilde{r}(x,t)} \rho(r,t)\,r^m \dif r \Big) \chi(t) \dif t \\
&= \int_{0}^{T}\int_{\tilde{r}(x,t)}^{\tilde{r}_k(x,t)} \rho_k (r,t) \chi(t)\,r^m \dif r \dif t
   + \int_{0}^{T}\int_{a}^{\tilde{r}(x,t)} (\rho_k -\rho)(r,t) \chi(t)\,r^m \dif r \dif t\\
&=: I_k^{(1)} + I_k^{(2)}.
\end{align*}
$I_k^{(1)}$ is estimated by using Lemma \ref{lemma:rkcvg} and the upper bound $\rho_k \le C(a)$
from Lemma \ref{lemma:HREext}:
\begin{align*}
\snorm{I_k^{(1)}}
\le \dfrac{C(a)T}{n}\sbnorm{\chi(\cdot)}_{L^\infty}
 \sup_{ 0\le t \le  T}\snorm{\tilde{r}_k^n(x,t)-\tilde{r}^n(x,t)} \to 0 \qquad\text{as $k\to\infty$}.
\end{align*}
For $I_k^{(2)}$, observe that
$\chi(t) r^m \mathbbm{1}_{[a,\tilde{r}(x,t)]}(r) \in L^{1}(0,T;L^2( [a,\infty), \dif r))$, where $\mathbbm{1}_E(r)$ denotes the indicator function for spatial set $E\subset[a,\infty)$.
By Lemma \ref{lemma:rhoWeak}, it follows that
\begin{align*}
\lim\limits_{k\to\infty} I_k^{(2)}
=\lim\limits_{k\to\infty}\int_{0}^{T}\int_{a}^{\infty}(\rho_k-\rho)(r,t) \chi(t)
\mathbbm{1}_{[a,\tilde{r}(x,t)]}(r)\,r^m \dif r \dif t = 0.
\end{align*}
Since $|I_k^{(1)}|+|I_k^{(2)}|\to 0$ as $k\to\infty$, then
\begin{align*}
\int_{0}^{T} \Big( x-\int_{a}^{\tilde{r}(x,t)} \rho(r,t) r^m \dif r \Big) \chi(t)\,\dif t=0
\qquad \text{for each $\chi(t) \in C_c^{\infty}([0,T])$}
\end{align*}
so that $\eqref{temp:relation}_1$ follows.

Next, by $\eqref{eqs:LFNS-k}_1$ and $\eqref{eqs:LFNS-k}_4$, $\tilde{r}_k(x,t)$ satisfies
\begin{equation*}
\tilde{r}_k(x,t) = \tilde{r}_k(x,0) + \int_{0}^{t} \tilde{u}_k(x,s)\,\dif s
\qquad\text{ for $(x,t)\in[0,k]\times[0,T]$},
\end{equation*}
where, from $\eqref{eqs:LFNS-k}_4$ and \S \ref{subsec:ffaprox}, $\tilde{r}_k(x,0)$ is given by
	\begin{equation*}
		\tilde{r}_k(x,0) = \Big( a^n + n \int_{0}^{x} \tilde{v}_{a,k}^0(y)\,\dif y \Big)^{\frac{1}{n}}
		= \Big( a^n + n \int_{0}^{x} \dfrac{1}{\rho_{a,k}^0(\tilde{r}_a^0(y))} \dif y \Big)^{\frac{1}{n}}.
	\end{equation*}
Fix $(x,t)\in[0,\infty)\times[0,T]$. If $k \ge M(a,x)$, then, by \eqref{eqs:extfunc},
$u_k(\tilde{r}_k(x,t),t) = \tilde{u}_k(x,t)$ and
\begin{equation}\label{temp:rela1}
\tilde{r}_k(x,t) = \tilde{r}_k(x,0) + \int_{0}^{t} u_k(\tilde{r}_k(x,s),s)\,\dif s.
\end{equation}
Moreover, by Proposition \ref{prop:r0diff} in \S \ref{subsec:ffaprox},
if $k \ge N(x) \vcentcolon=\frac{2^n-1}{n}C_0a^n+ 2^n C_0^2 x$, then
\begin{align*}
\dfrac{1}{2}\tilde{r}_a^0(k)
=& \dfrac{1}{2} \Big( a^n + n \int_{0}^{k} \dfrac{1}{\rho_{a,k}^0(\tilde{r}_a^0(y))} \dif y \Big)^{\frac{1}{n}}
\ge \dfrac{1}{2}\big( a^n + n C_0^{-1} k \big)^{\frac{1}{n}} \\
&\ge ( a^n + n C_0 x )^{\frac{1}{n}} \ge \Big( a^n + n \int_{0}^{x} \dfrac{1}{\rho_a^0(\tilde{r}_a^0(y))} \dif y\Big)^{\frac{1}{n}} =\tilde{r}_a^0(x).
\end{align*}
By the construction of $\rho_k^0(r)$ in \eqref{def:appInitr1},
it follows that, if $k \ge N(x)$, then
\begin{align*}
\tilde{v}_{a,k}^0(x) = \dfrac{1}{\rho_{a,k}^0(\tilde{r}_a^0(x))}
= \Big( \rho_a^0(\tilde{r}_a^0(x)) \varphi_{a,k}^0(\tilde{r}_a^0(x))
      + 1 - \varphi_{a,k}^0(\tilde{r}_a^0(x)) \Big)^{-1}
= \dfrac{1}{\rho_a^0(\tilde{r}_a^0(x))}.
\end{align*}
Thus, by Proposition \ref{prop:r0diff}, for all $k\ge N(x)$,
\begin{align*}
\tilde{r}_k(x,0)
= \Big( a^n + n \int_{0}^{x} \dfrac{1}{\rho_{a,k}^0(\tilde{r}_a^0(y))} \dif y \Big)^{\frac{1}{n}}
= \Big( a^n + n \int_{0}^{x} \dfrac{1}{\rho_a^0(\tilde{r}_a^0(y))} \dif y \Big)^{\frac{1}{n}}
= \tilde{r}_a^0(x).
\end{align*}
This shows that, for each $x\in[0,\infty)$, $\tilde{r}_k(x,0)\to \tilde{r}_a^0(x)$ as $k\to\infty$.
Moreover, by Lemma \ref{lemma:rkcvg} and inequality \eqref{aholtemp2} obtained in Lemma \ref{lemma:klimholder},
it follows that, for all $(x,s)\in[0,\infty)\times[0,T]$,
\begin{equation*}
\snorm{u_k(\tilde{r}_k(x,s),s) - u_k(\tilde{r}(x,s),s)} \le C(a) \snorm{\tilde{r}_k(x,s)-\tilde{r}(x,s)} \to 0
\qquad\text{as $k\to\infty$.}
 \end{equation*}
By Lemma \ref{lemma:klimholder}, $\snorm{u_k(\tilde{r}(x,s),s)-u(\tilde{r}(x,s),s)} \to 0 $ for all $(x,s)\in[0,\infty)\times[0,T]$ as $k\to \infty$.
Combining these limits in \eqref{temp:rela1},
$\eqref{temp:relation}_2$ is shown by the dominated convergence theorem.
\end{proof}

\begin{lemma}\label{lemma:weak}
There exists a subsequence $($still denoted$)$ $(\rho_k,u_k, e_k )_{k\in\mathbb{N}}$ such that, as $k\to \infty$,
\begin{equation*}
\begin{dcases*}
(\sqrt{\sigma}\partial_r u_k, \sigma\partial_r e_k ) \overset{\ast}{\rightharpoonup} ( \sqrt{\sigma}\partial_r u, \sigma\partial_r e ) & in $L^{\infty}\big(0,T;L^2([a,\infty),r^m \dif r)\big)$,\\
( \partial_r u_k, \partial_r e_k, \sqrt{\sigma}\partial_t u_k, \sigma \partial_t e_k )
\rightharpoonup ( \partial_r u, \partial_r e, \sqrt{\sigma}\partial_t u, \sigma \partial_t e )
& in $L^{2}\big(0,T;L^2([a,\infty), r^m \dif r)\big)$,
\end{dcases*}
\end{equation*}
where $\sigma=\min\{1,t\}$. In addition,
\begin{align}\label{weakstmt1}
&\sup_{t\in[0,T]}\int_{a}^{\infty}
\big((\rho-1)^2 + u^4 + (e-1)^2 + \sigma\snorm{\partial_r u}^2 +\sigma^2 \snorm{\partial_r e}^2 \big)(r,t)\,r^m \dif r\nonumber\\
& +\int_{0}^{T}\int_{a}^{\infty} \Big\{ \snorm{\partial_r u}^2 + \snorm{u\partial_r u}^2 + \snorm{\partial_r e}^2  + \sigma\snorm{\partial_t u}^2 + \sigma^2 \snorm{\partial_t e}^2 \Big\}\,r^m \dif r \dif t \le C(a).
\end{align}
Furthermore, for each $\ve>0$,
\begin{align}\label{weakstmt2}
&\sup_{ 0\le t \le  T} \int_{\tilde{r}(\ve,t)}^{\infty}
\big((\rho-1)^2 + u^4 + (e-1)^2 + \sigma\snorm{\partial_r u}^2
  + \sigma^2 \snorm{\partial_r e}^2 \big)(r,t)\,r^m \dif r \nonumber \\
& +\int_{0}^{T} \int_{\tilde{r}(\ve,t)}^{\infty} \Big\{ \snorm{\partial_r u}^2 + \snorm{u\partial_r u}^2
 + \snorm{\partial_r e}^2  + \sigma\snorm{\partial_t u}^2 + \sigma^2 \snorm{\partial_t e}^2 \Big\}
 \,r^m \dif r\dif t \le C(\ve).
\end{align}
\end{lemma}

\begin{proof}
The proof only for $\sqrt{\sigma} \partial_r u_k \overset{\ast}{\rightharpoonup} \sqrt{\sigma} \partial_r u$
is presented, since the other limits can be shown via the same argument.

By Lemma \ref{lemma:HREext},
\begin{equation*}
\sup_{k\in\mathbb{N}} \sup_{t\in[0,T]} \int_{a}^{\infty} (\sigma \snorm{\partial_r u_k}^2)(r,t)\,\dif r
\le a^{-m} \sup_{k\in\mathbb{N}} \sup_{t\in[0,T]} \int_{a}^{\infty}
(\sigma\snorm{\partial_r u_k}^2)(r,t)\,r^m \dif r \le a^{-m}C(a).
\end{equation*}
Thus, there exists a subsequence (still denoted) $\{ \partial_r u_k \}_{k\in\mathbb{N}}$ such that
\begin{equation}\label{temp:weak1}
(\sqrt{\sigma}\partial_r u_k,\, r^{\frac{m}{2}}\sqrt{\sigma}\partial_r u_k)\,
\overset{\ast}{\rightharpoonup}\, (w_1,\, w_2) \qquad \text{ in $L^{\infty}\big(0,T; L^2( [a,\infty), \dif r )\big)$}.
\end{equation}
By \eqref{temp:weak1}, Lemma \ref{lemma:klimholder},
and the dominated convergence theorem,
we see that	the spatial weak derivative of $\sqrt{\sigma} u$ exists and $w_1 = \sqrt{\sigma} \partial_r u$.
As a result,
by \eqref{temp:weak1} and the fundamental lemma of calculus,
$w_2=r^{\frac{m}{2}}\sqrt{\sigma}\partial_r u$ for {\it a.e.} $(r,t)\in[a,\infty)\times[0,T]$.
Then \eqref{temp:weak1} implies that
$\sqrt{\sigma} \partial_r u_k \overset{\ast}{\rightharpoonup} \sqrt{\sigma} \partial_r u$
in $L^{\infty}(0,T;L^2([a,\infty),r^m \dif r))$.
By the weak-star lower semi-continuity of the Sobolev norm,
\begin{equation*}
\sbnorm{\sqrt{\sigma}\partial_r u}_{L^{\infty}(0,T;L^2([a,\infty),r^m \dif r))}
\le \liminf_{k\to\infty} \sbnorm{\sqrt{\sigma}\partial_r u_k}_{L^{\infty}(0,T;L^2([a,\infty),r^m \dif r))} \le C(a).
\end{equation*}
The other terms in \eqref{weakstmt1} can be proved by the same argument.

\smallskip
Next, we show \eqref{weakstmt2}. Only the proof for $\sqrt{\sigma}\partial_r u$ is provided, since the other terms
in \eqref{weakstmt2} can be obtained in the same way.
Let $\phi \in L^{1}(0,T;L^2([a,\infty),r^m \dif r))$. Then, by Lemma \ref{lemma:HREext},
\begin{align*}
&\Bignorm{\int_{0}^{T}\int_{\tilde{r}_k(\ve,t)}^{\infty} \sqrt{\sigma} \partial_r u_k \phi\,r^m \dif r \dif t
- \int_{0}^{T}\int_{\tilde{r}(\ve,t)}^{\infty} \sqrt{\sigma} \partial_r u \phi\,r^m \dif r \dif t }\\
&\le \Bignorm{\int_{0}^T \int_{\tilde{r}_k(\ve,t)}^{\tilde{r}(\ve,t)} \sqrt{\sigma} \partial_r u_k \phi\,r^m \dif r \dif t}
+ \Bignorm{\int_{0}^T \int_{\tilde{r}(\ve,t)}^{\infty} \sqrt{\sigma} (\partial_r u_k-\partial_r u) \phi\,r^m \dif r\dif t}\\
&\le  C(a)\int_{0}^{T}\Big(\int_{\tilde{r}_k(\ve,t)}^{\tilde{r}(\ve,t)}\snorm{\phi}^2\,r^m\dif r\Big)^{\frac{1}{2}}
\dif t + \Bignorm{\int_{0}^T \int_{\tilde{r}(\ve,t)}^{\infty} \sqrt{\sigma} (\partial_r u_k-\partial_r u) \phi\, r^m  \dif r\dif t}\\
&=: I_k^{(1)} + I_k^{(2)}.
\end{align*}
$I_k^{(2)}\to 0$ as $k\to\infty$ by \eqref{temp:weak1}.
For $I_k^{(1)}$, define
\begin{equation*}
\chi^{(k,\ve)}(r,t) \vcentcolon = \begin{dcases*}
1 & if $r\in[\tilde{r}_k(\ve,t),\tilde{r}(\ve,t)]$ or $r\in[\tilde{r}(\ve,t),\tilde{r}_k(\ve,t)]$,\\
0 & otherwise.
\end{dcases*}
\end{equation*}
By
Lemma \ref{lemma:rkcvg}, for each fixed $\ve>0$, $\chi^{(k,\ve)}(r,t) \to 0$ as $k\to\infty$
for {\it a.e.} $(r,t)\in[a,\infty)\times[0,T]$.
Moreover, since $\snorm{\chi^{(k,\ve)} \phi} \le \snorm{\phi}$,
by the dominated convergence theorem, as $k\to\infty$,
\begin{equation*}
I_k^{(1)} \vcentcolon
= C(a)\int_{0}^{T} \Big(\int_{\tilde{r}_k(\ve,t)}^{\tilde{r}(\ve,t)} \snorm{\phi}^2\,r^m \dif r\Big)^{\frac{1}{2}}
\dif t
=  C(a) \int_{0}^{T} \Big(\int_{a}^{\infty} \chi^{(\ve,k)}\snorm{\phi}^2\,r^m \dif r\Big)^{\frac{1}{2}} \dif t \to 0.
\end{equation*}
Let $\chi_\ve$ and $\chi_\ve^k$ be defined in \eqref{temp:chive}--\eqref{temp:chivek}.
Then the limits: $I_k^{(1)}$, $I_k^{(2)}\to 0$ imply the weak-star convergence:
$\chi_{\ve}^k\sqrt{\sigma} \partial_r u_k  \overset{\ast}{\rightharpoonup} \chi_\ve \sqrt{\sigma} \partial_r u$
as $k\to\infty$ in $L^{\infty}(0,T;L^2( [a,\infty), r^m \dif r ))$.
By the weak-star lower semi-continuity of the Sobolev norm,
it follows that
\begin{align*}
\sup_{t\in[0,T]}\int_{\tilde{r}(\ve,t)}^{\infty} (\sigma \snorm{\partial_r u}^2)(r,t) r^m \dif r
&= \sbnorm{\chi_\ve \sqrt{\sigma} \partial_r u}_{L^{\infty}(0,T;L^2([a,\infty),r^m \dif r))}\\
&\le  \liminf_{k\to\infty}\sbnorm{\chi_\ve^k \sqrt{\sigma} \partial_r u_k}_{L^{\infty}(0,T;L^2([a,\infty),r^m \dif r))} \\
&= \sup_{k\in\mathbb{N}}\sup_{t\in[0,T]}\int_{\tilde{r}_k(\ve,t)}^{\infty} \sigma \snorm{\partial_r u_k}^2(r,t) r^m \dif r  \le C(\ve).
\end{align*}
\end{proof}

In \S \ref{subsec:klimEnt}--\S \ref{subsec:indepPath} below,
we denote $(\rho,u,e)(r,t)$ in $[a,\infty)\times[0,T]$ as the limit functions obtained in Lemma \ref{lemma:klimholder}--\ref{lemma:rhoWeak} and keep the simplified notation \eqref{asuppress}
as in \S 5.4.

\subsection{Entropy estimate for the limit solution}\label{subsec:klimEnt}
\begin{lemma}\label{lemma:Ent}
There exists $C(T)>0$ such that
\begin{align}\label{temp:Ent0}
&\esssup\limits_{t\in[0,T]}\int_{a}^{\infty}
\Big(\rho\big(\frac{1}{2}\snorm{u}^2+ \psi(e)\big)+ G(\rho) \Big)(r,t)\, r^m \dif r
 + \kappa\int_{0}^{T}\int_{a}^{\infty}\dfrac{\snorm{\partial_r e}^2}{ e^2}\,r^m \dif r \dif t\nonumber \\
&+\int_{0}^{T}\int_{a}^{\infty} \Big\{ \big( \dfrac{2\mu}{n} + \lambda \big)
\dfrac{1}{e}\Bignorm{\partial_r u+m \dfrac{u}{r}}^2 + \dfrac{2m\mu}{n} \dfrac{1}{e}
\Bignorm{\partial_r u-\dfrac{u}{r}}^2\Big\}\,r^m\dif r\dif t \le C(T).
\end{align}
\end{lemma}

\begin{proof}  We divide the proof into four steps.

\smallskip
1. For each integer $L\ge 2$, let $\phi_L\in C^{\infty}([a,\infty))$ be such that $\phi_L(r)=1$ if $r\in[a,L-1]$ and $\phi_L(r)=0$ if $r\in[L,\infty)$.
Then $\phi_L\in \tilde{H}^1_0( [a,L], r^m \dif r )$, where $\tilde{H}_0^1$ is the Sobolev space defined in \eqref{eqs:tildeH1}. Moreover,
$\partial_r \psi(e_k) = \frac{e_k-1}{e_k} \partial_r e_k$.
Thus, for each fixed $t\in(0,T]$,
$(\psi(e_k)\phi_L)(\cdot,t) \in \tilde{H}^1_0( [a,L], r^m \dif r)$
for all $(k,t)\in\mathbb{N}\times(0,T]$ and
\begin{align*}
\int_{a}^{\infty} \rho_k  \psi(e_k)\phi_L\,r^m \dif r - \int_{a}^{\infty} \rho \psi(e) \phi_L\,r^m \dif r
&\le  \int_{a}^{L}(\rho_k-\rho)\psi(e_k)\phi_L\,r^m \dif r
 + \int_{a}^{\infty}\rho  ( \psi(e_k) - \psi(e) )\phi_L\,r^m \dif r \\
&=\vcentcolon I_k^{(1)} + I_k^{(2)}.
\end{align*}
For all $t\in(0,T]$, $I_k^{(1)}$ is estimated by using Lemmas \ref{lemma:HREext} and \ref{lemma:rhoScvg} as
\begin{align*}
I_k^{(1)}
&\le  \sbnorm{\rho_k-\rho}_{\tilde{H}^{-1}([a,L];r^m \dif r)} \sbnorm{ \psi(e_k)\phi_L}_{H^{1}([a,L];r^m \dif r)}\\
&= \sbnorm{\rho_k-\rho}_{\tilde{H}^{-1}}  \Big\{ \sbnorm{\phi_L}_{H^1}\sbnorm{\psi(e_k)}_{L^\infty}
   + \sigma^{-1}(t)\Bigbnorm{\dfrac{e_k-1}{e_k}\phi_L}_{\infty} \sbnorm{\sigma\partial_r e_k}_{L^2}  \Big\}\\
	&\le  \sbnorm{\rho_k-\rho}_{\tilde{H}^{-1}} \big\{ C(a)\sbnorm{\phi_L}_{H^1} + C(a) \sigma^{-1}(t)\big\} \to 0
	\qquad\ \text{as $k\to\infty$}.
\end{align*}
From Lemma \ref{lemma:klimholder},
$$
e_k(r,t)\to e(r,t) \qquad  \mbox{as $k\to\infty$ for $(r,t)\in[a,\infty)\times[0,T]$},
$$
and, from Lemmas \ref{lemma:klimholder}--\ref{lemma:rhoWeak},
$$
C^{-1}(a)\le \rho(r,t) \le C(a), \,\, C^{-1}(a) \le e_k(r,t),\, e(r,t) \le C(a) \sigma^{-\frac{1}{2}}(t)
$$
for {\it a.e.} $(r,t)\in[a,\infty)\times(0,T]$.
Using these and the dominated convergence theorem, it follows that
\begin{align*}
\lim\limits_{k\to\infty} I_k^{(2)}
\vcentcolon= \lim\limits_{k\to\infty} \int_{a}^{\infty}  \rho \phi_n ( \psi(e_k) - \psi(e) )(r,t)\,r^m \dif r
= 0 \qquad \text{ for each $t\in(0,T]$}.
\end{align*}
Using Lemma \ref{lemma:entEstext}, the limits: $I_k^{(1)}, I_k^{(2)}\to 0$ then imply that, for each $t>0$,
\begin{align*}
\int_{a}^{\infty} \phi_L(r,t)(\rho \psi(e))(r,t)\,r^m \dif r
= \lim\limits_{k\to\infty}\int_{a}^{\infty}\phi_L(r,t)(\rho_k \psi(e_k))(r,t)\,r^m \dif r \le C(T).
\end{align*}
Now, by construction, $\phi_L(r)\to 1$ for all $r\in[a,\infty)$ as $L\to \infty$.
Since $\rho \psi(e) >0$, we apply the monotone convergence theorem to obtain
\begin{align*}
\int_{a}^{\infty} (\rho \psi(e))(r,t)\,r^m \dif r
=\lim\limits_{L\to\infty} \int_{a}^{\infty} \phi_L(r,t) (\rho \psi(e))(r,t)\,r^m \dif r \le C(T)
\qquad \text{ for each $t\in(0,T]$}.
\end{align*}

\smallskip	
2. From Lemma \ref{lemma:HREext}, $u_k^2 \in L^2( [a,\infty), r^m \dif r )$ and, for all $(k,t)\in\mathbb{N}\times(0,T]$,
\begin{equation*}
\sbnorm{\partial_r u_k^2(\cdot,t)}_{L^2([a,\infty), r^m \dif r)} = \sbnorm{u_k}_{L^\infty} \sbnorm{\partial_r u_k(\cdot,t)}_{L^2( [a,\infty), r^m \dif r)} \le \sigma^{-\frac{3}{4}}(t)C(a).
\end{equation*}
From these, $ u_k\phi_L\in \tilde{H}_0^1([a,L],r^m \dif r)$ and
\begin{align*}
&\Bignorm{\int_{a}^{\infty} \rho_k u_k^2\phi_L\, r^m \dif r - \int_{a}^{\infty} \phi_L \rho u^2\, r^m \dif r}\\
&\le \bnorm{\rho_k-\rho}_{\tilde{H}^{-1}([a,L], r^m\dif r)} \sbnorm{u_k^2\phi_L}_{H^1([a,L],r^m \dif r)} + \int_{a}^{\infty} \rho ( u_k^2 - u^2 ) \phi_L\, r^m \dif r\\
&\le  \int_{a}^{\infty} \phi_L \rho ( u_k^2 - u^2)\, r^m \dif r 
     +  C(a)\big(1+\sigma^{-\frac{3}{4}}(t)\big)
\sbnorm{\rho_k-\rho}_{\tilde{H}^{-1}([a,L],r^m\dif r)}.
\end{align*}
Since $u_k(r,t)\to u(r,t)$ from Lemma \ref{lemma:klimholder},
$\snorm{u_k(r,t)}+\snorm{u(r,t)}\le C(a) \sigma^{-\frac{1}{4}}(t)$
from Lemmas \ref{lemma:HREext} and \ref{lemma:klimholder},
and $\rho_k\to \rho$ strongly in $\tilde{H}^{-1}$ from Lemma \ref{lemma:rhoScvg},
then, by the dominated convergence theorem,
\begin{align*}
\int_{a}^{\infty}  \rho u^2\phi_L\,r^m \dif r
= \lim\limits_{k\to\infty} \int_{a}^{\infty} \rho_k u_k^2\phi_L\,r^m \dif r
\le \limsup\limits_{k\to\infty} \int_{a}^{\infty} \rho_k u_k^2\,r^m \dif r \le  C(T).
\end{align*}
Then $\sbnorm{(\rho u^2)(\cdot,t)}_{L^1([a,\infty),r^m\dif r)}\le C(T)$ follows
from the monotone convergence theorem for $L\to\infty$.
	
\smallskip
3. Let $\phi\in L^2(0,T;L^2([a,\infty),r^m \dif r))$. By Lemma \ref{lemma:HREext},
\begin{align*}
&\Bignorm{\int_{0}^{T}\int_{a}^{\infty} \dfrac{\partial_r e_k}{e_k} \phi\,r^m \dif r \dif t
- \int_{0}^{T}\int_{a}^{\infty} \dfrac{\partial_r e}{e} \phi\,r^m \dif r \dif t}\\
&\le  C(a) \Big( \int_{0}^{T}\int_{a}^{\infty} \big(e_k^{-1}-e^{-1}\big)^2 \snorm{\phi}^2\,r^m \dif r \dif t
   \Big)^{\frac{1}{2}}
   + \Bignorm{\int_{0}^{T}\int_{a}^{\infty} (\partial_r e_k-\partial_r e)\dfrac{\phi}{e}\,r^m\dif r\dif t}\\
& =\vcentcolon \textrm{II}_k^{(1)} + \textrm{II}_k^{(2)}.
\end{align*}
Using $e_k\to e$
from Lemma \ref{lemma:klimholder},
and the estimates: $C(a)\le e_k(r,t), e(r,t) \le C(a)\sigma^{-\frac{1}{2}}(t)$ from Lemmas \ref{lemma:HREext} and \ref{lemma:klimholder},
we use the dominated convergence theorem to obtain
\begin{align*}
\lim\limits_{k\to\infty} \textrm{II}_k^{(1)}
\vcentcolon= \lim\limits_{k\to\infty}  C(a) \Big( \int_{0}^{T}\int_{a}^{\infty} \big(e_k^{-1}-e^{-1}\big)^2 \snorm{\phi}^2\,r^m \dif r \dif t \Big)^{\frac{1}{2}} =0.
\end{align*}
For $\textrm{II}_k^{(2)}$, since $e(r,t)\ge C^{-1}(a)$
and $\partial_r e_k \rightharpoonup \partial_r e$ in Lemma \ref{lemma:weak}, we have
\begin{align*}
\lim\limits_{k\to\infty} \textrm{II}_k^{(2)} \vcentcolon
= \lim\limits_{k\to\infty}\Bignorm{\int_{0}^{T}\int_{a}^{\infty} (\partial_r e_k - \partial_r e)
\dfrac{\phi}{e} (r,t)\,r^m \dif r \dif t} = 0.
\end{align*}
The above limits imply that $\frac{\partial_r e_k}{e_k} \rightharpoonup \frac{\partial_r e}{e}$ as $k\to\infty$
weakly in $L^2(0,T;L^2([a,\infty),r^m \dif r))$.
By Lemma \ref{lemma:entEstext} and the weak-star lower semi-continuity of the Sobolev norm,
it follows that
\begin{align*}
\int_{0}^{T}\int_{a}^{\infty} \dfrac{\snorm{\partial_r e}^2}{e^2}\,r^m \dif r \dif t
\le \liminf_{k\to\infty} \int_{0}^{T}\int_{a}^{\infty} \dfrac{\snorm{\partial_r e_k}^2}{e_k^2}\,r^m \dif r \dif t
\le C(T).
\end{align*}
By the same argument,
we can also obtain the estimates
for $\sbnorm{(\partial_r u + m \frac{u}{r})e^{-1}}_{L^2}$ and $\sbnorm{(\partial_r u - \frac{u}{r})e^{-1}}_{L^2}$
in \eqref{temp:Ent0}.
	
\smallskip
4. From Lemmas \ref{lemma:entEstext}--\ref{lemma:HREext} and \ref{lemma:rhoWeak}, we have
\begin{equation}
\begin{dcases}
\rho_k \rightharpoonup \rho \ \text{ weakly in } \ L^2\big(0,T;L^2([a,\infty),r^m \dif r)\big)
\qquad \text{ as $k\to\infty$}.\\
\sup\limits_{k\in\mathbb{N}}\esssup_{t\in[0,T]}\int_{a}^{\infty} G(\rho_k)(r,t)\,r^m \dif r \le C(T), \\
C^{-1}(a)\le \rho_k(r,t)\le C(a)  \qquad \text{for {\it a.e.} \ $(r,t)\in[a,\infty)\times[0,T]$ and for each
$k\in\mathbb{N}$}.
\end{dcases}
\end{equation}
By Proposition \ref{prop:mazur},
$\sbnorm{G(\rho)(\cdot,t)}_{L^1([a,\infty),r^m \dif r)}\le C(T)$ for {\it a.e.} $t\in[0,T]$.
\end{proof}

\subsection{Weak forms of the exterior problem}\label{subsec:klimWF}
We now show that $(\rho,u,e)$ is indeed a weak solution of problem \eqref{eqs:SFNS} and \eqref{eqs:eSFNS}.

\begin{lemma}\label{lemma:klimWF}
Let $(\rho,u,e)(r,t)$ be the limit function obtained
in {\rm Lemmas \ref{lemma:klimholder}}--{\rm \ref{lemma:rhoWeak}}.
Then $(\rho,u,e)(r,t)$ is a weak solution of problem \eqref{eqs:SFNS} and \eqref{eqs:eSFNS} as stated in {\rm Definition \ref{def:SWeak}}.
\end{lemma}

\begin{proof} We divide the proof into three steps.

\smallskip
1. We start with the {\it continuity equation}.
For any $\phi\in \mathcal{D}^a$, then there exists an integer $N\in\mathbb{N}$
such that $\phi(r,t) = 0$ for $r\ge N$ and $t\in[0,T]$.
Take $k \in \mathbb{N}$ to be $k\ge M(\phi,a) \vcentcolon= 2^n N^n\max\{C_0,C(a)\}$.
Then, for all $t\in[0,T]$,
\begin{align*}
&\tilde{r}_k(k,t)= \Big( a^n + \int_{0}^{k} \tilde{v}_k(x,t) \dif x \Big)^{\frac{1}{n}}
\ge \Big( a^n + \frac{k}{C(a)}\Big)^{\frac{1}{n}}  \ge ( a^n + 2^n N^n )^{\frac{1}{n}} > 2N,\\
& \tilde{r}_a^0(k) = \Big( a^n + \int_{0}^{k}\dfrac{1}{\rho_a^0(\tilde{r}_a^0(x))} \dif x \Big)^{\frac{1}{n}}
\ge \Big( a^n + \frac{k}{C_0}\Big)^{\frac{1}{n}} \ge ( a^n + 2^n N^n )^{\frac{1}{n}} > 2N.
\end{align*}
Thus, it follows by construction that, if $(r,t)\in \supp(\phi)$ and $k\ge M(\phi,a)$,
\begin{equation}\label{wtemp1}
\begin{split}
\rho_k(r,t) =&\, \overline{v}_k^{-1}\varphi_k(r,t) + 1 - \varphi_k(r,t) =  \overline{v}_k^{-1}(r,t), \ \ u_k(r,t) = \overline{u}_k\varphi_k(r,t) = \overline{u}_k(r,t),\\
e_k(r,t) =&\, \overline{e}_k\varphi_k(r,t) + 1 - \varphi_k(r,t) =  \overline{e}_k(r,t).
\end{split}
\end{equation}
Moreover, using that $\tilde{r}_a^0(k)> 2N$
for $k \ge M(\phi,a) \vcentcolon= 2^n N^n \max\{C_0,C(a)\}$,
we see from construction \eqref{def:appInitr1} in \S\ref{subsec:ffaprox} that,
if $r\in \supp(\phi(\cdot,0))$, then
\begin{equation}\label{wtemp6}
\begin{split}
\rho_{a,k}^0(r)
=&\, \rho_a^0(r) \varphi_{a,k}^0(r) + 1 - \varphi_{a,k}^0(r)
= \rho_a^0(r),\ \ u_{a,k}^0(r) = u_a^0(r)\varphi_{a,k}^0(r) = u_a^0(r),\\
e_{a,k}^0(r) =&\, e_a^0(r)\varphi_{a,k}^0(r) + 1 - \varphi_{a,k}^0(r) =  e_a^0(r).
\end{split}
\end{equation}
With this, if $x\in[0,\infty)$ is such that $\tilde{r}_a^0(x)\in \supp( \phi(\cdot,0) )$, then
it follows from Proposition \ref{prop:r0diff} that,
for all $k\ge M(\phi,a)$,
\begin{equation}\label{wtemp7}
\tilde{r}_a^0(x) = \Big( a^n + n\int_{0}^{x} \dfrac{1}{\rho_a^0(\tilde{r}_a^0(y))} \dif y \Big)^{\frac{1}{n}}
= \Big( a^n + n\int_{0}^{x} \tilde{v}_{a,k}^0(y) \dif y \Big)^{\frac{1}{n}} = \tilde{r}_{k}(x,0).
\end{equation}
Thus, from the coordinate transformation \eqref{eqs:coorderiv} and the continuity equation for which the approximate Lagrangian functions satisfy, we obtain that, for all $k\ge M(\phi,a)$,
\begin{align*}
&\int_{a}^{\infty} \Big\{ \rho_k(r,t) \phi(r,t) - \rho_a^0(r)\phi(r,0) \Big\}\, r^m \dif r\\
&= \int_{a}^{\tilde{r}_k(k,t)} \overline{v}_k^{-1} \phi(r,t)\, r^m \dif r
 - \int_{a}^{\tilde{r}_a^0(k)} \phi(r,0) \rho_a^0(r)\,r^m \dif r \\
&= \int_{0}^{k} \phi(\tilde{r}_k(x,t),t)\,\dif x - \int_{0}^{k} \phi(\tilde{r}_k(x,0),0)\,\dif x \\
&= \int_{0}^{t} \int_{0}^{k} \big( \tilde{u}_k \partial_r \phi+ \partial_t \phi \big)(\tilde{r}_k(x,s),s)\,\dif x \dif s\\
&= \int_{0}^{t}\int_{a}^{\tilde{r}_k(k,t)} \big(\overline{v}_k^{-1}\overline{u}_k \partial_r \phi
 + \tilde{v}_k^{-1} \partial_t \phi \big)(r,t)\,r^m \dif r \dif s\\
&= \int_{0}^{t}\int_{a}^{\infty} \big(\rho_k u_k \partial_r \phi  + \rho_k \partial_t \phi \big)(r,t)\,r^m \dif r \dif s,
\end{align*}
where  we have used $\tilde{r}_k(k,t)\ge 2N$ and \eqref{wtemp1} in the last equality.
Also, by Lemma \ref{lemma:klimholder}--\ref{lemma:rhoScvg}, we have
\begin{equation}\label{temp:klimWF1}
\begin{aligned}
&\lim\limits_{k\to\infty} \sup_{t\in[0,T]}\sbnorm{\rho_k(\cdot, t)
  - \rho(\cdot,t)}_{\tilde{H}^{-1}([a,L], r^m \dif r)} = 0 && \text{ for all $L\in\mathbb{N}$,}\\
&\rho_k-1 \overset{\ast}{\rightharpoonup}  \rho-1
&& \text{ in $L^{\infty}\big(0,T;L^2([a,\infty), r^m \dif r)\big)$,}\\
&u_k(r,t) \to u(r,t) \ \text{ and } \  \rho_k(r,t),\,\rho(r,t) \le C(a)  && \text{ for {\it a.e.} $(r,t) \in[a,\infty)\times[0,T]$.}
\end{aligned}
\end{equation}
Using these, we can apply the dominated convergence theorem to obtain
\begin{align*}
\int_{a}^{\infty} \rho(r,t) \phi(r,t)\,r^m \dif r
- \int_{a}^{\infty} \rho_a^0(r) \phi(r,0)\,r^m \dif r
= \int_{0}^{t}\int_{a}^{\infty} \big(\rho u \partial_r \phi + \rho \partial_t \phi\big)\,r^m \dif r \dif s.
\end{align*}
	
\smallskip	
2. We show the weak form for the {\it momentum equation}.
Let $\varphi\in \mathcal{D}_0^a$.
Then there exists $N\in\mathbb{N}$ such that $\varphi(r,t) = 0$ for $r\ge N$ and $t\in[0,T]$.
If $k \in \mathbb{N}$ is such that $k \ge M(\varphi,a)$, then, as previously shown,
$\tilde{r}_k(k,t) > 2N$ for all $t\in[0,T]$, and $\tilde{r}_a^0(k)>2N$.
By construction, if $(r,t)\in \supp (\varphi)$ and $r\in \supp (\varphi(\cdot,0))$,
then \eqref{wtemp1}--\eqref{wtemp6} hold so that
\begin{align*}
&\int_{a}^{\infty} (\rho_k u_k)(r,t)\varphi(r,t)\,r^m \dif r
  - \int_{a}^{\infty} (\rho_a^0 u_a^0)(r) \varphi(r,0)\,r^m \dif r\\
&=\int_{a}^{\tilde{r}_k(k,t)} (\overline{v}_k^{-1} \overline{u}_k) (r,t)\varphi(r,t)\,r^m \dif r
- \int_{a}^{\tilde{r}_a^0(k)}  u_{a,k}^0 (r) \varphi(r,0) \rho_{a}^0(r)\,r^m \dif r\\
&= \int_{0}^{k} \tilde{u}_k (x,t)\varphi(\tilde{r}_k(x,t),t)\,\dif x
  - \int_{0}^{k} \tilde{u}_{a,k}^0 (x) \varphi(\tilde{r}_{k}(x,0),0)\,\dif x\\
&= \int_{0}^{t}\int_{0}^{k} \big(\varphi(\tilde{r}_k(x,s),s) \D_t \tilde{u}_k
  +  \partial_t \varphi (\tilde{r}_k(x,s),s) \tilde{u}_k
  +  \partial_r \varphi(\tilde{r}_k(x,s),s) \tilde{u}_k^2 \big)\,\dif x \dif s.
\end{align*}
By $\eqref{eqs:LFNS-k}_2$
and the boundary condition: $\varphi(a,t)=0$, $\tilde{u}_k(k,t)=0$ for $t\in[0,T]$,
it follows that
\begin{align*}
&\int_{a}^{\infty} (\rho_k u_k)(r,t)\varphi(r,t)\,r^m \dif r
  - \int_{a}^{\infty} (\rho_a^0 u_a^0)(r) \varphi(r,0)\,r^m \dif r \\
&= \int_{0}^{t}\int_{0}^{k} \big(\varphi(\tilde{r}_k(x,s),s) \D_t \tilde{u}_k
  + \partial_t \varphi(\tilde{r}_k(x,s),s) \tilde{u}_k
  +  \partial_r \varphi(\tilde{r}_k(x,s),s) \tilde{u}_k^2  \big)\,\dif x \dif s\\
&= \int_{0}^{t}\int_{0}^{k} \big(\tilde{u}_k(x,s)\,\partial_t \varphi(\tilde{r}_k(x,s),s)
  + \tilde{u}_k^2(x,s)\,\partial_r \varphi(\tilde{r}_k(x,s),s)  \big)\,\dif x \dif s\\
&\quad + \int_{0}^{t}\int_{0}^{k} \tilde{r}_k^m(x,s)\varphi(\tilde{r}_k(x,s),s) \D_x\Big( \beta \dfrac{\D_x(\tilde{r}_k^m \tilde{u}_k)}{\tilde{v}_k} - (\gamma-1) \dfrac{\tilde{e}_k}{\tilde{v}_k} \Big)\,\dif x \dif s\\
&= \int_{0}^{t}\int_{0}^{k}\big(  \tilde{u}_k^2(x,s)\,\partial_r \varphi(\tilde{r}_k(x,s),s)
  + \tilde{u}_k(x,s)\,\partial_t \varphi(\tilde{r}_k(x,s),s)\big)\,\dif x \dif s \\
&\quad+ \int_{0}^{t} \int_{0}^{k} \Big( (\gamma-1)\tilde{e}_k - \beta \tilde{r}_k^m \D_x\tilde{u}_k
 - m\beta \dfrac{\tilde{v}_k\tilde{u}_k}{\tilde{r}_k}  \Big)
  \Big( m \dfrac{\varphi(\tilde{r}_k(x,s),s)}{\tilde{r}_k(x,s)}
    + \partial_r \varphi(\tilde{r}_k(x,s),s) \Big)\,\dif x \dif s\\
&= \int_{0}^{t}\int_{a}^{\tilde{r}_{k}(k,t)} \big(\overline{v}_k^{-1} \overline{u}_k^2 \partial_r\varphi + \overline{v}_k^{-1}\overline{u}_k \partial_t \varphi \big)\,r^m \dif r \dif s \\
&\quad + \int_{0}^{t}\int_{a}^{\tilde{r}_k(k,t)} \Big((\gamma-1)\overline{v}_k^{-1}\overline{e}_k
  -\beta \partial_r \overline{u}_k - m\beta \dfrac{\overline{u}_k}{r} \Big)
  \Big( \partial_r \varphi + m \dfrac{\varphi}{r} \Big)\,r^m \dif r \dif s\\
&= \int_{0}^{t}\int_{a}^{\infty} \big(\rho_k u_k^2 \partial_r\varphi
  + \rho_k u_k \partial_t \varphi \big)\, r^m \dif r \dif s \\
&\quad + \int_{0}^{t}\int_{a}^{\infty} \Big((\gamma-1)\rho_k e_k
 -\beta \partial_r u_k - m\beta \dfrac{u_k}{r} \Big)
  \Big( \partial_r \varphi + m \dfrac{\varphi}{r} \Big)\,r^m \dif r \dif s.
\end{align*}
In addition, from Lemmas \ref{lemma:HREext} and \ref{lemma:klimholder}--\ref{lemma:weak}, we have
\begin{equation}\label{wtemp4}
\begin{aligned}
&\partial_r u_k \rightharpoonup \partial_r u &&  \text{ in $L^{2}\big(0,T;L^2([a,\infty), r^m \dif r)\big)$,}\\
&e_k(r,t) \to e(r,t) &&  \text{ for $(r,t) \in[a,\infty)\times[0,T]$,}\\
&\sigma^{-\frac{1}{4}}(t) \snorm{u(r,t)} + \sigma^{-\frac{1}{2}}(t) e(r,t)  \le  C(a)
&& \text{ for $(r,t)\in[a,\infty)\times[0,T]$,} \\
&\sigma^{-\frac{1}{4}}(t) \snorm{u_k(r,t)} + \sigma^{-\frac{1}{2}}(t) e_k(r,t)  \le  C(a)
&& \text{ for $(k,r,t)\in \mathbb{N}\times[a,\infty)\times[0,T]$.}
\end{aligned}
\end{equation}
Using \eqref{temp:klimWF1} and \eqref{wtemp4},
we conclude  Definition \ref{def:SWeak}\ref{item:SWeak2}.

\smallskip	
3. We now show the {\it energy equation}.
Let $\phi\in \mathcal{D}^a$. Then, for $k\ge M(\phi,a)$,
\begin{align}\label{wtemp2}
&\int_{a}^{\infty} (\rho_k e_k)(r,t) \phi (r,t)\,r^m \dif r
  - \int_{a}^{\infty} (\rho_a^0 e_a^0) (r) \phi(r,0)\,r^m \dif r \nonumber\\
&= \int_{0}^{k} \tilde{e}_k(x,t) \phi(\tilde{r}_k(x,t),t)\,\dif x
  - \int_{0}^{k} \tilde{e}_{a,k}^0(x) \phi(\tilde{r}_{k}(x,0),0)\,\dif x \nonumber\\
&= \int_{0}^{t}\int_{0}^{k} \Big\{ \phi(\tilde{r}_k(x,s),s) \D_t \tilde{e}_k(x,s)
  + \tilde{e}_k(x,s)\big( \partial_t \phi + \tilde{u}_k(x,s) \partial_r \phi\big)
   (\tilde{r}_k(x,s),s)\Big\}\,\dif x \dif s \nonumber\\
&= \int_{0}^{t} \int_{0}^{k} \phi(\tilde{r}_k(x,s),s) \D_t \tilde{e}_k\, \dif x \dif s
  + \int_{0}^{t} \int_{a}^{\infty} \rho_k e_k ( \partial_t \phi +  u_k \partial_r \phi )\,r^m \dif r \dif s\nonumber \\
& =\vcentcolon I_1+I_2.
\end{align}
Denote $\tilde{F}_k\vcentcolon= \beta \tilde{v}_k^{-1}\D_x(\tilde{r}_k^m \tilde{u}_k)
- (\gamma-1)\tilde{v}_k^{-1}\tilde{e}_k$.
It follows from the momentum equation that $\D_t \tilde{u}_k = \tilde{r}_k^{m} \D_x \tilde{F}_k$.
Then substituting this into $\eqref{eqs:LFNS-k}_3$ yields
	\begin{align*}
		\D_t \tilde{e}_k
		= \D_x\big( \tilde{r}_k^m \tilde{u}_k \tilde{F}_k  \big) - \dfrac{1}{2}\D_t \snorm{\tilde{u}_k}^2  - 2 m \mu \D_x \big(\tilde{r}_k^{m-1} \tilde{u}_k^2\big) + \kappa \D_x\Big( \dfrac{\tilde{r}_k^{2m}}{\tilde{v}_k} \D_x \tilde{e}_k \Big)
	\end{align*}
for $I_1$ which, along with  $\D_x \tilde{e}_k(0,t)=\D_x \tilde{e}_k (k,t) = 0$ for all $t\in[0,T]$, leads to
\begin{align*}
I_1
=&\, \int_{0}^{t} \int_{0}^{k} \phi(\tilde{r}_k(x,s),s) \Big\{ \D_x( \tilde{r}_k^m \tilde{u}_k \tilde{F}_k )
  - \dfrac{1}{2}\D_t \snorm{\tilde{u}_k}^2  \Big\}\,\dif x \dif s \\
&\,+ \int_{0}^{t} \int_{0}^{k} \phi(\tilde{r}_k(x,s),s) \Big\{ - 2 m \mu \D_x (\tilde{r}_k^{m-1} \tilde{u}_k^2)
 + \kappa \D_x\Big( \dfrac{\tilde{r}_k^{2m}}{\tilde{v}_k} \D_x \tilde{e}_k \Big) \Big\}\,\dif x \dif s\\
=&\, \int_{0}^{t} \int_{0}^{k} \big\{ -m \beta \tilde{r}_k^{-1} \tilde{v}_k \snorm{\tilde{u}_k}^2
  - \beta \tilde{r}_k^m \tilde{u}_k  \D_x\tilde{u}_k + (\gamma-1) \tilde{u}_k \tilde{e}_k \big\} \partial_r\phi(\tilde{r}_k(x,s),s)\,\dif x \dif s \\
&\, + \frac{1}{2}\int_{0}^{t}\int_{0}^{k} \snorm{\tilde{u}_k}^2\big\{ \partial_t \phi(\tilde{r}_k(x,s),s)
    + \tilde{u}_k(x,s) \partial_r \phi(\tilde{r}_k(x,s),s)  \big\}\,\dif x \dif s \\
&\,- \frac{1}{2}\int_{0}^{k} \snorm{\tilde{u}_k}^2(x,t) \phi(\tilde{r}_k(x,t),t)\,\dif x
    + \frac{1}{2}\int_{0}^{k} \snorm{\tilde{u}_{a,k}^0(x)}^2 \phi(\tilde{r}_k(x,0),0)\,\dif x \\
&\,+ \int_{0}^{t} \int_{0}^{k} \Big\{ 2 m \mu \tilde{r}_k^{-1} \tilde{v}_k \snorm{\tilde{u}_k}^2\,\partial_r \phi(\tilde{r}_k(x,s),s)
-\kappa\tilde{r}_k^{m} \D_x\tilde{e}_k\,\partial_r \phi(\tilde{r}_k(x,s),s) \Big\}
  \,\dif x \dif s.
\end{align*}
From \eqref{wtemp7}, $\tilde{r}_k(x,0)=\tilde{r}_a^0(x)$
if $\tilde{r}_a^0(x)\in \supp(\phi(\cdot,0))$ and $k\ge M(\phi,a)$.
By the coordinate transformation \eqref{eqs:coordtrans}--\eqref{eqs:coorderiv},
and \eqref{wtemp1}--\eqref{wtemp6}, it follows that
\begin{align*}
I_1
&= \int_{0}^{t} \int_{a}^{\tilde{r}_k(k,t)} \Big\{ m(2\mu-\beta) \dfrac{\snorm{\overline{u}_k}^2}{r}
  - \beta \overline{u}_k  \partial_r\overline{u}_k + (\gamma-1) \overline{u}_k \overline{v}_k^{-1} \overline{e}_k
   - \kappa \partial_r \overline{e}_k \Big\} \partial_r\phi\,r^m \dif r \dif s \\
&\quad + \frac{1}{2}\int_{0}^{t}\int_{a}^{\tilde{r}_k(k,t)}
   \overline{v}_k^{-1} \snorm{\overline{u}_k}^2
    \big\{\partial_t \phi  + \overline{u}_k \partial_r\phi\big\}\,r^m  \dif r \dif s \\
&\quad - \frac{1}{2}\int_{a}^{\tilde{r}_k(k,t)}
  (\overline{v}_k^{-1}\snorm{\overline{u}_k}^2)(r,t)\phi(r,t)\,r^m \dif r
  + \frac{1}{2}\int_{a}^{\tilde{r}_a^0(k)} (\rho_a^0 \snorm{u_{a}^0}^2)(r) \phi(r,0)\,r^m \dif r \\
&= \int_{0}^{t} \int_{a}^{\infty} \Big\{ -m \lambda \dfrac{\snorm{u_k}^2}{r}
 - \beta u_k  \partial_r u_k + (\gamma-1) \rho_k u_k e_k - \kappa  \partial_r e_k  \Big\}
   \partial_r\phi\,r^m \dif r \dif s \\
&\quad + \frac{1}{2}\int_{0}^{t}\int_{a}^{\infty} \rho_k \snorm{u_k}^2
  \big\{\partial_t \phi + u_k \partial_r \phi\big\}\,r^m  \dif r \dif s \\
&\quad - \frac{1}{2}\int_{a}^{\infty} (\rho_k \snorm{u_k}^2)(r,t) \phi(r,t)\, r^m \dif r
  + \frac{1}{2}\int_{a}^{\infty} (\rho_a^0 \snorm{u_a^0}^2)(r) \phi(r,0)\, r^m \dif r.
\end{align*}
Define $E_k(r,t) \vcentcolon= (\frac{1}{2}\snorm{u_k}^2 + e_k)(r,t)$. Then the above identity becomes:
\begin{align*}
I_1
=&\,\int_{0}^{t} \int_{a}^{\infty} \Big\{ -m \lambda \dfrac{\snorm{u_k}^2}{r}
  - \beta u_k  \partial_r u_k + (\gamma-1) \rho_k u_k e_k - \kappa  \partial_r e_k \Big\}
  \partial_r\phi\,r^m \dif r \dif s \\
&\, + \int_{0}^{t}\int_{a}^{\infty} \big\{ \rho_k E_k ( \partial_t \phi + u_k \partial_r \phi)
 - \rho_k e_k (\partial_t \phi + u_k \partial_r \phi) \big\} \,r^m  \dif r \dif s \\
&\,- \frac{1}{2}\int_{a}^{\infty} (\rho_k \snorm{u_k}^2)(r,t) \phi(r,t)\,r^m \dif r
  + \frac{1}{2}\int_{a}^{\infty} (\rho_a^0\snorm{u_a^0}^2)(r) \phi(r,0)\,r^m \dif r.
\end{align*}
Denote $P_k\vcentcolon= (\gamma-1)\rho_k e_k$ and $E_a^0\vcentcolon=\frac{1}{2}\snorm{u_a^0}^2+e_a^0$.
Substituting $I_1$ into \eqref{wtemp2}, we have
\begin{equation}\label{wtemp3}
\begin{aligned}
&\int_{a}^{\infty} (\rho_k E_k)(r,t) \phi (r,t)\,r^m \dif r
  - \int_{a}^{\infty} (\rho_a^0 E_a^0) (r) \phi(r,0) r^m \dif r \\
&= \int_{0}^{t} \int_{a}^{\infty} \Big\{ \Big( (\rho_kE_k + P_k) u_k -m \lambda \dfrac{\snorm{u_k}^2}{r}
- \beta u_k  \partial_r u_k - \kappa  \partial_r e_k \Big)\partial_r\phi
+ \rho_k E_k \partial_t \phi \Big\}\,r^m \dif r \dif s.
\end{aligned}
\end{equation}
	
In addition to \eqref{wtemp4}, we also uses the weak convergence from Lemma \ref{lemma:weak}
that $\partial_r e_k \rightharpoonup  \partial_r e$ in $L^2(0,T;L^2([a,\infty),r^m \dif r))$.
Then \eqref{wtemp3} becomes:
\begin{align*}
&\int_{a}^{\infty} (\rho E)(r,t) \phi (r,t)\,r^m \dif r
 - \int_{a}^{\infty} (\rho_a^0 E_a^0)(r) \phi(r,0)\,r^m \dif r \\
&= \int_{0}^{t} \int_{a}^{\infty} \Big\{ \Big( (\rho E + P ) u -m \lambda \dfrac{\snorm{u}^2}{r}
  - \beta u  \partial_r u - \kappa  \partial_r e \Big) \partial_r\phi
  + \rho E \partial_t \phi \Big\}\,r^m \dif r \dif s,
\end{align*}
where $P\vcentcolon= (\gamma-1) \rho e$ and $E\vcentcolon= \frac{1}{2}\norm{u}^2 + e$.
\end{proof}

\subsection{Estimates away from the origin, independent of the particle path function}\label{subsec:indepPath} 
The main aim of this subsection is to obtain
the uniform  estimates, independent of $a\in(0,1)$, which do not involve the particle path
function $\tilde{r}(x,t)$.

\begin{lemma}\label{lemma:dissip}
For each $\eta\in (a,1)$,
\begin{equation}\label{temp:dissip0}
\begin{dcases*}
\int_{0}^{T} \sup_{r\ge\eta}\,\dfrac{\snorm{u}}{\sqrt{e}}(r,t)\,\dif t
\le C(T) \big(\eta^{\frac{2-n}{2}} +  \eta^{2-n}\big) \ \ & if $n=2$, $3$,\\
\int_{0}^{T} \sup_{r\ge \eta} \log( \max\{ 1, e^{\pm 1}(r,t)\})\,\dif t \le C(T) \eta^{2-n} & if $n=3$,\\
\int_{0}^{T} \sup_{r\ge\eta} \log( \max\{1, e(r,t)\})\,\dif t
\le C(T)\big( 1 + \sqrt{|\log\eta|}\big) & if $n=2$.
\end{dcases*}
\end{equation}
\end{lemma}

\begin{proof}
We divide the proof into three steps.

\smallskip
1. Fix a point $r \in [\eta, \infty)$ and a small constant $0 < \delta < 1$. Since $\lim_{r\to\infty} u_k(r,t)=0$
and $\lim_{r\to\infty} e_k(r,t)=1$ for each $t\in[0,T]$ and $k\in\mathbb{N}$, we have
\begin{align*}
\dfrac{u_k}{r^{\delta}\sqrt{e_k}}(r,t)
&= \lim\limits_{r\to\infty}\dfrac{u_k}{r^{\delta}\sqrt{e_k}}(r,t)
- \int_{r}^{\infty} \Big\{ \zeta^{-\delta}\dfrac{\partial_r u_k}{\sqrt{e_k}} -\delta \zeta^{-\delta-1} \dfrac{u_k}{\sqrt{e_k}}
 -\dfrac{1}{2}\zeta^{-\delta} u_k e_k^{-\frac{3}{2}}\partial_r e_k \Big\} (\zeta,t)\,\dif \zeta\\
&\le  \Big( \int_{r}^{\infty} \zeta^{-2\delta-m}\dif \zeta \Big)^{\frac{1}{2}}\Big\{ \delta\Big( \int_{r}^{\infty} \dfrac{\snorm{u_k}^2}{\zeta^2 e_k}\,\zeta^m \dif \zeta \Big)^{\frac{1}{2}}
+ \Big( \int_{r}^{\infty} \dfrac{\snorm{\partial_r u_k}^2}{e_k}\,\zeta^m \dif\zeta  \Big)^{\frac{1}{2}} \Big\}\\
&\quad + \dfrac{1}{2}\Big( \int_{r}^{\infty} \dfrac{\snorm{u_k}^2}{\zeta^2 e_k}  \zeta^{-\delta+1} \dif \zeta  \Big)^{\frac{1}{2}} \Big( \int_{r}^{\infty} \dfrac{\snorm{\partial_r e_k}}{e_k^2} \zeta^{-\delta+1} \dif\zeta\Big)^{\frac{1}{2}}\\
&\le \dfrac{r^{-\delta} \eta^{\frac{1-m}{2}}}{\sqrt{2\delta+m-1}}\Big\{ \delta\Big( \int_{\eta}^{\infty} \dfrac{\snorm{u}_k^2}{\zeta^2 e_k}(\zeta,t)\,\zeta^m \dif \zeta \Big)^{\frac{1}{2}}
+ \Big( \int_{\eta}^{\infty} \dfrac{\snorm{\partial_r u_k}^2}{e_k}\,\zeta^m \dif\zeta\Big)^{\frac{1}{2}}\Big\} \\
&\quad +\dfrac{1}{2} r^{-\delta} \eta^{1-m}\Big(\int_{\eta}^{\infty}
 \dfrac{\snorm{u_k}^2}{\zeta^2 e_k}\,\zeta^{m} \dif \zeta  \Big)^{\frac{1}{2}}
 \Big( \int_{\eta}^{\infty} \dfrac{\snorm{\partial_r e_k}}{e_k^2}\,\zeta^{m} \dif \zeta \Big)^{\frac{1}{2}}
 =\vcentcolon r^{-\delta}g_k(\eta,t),
\end{align*}
where we have used the fact that $-2\delta-m+1<0$ and $-\delta+1-m<0$ for $n=2, 3$ since $0<\delta<1$.
Multiplying both sides of the above inequality by $r^{\delta}$ and taking the limit inferior over $k\to\infty$,
it follows by Lemma \ref{lemma:klimholder} that
\begin{equation*}
\dfrac{|u|}{\sqrt{e}}(r,t) \le \liminf\limits_{k\to\infty} g_k(\eta,t)
\qquad\ \text{for all $(r,t)\in[\eta,\infty)\times[0,T]$.}
\end{equation*}
Taking the supremum in $r\in[\eta,\infty)$ and integrating in $t\in[0,T]$, we obtain
\begin{equation}\label{temp:dissip1}
\int_{0}^{T} \sup_{r\ge\eta}\,\dfrac{|u|}{\sqrt{e}}(r,t)\,\dif t \le \int_{0}^{T} \liminf_{k\to\infty}g_{k}(\eta,t)\,\dif t
\le \liminf_{k\to\infty} \int_{0}^{T} g_{k}(\eta,t)\,\dif t,
\end{equation}
where we have used Fatou's lemma since $g_k(\eta,t)>0$.
By definition of $g_k(\eta,t)$
and Lemma \ref{lemma:entEstext}, we obtain that, for all $k\in\mathbb{N}$,
\begin{align*}
\int_{0}^{T} g_{k}(\eta,t) \dif t
&\le C\eta^{\frac{1-m}{2}}\Big\{\int_{0}^{T}\Big( \int_{\eta}^{\infty}
\dfrac{\snorm{u}_k^2}{\zeta^2 e_k}\,\zeta^m\dif \zeta \Big)^{\frac{1}{2}} \dif t
+ \int_{0}^{T}\Big( \int_{\eta}^{\infty}
\dfrac{\snorm{\partial_r u_k}^2}{e_k} \,\zeta^m \dif \zeta  \Big)^{\frac{1}{2}}\dif t  \Big\}   \\
&\quad + C \eta^{1-m} \int_{0}^{T} \Big( \int_{\eta}^{\infty} \dfrac{\snorm{u_k}^2}{\zeta^2 e_k}\,\zeta^{m} \dif \zeta  \Big)^{\frac{1}{2}} \Big( \int_{\eta}^{\infty} \dfrac{\snorm{\partial_r e_k}}{e_k^2}\,\zeta^{m}\dif\zeta\Big)^{\frac{1}{2}} \dif t \\
&\le C\eta^{\frac{1-m}{2}} T^{\frac{1}{2}} \Big\{ \Big( \int_{0}^{T} \int_{\eta}^{\infty}
 \dfrac{\norm{u}_k^2}{\zeta^2 e_k}\,\zeta^m \dif \zeta \dif t \Big)^{\frac{1}{2}}
 + \Big( \int_{0}^{T}\int_{\eta}^{\infty}
 \dfrac{\snorm{\partial_r u_k}^2}{e_k}\,\zeta^m \dif \zeta \dif t   \Big)^{\frac{1}{2}} \Big\}\\
&\quad + C \eta^{1-m} \Big( \int_{0}^{T}\int_{\eta}^{\infty}
 \dfrac{\snorm{u_k}^2}{\zeta^2 e_k}\,\zeta^{m} \dif \zeta  \dif t \Big)^{\frac{1}{2}}
 \Big( \int_{0}^{T} \int_{\eta}^{\infty} \dfrac{\snorm{\partial_r e_k}}{e_k^2}\,
 \zeta^{m} \dif \zeta \dif t \Big)^{\frac{1}{2}}\\
&\le C(T) ( \eta^{\frac{1-m}{2}} + \eta^{1-m}),
\end{align*}
which, together with \eqref{temp:dissip1}, leads to $\eqref{temp:dissip0}_1$.
	
\smallskip	
2.  Consider the case: $n=3$. Let $(r,t)\in[\eta,\infty)\times[0,T]$. By $\lim_{r\to\infty} e_k(r,t) =1$, we have
\begin{align*}
\log e_k(r,t)
&= -\int_{r}^{\infty} \dfrac{\partial_r e_k}{e_k}(\zeta,t)\,\dif r
\le \Big(\int_{r}^{\infty} \zeta^{-m} \dif \zeta\Big)^{\frac{1}{2}}
\Big( \int_{r}^{\infty} \dfrac{\snorm{\partial_r e_k}^2}{e_k^2}(\zeta,t)\,\zeta^m \dif \zeta \Big)^{\frac{1}{2}}\\
&\le \dfrac{\eta^{2-n}}{\sqrt{n-2}}\Big( \int_{\eta}^{\infty}
\dfrac{\snorm{\partial_r e_k}^2}{e_k^2}(\zeta,t) \,\zeta^m \dif \zeta \Big)^{\frac{1}{2}}
=\vcentcolon h_k(\eta,t).
\end{align*}
Thus, $0 \le \log ( \max\{ 1, e_k \} ) (r,t) \le h_k(\eta,t)$ for all $(r,t)\in[\eta,\infty)\times[0,T]$.
Replacing the term $e_k(r,t)$ with $e_k^{-1}(r,t)$ in the above derivation,
then $0 \le \log ( \max\{ 1, e_k^{-1}\} ) (r,t) \le h_k(\eta,t)$ for all $(r,t)\in[\eta,\infty)\times[0,T]$.
Taking the limit inferior over $k\to\infty$ on both sides of these inequalities,
it follows from Lemma \ref{lemma:klimholder} that
\begin{align*}
\log ( \max\{ 1, e^{\pm 1} \} ) (r,t) \le \liminf_{k\to\infty}h_{k}(\eta,t) \qquad\ \text{for all $(r,t)\in[\eta,\infty)\times[0,T]$.}
\end{align*}
Taking the supremum in $r\ge\eta$ on the above inequality and integrating in $t\in [0,T]$, we have	\begin{equation}\label{temp:dissip2}
\int_{0}^{T} \sup_{r\ge\eta} \log ( \max\{ 1, e^{\pm 1}(r,t) \} )\,\dif t
\le \int_{0}^{T} \liminf_{k\to\infty}h_{k}(\eta,t)\,\dif t
\le \liminf_{k\to\infty} \int_{0}^{T} h_{k}(\eta,t)\,\dif t,
\end{equation}
where we have used Fatou's lemma since $h_k(\eta,t)>0$.
By Lemma \ref{lemma:entEstext}, we obtain that, for all $k\in\mathbb{N}$,
\begin{align*}
\int_{0}^{T} h_{k}(\eta,t)\,\dif t 
\le C \eta^{2-n} \sqrt{T} \Big( \int_{0}^{T}\int_{\eta}^{\infty}
\dfrac{\snorm{\partial_r e_k}^2}{e_k^2}\,\zeta^m \dif \zeta \dif t\Big)^{\frac{1}{2}}
\le C(T) \eta^{2-n},
\end{align*}
which, together with \eqref{temp:dissip2}, proves $\eqref{temp:dissip0}_2$.
	
\smallskip	
3. We consider the case: $n=2$.
Set $M \vcentcolon= \sup_{k\in\mathbb{N}} \sup_{t\in[0,T]} \tilde{r}_k(1,t)$
so that, by Lemma \ref{lemma:rbd}, $M\le (2C_0)^{\frac{1}{n}}$.
If $r\ge M$, then
$r > \sup_{t\in[0,T]} \tilde{r}_k(1,t)$ for all $k\in\mathbb{N}$.
By Lemma \ref{lemma:HREext}, $e_k(r,t)\le C(\ve) \sigma^{-\frac{1}{2}}(t)$ for all $t\in[0,T]$ and $r\in[\tilde{r}_k(\ve,t),\infty)$.
It follows that there exists a constant $C(T)>1$
such that
\begin{equation}\label{temp:logM1}
\log (\max \{ 1,e_k(r,t) \}) \le \log (C(T) \sigma^{-\frac{1}{2}}(t)) \qquad\,
\text{for all $(r,t)\in[M,\infty)\times[0,T]$.}
\end{equation}
By
\eqref{temp:logM1}, for all $r\in[\eta,M)$,
\begin{equation}\label{temp:logM2}
\begin{aligned}
\log e_k(r,t) & = \log e_k(2M,t) - \int_{r}^{2M}\dfrac{\partial_r e_k}{e_k}(\zeta,t)\,\dif \zeta\\
&\le \log (\max\{1,e_k(2M,t)\}) + \Big(\int_{r}^{2M} \zeta^{-1} \dif \zeta \Big)^{\frac{1}{2}} \Big(\int_{r}^{2M} \dfrac{\snorm{\partial_r e_k}^2}{e_k^2}(\zeta,t)\,\zeta \dif \zeta\Big)^{\frac{1}{2}}\\
&\le \log (C(T) \sigma^{-\frac{1}{2}}(t))
+ \sqrt{\log (C(T)/\eta)} \Big(\int_{\eta}^{\infty}
   \dfrac{\snorm{\partial_r e_k}^2}{e_k^2}(\zeta,t)\,\zeta \dif \zeta\Big)^{\frac{1}{2}}.
\end{aligned}
\end{equation}
Finally, in a similar way to that for the case $n=3$,
\begin{align*}
\int_{0}^{T} \sup_{r\ge\eta}\log (\max\{ 1, e(r,t) \}) \dif t
&\le \dfrac{C^2(T)}{2} + \int_{0}^{T} \sup_{r\in[\eta,M)}\log (\max\{ 1, e(r,t)\})\,\dif t \\
&\le C(T) + \sqrt{T\log(C(T)/\eta)} \sup_{k\in\mathbb{N}}\Big( \int_{0}^{T}\int_{\eta}^{\infty} \dfrac{\snorm{\partial_r e_k}^2}{e_k^2}\,\zeta^m \dif \zeta \dif t\Big)^{\frac{1}{2}}\\
&\le  C(T)\big(1+ \sqrt{|\log\eta|}\big).
\end{align*}
\end{proof}

\section{Global Weak Solutions including the Origin}\label{sec:ato0}

We denote $(\rho_a,u_a,e_a)(r,t)$ in $[a,\infty)\times[0,T]$ as the solution
of problem \eqref{eqs:SFNS} and \eqref{eqs:eSFNS}, which is guaranteed by Theorem \ref{thm:WSCEx},
with the modified initial data $(\rho_a^0,u_a^0,e_a^0)(r)$ constructed in \S \ref{subsec:mollify}
for each $a\in(0,1)$. We also denote $\tilde{r}_a(x,t)$ as the particle path function associated
with $\rho_a$, which satisfies
\begin{equation*}
	x= \int_{a}^{\tilde{r}_a(x,t)} \rho_a(y,t)\,r^m \dif r \qquad
	\text{for {\it a.e.} $(x,t)\in [0,\infty)\times[0,T]$}.
\end{equation*}
The main aim of this section is to take the limit: $a\searrow 0$ of the solution sequence
$(\rho_a,u_a,e_a)(r,t)$ of problems \eqref{eqs:SFNS} and \eqref{eqs:eSFNS}
to obtain a global spherically symmetric weak solution of the original Cauchy
problem \eqref{eqs:FNS}--\eqref{eqs:CauchyInit}.

\subsection{Uniform inequalities from the entropy estimates} \label{subsec:meas}
We obtain a set of inequalities involving measurable subsets of $[a,\infty)$, which
are uniform in $a\in(0,1)$.
In this section, we denote
\begin{equation}\label{eqs:GpsiH}
    G(\zeta) \vcentcolon= 1-\zeta+\zeta \log \zeta,
    \quad \psi(\zeta) \vcentcolon= \zeta -1 - \log \zeta,
    \quad H(\zeta) \vcentcolon= \zeta \log \zeta,
\end{equation}
which are all convex functions.
Let $G^{-1}_{+}$, $\psi_{+}^{-1}$, and $H_{+}^{-1}$ be the right branch inverses of $G$, $\psi$,
and $H$, respectively.
Then, by Proposition \ref{prop:omega},
the following functions are well defined for $(y,z)\in(0,\infty)^2\vcentcolon$
\begin{equation}\label{eqs:fis}
f_1(y;z)\vcentcolon= y G_{+}^{-1}( \dfrac{z}{y} ), \quad
f_2( y; z) \vcentcolon= y \psi_{+}^{-1}(\dfrac{z}{y}) - z,
\quad f_3(y;z) \vcentcolon= y H_{+}^{-1}(\dfrac{z}{y}).
\end{equation}
Some further properties of $(y,z)\mapsto f_i(y;z)$ for $i=1,2,3$ can be found
in Proposition \ref{prop:omega}.

\begin{lemma}\label{lemma:rhoM}
For any measurable set $E\subset [a,\infty)$ and each $t\in[0,T]$,
\begin{align}
&\sup_{a\in(0,1)}\int_{E} \rho_a(r,t)\,r^m \dif r \le \omega_1(E; C(T)),\label{6.3a} \\
&\sup_{a\in(0,1)}\int_{E} (\rho_a e_a)(r,t)\,r^m \dif r \le C(T) + \omega_2(E;C(T)),\label{6.3b}
\end{align}
where $\omega_1(E;z) \vcentcolon = f_1 (\int_{E} r^m \dif r ; z)$
and
$\omega_2(E;z) \vcentcolon = f_2(\omega_1(E;z);z)
= f_2( f_1(\int_{E}r^m \dif r ; z);z)$ for $z>0$.
\end{lemma}

\begin{proof}
Denote $\mathcal{L}_n(E) \vcentcolon= \int_{E} r^m \dif r$ for each $E\subseteq [a,\infty)$. 
By Jensen's inequality and Theorem \ref{thm:WSCEx}\ref{item:WSCEx1}, it follows that
	\begin{align*}
		G( \dfrac{1}{\mathcal{L}_n(E)} \int_{E} \rho_a(r,t)\,r^m \dif r)
		\le \dfrac{1}{\mathcal{L}_n(E)} \int_{E} G(\rho_a(r,t))\,r^m \dif r \le \dfrac{C(T)}{\mathcal{L}_n(E)}.
	\end{align*}
Taking the right branch inverse $G_{+}^{-1}(z)$ on both sides of the above
and using the fact that $z\mapsto G_{+}^{-1}(z)$ is monotone increasing, it follows that, for all $a\in(0,1)$,
\begin{align*}
		\int_{E} \rho_a(r,t)\,r^m \dif r
		\le \mathcal{L}_n(E) G_{+}^{-1}(\dfrac{C(T)}{\mathcal{L}_n(E)})
		= f_1( \mathcal{L}_n(E) ; C(T)) = \vcentcolon \omega_1(E;C(T)).
\end{align*}

Now, fix $t\in[0,T]$ and set $m_{a,t}(E) \vcentcolon= \int_{E} \rho_a(r,t)\,r^m \dif r$.
By Jensen's inequality and Theorem \ref{thm:WSCEx}\ref{item:WSCEx1}, it follows that
\begin{align*}
\psi( \dfrac{1}{m_{a,t}(E)} \int_E (\rho_a e_a) (r,t)\,r^m \dif r)
\le \dfrac{1}{m_{a,t}(E)} \int_{E} (\rho_a \psi(e_a)) (r,t)\,r^m \dif r \le \dfrac{C(T)}{m_{a,t}(E)}.
\end{align*}
Taking $\psi_{+}^{-1}(\cdot)$ on both sides of the above, it follows that
\begin{align*}
\int_E (\rho_a e_a) (r,t)\,r^m \dif r \le m_{a,t}(E) \psi_{+}^{-1}(\dfrac{C(T)}{m_{a,t}(E)})
= f_2(m_{a,t}(E); C(T)) + C(T).
\end{align*}
Using \eqref{6.3a}, $m_{a,t}(E) \le \omega_1(E ;C(T))$.
Since $y\mapsto f_2(y;z)$ is increasing in $y$ for fixed $z>0$ from Proposition \ref{prop:omega},
it follows that, for all $a\in(0,1)$,
\begin{align*}
\int_E (\rho_a e_a) (r,t)\,r^m \dif r
\le C(T) + f_2(\omega_1(E;C(T)); C(T)) = \vcentcolon C(T) + \omega_2(E;C(T)).
\end{align*}
\end{proof}

\subsection{Existence of particle paths and vacuum interfaces}\label{subsec:alimpath}

\begin{lemma}\label{lemma:alimpath}
There exist both a subsequence $\{a_j\}_{j\in\mathbb{N}}$ with $a_j \searrow 0$ as $j\to \infty$
and a continuous function $\tilde{r}(x,t)\vcentcolon [0,\infty)\times[0,T]\to [0,\infty)$ such that
\begin{enumerate}[label=(\roman*),ref=(\roman*),font={\normalfont\rmfamily}]
\item\label{item:apath1} for each compact subset $K\Subset (0,\infty)\times[0,T]$,
\begin{align*}
\lim\limits_{j\to \infty} \sup_{(x,t)\in K} \snorm{\tilde{r}_{a_j}(x,t)-\tilde{r}(x,t)} = 0.
\end{align*}

\item\label{item:apath2} for each fixed $\ve>0$,
\begin{equation*}
\begin{aligned}
&\snorm{\tilde{r}(x_1,t)-\tilde{r}(x_2,t)} \le C(\ve) \snorm{x_1-x_2}
  &&\quad \text{for all $(x_1,x_2,t)\in [\ve,\infty)^2\times[0,T]$},\\
&\snorm{\tilde{r}(x,t_1)-\tilde{r}(x,t_2)} \le C(\ve) \snorm{t_1^{\frac{3}{4}}-t_2^{\frac{3}{4}}}
  &&\quad \text{for all $(x,t_1,t_2)\in [\ve,\infty)\times[0,T]^2$}.
\end{aligned}
\end{equation*}
		
\item\label{item:apath3} for all $t\in[0,T]$, the function{\rm :} $x\mapsto \tilde{r}(x,t)$ is strictly increasing
so that $\ul{r}(t) \vcentcolon= \lim_{x\to 0^+} \tilde{r}(x,t)$ exists for all $t\in[0,T]$
and
\begin{equation*}
\lim_{t\searrow 0} \ul{r}(t) = 0, \qquad  t\mapsto \ul{r}(t) \ \text{is upper semi-continuous}.
\end{equation*}
In addition, $\mathring{\mathrm{F}}\vcentcolon= \{(r,t)\in(0,\infty)\times(0,T)\,\vcentcolon\, \ul{r}(t)<r\}$
is an open set.

\item\label{item:apath4} $\tilde{r}(x,t)$ satisfies the pointwise bound{\rm :}
\begin{equation*}
n x \psi^{-1}_{-}(\dfrac{C_0}{x}) \le  r^n(x,t) \le C_0( 1 +x )
\qquad  \text{for all $(x,t)\in[0,\infty)\times[0,T]$}.
\end{equation*}

\item\label{item:apath5} if $\ul{r}(t)>0$ for some $t\in(0,T]$, then
\begin{equation*}
\lim\limits_{j\to \infty} \int_{a_j}^{\ul{r}(t)} \rho_{a_j}(r,t)\,r^m \dif r =0.
\end{equation*}
In particular, for such $t\in(0,T]$,
\begin{equation*}
\liminf\limits_{j\to \infty} \rho_{a_j} (r,t) = 0 \qquad  \text{for a.e. $r\in(0,\ul{r}(t))$}.
\end{equation*}
\end{enumerate}
\end{lemma}

\begin{proof}  We divide the proof into six steps.

\smallskip
1. For each integer $i\in\mathbb{N}$, let $S_{i}\vcentcolon= [i^{-1},i]\times[0,T]$.
From Theorem \ref{thm:WSCEx}\ref{item:WSCEx3}--\ref{item:WSCEx4}, it follows that, for all $(x,t)\in S_i$,
\begin{equation}\label{temp:apath1}
\tilde{r}_a(x,t) = \tilde{r}_a^0(x) + \int_{0}^{t} u_a( \tilde{r}(x,s), s)\,\dif s
\le \tilde{r}_a^0(x) + C(i) \int_{0}^{T} \sigma^{-\frac{1}{4}}(s)\,\dif s \le C(i)(1+T),
\end{equation}
with $C(i)>0$ a constant independent of $a\in(0,1)$, and
\begin{equation}\label{temp:apath2}
\begin{split}
&\snorm{\tilde{r}_a(x_1,t)-\tilde{r}_a(x_2,t)} \le C(i) \snorm{x_1-x_2} \qquad\text{for all $(x_1,x_2,t)\in [i^{-1},i]^2\times[0,T]$},\\
&\snorm{\tilde{r}_a(x,t_1)-\tilde{r}_a(x,t_2)} \le C(i) \snorm{t_1^{\frac{3}{4}}-t_2^{\frac{3}{4}}} \qquad\,
  \text{for all $(x,t_1,t_2)\in [i^{-1},i]\times[0,T]^2$}.
\end{split}
\end{equation}
In addition, $S_i\subset S_{i+1}$, $\text{Int}(S_i)\neq \emptyset$ for each $i\in\mathbb{N}$, and $(0,\infty)\times[0,T]=\cup_{i\in\mathbb{N}}S_i$.
Thus, by Proposition \ref{prop:aaExt}, it follows that there exists a continuous function $\tilde{r}(x,t)\vcentcolon(0,\infty)\times[0,T]\to [0,\infty)$ such that,
 for all compact subset $K\Subset (0,\infty)\times[0,T]$,
\begin{equation}\label{temp:apath3}
\lim\limits_{j\to\infty} \sup\limits_{(x,t)\in K} \snorm{\tilde{r}_{a_j}(x,t)-\tilde{r}(x,t)} = 0,
\end{equation}
which, along with \eqref{temp:apath2} and Theorem \ref{thm:WSCEx}\ref{item:WSCEx4}, implies Lemma  \ref{lemma:alimpath}\ref{item:apath2}.

\smallskip
2. Let $0<x_1<x_2<\infty$. It follows from Theorem \ref{thm:WSCEx}\ref{item:WSCEx3}--\ref{item:WSCEx4} that,
for {\it a.e.} $t\in[0,T]$,
\begin{align*}
x_2 - x_1 =\int_{\tilde{r}_a(x_1,t)}^{\tilde{r}_a(x_2,t)} \rho_a(r,t)\,r^m \dif r
\le C(x_1)\int_{\tilde{r}_a(x_1,t)}^{\tilde{r}_a(x_2,t)} r^m \dif r
= \dfrac{C(x_1)}{n} \big(\tilde{r}_a^{\,n}(x_2,t) - \tilde{r}_a^{\,n}(x_1,t)\big),
\end{align*}
where
we have used that $\rho_a(r,t)\le C(x_1)$ for {\it a.e.} $(r,t)$
such that $r \ge \tilde{r}_a(x_1,t)$.
Taking $j\to\infty$ and using \eqref{temp:apath3} yield
\begin{equation*}
0 < x_2 - x_1 \le \dfrac{C(x_1)}{n} \big(\tilde{r}^{\,n}(x_2,t)-\tilde{r}^{\,n}(x_1,t)\big).
\end{equation*}
This shows that $x\mapsto\tilde{r}(x,t)$ is strictly increasing for {\it a.e.} $t\in[0,T]$.
Since $(x,t)\mapsto \tilde{r}(x,t)$ is continuous, then $x\mapsto\tilde{r}(x,t)$ is strictly increasing
for all $t\in[0,T]$. Since $\tilde{r}(x,t)\ge 0$,
\begin{equation}\label{temp:apath4}
\ul{r}(t) \vcentcolon= \lim\limits_{x\to 0^+} \tilde{r}(x,t) \quad \text{ exists for all $t\in[0,T]$}.
\end{equation}
From Theorem \ref{thm:WSCEx}\ref{item:WSCEx3}, it follows that $\tilde{r}_a(x,0)=\tilde{r}_a^0(x)$
for each $a\in(0,1)$ and $x\in[0,\infty)$, where $\tilde{r}_a^{0}(x)$ is the function constructed
in \S \ref{subsec:ffaprox}, satisfying
\begin{equation*}
x=\int_{a}^{\tilde{r}_a^0(x)} \rho_a^{0}(r)\,r^m \dif r\qquad \text{for all $x\in[0,\infty]$}.
\end{equation*}
Then, by the uniform convergence \eqref{temp:apath3} and Proposition \ref{prop:mollify} in \S \ref{subsec:mollify},
we have
\begin{equation*}
x=\int_{0}^{\tilde{r}(x,0)}\rho_0(r)\,r^m \dif r \ge \dfrac{\tilde{r}^{\,n}(x,0)}{C_0 n}\qquad
\text{for all $x>0$}.
\end{equation*}
This implies the inequality: $\tilde{r}(x,0) \le (n C_0 x)^{\frac{1}{n}}$.
Given $\delta>0$, let $x_\delta>0$ be such that $\tilde{r}(x_\delta,0)\le ( n C_0 x_\delta )^{\frac{1}{n}}
\le \frac{\delta}{2}$.
Since $t\mapsto \tilde{r}(x_\delta,t)$ is continuous in $[0,T]$, there exists $t_\delta>0$ such that,
if $t\in(0,t_\delta)$, then $\tilde{r}(x_\delta,t) < \tilde{r}(x_\delta,0) + \frac{\delta}{2}\le
\delta$. By \eqref{temp:apath4},  $\ul{r}(t)=\inf\{ \tilde{r}(x,t) \vcentcolon x>0 \}$ for each $t\in[0,T]$.
It follows that $\ul{r}(t)\le \tilde{r}(x_\delta,t) < \delta$ if $t\in(0,t_\delta)$.
This shows that $\lim_{t\searrow 0}\ul{r}(t)=0$.
	
\smallskip	
3. We now show that $t\mapsto \ul{r}(t)$ is upper semi-continuous.
To start with, a function $f\vcentcolon \R \to \R$ is upper semi-continuous if and only if
the set $\{x\in\R\vcentcolon f(x)<y \}$ is open for all $y\in\R$.
Define the sets:
\begin{equation*}
\begin{aligned}
&U_{\zeta} \vcentcolon = \{ s\in(0,T) \vcentcolon \ul{r}(s) < \zeta \} \  &&\text{ for each $\zeta\in(0,\infty)$,} \\
&U_{\zeta}^x \vcentcolon =  \{ s\in(0,T) \vcentcolon \tilde{r}(x,s) < \zeta \} \
&&\text{ for each $(x,\zeta)\in(0,\infty)^2$.}
\end{aligned}
\end{equation*}
Suppose that $t\in U_{\zeta}$.
Since $\ul{r}(t)=\inf_{x>0} \tilde{r}(x,t)$, there exists $\delta>0$
such that $\zeta > \tilde{r}(\delta,t) \ge \ul{r}(t)$.
Thus, $t\in U_{\zeta}^{\delta} \subseteq \cup_{x>0} U_{\zeta}^{x}$, which implies that
$U_{\zeta}\subseteq \cup_{x>0} U_{\zeta}^{x}$.
On the other hand, if $t\in U_{\zeta}^{\delta}$ for some $\delta>0$,
then $\zeta>\tilde{r}(\delta,t)\ge \ul{r}(t)$ so that $t\in U_{\zeta}$,
which implies that $U_{\zeta}= \cup_{x>0}U_{\zeta}^{x}$.
Since $t\mapsto \tilde{r}(x,t)$ is continuous,
$U_{\zeta}^{x}$ is an open set for each $x>0$.
Therefore, $U_{\zeta}$ is also an open set as it is the union of open sets.
Since this is true for arbitrary $\zeta>0$,
and $\{s\in[0,T]\vcentcolon \ul{r}(s)<\zeta \}=\emptyset$ for $\zeta\le 0$,
$t\mapsto \ul{r}(t)$ is upper semi-continuous.
	
\smallskip	
4. For each $\ve\in(0,1]$, define
$\mathring{\mathrm{F}}_{\ve} \vcentcolon= \{(r,t)\in(0,\infty)\times(0,T)\,\vcentcolon\,\tilde{r}(\ve,t)< r\}$.
Then $\mathring{\mathrm{F}}_{\ve}$ is an open set, since $t\mapsto \tilde{r}(\ve,t)$ is continuous in $t\in[0,T]$.
Since $\ul{r}(t)=\inf_{x>0}\tilde{r}(x,t)$ for all $t\in[0,T]$,
it follows that $\mathring{\mathrm{F}}_{\ve}\subset \mathring{\mathrm{F}}$ for each $\ve>0$.
Furthermore, if $(r,t)\in\mathring{\mathrm{F}}$, then there exists $\ve>0$ such
that $\ul{r}(t) \le \tilde{r}(\ve,t)<r$, which implies $(r,t)\in\mathring{\mathrm{F}}_{\ve}$
so that $\mathring{\mathrm{F}}=\cup_{\ve>0}\mathring{\mathrm{F}}_{\ve}$.
Thus, $\mathring{\mathrm{F}}$ is open as it is a union of open sets. This concludes the proof for Lemma \ref{lemma:alimpath}\ref{item:apath3}.

\smallskip	
5. We show the upper and lower bounds of $\tilde{r}(x,t)$.
From Theorem \ref{thm:WSCEx}\ref{item:WSCEx3}, we have
	\begin{equation*}
		nx \psi_{-}^{-1}(\frac{C_0}{x}) \le \tilde{r}_a^{\,n}(x,t) \le C_0(1 + x) \qquad\ \text{for all $(a,x,t)\in(0,1)\times[0,\infty)\times[0,T]$,}
	\end{equation*}
    which, along with the uniform convergence in  \eqref{temp:apath3}, implies Lemma \ref{lemma:alimpath}\ref{item:apath4}.

\smallskip	
6. We now prove Lemma \ref{lemma:alimpath}\ref{item:apath5}.
If $\ul{r}(t)>0$ for some $t\in(0,T]$, then, for each $a\in (0,\ul{r}(t))$,
Theorem \ref{thm:WSCEx}\ref{item:WSCEx3} and \eqref{6.3a} imply that
	\begin{align*}
		\Bignorm{\int_{a}^{\ul{r}(t)} \rho_a(r,t)\,r^m \dif r}
		&\le \Bignorm{\int_{a}^{\tilde{r}_a(x,t)} \rho_a(r,t)\,r^m \dif r}
		 + \Bignorm{\int_{\tilde{r}_a(x,t)}^{\tilde{r}(x,t)} \rho_a(r,t) r^m \dif r} + \Bignorm{\int_{\tilde{r}(x,t)}^{\tilde{r}(t)}\rho_a(r,t)\,r^m \dif r}\\
		&\le x + \omega_1\big( [\tilde{r}(x,t),\tilde{r}_a(x,t)]; C(T) \big)
		+ \omega_1\big( [\ul{r}(t),\tilde{r}(x,t)]; C(T) \big),
	\end{align*}
where $\omega_1(E;C) \vcentcolon= f_1( \int_{E}r^m \dif r ; C)$  with
$f_1(y;C) \vcentcolon= y G_+^{-1}(\frac{C}{y})$.
By Proposition \ref{prop:omega}, for each fixed $C>0$, $y\mapsto f_1(y;C)$ is a positive, monotone increasing,
continuous function with $\lim_{y\searrow 0} f_1(y;C) = 0$.
Hence, given $\ve>0$, there exists $\delta^{\prime}>0$ such that
$$
f_1( \frac{\snorm{\tilde{r}^{\,n}(x,t)-\ul{r}^{\,n}(t)}}{n};C)\le \frac{\ve}{2}   \qquad\mbox{for $x\in(0,\delta^{\prime})$}.
$$

Let $\delta\vcentcolon= \min\{ \delta^{\prime}, \frac{\ve}{2}\}$. We have
\begin{equation*}
\int_{a}^{\ul{r}(t)} \rho_a(r,t)\,r^m \dif r
\le \ve + f_1( \dfrac{\tilde{r}_a^{\,n}(x,t)-\tilde{r}^{\,n}(x,t)}{n}; C(T))
\qquad \mbox{if $x\in (0,\delta)$.}
\end{equation*}
Now, fix $x\in(0,\delta)$. It follows from \eqref{temp:apath3} and $\lim_{y\searrow 0} f_1(y;C) = 0$ that
\begin{equation*}
\limsup\limits_{a_j\searrow 0} \int_{a_j}^{\ul{r}(t)} \rho_{a_j}(r,t)\,r^m \dif r
= \ve + \lim\limits_{a_j\searrow 0}f_1(\dfrac{(\tilde{r}_{a_j}(x,t))^n-(\tilde{r}(x,t))^n}{n}; C(T)) = \ve.
\end{equation*}
Since this is true for arbitrarily small $\ve>0$, taking $\ve\searrow 0$ yields
\begin{equation}\label{temp:alimpath1}
\limsup\limits_{j\to\infty} \int_{a_j}^{\ul{r}(t)} \rho_{a_j}(r,t)\,r^m \dif r =0.
\end{equation}
Since $\rho_{a}(r,t)\ge C^{-1}(a)$ for {\it a.e.} $(r,t)\in[a,\infty)\times[0,T]$ from
Theorem \ref{thm:WSCEx}\ref{item:WSCEx2},
\begin{align*}
\int_{a}^{\ul{r}(t)} \rho_a(r,t)\,r^m \dif r
\ge \frac{1}{C(a)}\int_{a}^{\ul{r}(t)} r^m \dif r >0 \qquad
\text{for all $a\in(0,\ul{r}(t))$},
\end{align*}
which, together with \eqref{temp:alimpath1},
implies the first statement of Lemma \ref{lemma:alimpath}\ref{item:apath5}.
Since $\rho_a(r,t)>0$ for each $a\in(0,1)$, by Fatou's lemma,
\begin{align*}
\int_{0}^{\infty} \liminf\limits_{a_j\searrow 0}\rho_{a_j}(r,t) \mathbbm{1}_{(a_j,\ul{r}(t))}(r) r^m \dif r \le \liminf\limits_{a_j\searrow 0} \int_{a_j}^{\ul{r}(t)} \rho_{a_j}(r,t)\,r^m \dif r =0.
\end{align*}
Thus, $\liminf_{j\to\infty}\rho_{a_j}(r,t) \mathbbm{1}_{(a_j,\ul{r}(t))}(r) = 0$ for {\it a.e.} $r\in(0,\ul{r}(t))$.
Now fix such $r\in(0,\ul{r}(t))$.
It follows that $\rho_{a_j}(r,t) = \rho_{a_j}(r,t) \mathbbm{1}_{(a_j,\ul{r}(t))}(r)$ for all $a_j\in(0,r)$.
Taking the limit inferior as $j\to\infty$ on both sides, we obtain the second statement
of Lemma \ref{lemma:alimpath}\ref{item:apath5}.
\end{proof}

In what follows, we denote the fluid region $\mathrm{F}$
as the set:
\begin{align*}
	\mathrm{F}\vcentcolon= \{ (r,t)\in(0,\infty)\times[0,T]\,\vcentcolon\, r> \ul{r}(t) \}.
\end{align*}

\subsection{Compactness results in the exterior domain} \label{subsec:alim}
In this subsection, we show the convergence of the approximate solutions in the exterior domain
and obtain several properties of the limit solutions.

\begin{lemma} \label{lemma:alimue}
There exist both a subsequence $a_j\searrow 0$ as $j\to\infty$
and H\"older continuous functions
$(u,e) \vcentcolon \mathrm{F}\backslash\{t=0\}\to \mathbb{R}\times [0,\infty)$
such that, for any compact subset $K\Subset \mathrm{F}\backslash\{t=0\}$,
\begin{equation*}
\lim\limits_{j\to \infty} \sup\limits_{(r,t)\in K}
\snorm{ ( u_{a_j} - u , e_{a_j} - e )(r,t) } = 0.
\end{equation*}
Moreover, for each $\ve>0$, there exists $C(\ve)>0$ such that
\begin{enumerate}[label=(\roman*),ref=(\roman*),font={\normalfont\rmfamily}]
\item If $t\in[0,T]$ and $r\in[\tilde{r}(\ve,t),\infty)$,
\begin{equation}\label{temp:alimue1}
\sigma^{\frac{1}{4}}(t) \snorm{u(r,t)}
+ \sigma^{\frac{1}{2}}(t) e(r,t)\le C(\ve);
\end{equation}

\item If $t\in(0,T]$ and $r_1,\, r_2 \ge \tilde{r}(\ve,t)$,
\begin{equation}\label{temp:alimue2}
\sigma^{\frac{1}{2}}(t)\snorm{u(r_1,t)-u(r_2,t)} + \sigma(t)\snorm{e(r_1,t)-e(r_2,t)}
\le C(\ve) \snorm{r_1-r_2}^{\frac{1}{2}};
\end{equation}

\item If $0<t_1<t_2\le T$ and $r\ge \sup_{t_1\le t\le t_2}\tilde{r}(\ve,t)$,
\begin{equation}\label{temp:alimue3}
\sigma^{\frac{1}{2}}(t_1)\snorm{u(r,t_1)-u(r,t_2)} + \sigma(t_1)\snorm{e(r,t_1)-e(r,t_2)} \le C(\ve) \snorm{t_2-t_1}^{\frac{1}{4}}.
\end{equation}

\item $(u,e)(r,t)$ satisfy
\begin{align}
&\sup_{ 0\le t \le  T} \int_{\tilde{r}(\ve,t)}^{\infty}
\big(\snorm{(u,e-1)}^2 + \snorm{(\sqrt{\sigma}\partial_r u, \sigma\partial_r e)}^2 \big)(r,t) r^m \dif r\nonumber\\
&\,  +\int_{0}^{T}\int_{\tilde{r}(\ve,t)}^{\infty}\big(\snorm{( \partial_r u, \partial_r e, u \partial_r u )}^2
+\snorm{(\sqrt{\sigma} \partial_t u, \sigma \partial_t e)}^2\big) r^m \dif r \dif t\le C(\ve).\label{6.13b}
\end{align}
\end{enumerate}
\end{lemma}

\begin{proof}
We present the proof only for $u(r,t)$, since the case for $e(r,t)$ follows in the same way.

First, fix $\ve\in(0,1)$ and define the domain:
\begin{equation*}
\mathrm{F}_\ve \vcentcolon
=\big\{(r,t)\in (0,\infty)\times[0,T]\, \vcentcolon\, t\in[\ve,T], \,\tilde{r}(\ve,t) \le r \le \ve^{-1}\big\}.
\end{equation*}
By Lemma \ref{lemma:alimpath}\ref{item:apath4}, there exists $\ve_0>0$ such that
$\text{Int}(\mathrm{F}_\ve)\neq \emptyset$ for each $\ve\in(0,\ve_0)$.
Since $x\mapsto\tilde{r}(x,t)$ is strictly increasing, then $\mathrm{F}_{\ve_1} \subset \mathrm{F}_{\ve_2}$ if $\ve_1 > \ve_2$.

\smallskip
We claim:
\begin{equation} \label{fltemp}
\mathrm{F}\backslash \{t=0\} = \bigcup\limits_{0<\ve<1} \mathrm{F}_\ve.
\end{equation}
This can be seen as follows:
If $(r,t)\in \mathrm{F}_\ve$ for some $\ve\in(0,1)$, then $r> \ul{r}(t)$ since $ \tilde{r}(x,t) \searrow \ul{r}(t)$
as $x\searrow 0$ for each $t\in[0,T]$, which implies that $(r,t)\in \mathrm{F}$.
On the other hand, if $(r,t)\in \mathrm{F}\backslash\{t=0\}$, then $\ul{r}(t)< r$ and
there exists $\ve_1>0$ such that $t>\ve_1$.
Since $\lim_{x\searrow 0}\tilde{r}(x,t) = \ul{r}(t)$ for all $t\in[0,T]$, there exists $\ve_2\in(0,1)$
such that $\ul{r}(t) \le \tilde{r}(\ve_2,t) \le r$.
Thus, $(r,t)\in \mathrm{F}_\ve$ for $\ve\vcentcolon=\min\{\ve_1,\ve_2\}$.
Then \eqref{fltemp} follows.
	
\smallskip
Next, since $x\mapsto \tilde{r}(x,t)$ is strictly increasing for each $t\in[0,T]$,
and $(x,t)\mapsto \tilde{r}(x,t)$ is continuous, then
$\D_{\ve}\vcentcolon=\inf_{t\in[0,T]}\{ \tilde{r}(\ve,t) - \tilde{r}(\frac{\ve}{2},t) \}>0$ by the compactness of $[0,T]$.
By Lemma \ref{lemma:alimpath}, there exists $N_{\ve}\in\mathbb{N}$ such that
$\sup_{t\in[0,T]} \snorm{\tilde{r}_{a_j}(\frac{\ve}{2},t)-\tilde{r}(\frac{\ve}{2},t)}\le \frac{\D_{\ve}}{2}$ for all $j \ge N_\ve$.
Hence, $\tilde{r}_{a_j}(\frac{\ve}{2},t) < \tilde{r}(\ve,t)$
for all $j\ge N_\ve$ and $t\in[0,T]$, which implies that,
if $(r,t)\in\mathrm{F}_\ve$, then $\tilde{r}_{a_j}(\frac{\ve}{2},t)< r$ for all $j\ge N_\ve$.
Combining this with Theorem \ref{thm:WSCEx}\ref{item:WSCEx4}, we have
	\begin{equation}\label{fltemp2}
		\begin{dcases*}
			\sup\limits_{j\ge N_\ve}\snorm{u_{a_j}(r_1,t)-u_{a_j}(r_2,t)}
			\le C(\frac{\ve}{2}) \sigma^{-\frac{1}{2}}(t) \snorm{r_1-r_2}^{\frac{1}{2}} & for all  $(r_1,t), (r_2,t)\in \mathrm{F}_\ve$,\\
			\sup\limits_{j\ge N_\ve}\snorm{u_{a_j}(r,t_1)-u_{a_j}(r,t_2)} \le C(\frac{\ve}{2}) \sigma^{-\frac{1}{2}}(t_1) \snorm{t_2-t_1}^{\frac{1}{4}}
			& for all  $(r,t_1), (r,t_2)\in \mathrm{F}_\ve$.
		\end{dcases*}
	\end{equation}
In addition, it also follows from Theorem \ref{thm:WSCEx}\ref{item:WSCEx4} that
\begin{equation}\label{fltemp3}
\sup_{j\ge N_\ve}\snorm{u_{a_j}(r,t)} \le C(\ve) \sigma^{-\frac{1}{4}}(t)
\le \ve^{-\frac{1}{4}}C(\ve)
\qquad \text{ for all $(r,t)\in \mathrm{F}_{\ve}$}.
\end{equation}
In light of \eqref{fltemp}--\eqref{fltemp3}, we can apply Proposition \ref{prop:aaExt} to obtain a further
subsequence (still denoted) $a_j$ and a continuous function
$u(r,t)\vcentcolon \mathrm{F}\backslash\{t=0\}\to \mathbb{R}$
such that
\begin{equation}\label{fltemp4}
\lim\limits_{j\to\infty} \sup\limits_{(r,t)\in K} \snorm{u_{a_j}(r,t)-u(r,t)}=0 \qquad
\text{for all compact subset $K \Subset \mathrm{F}\backslash\{t=0\}$},
\end{equation}
which together with Theorem \ref{thm:WSCEx}\ref{item:WSCEx4} implies \eqref{temp:alimue1}--\eqref{temp:alimue3}.

Similarly, we can establish similar estimates for $e(r,t)$ in the same way.

The proof for
\eqref{6.13b} follows the same as that of Lemma \ref{lemma:weak}.
\end{proof}

\begin{lemma}\label{lemma:alimrho}
For each $\ve>0$, define
$$
\rho_{a}^{(\ve)}(r,t) \vcentcolon= \rho_a(r,t) \chi_\ve^a(r,t), \quad
\rho^{(\ve)}(r,t) \vcentcolon= \rho(r,t) \chi_\ve(r,t),
$$
where $\chi_\ve^a$ and $\chi_\ve$ are the indicator functions defined as
\begin{equation}\label{6.6indicator}
		\chi_\ve(r,t) \vcentcolon = \begin{dcases*}
			1 & if $r\in [\tilde{r}(\ve,t),\infty),$\\
			0 & otherwise,
		\end{dcases*} \qquad \  \chi_\ve^a(r,t)
		\vcentcolon = \begin{dcases*}
			1 & if $r\in [\tilde{r}_a(\ve,t),\infty)$,\\
			0 & otherwise.
		\end{dcases*}
\end{equation}
Then there exists $C(\ve)>0$ such that the following statements hold{\rm :}
\begin{enumerate}[label=(\roman*), ref=(\roman*),font=\normalfont\rmfamily]
\item\label{item:alimrho1}
There exist both a nonnegative function $\rho(r,t) \in L^2_{\text{\rm loc}}(\mathrm{F},r^m\dif r \dif t)$
and a subsequence $a_j\searrow 0$ as $j\to\infty$ such that, for each $\varepsilon>0$,
\begin{equation*}
\begin{cases}
\rho_{a_j}^{(\ve)} - 1 \overset{\ast}{\rightharpoonup}  \rho^{(\ve)} - 1 \qquad
 \text{in $L^{\infty}\big(0,T; L^{2}([0,\infty), r^m\dif r)\big)$ \,\, as $j\to\infty$},\\
\sup_{t\in[0,T]}\int_{\tilde{r}(\ve,t)}^{\infty} \snorm{\rho^{(\ve)}-1}^2(r,t)\,r^m \dif r \le C(\ve).
\end{cases}
\end{equation*}

\smallskip
\item\label{item:alimrho2}
For all $L\in\mathbb{N}$ with $X_L\vcentcolon= \tilde{H}^{-1}( [0,L], r^m \dif r)$,
\begin{equation*}
\begin{cases}
 \rho^{(\ve)} \in C^{0}([0,T];X_L), \qquad \lim\limits_{j\to\infty}\sup\limits_{t\in[0,T]}\sbnorm{\rho_{a_j}^{(\ve)}(\cdot,t)-\rho^{(\ve)}(\cdot,t)}_{X_L} = 0,\\
\sbnorm{\rho^{(\ve)}(\cdot,t_1)-\rho^{(\ve)}(\cdot,t_2)}_{X_L} \le C(\ve) \snorm{t_1-t_2} \qquad
   \text{for all $t_1,t_2 \in[0,T]$}.
\end{cases}
\end{equation*}
		
\smallskip
\item\label{item:alimrho3}
For all $x>0$,
$x=\int_{\ul{r}(t)}^{\tilde{r}(x,t)} \rho(r,t)\,r^m \dif r$
for $t\in[0,T]$ almost everywhere,
where $\tilde{r}(x,t)$ and $\ul{r}(t)$ are the functions obtained in {\rm Lemma \ref{lemma:alimpath}}.

\smallskip
\item\label{item:alimrho4}
$C^{-1}(\ve) \le \rho(r,t) \le C(\ve)$ for almost every $t\in[0,T]$ and $r\in[\tilde{r}(\ve,t),\infty)$.
\end{enumerate}
\end{lemma}

\begin{proof}
 We divide the proof into four steps, showing Lemma \ref{lemma:alimrho}\ref{item:alimrho1}--\ref{item:alimrho4}
 in chronological order.

 \smallskip
1. Proof of Lemma \ref{lemma:alimrho}\ref{item:alimrho1}.
From Theorem \ref{thm:WSCEx}\ref{item:WSCEx4}, we obtain that, for all $j\in\mathbb{N}$,
\begin{align*}
\sup_{t\in[0,T]} \int_{0}^{\infty} \snorm{\rho_{a_j}^{(\ve)}-1}^2\, r^m \dif r
\le \dfrac{\tilde{r}_{a_j}^{\,n}(\ve,t)}{n}
+ \sup_{t\in[0,T]} \int_{\tilde{r}_{a_j}(\ve,t)}^{\infty} \snorm{\rho_{a_j}-1}^2\,r^m \dif r  \le C(\ve).
\end{align*}
Then there exists $\tilde{w}_\ve-1 \in L^{\infty}(0,T; L^2([0,\infty),r^m \dif r) )$ such that
\begin{equation}\label{Ftemp0}
\rho_{a_j}^{(\ve)}-1 \overset{\ast}{\rightharpoonup} \tilde{w}_\ve-1 \qquad
\text{weakly star in $L^{\infty}\big(0,T; L^{2}([0,\infty), r^m\dif r)\big)\,\,$ as $j\to\infty$}.
\end{equation}
For each $l \in \mathbb{N}$, denote $\ve_l \vcentcolon= l^{-1}$ and $\tilde{w}^{(l)} \vcentcolon= \tilde{w}_{\ve_l}$.
Fix $l,\,q\in\mathbb{N}$ such that $l< q$.
Then $\tilde{r}_a(\ve_q,t)< \tilde{r}_a(\ve_l,t)$ for each $(a,t)\in(0,1)\times[0,T]$.
Thus, by construction, it follows that, for all $j\in\mathbb{N}$,
$\chi_{\ve_l}^{a_j}(r,t)=1=\chi_{\ve_q}^{a_j}(r,t)$ if $r\in [\tilde{r}_{a_j}(\ve_l,t),\infty)$ so that
\begin{equation}\label{Ftemp1}
\rho_{a_j}^{(\ve_q)}(r,t) = \rho_{a_j}^{(\ve_l)}(r,t) \ \ \text{for {\it a.e.} $t\in[0,T]$}, \qquad\,\,\, r\in[\tilde{r}_{a_j}(\ve_l,t),\infty) \ \ \text{for each $j\in\mathbb{N}$}.
\end{equation}
Let $\phi \in C^{\infty}_{\rm c}(\mathrm{F}^{(l)})$ for
$\mathrm{F}^{(l)}\vcentcolon= \{ (r,t) \vcentcolon t\in[0,T], r\in [\tilde{r}(\ve_l,t),\infty) \}$.
Then it follows that $\phi \chi_{\ve_l} \in L^{1}(0,T;L^2([0,\infty),r^m\dif r))$ and
\begin{align*}
& \int_{0}^{T} \int_{\tilde{r}(\ve_l,t)}^{\infty}  \{\tilde{w}^{(q)}-\tilde{w}^{(l)}\} \phi\, r^m \dif r \dif t\\
&=\int_{0}^{T} \int_{0}^{\infty} \{\tilde{w}^{(q)}-\rho_{a_j}^{(\ve_q)}\} \chi_{\ve_l} \phi\, r^m \dif r \dif t + \int_{0}^{T}\int_{\tilde{r}(\ve_l,t)}^{\tilde{r}_{a_j}(\ve_l,t)} \{\rho_{a_j}^{(\ve_q)}-\rho_{a_j}^{(\ve_l)}\}
   \phi\, r^m \dif r \dif t \\
&\quad+ \int_{0}^{T}\int_{\tilde{r}_{a_j}(\ve_l,t)}^{\infty} \{\rho_{a_j}^{(\ve_q)}-\rho_{a_j}^{(\ve_l)}\}
\phi\,r^m \dif r\dif t
 +\int_{0}^{T} \int_{0}^{\infty}  \{\rho_{a_j}^{(\ve_l)}-\tilde{w}^{(l)}\} \chi_{\ve_l} \phi\, r^m \dif r \dif t \\
&= \sum_{i=1}^4 I_i^{(j)} .
\end{align*}
The limits of $I_1^{(j)}$ and $I_4^{(j)}$ are evaluated by the weak-star convergence \eqref{Ftemp0}.
As $j\to\infty$,
\begin{align*}
I_1^{(j)} + I_4^{(j)}\vcentcolon
= \int_{0}^{T} \int_{0}^{\infty}  \{\tilde{w}^{(q)}-\rho_{a_j}^{(\ve_q)}\} \chi_{\ve_l} \phi\, r^m \dif r \dif t
+ \int_{0}^{T} \int_{0}^{\infty}  \{\rho_{a_j}^{(\ve_l)}-\tilde{w}^{(l)}\}
\chi_{\ve_l} \phi\, r^m \dif r \dif t \to 0.
\end{align*}
Using \eqref{Ftemp1}, we also see that
\begin{align*}
I_3^{(j)} \vcentcolon= \int_{0}^{T}\int_{\tilde{r}_{a_j}(\ve_l,t)}^{\infty} \{\rho_{a_j}^{(\ve_q)}-\rho_{a_j}^{(\ve_l)}\}
\phi\, r^m \dif r \dif t = 0.
\end{align*}
For $I_2^{(j)}$, observe that
\begin{equation*}
\snorm{\chi_{\ve_l}^{a_j}-\chi_{\ve_l}}(r,t) = \begin{dcases*}
			1 & if $r\in [\tilde{r}_{a_j}(\ve_l,t),\tilde{r}(\ve_l,t)]$ or $r\in[\tilde{r}(\ve_l,t),\tilde{r}_{a_j}(\ve_l,t)]$,\\
			0 & otherwise.
\end{dcases*}
\end{equation*}
By the uniform convergence from Lemma \ref{lemma:alimpath},
it follows that $\snorm{\chi_{\ve_l}^{a_j}-\chi_{\ve_l}}(r,t)\to 0$ as $j\to\infty$ for {\it a.e.} $(r,t)\in[0,\infty)\times[0,T]$.
Using the estimates from Theorem \ref{thm:WSCEx}\ref{item:WSCEx4} and the dominated convergence theorem yield
that, as $j\to\infty$,
\begin{align*}
\snorm{I_2^{(j)}}
&\le \int_{0}^{T}\int_{0}^{\infty} \snorm{\chi_{\ve_l}^{a_j}-\chi_{\ve_l}}\snorm{\rho_{a_j}^{(\ve_q)}-\rho_{a_j}^{(\ve_l)}} \snorm{\phi}\,r^m \dif r \dif t \\
&\le \int_{0}^{T} \Big( \int_{0}^{\infty} \snorm{\chi_{\ve_l}^{a_j}-\chi_{\ve_l}}^2 \snorm{\phi}^2\, r^m \dif r \Big)^{\frac{1}{2}}
\big(\sbnorm{\rho_{a_j}^{(\ve_q)}(\cdot,t) -1}_{L^2} + \sbnorm{\rho_{a_j}^{(\ve_l)}(\cdot,t) -1}_{L^2}\big)\,\dif t\\
&\le \big( C(\ve_l) + C(\ve_q)\big)\int_{0}^{T} \Big( \int_{0}^{\infty} \snorm{\chi_{\ve_l}^{a_j}-\chi_{\ve_l}}^2
\snorm{\phi}^2\, r^m \dif r \Big)^{\frac{1}{2}} \dif t \to 0 .
\end{align*}
Combining the convergence for $I_i^{(j)}$ as $j\to\infty$ for $i=1,2,3,4$, it follows that
\begin{align*}
\int_{0}^{T} \int_{\tilde{r}(\ve_l,t)}^{\infty}  \{\tilde{w}^{(q)}-\tilde{w}^{(l)}\} \phi\, r^m \dif r \dif t = 0
\qquad \text{for all $\phi \in C^{\infty}_{\rm c}(\mathrm{F}^{(l)})$.}
\end{align*}
It follows that, if $l\le q \in \mathbb{N}$, then
\begin{equation}\label{Ftemp6}
\tilde{w}^{(q)}(r,t) = \tilde{w}^{(l)}(r,t)  \qquad
\text{for almost every $t\in[0,T]$ and $r\in [\tilde{r}(\ve_l,t),\infty)$.}
\end{equation}
In light of above, $\tilde{w}^{(q)}(r,t)$ is truncated in domain $\mathrm{F}$ for each $q\in\mathbb{N}$ as
\begin{equation*}
w^{(q)}(r,t) \vcentcolon=
\begin{dcases*}
\tilde{w}^{(q)}(r,t) & if $t\in[0,T]$ and $r\in[\tilde{r}(\ve_q,t),\infty)$,\\
0 & otherwise.
\end{dcases*}
\end{equation*}
Using this, for each $N\in\mathbb{N}$, define
\begin{equation*}
h_N(r,t) \vcentcolon= w^{(1)}(r,t) + \sum_{i=1}^{N} \{ w^{(i+1)}(r,t) - w^{(i)}(r,t) \} \qquad
\text{for $(r,t)\in \mathrm{F}$.}
\end{equation*}
Fix $(r,t)\in \mathrm{F}$.
Lemma \ref{lemma:alimpath}\ref{item:apath3}--\ref{item:apath4} imply that
$\lim_{x\searrow 0}\tilde{r}(x,t)=\ul{r}(t)$ and $\lim_{x\to \infty}\tilde{r}(x,t)=\infty$.
Since $\ul{r}(t)< r$, there exists $q\in\mathbb{N}$ such that $\tilde{r}(\ve_q,t)\le r < \tilde{r}(\ve_{q-1},t)$.
Then \eqref{Ftemp6} implies
\begin{equation*}
h_{N}(r,t) = w^{(1)}(r,t) + \sum_{i=1}^N \{ w^{(i+1)}(r,t) - w^{(i)}(r,t) \} = w^{(q)}(r,t) = \tilde{w}^{(q)}(r,t)
\quad \text{ if $N\ge q+1$.}
\end{equation*}
Thus, for {\it a.e.} $(r,t)\in \mathrm{F}$,
$\lim_{N\to\infty} h_N(r,t) = \tilde{w}^{(q)}(r,t)$ exists for some $q\in\mathbb{N}$,
and the following function is well-defined almost everywhere in $(0,\infty)\times[0,T]$:
\begin{equation*}
\rho(r,t) \vcentcolon=\begin{dcases*}
\lim\limits_{N\to\infty} h_N(r,t) & for {\it a.e.} $(r,t)\in \mathrm{F}$,\\
0 & for $(r,t)\in (0,\infty)\times[0,T] \backslash \mathrm{F}$.
\end{dcases*}
\end{equation*}
	
Now let $\ve>0$. Then there exists some integer $q\in\mathbb{N}$ such that $\ve_q \le \ve$ and
\begin{equation*}
\rho^{(\ve)}(r,t) = \chi_\ve(r,t) \rho(r,t)  = \chi_\ve(r,t) \tilde{w}^{(q)}(r,t) \quad \
\text{for {\it a.e.} $(r,t)\in (0,\infty)\times[0,T]$},
\end{equation*}
from the above construction.
	
Let $\phi \in L^{1}(0,T;L^2([0,\infty);r^m \dif r))$. Using $\tilde{r}_{a_j}(\ve_q,t)< \tilde{r}_{a_j}(\ve,t)$ for each $j\in\mathbb{N}$, one has
\begin{equation}\label{Ftemp2}
\chi_\ve^{(a_j)}= \chi_{\ve_q}^{(a_j)}\chi_{\ve}^{(a_j)} \qquad
\text{for all $(r,t)\in[0,\infty)\times[0,T]$ and $j\in\mathbb{N}$.}
\end{equation}
Using this, it then follows that
\begin{align*}
\int_{0}^{T}\int_{0}^{\infty} \{ \rho^{(\ve)} - \rho_{a_j}^{(\ve)} \} \phi\, r^m \dif r \dif t
&= \int_{0}^{T} \int_{0}^{\infty} \chi_\ve \phi \{ \tilde{w}^{(q)} - \rho_{a_j}^{(\ve_q)} \}\, r^m \dif r \dif t
+\int_{0}^{T} \int_{0}^{\infty}  \{ \chi_{\ve} - \chi_{\ve}^{a_j} \} \rho_{a_j}^{(\ve_q)} \phi\, r^m \dif r \dif t\\
&= \vcentcolon J_1^{(j)} + J_2^{(j)}.
\end{align*}
By the weak-star convergence for each $q\in\mathbb{N}$ in \eqref{Ftemp0},
$J_1^{(j)}\to 0$ as $j\to\infty$.
For $J_2^{(j)}$, one uses the pointwise bound of $\rho_a$ from Theorem \ref{thm:WSCEx}\ref{item:WSCEx4} to obtain
	\begin{equation*}
		\snorm{(\chi_\ve - \chi_\ve^{a_j})\rho_{a_j}^{(\ve_q)} \varphi } \le C(\ve_q) \snorm{\varphi} \in L^{1}\big(0,T;L^2( [0,\infty), r^m \dif r)\big).
	\end{equation*}
Moreover, due to the uniform convergence from Lemma \ref{lemma:alimpath},
$(\chi_\ve - \chi_\ve^{(a_j)})(r,t)\to 0$ as $j\to \infty$ for {\it a.e.} $(r,t)\in [0,\infty)\times[0,T]$.
Hence, by the dominated convergence theorem,
\begin{align*}
J_2^{(j)} \vcentcolon
= \int_{0}^{T} \int_{0}^{\infty}  ( \chi_{\ve} - \chi_{\ve}^{a_j} ) \rho_{a_j}^{(\ve_q)} \varphi\, r^m \dif r \dif t \to 0
\qquad \text{as $j\to\infty$.}
\end{align*}
Combining the limits of $J_1^{(j)}$ and $J_2^{(j)}$ as $j\to\infty$, it follows that there exists
$\rho \in L^2_{\text{loc}}(\mathrm{F})$ such that,  for each $\ve>0$,
\begin{equation}\label{Ftemp7}
\rho^{(\ve)}_{a_j} \overset{\ast}{\rightharpoonup} \rho^{(\ve)} \qquad
\text{in $L^{\infty}\big(0,T; L^2([0,\infty),r^m \dif r)\big)\,\,$  as $j\to \infty$}.
\end{equation}
This completes the proof
of Lemma \ref{lemma:alimrho}\ref{item:alimrho1}.

\smallskip	
2. Proof of Lemma \ref{lemma:alimrho}\ref{item:alimrho2}.
Define $r_\ve\vcentcolon= \frac{1}{2}\inf_{t\in[0,T]}\tilde{r}(\ve,t)$.
By Lemma \ref{lemma:alimpath}, there exists $N_\ve\in\mathbb{N}$ such that
$r_\ve < \tilde{r}_{a_j}(\ve,t)$ for all $j\ge N_\ve$ and $t\in[0,T]$.
By Theorem \ref{thm:WSCEx}\ref{item:WSCEx5}, we see that, for each $\ve>0$ and $t_1, t_2 \in [0,T]$,
\begin{align*}
\sup\limits_{j \ge N_{\ve}}\sbnorm{\rho_{a_j}^{(\ve)}(\cdot, t_1)- \rho_{a_j}^{(\ve)}(\cdot,t_2)}_{\tilde{H}^{-1}([r_\ve,L], r^m \dif r)} \le C(\ve)\snorm{t_1-t_2} \qquad \text{for all $L\in\mathbb{N}$.}
\end{align*}
Moreover, it follows that, for all $j\ge N_\ve$ and $L\in\mathbb{N}$,
\begin{align*}
\rho_{a_j}^{(\ve)} \in C^0\big([0,T]; \tilde{H}^{-1}([r_\ve,L],r^m \dif r) \big), \quad \sup_{j \ge N_\ve } \sbnorm{\rho_{a_j}^{(\ve)}-1}_{L^{\infty}(0,T; L^{2}( [r_\ve,\infty), r^m \dif r))} \le C(\ve).
\end{align*}
Therefore, by Proposition \ref{prop:scHm1},
there exists $\tilde{h}_\ve$ such that, for all $L\in\mathbb{N}$,
\begin{equation}\label{Ftemp3}
\begin{cases}
\tilde{h}_\ve-1 \in L^{\infty}\big(0,T;L^2([r_\ve,\infty),r^m\dif r)\big), \quad \tilde{h}_\ve \in C^0\big([0,T];\tilde{H}^{-1}([r_\ve,L], r^m \dif r)\big),\\
\lim\limits_{j\to \infty} \sup_{t\in[0,T]}\sbnorm{\rho_{a_j}^{(\ve)}(\cdot,t)-\tilde{h}_\ve(\cdot,t)}_{H^{-1}([r_\ve,L], r^m \dif r)} = 0,\\
\sbnorm{\tilde{h}_\ve(\cdot,t_1)-\tilde{h}_\ve(\cdot,t_2)}_{H^{-1}([r_\ve,L],r^m\dif r)} \le C(\ve) \snorm{t_1-t_2}
\quad \text{ for each $t_1, t_2 \in[0,T]$.}
\end{cases}
\end{equation}

\smallskip
We claim that $\rho^{(\ve)} = \tilde{h}_\ve$ for {\it a.e.} $(r,t)\in[r_\ve,\infty)\times[0,T]$.
Let $\phi \in C_c^{\infty}( [r_\ve,\infty) \times [0,T])$. Then there exists $L\in\mathbb{N}$ such that $\supp(\phi)\subseteq [r_\ve,L]\times[0,T]$. Using \eqref{Ftemp7}--\eqref{Ftemp3}, we obtain that, as $j\to\infty$,
\begin{align*}
&\int_{0}^{T}\int_{r_\ve}^{\infty} (\rho^{(\ve)} - \tilde{h}_\ve) \phi\, r^m \dif r \dif t
= \int_{0}^{T}\int_{r_\ve}^{\infty} (\rho^{(\ve)} - \rho^{(\ve)}_{a_j}) \phi\, r^m \dif r \dif t
+ \int_{0}^{T}\int_{r_\ve}^{L} (\rho_{a_j}^{(\ve)} - \tilde{h}_\ve) \phi\, r^m \dif r \dif t  \\
&\le  \int_{0}^{T}\int_{r_\ve}^{\infty} \{\rho^{(\ve)} - \rho^{(\ve)}_{a_j}\} \phi\, r^m \dif r \dif t + \int_{0}^{T} \sbnorm{\rho_{a_j}^{(\ve)}(\cdot,t)-\tilde{h}_{\ve}(\cdot,t)}_{\tilde{H}^{-1}}\sbnorm{\phi}_{H^1} \dif t\to 0.
\end{align*}
This implies that
$\rho^{(\ve)}=\tilde{h}_\ve$ for {\it a.e.} $(r,t)\in[r_\ve,\infty)\times[0,T]$.
Lemma \ref{lemma:alimrho}\ref{item:alimrho2} then follows since $\supp(\rho^{(\ve)}(\cdot,t))=[\tilde{r}(\ve,t),\infty)$,
and $r_\ve < \tilde{r}(\ve,t)$ for all $t\in[0,T]$.
	
\medskip
3. Proof of Lemma \ref{lemma:alimrho}\ref{item:alimrho3}.
Let $\xi\in C^{\infty}_{\rm c}(0,T)$ and $0<y< x$. By \eqref{Ftemp7}, we have
\begin{align*}
&\int_{0}^{T} \xi(t) \Big(\int_{\tilde{r}_{a_j}(y,t)}^{\tilde{r}(x,t)} \rho_{a_j}(r,t)\,r^m \dif r\Big) \dif t
 -\int_{0}^{T} \xi(t) \Big(\int_{\tilde{r}(y,t)}^{\tilde{r}(x,t)} \rho(r,t)\, r^m \dif r\Big)\dif t\\
&= \int_{0}^{T}  \int_{0}^{\tilde{r}(x,t)} \xi (\chi_y^{a_j}\rho_{a_j} - \chi_{y}\rho)\, r^m \dif r \dif t
= \int_{0}^{T}  \int_{0}^{\infty} (1-\chi_{x})\xi \{\rho_{a_j}^{(y)} - \rho^{(y)}\}\,r^m \dif r \dif t \to 0
\end{align*}
as $j\to\infty$, where we have used the fact that $(1-\chi_x) \xi \in L^{1}(0,T; L^2([0,\infty),r^m \dif r))$.
Then
\begin{equation}\label{Ftemp4}
\lim\limits_{j\to\infty}\int_{0}^{T}  \int_{\tilde{r}_{a_j}(y,t)}^{\tilde{r}(x,t)} \xi(t) \rho_{a_j}(r,t)\,r^m \dif r \dif t
= \int_{0}^{T} \xi(t) \Big(\int_{\tilde{r}(y,t)}^{\tilde{r}(x,t)}\rho(r,t)\,r^m\dif r\Big)\dif t.
\end{equation}
Moreover, using \eqref{6.3a}, $\lim_{y\searrow 0}\tilde{r}(y,t)=\ul{r}(t)$, and the dominated convergence theorem,
\begin{equation*}
\lim\limits_{y\searrow 0}\sup\limits_{j\in\mathbb{N}}\Bignorm{\int_{0}^{T}\int_{\ul{r}(t)}^{\tilde{r}(y,t)} \xi \rho_{a_j}
\,r^m \dif r \dif t}
\le \lim\limits_{y\searrow 0} \int_{0}^{T} \omega_{1}( [\ul{r}(t),\tilde{r}(y,t)]; C(T) ) \snorm{\xi(t)}\,\dif t = 0.
\end{equation*}
For a given $\delta>0$, the above limit implies that there exists $z_{\delta}>0$ such that, if $y\in(0,z_{\delta})$, then
\begin{equation}\label{Ftemp8}
\sup_{j\in\mathbb{N}}\Bignorm{\int_{0}^{T}\int_{\ul{r}(t)}^{\tilde{r}(y,t)} \xi(t) \rho_{a_j}(r,t)\, r^m \dif r \dif t}
\le \delta.
\end{equation}
Now fix $y\in(0,z_\delta)$. By Theorem \ref{thm:WSCEx}\ref{item:WSCEx3}, for {\it a.e.} $t\in[0,T]$,
\begin{align*}
x = \int_{a_j}^{\tilde{r}_{a_j}(x,t)} \rho_{a_j}(r,t)r^m \dif r
= \Big(\int_{a_j}^{\ul{r}(t)} + \int_{\ul{r}(t)}^{\tilde{r}(y,t)} + \int_{\tilde{r}(y,t)}^{\tilde{r}_{a_j}(y,t)} + \int_{\tilde{r}_{a_j}(y,t)}^{\tilde{r}(x,t)} +  \int_{\tilde{r}(x,t)}^{\tilde{r}_{a_j}(x,t)} \Big) \rho_{a_j}\,r^m \dif r.
\end{align*}
Multiplying the above equality with $\xi(t)$, then integrating in $t\in[0,T]$ and letting $j\to\infty$,
it follows from \eqref{Ftemp4}--\eqref{Ftemp8} that
\begin{equation}\label{Ftemp9}
-\delta \le \lim\limits_{j\to\infty}\sum_{i=1}^{3} \mathcal{I}_i^{(j)} +\int_{0}^{T} \xi(t) \Big\{ -x + \int_{\tilde{r}(y,t)}^{\tilde{r}(x,t)} \rho(r,t) \,r^m \dif r \Big\}\,\dif t \le \delta.
\end{equation}
Repeating the same proof for Lemma \ref{lemma:alimpath}\ref{item:apath5}, we have
\begin{align*}
\lim\limits_{j\to\infty}\snorm{\mathcal{I}_1^{(j)}} = \lim\limits_{j\to\infty}\Bignorm{\int_{0}^T \xi(t)\big(\int_{a_j}^{\ul{r}(t)} \rho_{a_j}(r,t)\,r^m \dif r\big)\dif t} = 0.
\end{align*}
Moreover, using \eqref{6.3a}, Lemma \ref{lemma:alimpath}, and the dominated convergence theorem,
we see that
\begin{align*}
\lim\limits_{j\to\infty}\snorm{\mathcal{I}_2^{(j)}}+\snorm{\mathcal{I}_3^{(j)}}
&\le\lim\limits_{j\to\infty} \int_{0}^{T} \snorm{\xi(t)} \Big(\int_{\tilde{r}(y,t)}^{\tilde{r}_{a_j}(y,t)}+\int_{\tilde{r}(x,t)}^{\tilde{r}_{a_j}(x,t)}\Big) \rho_{a_j}(r,t)\,r^m \dif r \dif t \\
&\le \lim\limits_{j\to\infty} \int_{0}^{T} \snorm{\xi(t)}
\Big\{ \omega_1\big([\tilde{r}(y,t),\tilde{r}_{a_j}(y,t)];C(T)\big) + \omega_1\big([\tilde{r}(x,t),\tilde{r}_{a_j}(x,t)];C(T)\big) \Big\}\,\dif t  = 0.
\end{align*}
It follows from \eqref{Ftemp9} that, for each $\delta>0$, there exists $z_{\delta}>0$ such that, if $y\in(0,z_{\delta})$, then
\begin{equation*}
-\delta \le  \int_{0}^{T} \xi(t) \Big\{ -x+ \int_{\tilde{r}(y,t)}^{\tilde{r}(x,t)} \rho(r,t) r^m \dif r \Big\}\,\dif t
\le \delta.
\end{equation*}
Let $\{\xi_k\}_{k\in\mathbb{N}} \in C^{\infty}_{\rm c}(0,T)$ be such that $\xi_k\ge 0$,
$\sbnorm{\xi_k}_{L^1([0,T])}=1$, and $\{\xi_k\}_{k\in\mathbb{N}}$ is a sequence of kernels approximating
the Dirac measure.
Then, for all $(k,y)\in\mathbb{N}\times(0,z_{\delta})$,
\begin{align*}
-\delta \le - x + \int_{0}^{T}\xi_k(t)\Big(\int_{\tilde{r}(y,t)}^{\tilde{r}(x,t)} \rho(r,t)\,r^m \dif r\Big) \dif t
\le \delta.
\end{align*}
Since $\rho(r,t)\ge 0$, we apply the monotone convergence theorem for $y\searrow 0$ to obtain
\begin{align*}
-\delta \le - x
+ \int_{0}^{T}\xi_k(t)\Big(\int_{\ul{r}(t)}^{\tilde{r}(x,t)}\rho(r,t)\,r^m \dif r\Big)\,\dif t \le \delta.
\end{align*}
Since this is true for arbitrarily small $\delta>0$, it follows that
\begin{align*}
x=\int_{0}^{T}\xi_k(t)\Big(\int_{\ul{r}(t)}^{\tilde{r}(x,t)} \rho(r,t)\,r^m \dif r\Big)\,\dif t \qquad \
\text{for each $i\in\mathbb{N}$}.
\end{align*}
Finally,  since $\xi_k$ converges to the Dirac measure in the sense of distributions, we take $k\to \infty$ to
conclude that
\begin{align*}
x=\int_{\ul{r}(t)}^{\tilde{r}(x,t)} \rho(r,t)\,r^m \dif r  \qquad \ \text{for $t\in[0,T]$  almost everywhere.}
\end{align*}

\medskip
4. The proof for \ref{item:alimrho4}
is  the same as that of Proposition \ref{prop:veDomain} and Lemma \ref{lemma:rhoWeak}.
\end{proof}

In light of Lemmas \ref{lemma:alimpath}\ref{item:apath5} and \ref{lemma:alimrho}\ref{item:alimrho3},
we extend $\rho\in L^2_{\text{loc}}(\mathrm{F})$ obtained above by setting $\rho(r,t)=0$
for $(r,t)\in \R^n\times[0,T] \backslash \mathrm{F}$.

\begin{lemma}\label{lemma:alimEnt}
Let $(\rho,u,e)(r,t)$ be the limit function obtained in {\rm Lemmas \ref{lemma:alimue}}--{\rm \ref{lemma:alimrho}}.
Then
\begin{align*}
&\esssup_{t\in(0,T)}\int_{0}^{\infty}\big(\rho \snorm{u}^2 +G(\rho)\big)(r,t)\,r^m \dif r \le C(T),\\
&\esssup_{t\in[0,T]}\int_{E} (\rho e)(r,t)\,r^m \dif r \le C(T) + \omega_2(E;C(T)) \quad
  \text{for bounded measurable $E\subset [0,\infty)$,}
\end{align*}
where $\omega_2(E,z)$ for $z>0$ is the set function given in {\rm Lemma \ref{lemma:rhoM}}.
\end{lemma}

\begin{proof}
Let $\ve\in (0,1)$. Denote $\rho^{(\ve)}$ and $\rho^{(\ve)}_{a_j}$ as in Lemma \ref{lemma:alimrho}.
Let $E\subset [0,\infty)$ be a bounded measurable set, and let $\mathbbm{1}_E(r)$ be the indicator function.
Since $x\to \tilde{r}(x,t)$ is strictly increasing and $(x,t)\mapsto\tilde{r}(x,t)$ is a continuous function from Lemma \ref{lemma:alimpath}, combining with the fact that $[0,T]$ is a compact set yields that
$\D_\ve  =\vcentcolon 2\inf_{t\in[0,T]}\{\tilde{r}(\ve,t) - \tilde{r}(\frac{\ve}{2},t)\}>0$.
By Lemma \ref{lemma:alimpath}, there exists $N_{\ve}\in\mathbb{N}$ such that,
if $j\ge N_\ve$, $\sup_{ 0\le t \le  T} \snorm{\tilde{r}_{a_j}(\ve,t)-\tilde{r}(\ve,t)} \le \frac{d_\ve}{2}$.
Then
\begin{equation}\label{temp:alimEnt2}
\tilde{r}_{a_j}(\ve,t) > \tilde{r}(\frac{\ve}{2},t)   \qquad\ \text{for all $t\in[0,T]$ and $j\ge N_\ve$}.
\end{equation}
Let $[t_1,t_2]\subset (0,T]$. Then it follows from \eqref{temp:alimEnt2} that, as $j\to\infty$,
\begin{align*}
&\int_{t_1}^{t_2}\int_{\tilde{r}_{a_j}(\ve,t)}^{\infty}
 \rho_{a_j} \snorm{u_{a_j}}^2\mathbbm{1}_E(r)\,r^m \dif r \dif t
 - \int_{t_1}^{t_2}\int_{\tilde{r}(\ve,t)}^{\infty} \rho \snorm{u}^2\mathbbm{1}_E(r)\,r^m \dif r \dif t\\
&= \int_{t_1}^{t_2}\int_{\tilde{r}(\frac{\ve}{2},t)}^{\infty} \rho_{a_j}^{(\ve)}
 \snorm{u_{a_j}}^2\mathbbm{1}_E(r)\,r^m \dif r \dif t
 - \int_{t_1}^{t_2}\int_{\tilde{r}(\frac{\ve}{2},t)}^{\infty}\rho^{(\ve)} \snorm{u}^2\mathbbm{1}_E(r)\,r^m \dif r \dif t\\
&= \int_{t_1}^{t_2}\int_{\tilde{r}(\frac{\ve}{2},t)}^{\infty} \rho_{a_j}^{(\ve)}
  \big\{ \snorm{u_{a_j}}^2 - \snorm{u}^2\big\}\mathbbm{1}_E(r)\,r^m \dif r \dif t
  + \int_{t_1}^{t_2} \int_{\tilde{r}(\frac{\ve}{2},t)}^{\infty}
 \mathbbm{1}_E \big\{\rho_{a_j}^{(\ve)}-\rho^{(\ve)}\big\}\snorm{u}^2\mathbbm{1}_E(r)\,r^m \dif r \dif t\to 0,
\end{align*}
where the first-term convergence follows from the dominated convergence theorem,
Lemma \ref{lemma:alimue}, and the pointwise bound for $\snorm{u_{a_j}}$ and $\rho_{a_j}^{(\ve)}$
from Theorem \ref{thm:WSCEx}\ref{item:WSCEx4}, while the second-term convergence follows from
Lemma \ref{lemma:alimrho}.

Let $L\in\mathbb{N}$ and set $E\equiv [0,L]$. By Theorem \ref{thm:WSCEx}\ref{item:WSCEx1},
	\begin{align*}
		 \int_{t_1}^{t_2}\int_{\tilde{r}(\ve,t)}^{L} \rho \snorm{u}^2\,r^m \dif r \dif t
		 =&\, \lim_{j\to \infty}\int_{t_1}^{t_2}\int_{\tilde{r}_{a_j}(\ve,t)}^{L}
		   \rho_{a_j} \snorm{u_{a_j}}^2\,r^m \dif r \dif t\\
		 \le&\, \snorm{t_2-t_1} \esssup_{t\in(0,T)}\int_{a_j}^{\infty}
		   \rho_{a_j} \snorm{u_{a_j}}^2(r,t) r^m \dif r = C(T) \snorm{t_2-t_1},
	\end{align*}
	where $C(T)=C(T,C_0)$ is independent of $a\in(0,1)$, $\ve>0$, and $L\in\mathbb{N}$. Thus,
	by the Lebesgue differentiation theorem, it follows that
	\begin{align*}
		\esssup_{t\in(0,T)} \int_{\tilde{r}(\ve,t)}^{L} (\rho \snorm{u}^2)(r,t)\,r^m \dif r \le C(T).
	\end{align*}
	Using the monotone convergence theorem with the consecutive limits $L\to\infty$ and then $\ve\searrow 0$,
	\begin{align*}
		\esssup_{t\in(0,T)} \int_{\ul{r}(t)}^{\infty} (\rho \snorm{u}^2)(r,t)\,r^m \dif r
		=\lim\limits_{\substack{L\to\infty\\ \ve\to 0^+}}\esssup_{t\in(0,T)}
		\int_{\tilde{r}(\ve,t)}^{L} (\rho \snorm{u}^2)(r,t)\,r^m \dif r \le C(T).
	\end{align*}
	
Next, let $[t_1,t_2]\subseteq(0,T]$, and let $E\subset [0,\infty)$ be a bounded measurable set.
Using \eqref{temp:alimEnt2}, the convergence results from Lemma \ref{lemma:alimue}--\ref{lemma:alimrho},
and the bound: $e(r,t)\le C(\ve) \sigma^{-\frac{1}{2}}(t)$ for $r\ge \tilde{r}(\frac{\ve}{2},t)$ from Lemma \ref{lemma:alimue},
we see that, as $j\to\infty$,
\begin{align*}
&\int_{t_1}^{t_2}\int_{\tilde{r}_{a_j}(\ve,t)}^{\infty}  \rho_{a_j} e_{a_j}\mathbbm{1}_E(r)\,r^m \dif r \dif t
 -\int_{t_1}^{t_2}\int_{\tilde{r}(\ve,t)}^{\infty}  \rho e\,\mathbbm{1}_E(r)\,r^m \dif r \dif t \\
&=\int_{t_1}^{t_2}\int_{\tilde{r}(\frac{\ve}{2},t)}^{\infty}\rho_{a_j}^{(\ve)}e_{a_j}\mathbbm{1}_E(r)\,r^m\dif r\dif t -\int_{t_1}^{t_2}\int_{\tilde{r}(\frac{\ve}{2},t)}^{\infty} \rho^{(\ve)} e\,\mathbbm{1}_E(r)\,r^m \dif r \dif t\\
&=\int_{t_1}^{t_2}\int_{\tilde{r}(\frac{\ve}{2},t)}^{\infty} \rho_{a_j}^{(\ve)}
  (e_{a_j}-e)\mathbbm{1}_E(r)\,r^m \dif r \dif t  + \int_{t_1}^{t_2}\int_{\tilde{r}(\frac{\ve}{2},t)}^{\infty}
\big(\rho_{a_j}^{(\ve)}-\rho^{(\ve)}\big) e\,\mathbbm{1}_E(r)\,r^m \dif r \dif t \to 0.
\end{align*}
Let $E_j^{\ve}(t)\vcentcolon= E\cap [\tilde{r}_{a_j}(\ve,t),\infty)$. Then, by \eqref{6.3b},
\begin{equation}\label{aEtemp1}
    \begin{aligned}
	        \int_{t_1}^{t_2}\int_{\tilde{r}(\ve,t)}^{\infty}\rho e\,\mathbbm{1}_E(r)\, r^m \dif r \dif t
	         &= \lim\limits_{j\to\infty} \int_{t_1}^{t_2}\int_{\tilde{r}_{a_j}(\ve,t)}^{\infty}
	           \rho_{a_j} e_{a_j} \,\mathbbm{1}_E(r)\,r^m \dif r \dif t \nonumber\\
		    &\le \sup\limits_{j\in\mathbb{N}} \int_{t_1}^{t_2}\int_{E_j^\ve(t)}\rho_{a_j} e_{a_j}\,r^m \dif r\dif t\\
		    &\le \int_{t_1}^{t_2} \big(C(T) + \omega_2(E_{j}^{\ve}(t);C(T)) \big)\dif t,
	    \end{aligned}
\end{equation}
where $\omega_2\big(E_{j}^{\ve}(t);C(T)\big) = f_2(f_1(\int_{E_{j}^{\ve}(t)}r^m \dif r;C(T));C(T))$.

Since $E_{j}^{\ve}(t) \subseteq E$ for all $(j,\ve,t)\in\mathbb{N}\times(0,1)\times[0,T]$,
and $f_1(y;C)$ and $f_2(y;C)$ are monotone increasing functions in $y$ (see Proposition \ref{prop:omega}),
then, for all $(j,\ve,t)\in\mathbb{N}\times(0,1)\times[0,T]$,
\begin{align*}
\omega_2(E_{j}^{\ve}(t);C(T))
\le  f_2(f_1(\int_{E}r^m \dif r;C(T));C(T)) = \omega_2(E;C(T)).
\end{align*}
Using this, \eqref{aEtemp1} becomes
	\begin{align*}
		\int_{t_1}^{t_2}\int_{\tilde{r}(\ve,t)}^{\infty} (\rho e)(r,t)\,\mathbbm{1}_E(r)\, r^m \dif r \dif t
		\le  \snorm{t_2-t_1} \big(C(T)+ \omega_2(E;C(T))\big).
	\end{align*}
Thus, by the Lebesgue differentiation theorem, it follows that
\begin{align*}
\esssup\limits_{t\in[0,T]}\int_{\tilde{r}(\ve,t)}^{\infty}(\rho e) (r,t)\mathbbm{1}_E(r)\,r^m \dif r
\le C(T)+ \omega_2(E;C(T)).
\end{align*}
Since $\rho=0$ in $\{r< \ul{r}(t)\}$ and $\lim_{\ve\searrow 0}\tilde{r}(\ve,t)= \ul{r}(t)$,
by the monotone convergence theorem,
\begin{align*}
		\esssup\limits_{t\in[0,T]}\int_{E}(\rho e)(r,t)\,r^m \dif r \le C(T)+ \omega_2(E;C(T)).
\end{align*}
	
Let $\ve\in(0,1]$ and $r_\ve \vcentcolon=\frac{1}{2} \inf_{t\in[0,T]}\tilde{r}(\ve,t)$.
By Lemma \ref{lemma:alimpath}, there exists $M_\ve\in\mathbb{N}$ such that
\begin{equation}\label{temp:alimEnt1}
	    \inf_{t\in[0,T]}\tilde{r}_{a_j}(\ve,t)> r_\ve \qquad\ \text{for all $j\ge M_\ve$}.
\end{equation}
Moreover, by Theorem \ref{thm:WSCEx}\ref{item:WSCEx3}, $\tilde{r}_{a_j}(\ve,t)\le \tilde{r}_{a_j}(1,t)\le C_0$
for all $(j,t)\in\mathbb{N}\times[0,T]$.
By \eqref{temp:alimEnt1}, Theorem \ref{thm:WSCEx}\ref{item:WSCEx1}, Lemma \ref{lemma:alimrho},
and the fact that $G(0)=1$, we have
\begin{equation*}
\begin{dcases}
\rho_{a_j}^{(\ve)} - 1 \rightharpoonup \rho^{(\ve)} - 1
\quad\text{ weakly in $L^{2}\big(0,T; L^2([r_{\ve},\infty),r^m \dif r) \big)$  as $j\to\infty$},\\
\sup_{j\ge M_\ve} \esssup_{t\in[0,T]} \int_{r_\ve}^{\infty} G(\rho_{a_j}^{(\ve)})(r,t)\,r^m \dif r \le C(T), \\
0 \le \rho_{a_j}^{(\ve)}(r,t) \le C(\ve) \quad \text{ for {\it a.e.}\, $(r,t)\in[r_\ve,\infty)\times[0,T]$
   and for each $j\in\mathbb{N}$}.
\end{dcases}
\end{equation*}
Thus, applying Proposition \ref{prop:mazur},
it follows that
\begin{equation*}
	\esssup_{t\in(0,T)}\int_{r_\ve}^{\infty} G(\rho^{(\ve)})(r,t)\, r^m \dif r \le C(T).
\end{equation*}
By Lemma \ref{lemma:alimpath}\ref{item:apath4},
$r_\ve < \tilde{r}(\ve,t)\le C_0(1+\ve)^{\frac{1}{n}}\le 2^{\frac{1}{n}}C_0$ for $t\in[0,T]$.
Since $G(\rho^{(\ve)})=G(0)=1$ for $r\in[r_\ve,\tilde{r}(\ve,t)]$ and $\rho^{(\ve)}=\rho$
for {\it a.e.} $r\in[\tilde{r}(\ve,t),\infty)$, then,  for {\it a.e.} $t\in[0,T]$,
	\begin{equation*}
	    \int_{\tilde{r}(\ve,t)}^{\infty} G(\rho)(r,t)\,r^m \dif r
	    \le \int_{r_\ve}^{\tilde{r}(\ve,t)}\,r^m \dif r
	    + \int_{\tilde{r}(\ve,t)}^{\infty} G(\rho)(r,t)\,r^m \dif r \le C(T).
	\end{equation*}
Since the above is true for all $\ve\in(0,1]$,
using $\lim_{\ve\searrow 0} \tilde{r}(\ve,t)= \ul{r}(t)$, $\rho \equiv 0$ for $r<\ul{r}(t)$,
and the monotone convergence Theorem, it follows that
\begin{equation*}
\esssup_{t\in[0,T]}\int_{0}^{\infty} G(\rho)(r,t)\,r^m \dif r
=\lim\limits_{\ve\searrow 0} \esssup_{t\in[0,T]}\int_{\tilde{r}(\ve,t)}^{\infty} G(\rho)(r,t)\,r^m \dif r
\le C(T).
\end{equation*}
\end{proof}

\begin{lemma}\label{lemma:entdis}
Let $(\rho,u,e)(r,t)$ be the limit function obtained in {\rm Lemmas \ref{lemma:alimue}}--{\rm \ref{lemma:alimrho}}. Then for $\eta>0$,
\begin{align*}
&\int_{0}^{T} \sup_{r\ge\ul{r}(t)+\eta}\dfrac{\snorm{u}}{\sqrt{e}}(r,t)\,\dif t
\le C(T)\big(\eta^{\frac{2-n}{2}} + \eta^{2-n} \big) && \text{if } \ n=2,\, 3,\\
&\int_{0}^{T} \sup_{r\ge \ul{r}(t)+\eta} \log \big( \max\big\{ 1, e(r,t)\big\} \big)\, \dif t
\le C(T)\big(1+\sqrt{|\log\eta|}\big)&& \text{if } \ n=2,\\
&\int_{0}^{T} \sup_{r\ge\ul{r}(t)+\eta} \log \big( \max \big\{ 1, e^{\pm 1}(r,t) \big\} \big)\,\dif t
\le C(T) \eta^{2-n} && \text{if } \ n=3.
\end{align*}
\end{lemma}
\begin{proof}
We will only prove the first estimate, as the other two statements follow the exact same proof. Fix $\eta>0$. Since $a_j\searrow 0$ as $j\to\infty$, there exists $N_1\in\mathbb{N}$ such that $\eta > a_j$ for $j\ge N_1$. Let $\ve\in(0,1]$. It follows from Theorem \ref{thm:WSCEx}\ref{item:WSCEx6} that,
\begin{equation*}
\int_{0}^{T} \sup\limits_{r\ge \tilde{r}(2\ve,t)+\eta} \dfrac{|u_{a_j}|}{\sqrt{e_{a_j}}}(r,t)\, \dif t \le \int_{0}^{T} \sup\limits_{r\ge \eta} \dfrac{|u_{a_j}|}{\sqrt{e_{a_j}}}(r,t)\, \dif t \le C(T) \big(\eta^{\frac{2-n}{2}} + \eta^{2-n} \big) \qquad \text{for all } \ j\ge N_1.
\end{equation*}
By Lemma \ref{lemma:alimpath}\ref{item:apath1} and the fact that $\ve\mapsto\tilde{r}(\ve,t)$ is increasing, 
there exists $N_2\in\mathbb{N}$ such that $\tilde{r}(2\ve,t) \ge \tilde{r}_{a_j}(\ve,t) $ for all $j\ge N_2$.
Thus, by the uniform convergence in Lemma \ref{lemma:alimue} and Fatou's lemma, we have
\begin{equation*}
	\int_{0}^{T} \sup\limits_{r\ge \tilde{r}(2\ve,t)+\eta} \dfrac{|u|}{\sqrt{e}}(r,t)\, \dif t
	\le C(T) \big( \eta^{\frac{2-n}{2}} + \eta^{2-n} \big).
\end{equation*}
Finally, since $\tilde{r}(\ve,t) \searrow \ul{r}(t)$ as $\ve \searrow 0$, the desired estimate is obtained by an application of Monotone Convergence theorem with the limit $\ve\searrow 0$.
\end{proof}

\subsection{Weak forms away from the vacuum}\label{subsec:alimWF}
In this subsection, we denote $(\rho,u,e)(r,t)$ on $\mathrm{F}$ and $\tilde{r}(x,t)$ on $[0,\infty)\times[0,T]$
as the limit functions obtained in Lemmas \ref{lemma:alimpath}--\ref{lemma:alimrho}.
The main aim is to show that they satisfy the weak forms described in Definition \ref{def:WFAV}.
Moreover, we also define a class of test functions $\mathcal{D}_T$ as
\begin{equation*}
\mathcal{D}_T\vcentcolon
=\{ \phi\in C^{1}([0,\infty)\times[0,T])\,\vcentcolon\,\exists L\in\mathbb{N} \text{ such that }  \phi(r,t)=0
\text{ for $(r,t)\in[L,\infty)\times[0,T]$} \}.
\end{equation*}

\begin{lemma}\label{lemma:alimID}
Let $(\rho_0,u_0,e_0)(r)$ be the spherically symmetric initial data in {\rm Theorem \ref{thm:WSAV}},
and let $(\rho_a^0,u_a^0 ,e_a^0)(r)$ be the modified initial data constructed in {\rm \S \ref{subsec:mollify}}.
Then, for all $\phi\in \mathcal{D}_T$,
\begin{equation*}
\begin{split}
&\lim\limits_{a\searrow 0} \int_{a}^{\infty} \rho_a^0(r) \phi(r,0)\,r^m \dif r
 = \int_{0}^{\infty} \rho_0(r) \phi (r,0)\,r^m \dif r,\\
&\lim\limits_{a\searrow 0} \int_{a}^{\infty} (\rho_a^0 u_a^0) (r) \phi(r,0)\,r^m \dif r
 = \int_{0}^{\infty} (\rho_0 u_0)(r) \phi(r,0)\,r^m \dif r,\\
&\lim\limits_{a\searrow 0} \int_{a}^{\infty} \Big(\frac{1}{2}\rho_a^0 \norm{u_a^0}^2 +\rho_a^0 e_a^0\Big)(r) \phi(r,0)\, r^m \dif r
= \int_{0}^{\infty} \Big(\dfrac{1}{2}\rho_0\snorm{u_0}^2 + \rho_0 e_0 \Big)(r) \phi(r,0)\,r^m \dif r.
\end{split}
\end{equation*}
\end{lemma}

\begin{proof}
By construction \eqref{eqs:amolint} in \S \ref{subsec:mollify},
$(\rho_a^0,u_a^0,e_a^0)(r)=(1,0,1)$ for $r\in[0,a]$. Then
we use Proposition \ref{prop:mollify} to obtain
\begin{align*}
&\Bignorm{\int_{a}^{\infty}\rho_a^0(r)\phi(r,0)\,r^m \dif r
  - \int_{0}^{\infty}\rho_0(r)\phi(r,0)\,r^m \dif r}\\
	&\quad\le \int_{0}^a \snorm{\phi(r,0)}\,r^m \dif r
	 +\sbnorm{\phi(\cdot,0)}_{L^2([0,\infty),r^m \dif r)}
	 \sbnorm{\rho_a^0 - \rho_0}_{L^2([0,\infty),r^m \dif r)} \to 0 \qquad\text{ as $a\searrow 0$},\\[2mm]
&\Bignorm{\int_{a}^{\infty}(\rho_a^0 u_a^0)(r) \phi(r,0)\,r^m \dif r
 - \int_{0}^{\infty}(\rho_0 u_0)(r) \phi(r,0)\,r^m \dif r}\\
&\quad\le \sbnorm{\rho_a^0}_{L^\infty} \sbnorm{\phi(\cdot,0)}_{L^2([0,\infty),r^m\dif r)} \sbnorm{u_a^0-u_0}_{L^2([0,\infty),r^m\dif r)} \\[1mm]
&\quad\quad\, + \sbnorm{\phi(\cdot,0)}_{L^\infty} \sbnorm{u_0}_{L^2([0,\infty),r^m\dif r)} \sbnorm{\rho_a^0-\rho_0}_{L^2([0,\infty),r^m\dif r)} \to 0 \qquad\text{ as $a\searrow 0$}.
\end{align*}
Furthermore, it follows from \eqref{eqs:amolint} and Proposition \ref{prop:mollify} that,
as $a\searrow 0$,
\begin{align*}
&\Bignorm{\int_{a}^{\infty} \rho_a^0 \snorm{u_a^0}^2(r) \phi(r,0)\,r^m \dif r
  - \int_{0}^{\infty} \rho_0 \snorm{u_0}^2(r) \phi(r,0)\,r^m \dif r }\\
&\quad \le \sbnorm{\rho_a^0}_{L^{\infty}} \sbnorm{\phi(\cdot,t)}_{L^2([0,\infty),r^m \dif r)} \sbnorm{\snorm{u_a^0}^2 - \snorm{u_0}^2}_{L^2([0,\infty),r^m \dif r)}\\[1mm]
&\quad\quad + \sbnorm{\phi(\cdot,0)}_{L^{\infty}}\sbnorm{u_0}_{L^4([0,\infty),r^m\dif r)}^2
  \sbnorm{ \rho_a^0 - \rho_0}_{L^2([0,\infty),r^m \dif r)}\to 0 \qquad\text{ as $a\searrow 0$}.\\[2mm]
&\Bignorm{\int_{a}^{\infty}\rho_a^0 e_a^0(r) \phi(r,0)\,r^m \dif r
  - \int_{0}^{\infty} \rho_0 e_0(r) \phi(r,0)\,r^m \dif r}\\
&\quad\le \Bignorm{\int_{0}^{a}\phi(r,0)\,r^m\dif r}
  + \Bignorm{\int_{0}^{\infty} \phi(r,0) \big(\rho_a^0 (e_a^0 - e_0) + e_0 (\rho_a^0-\rho_0)\big)\,r^m \dif r }\\
&\quad\le \sbnorm{\phi(\cdot,0)}_{L^{\infty}}\frac{a^n}{n} + \sbnorm{\rho_a^0}_{L^{\infty}} \sbnorm{\phi(\cdot,0)}_{L^2([0,\infty),r^m \dif r)} \sbnorm{e_a^0-e_0}_{L^2([0,\infty),r^m \dif r)} \\[1mm]
&\quad\quad\, + \big(\sbnorm{\phi(\cdot,0)}_{L^{\infty}}\sbnorm{e_0-1}_{L^2([0,\infty),r^m\dif r)} + \sbnorm{\phi(\cdot,0)}_{L^2([0,\infty),r^m\dif r)}\big)
\sbnorm{\rho_a^0-\rho_0}_{L^2([0,\infty),r^m \dif r)} \to 0.
\end{align*}
This completes the proof.
\end{proof}

\begin{lemma}\label{lemma:alimRhoWF}
For any $\phi\in \mathcal{D}_T$,
\begin{align*}
\int_{0}^{\infty} \rho(r,t) \phi(r,t)\,r^m \dif r - \int_{0}^{\infty} \rho_0(r)\phi(r,0)\,r^m \dif r
= \int_{0}^{t}\int_{0}^{\infty} \big(\rho \partial_t \phi + \rho u \partial_r\phi\big)(r,s)\,r^m \dif r \dif s.
\end{align*}
\end{lemma}

\begin{proof}
First, from the continuity equation in Theorem \ref{thm:WSCEx},
we obtain that, for all $j\in\mathbb{N}$ and $a_j\in(0,1)$,
\begin{equation}\label{temp:alimRhoWF1}
\int_{a_j}^{\infty} \rho_{a_j}(r,s) \phi(r,s)\,r^m \dif r \Big\vert_{s=0}^{s=t}
= \int_{0}^{t}\int_{a_j}^{\infty} \big( \rho_{a_j} \partial_t \phi
+ \rho_{a_j} u_{a_j} \partial_r \phi\big)(r,t)\,r^m \dif r \dif t.
\end{equation}

Since $\rho\equiv 0$ in $\{ r < \ul{r}(t) \}$ by construction, it follows from Lemmas \ref{lemma:rhoM} and \ref{lemma:alimrho} that
\begin{align*}
\mathcal{I}_1^{j}\vcentcolon
=&\,\Bignorm{\int_{a_j}^{\infty} \rho_{a_j}(r,t) \phi(r,t)\,r^m \dif r
  - \int_{\ul{r}(t)}^{\infty} \rho(r,t) \phi(r,t)\,r^m \dif r}\\
\le\,&\, \Bignorm{\int_{\tilde{r}_{a_j}(x,t)}^{\infty}\rho_{a_j}(r,t)\phi(r,t)\,r^m\dif  r
    -\int_{\tilde{r}(x,t)}^{\infty}\rho(r,t)\phi(r,t)\,r^m\dif r}
  + \sbnorm{\phi}_{\infty}\int_{\ul{r}(t)}^{\tilde{r}(x,t)}\rho(r,t)\,r^m \dif r  \\
\,& +\sbnorm{\phi}_{\infty} \bigg( \int_{a_j}^{\ul{r}(t)} + \int_{\ul{r}(t)}^{\tilde{r}(x,t)} + \int_{\tilde{r}(x,t)}^{\tilde{r}_{a_j}(x,t)} \bigg)\rho_{a_j}(r,t)\,r^m \dif r\\
\le &\, \Bignorm{\int_{0}^{\infty}\big(\rho_{a_j}^{(x)}-\rho^{(x)}\big)(r,t) \phi(r,t)\, r^m\dif r}
   + \sbnorm{\phi}_{\infty} x + \sbnorm{\phi}_{\infty} \int_{a_j}^{\ul{r}(t)}\rho_{a_j}(r,t)\,r^m \dif r\\
\,&+ \sbnorm{\phi}_{\infty}\omega_1( [\ul{r}(t),\tilde{r}(x,t)];C ) + \sbnorm{\phi}_{\infty}\omega_1([\tilde{r}(x,t),\tilde{r}_{a_j}(x,t)];C).
\end{align*}
Since $\tilde{r}(x,t)\searrow \ul{r}(t)$ as $x\searrow 0$, for any $\delta>0$,
there exists $y\in(0,1)$ such that
\begin{equation*}
		\sbnorm{\phi}_{\infty} x + \sbnorm{\phi}_{\infty} \omega_1( [\ul{r}(t),\tilde{r}(x,t)];C )
		\le \delta 		\qquad\ \text{ for all $0<x<y$.}
\end{equation*}
Then it follows that, for all $x\in (0,y)$,
\begin{equation*}
\mathcal{I}_1^{j}
		\le \delta + \Bignorm{\int_{0}^{\infty}\!\!\phi(r,t)\big(\rho_{a_j}^{(x)}-\rho^{(x)}\big)(r,t)\,r^m\dif r}
		+ \sbnorm{\phi}_{\infty}\int_{a_j}^{\ul{r}(t)}\!\!\rho_{a_j}(r,t)\,r^m \dif r +\sbnorm{\phi}_{\infty}\omega_1([\tilde{r}(x,t),\tilde{r}_{a_j}(x,t)];C).
\end{equation*}
Fix a point $x\in(0,y)$. Taking the limit $j\to\infty$ on the above inequality, then using
Lemmas \ref{lemma:alimrho}\ref{item:alimrho1} and \ref{lemma:alimpath}\ref{item:apath5},
we have
\begin{align*}
\lim\limits_{j\to\infty}\Bignorm{\int_{a_j}^{\infty} \rho_{a_j}(r,t) \phi(r,t)\,r^m \dif r - \int_{\ul{r}(t)}^{\infty} \rho(r,t) \phi(r,t)\,r^m \dif r} \le \delta,
\end{align*}
which implies that $\lim_{j\to\infty}\mathcal{I}_1^{j}=0$.
	
Next, for each $x>0$, we rewrite the term:
\begin{align*}
\mathcal{I}_{2}^{j}
&\vcentcolon=\Bignorm{\int_{0}^{t}\int_{a_j}^{\infty} \rho_{a_j} u_{a_j} \partial_r \phi(r,s)\,r^m \dif r \dif s
 - \int_{0}^{t}\int_{\ul{r}(s)}^{\infty} \rho u\,\partial_r \phi(r,s)\,r^m \dif r \dif s}\\
&\,\le  \sbnorm{\partial_r\phi}_{\infty} \int_{0}^{t}\int_{a_j}^{\tilde{r}_{a_j}(x,s)}
  \snorm{\rho_{a_j}u_{a_j}} \,r^m \dif r \dif s
  + \sbnorm{\partial_r\phi}_{\infty} \int_{0}^{t} \int_{\ul{r}(s)}^{\tilde{r}(x,s)}
    \snorm{\rho u}\,r^m \dif r \dif s \\
&\,\quad +\Bignorm{\int_{0}^{t}\int_{\tilde{r}_{a_j}(x,s)}^{\infty}\rho_{a_j}u_{a_j}
  \partial_r\phi\,r^m\dif r\dif s
  - \int_{0}^{t}\int_{\tilde{r}(x,s)}^{\infty}\rho u\partial_r \phi\,r^m\dif r \dif s}\\
& =\vcentcolon \mathcal{I}_{2,1}^{j}(x)+\mathcal{I}_{2,2}^{j}(x)+\mathcal{I}_{2,3}^{j}(x).
\end{align*}
Using Lemma \ref{lemma:rhoM}, the Cauchy-Schwartz inequality, and
Theorem \ref{thm:WSCEx}\ref{item:WSCEx3}, we have
\begin{align*}
\mathcal{I}_{2,1}^j(x)
\le \sbnorm{\partial_r\phi}_{\infty}\sqrt{x} \int_{0}^{t}\Big(\int_{a_j}^{\tilde{r}_{a_j}(x,t)} \rho_{a_j}u_{a_j}^2\,r^m \dif r\Big)^{\frac{1}{2}} \dif s
\le \sbnorm{\partial_r\phi}_{\infty} C(T) t \sqrt{x}.
\end{align*}
Moreover, using Lemmas \ref{lemma:alimrho}\ref{item:alimrho3} and \ref{lemma:alimEnt} leads to
\begin{align*}
\mathcal{I}_{2,2}^{j}(x)
 \le \sbnorm{\partial_r\phi}_{\infty}\int_{0}^{t} \Big( \int_{\ul{r}(s)}^{\tilde{r}(x,s)}
  \rho \snorm{u}^2\,r^m \dif r \Big)^{\frac{1}{2}}
  \Big( \int_{\ul{r}(s)}^{\tilde{r}(x,s)} \rho\,r^m \dif r \Big)^{\frac{1}{2}} \dif s
\le\sbnorm{\partial_r\phi}_{\infty} C(T) t \sqrt{x}.
\end{align*}
Thus, for any $\delta>0$, there exists $y_\delta\in(0,1)$ such that, if $x\in (0, y_{\delta})$,
then $\mathcal{I}_{2}^{j}(x)\le \delta + \mathcal{I}_{2,3}^{j}(x)$.
By Lemma \ref{lemma:alimpath}, $x\to \tilde{r}(x,t)$ is increasing,
and $(x,t)\mapsto\tilde{r}(x,t)$ is continuous. Since $[0,T]$ is compact,
$d:=\inf_{t\in[0,T]}\{\tilde{r}(x,t) - \tilde{r}(\frac{x}{2},t)\}>0$.
By Lemma \ref{lemma:alimpath}\ref{item:apath1}, there exists $N_x\in\mathbb{N}$
such that $\sup_{ 0\le t \le  T} \snorm{\tilde{r}_{a_j}(x,t)-\tilde{r}(x,t)}\le \frac{d}{2}$ if $j\ge N_x$.
Combining with the previous assertion, we have
	\begin{equation}\label{temp:alimRhoWF2}
	    \tilde{r}(\frac{x}{2},t) < \tilde{r}_{a_j}(x,t) \,\qquad \text{ for each $t\in[0,T]$ and $j\ge N_x$.}
	\end{equation}
Then $\chi_x^{a_j}=\chi_{\frac{x}{2}}\chi_x^{a_j}$ for all $j\ge N_x$,
where $\chi_\ve$ and $\chi_\ve^{a}$ are defined in \eqref{6.6indicator}.
Thus, for all $j\ge N_x$,
	\begin{align*}
		\mathcal{I}_{2,3}^{j}(x)
		\le \Bignorm{\int_{0}^{t}\int_{\tilde{r}(\frac{x}{2},s)}^{\infty}\rho_{a_j}^{(x)}(u_{a_j} - u)\partial_r
		  \phi\,r^m\dif r \dif s }
		  + \Bignorm{\int_{0}^{t}\int_{\tilde{r}(\frac{x}{2},s)}^{\infty}\big(\rho_{a_j}^{(x)} - \rho^{(x)}\big) u
		   \partial_r \phi\,r^m\dif r \dif s }.
	\end{align*}
	By the pointwise bound: $ \rho_a^{(x)}(r,t)\le C(x)$ from Theorem \ref{thm:WSCEx}\ref{item:WSCEx4},
	and the uniform convergence from Lemma \ref{lemma:alimue}, it follows that, as $j\to \infty$,
	\begin{align*}
		\Bignorm{\int_{0}^{t}\int_{\tilde{r}(\frac{x}{2},s)}^{\infty}\rho_{a_j}^{(x)}(u_{a_j} - u)
		 \partial_r \phi\, r^m\dif r \dif s }
		 \le C(x)\int_{0}^{t}\int_{\tilde{r}(\frac{x}{2},s)}^{\infty}\snorm{u_{a_j} - u}
		  \partial_r \phi\,r^m\dif r \dif s \to 0.
	\end{align*}
	By Lemma \ref{lemma:alimue}, $\chi_{\frac{x}{2}} u\partial_r \phi \in L^{1}(0,T;L^{2}([0,\infty),r^m \dif r))$.
Thus, by Lemma \ref{lemma:alimrho}\ref{item:alimrho1},
\begin{align*}
\Bignorm{\int_{0}^{t}\int_{\tilde{r}(\frac{x}{2},s)}^{\infty}\big(\rho_{a_j}^{(x)} - \rho^{(x)}\big)u
\partial_r \phi\,r^m\dif r \dif s } \to 0 \qquad\ \text{as $j\to \infty$},
\end{align*}
which implies that $\lim_{j\to\infty}\mathcal{I}_{2,3}^{j}(x)=0$ for each $x>0$.
Now, for any given $\delta>0$, we fix $x\in(0,y_\delta)$. Then $\mathcal{I}_{2}^j=\mathcal{I}_{2,1}^j(x)+\mathcal{I}_{2,2}^j(x)+\mathcal{I}_{2,3}^j(x)
\le \delta + \mathcal{I}_{2,3}^j(x)$.
Taking limit $j\to\infty$ on both sides of this,
we see that $\lim_{j\to\infty} \mathcal{I}_2^{j}\le \delta$.
Since $\delta>0$ is arbitrarily small, we conclude that $\lim_{j\to\infty} \mathcal{I}_2^{j} = 0$.
By the
same argument, we can also show that
\begin{align*}
\mathcal{I}_{3}^{j}
\vcentcolon= \Bignorm{\int_{0}^{t}\int_{a_j}^{\infty}\rho_{a_j}\partial_t \phi\, r^m \dif r \dif s - \int_{0}^{t}\int_{\ul{r}(s)}^{\infty}\rho \partial_t \phi\, r^m \dif r \dif s} \to 0 \qquad\ \text{ as $j\to\infty$.}
\end{align*}
Lemma \ref{lemma:alimRhoWF} is proved by substituting these limits in \eqref{temp:alimRhoWF1}
and using Lemma \ref{lemma:alimID}.
\end{proof}

\begin{lemma}\label{lemma:eWF}
 Let $\phi\in \mathcal{D}_T\cap C^{1}$ be such that $\supp(\phi)\Subset \mathrm{F}$.
 Then
 \begin{align}
&\int_{0}^{\infty} (\rho u)(r,t) \varphi(r,t)\,r^m \dif r
   - \int_{0}^{\infty} \rho_0(r) u_0(r) \varphi(r,0)\,r^m \dif r\nonumber\\
	        &\quad = \int_{0}^{t}\int_{0}^{\infty}\Big\{ \rho u \partial_t \varphi + \rho u^2 \partial_r \varphi + \big( P- \beta (\partial_r u + m \dfrac{u}{r})\big)
	        \big(\partial_r \varphi + m \dfrac{\varphi}{r}\big) \Big\}\,r^m \dif r \dif s,\label{eqs:WFAVmom}\\[2mm]
&\int_{0}^{\infty} (\rho E)(r,t) \phi (r,t)\, r^m \dif r
  - \int_{0}^{\infty} (\rho_0 E_0) (r) \phi(r,0)\,r^m \dif r
  - \int_{0}^{t}\int_{0}^{\infty} \rho E \partial_t \phi\, r^m \dif r \dif s\nonumber\\
&\quad =\int_{0}^{t} \int_{0}^{\infty} \Big\{ (\rho E + P ) u  - 2\mu u \partial_r u
    - \lambda u \big(\partial_r u + m \dfrac{u}{r}\big) - \kappa \partial_r e \Big\}
      \partial_r \phi\, r^m \dif r \dif s. \label{eqs:WFAVener}
 \end{align}
\end{lemma}

\begin{proof}
We give the proof only for \eqref{eqs:WFAVener}, since \eqref{eqs:WFAVmom} can be derived in the same way.

First, from the energy equation in Theorem \ref{thm:WSCEx}, we have the weak form:
\begin{equation}\label{temp:eWF1}
\begin{aligned}
&\int_{a_j}^{\infty} (\rho_{a_j} E_{a_j})(r,t) \phi (r,t)\, r^m \dif r
- \int_{a_j}^{\infty} \Big(\dfrac{1}{2}\rho_{a_j}^0\snorm{u_{a_j}^0}^2 +\rho_{a_j}^0 e_{a_j}^0 \Big)(r)
   \phi(r,0) \,r^m \dif r\\
&= -\int_{0}^{t} \int_{a_j}^{\infty} \Big\{  2\mu u_{a_j} \partial_r u_{a_j} +\lambda u_{a_j} \big( \partial_r u_{a_j} + m \dfrac{u_{a_j}}{r} \big) + \kappa  \partial_r e_{a_j} \Big\} \partial_r\phi\, r^m \dif r \dif s\\
&\quad\, + \int_{0}^{t}\int_{a_{j}}^{\infty}  \big\{ \rho_{a_j} E_{a_j} \partial_t \phi + (\rho_{a_j} E_{a_j} + P_{a_j}) u_{a_j} \partial_r \phi \big\}\, r^m  \dif r \dif s,
	    \end{aligned}
\end{equation}
where $P_{a_j}\vcentcolon= (\gamma-1) \rho_{a_j} e_{a_j}$ and
$E_{a_j}\vcentcolon= \frac{1}{2}\snorm{u_{a_j}}^2 + e_{a_j}$.
Now, for each $x\in(0,1)$, we define
$$
\tilde{\mathrm{F}}_{x}\vcentcolon= \{ (r,t)\,\vcentcolon\, t\in[0,T], \ \tilde{r}(x,t) < r <x^{-1} \}.
$$
Then $\tilde{\mathrm{F}}_{x}\cap\{0<t<T\}$ is open for each $x\in(0,1)$,
since $(x,t)\mapsto\tilde{r}(x,t)$ is continuous.
Moreover, since $x\mapsto \tilde{r}(x,t)$ is strictly increasing for each $t\in[0,T]$,
it follows that $\tilde{\mathrm{F}}_{x}\subsetneq \tilde{\mathrm{F}}_{y}$ if $y < x \in (0,1)$.
Using
$\tilde{r}(x,t)\searrow \ul{r}(t)$ as $x\searrow 0$ for all $t\in[0,T]$,
we see that $\mathrm{F} = \cup_{x\in(0,1)} \tilde{\mathrm{F}}_{x}$, so that  $\big\{\tilde{\mathrm{F}}_x\cap\{0<t<T\}\big\}_{x\in(0,1)}$ is an open
covering of $\mathrm{F}\cap\{0<t<T\}$.
Since $\supp(\phi) \Subset
\mathrm{F}$ is a compact subset,
there exists $\ve>0$ such that $\supp(\phi) \subset \tilde{\mathrm{F}}_{2\ve}$.
Defining $d:=\inf_{t\in[0,T]}\tilde{r}(2\ve,t)$, then $d>0$
due to Lemma \ref{lemma:alimpath}\ref{item:apath4},
and there exists $N_1\in\mathbb{N}$ such that $a_j < d$ for $j\ge N_1$.
This implies that $a_j < \tilde{r}(2\ve,t)$ if $t\in[0,T]$ and $j\ge N_1$.
In addition, by the same argument for \eqref{temp:alimRhoWF2} as in the proof of Lemma \ref{lemma:alimRhoWF},
there exists $N_2\in\mathbb{N}$ such that $\tilde{r}_{a_j}(\ve,t)< \tilde{r}(2\ve,t)$ if $j\ge N_2$
and $t\in[0,T]$.
Let $N\vcentcolon= \max\{N_1,N_2\}$. Then, by Lemmas \ref{lemma:alimue}--\ref{lemma:alimrho}
and the dominated convergence theorem, we have
\begin{align*}
&\Bignorm{\int_{a_j}^{\infty}(\rho_{a_j} E_{a_j})(r,t) \phi(r,t)\, r^m \dif r
 - \int_{\ul{r}(t)}^{\infty} (\rho E)(r,t) \phi(r,t)\,r^m \dif r }\\
&\le  \int_{\tilde{r}(2\ve,t)}^{\infty} \rho_{a_j}^{(\ve)}
 \Big(\frac{1}{2}\snorm{u_{a_j}^2 - u^2} + \snorm{e_{a_j}-e} \Big)\phi\, r^m \dif r + \Bignorm{\int_{\tilde{r}(2\ve,t)}^{\infty} \big(\rho_{a_j}^{(\ve)}-\rho^{(\ve)}\big)
 \big( \dfrac{1}{2}u^2 + e \big) \phi\, r^m \dif r} \\[2mm]
&\longrightarrow 0 \qquad\,\, \mbox{as $j\to\infty$}.
\end{align*}
By the same argument, we also obtain
	\begin{align*}
	\lim\limits_{j\to\infty} \int_{0}^{t}\int_{a_{j}}^{\infty}  \big\{ \rho_{a_j} E_{a_j} \partial_t \phi +(\rho_{a_j} E_{a_j} + P_{a_j} ) u_{a_j} \partial_r \phi \big\}\,r^m  \dif r \dif s
	= \int_{0}^{t}\int_{\ul{r}(s)}^{\infty}
		 \big\{ \rho E \partial_t \phi +(\rho E + P ) u \partial_r \phi\big\}\, r^m \dif r \dif s.
	\end{align*}
	
Next, integrating by parts and using Lemma \ref{lemma:alimue}, we obtain that, for all $j\ge N$,
\begin{align*}
\int_{0}^{t}\int_{a_j}^{\infty} u_{a_j} \partial_r u_{a_j} \partial_r \phi\, r^m \dif r \dif s
= -\int_{0}^{t}\int_{\tilde{r}(2\ve,s)}^{\infty}\dfrac{1}{2}\snorm{u_{a_j}}^2\partial_r^2\phi\, r^m\dif r\dif s
\to -\int_{0}^{t}\int_{\tilde{r}(2\ve,s)}^{\infty}
\dfrac{1}{2} \snorm{u}^2 \partial_r^2 \phi\, r^m \dif r \dif s,
\end{align*}
as $j\to\infty$.
From
\eqref{6.13b}, the spatial weak derivative of $u$ exists with
$\partial_r u \in L^2_{\text{loc}}(\mathrm{F})$ and $\partial_r u \in L^2(\tilde{\mathrm{F}}_\ve)$ for each $\ve\in(0,1)$.
It then follows that, as $j\to\infty$,
\begin{align*}
\int_{0}^{t}\int_{a_j}^{\infty} u_{a_j} \partial_r u_{a_j} \partial_r \phi\, r^m \dif r \dif s
\to
-\int_{0}^{t}\int_{\tilde{r}(2\ve,s)}^{\infty} \dfrac{1}{2} \norm{u}^2 \partial_r^2 \phi\,r^m \dif r \dif s
= \int_{0}^{t}\int_{\ul{r}(s)}^{\infty}  u \partial_r u \partial_r \phi\, r^m \dif r \dif s,
\end{align*}
where we have used $\supp(\phi) \Subset \mathrm{F}$.
By the same argument, we also have
\begin{align*}
\lim\limits_{j\to\infty}\int_{0}^{t} \int_{a_j}^{\infty}
\Big\{m\lambda \dfrac{\snorm{u_{a_j}}^2}{r} +\kappa\partial_r e_{a_j}\Big\}\partial_r\phi\, r^m \dif r \dif s
= \int_{0}^{t} \int_{\ul{r}(s)}^{\infty} \Big\{ m\lambda \dfrac{\snorm{u}^2}{r}  + \kappa  \partial_r e \Big\} \partial_r\phi\,r^m \dif r \dif s.
\end{align*}
Combining all the limits established above with \eqref{temp:eWF1}, we obtain Lemma \ref{lemma:eWF}.
\end{proof}

\smallskip
Finally, the weak forms obtained in Theorem \ref{lemma:alimRhoWF}--\ref{lemma:eWF}
are translated into the Eulerian coordinates $(\x,t)\in \R^n \times[0,T]$.
Let $\Phi \in C^1([0,T]; C_c^1 (\R^n))$.
Define
\begin{align*}
\rho(\x,t)\vcentcolon= \rho(\snorm{\x},t), \quad \u(\x,t)\vcentcolon= u(\snorm{\x},t)\frac{\x}{r},
\qquad\, \phi(r,t) \vcentcolon= \int_{S^{n-1}} \Phi( r \y, t ) \dif S_{\y},
\end{align*}
where $S^{n-1}$ is the $(n-1)$-dimensional sphere,
and $\dif S_{\y}$ is the surface measure on $S^{n-1}$.
Using this, we have
\begin{align*}
\int_{0}^{t}\int_{0}^{\infty} (\rho u)(r,s)\partial_r \phi (r,s)\, r^m \dif r \dif s
&= \int_{0}^{t}\int_{0}^{\infty} \int_{S^{n-1}} (\rho u)(r,s)\y\cdot\nabla\Phi(r\y,s)\,r^m\dif S_{\y}\dif r\dif s \\
&= \int_{0}^{t}\int_{\R^n}  (\rho u)(\snorm{\x},s) \dfrac{\x}{\snorm{\x}}  \cdot\nabla\Phi( \x ,s)\, \dif \x \dif s\\
&= \int_{0}^{t} \int_{\R^n} \rho(\x,s) \u(\x,s)\cdot \nabla \Phi(\x,s)\,\dif \x \dif s.
\end{align*}

By the same way, we also have
\begin{align*}
&\int_{0}^{\infty} \rho(r,s)\phi(r,s)\,r^m \dif r \Big\vert_{s=0}^{s=t}
= \int_{\R^n} \rho(\x,s) \Phi(\x,s)\,\dif \x - \int_{\mathbb{R}^n} \rho_0(\x) \Phi(\x,0)\,\dif \x,\\[2mm]
	& \int_{0}^{t}\int_{0}^{\infty}\rho(r,s) \partial_t \phi(r,s)\,r^m \dif r \dif s = \int_{0}^{t}\int_{\R^n} \rho(\x,s) \partial_t \Phi(\x,s)\, \dif \x \dif s.
\end{align*}
Since $\phi(r,t)$ satisfies the assumption in Lemma \ref{lemma:alimRhoWF},
using the above identities in Lemma \ref{lemma:alimRhoWF},
we obtain the weak form of the continuity equation.

Next, we show the momentum equations.
Let $\Psi\vcentcolon \R^n\times[0,T] \to \mathbb{R}$ be such that
$\Psi\in C^2([0,T];C^2_c(\R^n))$ and
$\supp(\Psi)\Subset\mathcal{F}\vcentcolon
= \{(\x,t)\in\R^n\times[0,T]\,\vcentcolon\,\ul{r}(t)<\snorm{\x}\}$. We define
\begin{equation}\label{temp:S2N1}
\varphi^i(r,t) \vcentcolon= \int_{S^{n-1}} \Psi(r\y,t) y^i\, \dif S_{\y} \quad\,\,\,\text{for $i=1,\dotsc,n$.}
\end{equation}
It can be verified that $\varphi^i$ satisfies the assumption for Lemma \ref{lemma:eWF},
hence \eqref{eqs:WFAVmom} holds with $\varphi^i$.
The derivations for the weak forms in Theorem \ref{thm:WSAV} are similar to that occurring
in the continuity equation, except for the term:
$\iint\{P-\beta(\partial_r u + m \frac{u}{r})\}(\partial_r\varphi^i + m \frac{\varphi^i}{r})\,r^m \dif r \dif s$.
To deal with this, we use the following identity that can be verified with few lines of calculation
on the spherical coordinate transformation:
\begin{equation*}
	\int_{\snorm{\x}\le r} \partial_{x^i} f (\x)\,\dif \x
	= \int_{\snorm{\x}=r} f (\x) \dfrac{x^i}{\snorm{\x}}\,\dif S_{\x}
	\qquad \ \text{for $i=1,\dotsc, n$, for $r>0$.}
\end{equation*}
Using this identity, we have
\begin{align*}
\partial_r ( r^m \varphi^i )
&=\partial_r \Big( \int_{S^{n-1}} y^i \Psi(r \y ,t)\,  r^m \dif S_{\y} \Big)
=\partial_r \Big( \int_{\snorm{\x}=r} \dfrac{x^i}{\snorm{\x}} \Psi( \x ,t)\,\dif S_{\x} \Big)\\
&= \partial_r \Big( \int_{0}^{r} \int_{S^{n-1}} \partial_{x^i} \Psi (\zeta\y,t)\,\zeta^m \dif S_{\y}\dif\zeta \Big)
= r^m \int_{S^{n-1}} \partial_{x^i} \Psi ( r \y ,t)\,\dif S_{\y}.
\end{align*}
Take the derivative on \eqref{temp:S2N1} to obtain that $\partial_r \varphi^i(r,t) = \int_{S^{n-1}} y^i y^j\partial_{\x^j}\Psi(r\y,t) \dif S_{\y}$ so that
\begin{align*}
	\int_{0}^{t}\int_{0}^{\infty} P \big(\partial_r \varphi^i + \dfrac{m}{r} \varphi^i  \big)\,r^m\dif r\dif s
	= \int_{0}^{t}\int_{0}^{\infty}   \int_{S^{n-1}} P \partial_{x^i} \Psi (r\y,s)\,r^m \dif S_{\y}\dif r\dif s
	= \int_{0}^{t}\int_{\R^n} P \partial_{x^i} \Psi(\x,s)\, \dif \x \dif s.
\end{align*}
Moreover, integrating by parts yields
\begin{equation*}
    \begin{aligned}
        &\int_{0}^T\int_{0}^{\infty} \big( \partial_r u + m \dfrac{u}{r} \big)\big(\partial_r \varphi^i+ m \dfrac{\varphi^i}{r}\big)\,r^m \dif r \dif t
        = \int_{0}^T\int_{0}^{\infty}
        \big(\partial_r u \partial_r \varphi^i + m \dfrac{\varphi^i u}{r^2} \big)\,r^m \dif r \dif t.
    \end{aligned}
\end{equation*}
Using the identity $\partial_{x^j} \u^i = \partial_r(\frac{u}{\snorm{\x}})\frac{x^i x^j}{\snorm{\x}} + \frac{u}{\snorm{\x}}\delta^{ij}$,
we have
\begin{equation*}
\begin{aligned}
\int_{0}^{T}\int_{\R^n} \partial_{x^j} \u^i \partial_{x^j} \Psi \dif \x \dif t
&= \int_{0}^{T}\int_{\R^n} \Big( \partial_r\big(\dfrac{u}{\snorm{\x}}\big)\dfrac{x^i x^j}{\snorm{\x}} + \dfrac{u}{\snorm{\x}}\delta^{ij} \Big) \partial_{x^j} \Psi \dif \x \dif t\\
&= \int_{0}^{T}\int_{0}^{\infty}\int_{\mathbb{S}^{n-1}} \Big(r\partial_r\big(\dfrac{u}{r}\big)y^i y^j \partial_{x^j}\Psi(r\y,t) + \dfrac{u}{r}\partial_{x^i}\Psi(r\y,t)  \Big)\,\dif S_{\y}\,r^m \dif r \dif t\\
&= \int_{0}^{T}\int_{0}^{\infty} \Big(r^{m+1} \partial_r\big(\dfrac{u}{r}\big) \partial_r \varphi^i + \dfrac{u}{r}\partial_r(r^m \varphi^i) \Big)\,\dif r \dif t\\
&=\int_{0}^T\int_{0}^{\infty} \big( \partial_r u + m \dfrac{u}{r} \big)\big(\partial_r \varphi^i+ m \dfrac{\varphi^i}{r}\big)\,r^m \dif r \dif t.
\end{aligned}
\end{equation*}
Similarly, since $\dd\,\u = \partial_r u + m \frac{u}{r}$, we also have
\begin{equation*}
    \begin{aligned}
        \int_{0}^{T}\int_{\R^n} \dd \u \partial_{x^i} \Psi\,\dif \x \dif t
        = \int_{0}^T\int_{0}^{\infty} \big( \partial_r u + m \dfrac{u}{r} \big)\big(\partial_r \varphi^i+ m \dfrac{\varphi^i}{r}\big)\,r^m \dif r \dif t.
    \end{aligned}
\end{equation*}

Finally, we show the weak form of energy equation. Let $\varPhi\in C^{2}([0,T];C^2_{\rm c}(\R^n))$
be a test function such that
$\supp(\Phi)\Subset \mathcal{F}=\{ (\x,t)\in \R^n\times[0,T]\,\vcentcolon\,\snorm{\x}>\ul{r}(t) \}$.
Define
\begin{align*}
	\phi(r,t) \vcentcolon= \int_{S^{n-1}} \Phi(r\y,t)\,\dif S_{\y}.
\end{align*}
It can be verified that $\phi$ satisfies the assumption of Lemma \ref{lemma:eWF}, so that
\eqref{eqs:WFAVener} holds with $\phi$.
Most of the terms can be derived by the same procedure presented previously,
except for the terms: $u \partial_r u$, $\frac{u^2}{r}$, and $\partial_r e$.
Since $\supp(\phi) \Subset \{(r,t)\,\vcentcolon\, r>\ul{r}(t)\}$, it follows that
\begin{align*}
\int_{0}^{t}\int_{\ul{r}(s)}^{\infty} u \partial_r u \partial_r \phi (r,s)\, r^m \dif r \dif s
&= \int_{0}^{t}\int_{\ul{r}(s)}^{\infty} u \partial_r u
\Big( \int_{S^{n-1}} y^i \partial_{x^i} \Phi(r\y,s) \dif S_{\y} \Big)\,r^m \dif r \dif s\\
&= \int_{0}^{t}\int_{\snorm{\x}>\ul{r}(s)}  \partial_r \big(\dfrac{u^2}{2}\big) (\snorm{\x},s)  \dfrac{x^i}{\snorm{\x}} \partial_{x^i} \Phi(\x,s)\, \dif \x \dif s\\
&=\int_{0}^{t}\int_{\snorm{\x}>\ul{r}(s)}  \partial_{x^i} \big(\dfrac{\snorm{\u}^2}{2}\big)
\partial_{x^i} \Phi(\x,s)\,\dif \x \dif s\\
&= \dfrac{1}{2} \int_{0}^{t}\int_{\snorm{\x}>\ul{r}(s)} \nabla \Phi(\x,s) \cdot \nabla \snorm{\u}^2\,
\dif \x \dif s,
\end{align*}
and with several lines of derivations, we also have
\begin{align*}
	&\int_{0}^{t}\int_{\ul{r}(s)}^{\infty}\big\{ \mu u \partial_r u + \lambda u\big( \partial_r u + m \dfrac{u}{r} \big) \big\} \partial_r \phi (r,s)\,r^m \dif r \dif s\\
	&= \int_{0}^{t}\int_{\ul{r}(s)}^{\infty} \int_{S^{n-1}} \big\{ \mu u \partial_r u + \lambda u\big( \partial_r u + m \dfrac{u}{r} \big) \big\} y^i \partial_{x^i} \Phi(r\y,s)\,r^m \dif S_{\y} \dif r \dif s\\
	&= \int_{0}^{t} \int_{\snorm{\x}>\ul{r}(s)}
	\{ \lambda (\dd \u) \u + \mu (\nabla \u) \cdot \u \} \nabla \Phi (\x,s)\,\dif \x \dif s,
\end{align*}
\begin{align*}
	\int_{0}^{t}\int_{\ul{r}(s)}^{\infty} \partial_r e\, \partial_r \phi (r,s)\,r^m \dif r \dif s
	&= \int_{0}^{t}\int_{\ul{r}(s)}^{\infty}\int_{S^{n-1}}
	\partial_r e (r,s) y^i \partial_{x^i} \Phi (r\y,s)\, r^m \dif S_{\y} \dif r \dif s \\
	&= \int_{0}^{t}\int_{\snorm{\x}>\ul{r}(s)} \nabla e \cdot \nabla \Phi\, \dif x \dif s.
\end{align*}
With these, we complete the proof of the weak form of the energy equation.

\appendix
\section{Some Variants of the Arzel\`a-Ascoli Theorem}\label{subsec:AAV}
This appendix is devoted to the proof of two variants of the Arzel\`a-Ascoli theorem,
which are used in \S\ref{subsec:klimpath}--\S\ref{subsec:compact} for the limit process
as $k\to\infty$
and in \S\ref{subsec:alimpath}--\S\ref{subsec:alim} for the limit process as $a\searrow 0$.

\begin{proposition} \label{prop:aaExt}
Let $D\subseteq \mathbb{R}^2$ be a
domain with Lipschitz boundary
such that $\text{Int}(D) \neq \emptyset$,
where $Int(\cdot)$ denotes the set of interior points.
Assume that $\{ f_{i}\vcentcolon D \to \mathbb{R}\}_{i\in\mathbb{N}}$ is a collection of functions,
and $\{(D_\alpha,C_\alpha)\}_{\alpha\in\mathbb{N}}$ is a collection of compact subsets
$D_\alpha\Subset D$ and constants $C_\alpha\in (0,\infty)$ such that,
for each $\alpha \in \mathbb{N}$,
\begin{align}
&D_\alpha \subseteq D_{\alpha+1}, \quad \text{Int}(D_\alpha)\neq \emptyset, \quad D = \cup_{\alpha\in\mathbb{N}} D_\alpha,\nonumber\\
&\sup\limits_{i\in\mathbb{N}}\sup\limits_{x\in D_{\alpha}} f_i(x) \le C_{\alpha},\qquad\, \sup_{i\in\mathbb{N}} \snorm{f_i(x)-f_i(y)} \le C_{\alpha} \snorm{x-y} \quad \text{for all $x, y \in D_\alpha$}.\label{lianxuxing}
\end{align}
Then there exist both a continuous function $f\vcentcolon D \to \mathbb{R}$
and a subsequence $\{ i_j \}_{j\in \mathbb{N}}$ with $i_j\to \infty$ as $j\to\infty$ such that
\begin{equation*}
\begin{aligned}
&\lim\limits_{j\to\infty} \sup_{x\in K} \snorm{ f_{i_j}(x) - f(x)} = 0 \ \  && \text{for all compact subset $K\Subset D$,}\\
&\snorm{f(x)-f(y)} \le C_{\alpha} \snorm{x-y} \ \  && \text{for all $x, y \in D_\alpha$, for each $\alpha\in\mathbb{N}$.}
\end{aligned}
\end{equation*}
\end{proposition}

\begin{proof}
We split the proof into three steps: the desired subsequence, the limit function,
and the uniform convergence and continuity.

\smallskip
1. {\it The desired  subsequence}.
By the Arzel\`a-Ascoli theorem, it follows that, for each $\alpha\in \mathbb{N}$,
there exist both a continuous function $f^{(\alpha)}(x,t) \vcentcolon D_\alpha \to \mathbb{R}$ and a subsequence $\{i_j^{(\alpha)}\}_{j\in\mathbb{N}}$ such that $i_j^{(\alpha)}\to \infty$ as $j\to\infty$, and
\begin{equation}\label{aatemp0}
\lim\limits_{j\to\infty} \sup_{x\in D_\alpha} \snorm{f^{(\alpha)}(x)-f_{i_j^{(\alpha)}}(x)} =0.
\end{equation}
Based on \eqref{aatemp0}, one can inductively construct a chain of subsequences
$\{ i_j^{(\alpha+1)} \}_{j\in\mathbb{N}}  \subseteq \{ i_j^{(\alpha)} \}_{j\in\mathbb{N}}$
for each $\alpha\in\mathbb{N}$.
Then taking the diagonal subsequence $i_j \vcentcolon= i_j^{(j)}$ yields that, for all $\alpha\in\mathbb{N}$,
\begin{equation}\label{aatemp1}
\lim\limits_{j\to\infty} \sup_{x\in D_\alpha} \snorm{f^{(\alpha)}(x)-f_{i_j}(x)} = 0.
\end{equation}
If $x\in D_\alpha$ for some integer $\alpha\in\mathbb{N}$, then, for any other integer $ q \ge \alpha+1$,
we see that, for all $x\in D_\alpha$,
\begin{align*}
\snorm{f^{(\alpha)}(x)-f^{(q)}(x)} \le \snorm{f^{(\alpha)}(x)-f_{i_j}(x)} + \snorm{f_{i_j}(x)-f^{(q)}(x)} \to 0  \qquad \text{as $j\to\infty$}.
\end{align*}
Thus, if $q\ge \alpha$ are two integers, then
\begin{equation}\label{aatemp2}
f^{(q)}(x) = f^{(\alpha)}(x) \qquad \text{ for all $x\in D_\alpha$}.
\end{equation}

\smallskip
2. {\it The limit function}. For each $\alpha\in\mathbb{N}$,
$f^{(\alpha)}(x)$ is extended to
$D$ by
\begin{equation*}
h^{(\alpha)}(x) \vcentcolon= \begin{dcases*}
			f^{(\alpha)}(x) & if $x\in D_\alpha$,\\
			0 & if $x\in D \backslash D_\alpha$.
\end{dcases*}
\end{equation*}
With this, we can define the following function $g_N(x)$ in
$D$:
\begin{align*}
g_N(x)\vcentcolon= h^{(1)}(x) + \sum_{k=1}^{N-1} \{ h^{(k+1)}(x) - h^{(k)}(x) \}
\qquad \
\text{for each $N \in \mathbb{N}$ and $N\ge 2$}.
\end{align*}
Then, for a fixed point $y\in D$, there exists $\alpha\in\mathbb{N}$ such that $y\in D_\alpha$ and $y\notin D_{\alpha-1}$,
with $D_0\vcentcolon= \emptyset$.
By \eqref{aatemp2} and the assumption that $D_q \subseteq D_{q+1}$ for each $q\in\mathbb{N}$,
it follows that $h^{(q)}(y) = 0$ if $q\le \alpha-1$, and $h^{(q)}(y)=h^{(\alpha)}(y)$ if $q\ge \alpha$.
Hence, it follows that, for all $N \ge \alpha+1$,
\begin{align*}
g_N(y) =  h^{(1)}(y) + \sum_{k=1}^{N-1} \{ h^{(k+1)}(y) - h^{(k)}(y) \} = h^{(N)}(y) = h^{(\alpha)}(y)=f^{(\alpha)}(y).
\end{align*}
Thus,   $\lim_{N\to\infty} g_N(y) = f^{(\alpha)}(y)$ for this point $y\in D$.
Since this is true for any arbitrary point $x \in D$, we obtain that
$f(x)\vcentcolon D \to \mathbb{R}$, given by
\begin{equation*}
f(x) \vcentcolon= \lim\limits_{N\to\infty} g_N (x),
\end{equation*}
is well-defined and, in particular, $f(x) = f^{(\alpha)}(x)$ if $x\in D_\alpha$ for some $\alpha\in\mathbb{N}$.
By the uniform convergence \eqref{aatemp1}, it follows that, for each $\alpha\in\mathbb{N}$,
\begin{equation}\label{aatemp3}
\lim\limits_{j\to\infty} \sup_{x\in D_\alpha} \snorm{f(x)-f_{i_j}(x)} = \lim\limits_{j\to\infty} \sup_{x\in D_\alpha} \snorm{f^{(\alpha)}(x)-f_{i_j}(x)} =0.
\end{equation}
	
\smallskip	
3. {\it The uniform convergence and continuity}.
Let $K \Subset D$ be any compact subset.
Then $\{\text{Int}(D_\alpha)\}_{\alpha\in\mathbb{N}}$ is an open covering for $K$, so that
there exists a finite subcovering $\{ D_{\alpha_k} \}_{k=1}^N$ of $K$ for some $N\in \mathbb{N}$.
Notice that $D_q \subseteq D_{q+1}$ for each $q\in\mathbb{N}$. Therefore,
there exists $\alpha\in \mathbb{N}$ such that $K \subseteq D_\alpha$.
Combining with \eqref{aatemp3}, it follows that
\begin{equation*}
\lim\limits_{j\to\infty} \sup_{x\in K} \snorm{ f_{i_j}(x) - f(x)} = 0.
\end{equation*}
For the continuity of $f$, fixing a point $x\in D$, then there exists $\alpha\in \mathbb{N}$ such that $x\in D_\alpha$.
Since $\text{Int}(D_\alpha)\neq \emptyset$,
there exists $\ve_1>0$ such that $B_{\ve_1}(x)\subset D_\alpha$.
According to  \eqref{lianxuxing}, 
\begin{equation*}
\sup\limits_{i\in \mathbb{N}} \snorm{f_i(x)- f_i(y)} \le C_{\alpha} \snorm{x-y}
\qquad \text{for all $y \in D_\alpha$}.
\end{equation*}
Applying the uniform convergence \eqref{aatemp3}, it follows that $\snorm{f(x)- f(y)} \le C_{\alpha} \snorm{x-y}$
for all $y \in D_\alpha$.
Given $\ve>0$, define $\delta\vcentcolon= \min\{ \ve_1, \frac{\ve}{C_\alpha}\}>0$.
Then, for $y\in B_{\delta}(x) \subseteq B_{\ve_1}(x) \subset D_{\alpha}$,
\begin{equation*}
\snorm{f(x)-f(y)} \le C_\alpha \snorm{x-y} \le \ve.
\end{equation*}
The continuity of $x\mapsto f(x)$ has been proved.
\end{proof}

Before stating the next proposition,
 we define
the functional space $\tilde{H}_0^1(I,r^m\dif r)$ for each connected interval $I\subseteq [0,\infty)$
as the closure of
\begin{align*}
 \mathcal{D}_0(I)\vcentcolon=\big\{ \phi\in C^{\infty}(I)\,\vcentcolon\,\exists N>a \text{ such that } [a,N]\subset I \text{ and } \phi(r)=0 \text{ for } r\in I\cap[N,\infty) \big\}
\end{align*}
via the $H^1(I,r^m\dif r)$--norm.
We also denote its dual space as $\tilde{H}^{-1}( I, r^m \dif r)$.
Then it follows that $\tilde{H}^{-1}( I, r^m \dif r)$ endowed with the norm:
\begin{equation*}
    \sbnorm{f}_{\tilde{H}^{-1}( I, r^m \dif r)}\vcentcolon
     =\sup\big\{ \langle f,g \rangle_{\tilde{H}^{-1},\tilde{H}_0^{1}}\,\vcentcolon\,
      \sbnorm{g}_{\tilde{H}_0^{1}(I,r^{m}\dif r)} \le 1  \big\}
\end{equation*}
is a Banach space.

\begin{proposition}\label{prop:scHm1}
Let $I=[a,b]$ be an interval for some $a, b\in \mathbb{R}$,
and let $D = [d,\infty)$ be an unbounded interval with some fixed $d\in(0,1)$.
Suppose that $\{f_i(r,t)\,\vcentcolon\, D\times I \to \mathbb{R}\}_{i\in\mathbb{N}}$
is a collection of functions satisfying $f_i(r,t) \in L^{\infty}( I ; \tilde{H}^{-1}(D_L,r^m \dif r) )$
for each $i, L\in\mathbb{N}$, where $D_{L}\vcentcolon= D\cap [0,L] = [d,L]$.
In addition, suppose that there exists a constant $C>0$ for which the following bounds
are satisfied{\rm :} for all $t_1, t_2 \in I$,
\begin{equation}\label{jiashe}
\begin{split}
&\sup\limits_{i\in\mathbb{N}} \sbnorm{f_i(\cdot,t_2)-f_i(\cdot,t_1)}_{\tilde{H}^{-1}(D_L,r^m \dif r)} \le C \snorm{t_2 - t_1} \qquad \text{for all $L\in \mathbb{N}$},\\
&\sup\limits_{i\in\mathbb{N}} \sup\limits_{t\in I} \int_{D} \snorm{f_i-1}^2(r,t)\,r^m \dif r \le C.
\end{split}
\end{equation}
Then there exist both a subsequence $\{i_k\}_{k\in\mathbb{N}}$ and a function $f$ such that, for all $L\in\mathbb{N}$,
\begin{equation}\label{temp:scHm0}
\begin{aligned}
&f-1\in L^{\infty}(I;L^2(D,r^m\dif r)), \quad\ f\in C^0\big(I;\tilde{H}^{-1}(D_L,r^m \dif r)\big),\\
&\lim\limits_{k\to\infty} \sup\limits_{t\in I}\sbnorm{f_{i_k}(\cdot,t)-f(\cdot,t)}_{\tilde{H}^{-1}(D_L,r^m \dif r)} = 0,\\
&\sbnorm{f(\cdot,t_1)-f(\cdot,t_2)}_{\tilde{H}^{-1}(D_L,r^m \dif r)} \le C \snorm{t_1-t_2}
\qquad\text{ for all $t_1$, $t_2\in I$},
\end{aligned}
\end{equation}
where
$C>0$ is independent of $L\in\mathbb{N}$.
\end{proposition}

\begin{proof} We divide the proof in four steps:
the desired subsequence, the desired Cauchy sequence, the uniform convergence,
and the independence of domain size.

\smallskip
1. {\it The desired subsequence}.
Fix a large integer $L\in\mathbb{N}$. It follows that, for all $i\in\mathbb{N}$,
\begin{equation*}
\begin{aligned}
\int_{D_L} \snorm{f_i}^2 \dif r
&\le 2d^{-m} \int_{D_L} \snorm{f_i-1}^2 r^{m} \dif r + 2L\le 2d^{-m}C + 2L <\infty.
\end{aligned}
\end{equation*}
Hence, the following uniform estimate hold:
\begin{equation}\label{stemp0}
\sup\limits_{i\in\mathbb{N}} \sup\limits_{t\in I} \int_{D_L} \snorm{f_i}^2(r,t)\,r^m \dif r
\le 2d^{-m}C+2L<\infty.
\end{equation}
Let $\{t_j\}_{j\in\mathbb{N}}\subset I$ be a countable dense subset.
Then, from \eqref{stemp0}, for each $j\in\mathbb{N}$,
there exist both a function $g_{j,\,L} \in L^{2}( D_L, \dif r )$ and
a subsequence $\{i^{(j)}_k\}_{k\in\mathbb{N}}$ such that
$f_{i_k^{(j)}}( \cdot , t_j ) \rightharpoonup g_{j,\,L}(\cdot)$
weakly as $k\to\infty$ in $L^2(D_L,\dif r)$.
Thus, one can inductively obtain a chain of subsequence
$\{i_k^{(j+1)}\}_{k\in\mathbb{N}} \subseteq \{i_k^{(j)}\}_{k\in\mathbb{N}}$
for each $j\in\mathbb{N}$.
Taking the diagonal sequence $i_k \vcentcolon= i_k^{(k)}$,
then, for all $j\in\mathbb{N}$,
\begin{equation}\label{stemp3}
f_{i_k}( \cdot , t_j ) \rightharpoonup g_{j,L}(\cdot) \qquad
\text{weakly as $k\to\infty$ in $L^2(D_L, \dif r)$}.
\end{equation}
	
\smallskip	
2. {\it The desired Cauchy sequence}. Define
\begin{equation*}
M^{i}(r,t) \vcentcolon= \int_{D\cap [0,r]} f_i (\zeta,t)\, \dif \zeta
\qquad\ \text{for each $r\le L$}.
\end{equation*}
Then it follows from \eqref{stemp3} that, for each $(r,j)\in [0,L]\times\mathbb{N}$,
\begin{equation}\label{stemp4}
\lim\limits_{k\to\infty}M^{i_k}(r,t_j)
= \lim\limits_{k\to\infty}\int_{D\cap[0,r]} f_{i_k}( \zeta,t_j)\,\dif\zeta
\longrightarrow
\int_{D\cap[0,r]} g_{j,\,L}(\zeta)\,\dif\zeta=\vcentcolon M_{j,\,L}(r).
\end{equation}
For all $\phi\in \tilde{H}_0^1(D_L,r^m \dif r)$, we have
\begin{align*}
&\int_{D_L} ( f_{i_k}(r,t_j) - g_j(r) ) \phi(r)\,r^m \dif r\\
&= \int_{D_L} \partial_r \big( M^{i_k}(r,t_j) - M_{j,\,L}(r) \big) \phi(r)\,r^m \dif r\\
&= -\int_{D_L} \big( M^{i_k}(r,t_j) - M_{j,L}(r)\big)\big(\partial_r \phi
  + m \dfrac{\phi}{r}\big)\,r^m \dif r\\
&\le \dfrac{C}{d} \sbnorm{\phi}_{H^1(D_L,r^m \dif r)} \Big(\int_{D_L}
  \snorm{M^{i_k}(r,t_j)-M_{j,\,L}(r)}^2\,r^m \dif r \Big)^{\frac{1}{2}}.
\end{align*}
Therefore, by \eqref{stemp4} and the dominated convergence theorem,
we see that, for $j\in\mathbb{N}$,
\begin{equation*}
\sbnorm{f_{i_k}(\cdot,t_j)-g_j(\cdot)}_{\tilde{H}^{-1}(D_L,r^m \dif r)} \le \dfrac{C}{d} \Big(\int_{D_L} \snorm{M^{i_k}(r,t_j)-M_{j,\,L}(r)}^2\,r^m \dif r \Big)^{\frac{1}{2}} \to 0
\qquad \mbox{as $k\to\infty$}.
\end{equation*}
In particular, since $L\in\mathbb{N}$ can be arbitrarily chosen, the following statement holds:
\begin{equation}\label{stempCauchy}
\{ f_{i_k}(\cdot,t_j) \}_{k\in\mathbb{N}} \,\,\, \text{ is a Cauchy sequence in }
\tilde{H}^{-1}( D_L, r^m \dif r) \text{ for each } (j,L)\in\mathbb{N}^2.
\end{equation}
	
\smallskip
3. {\it The uniform convergence}.
Given $\delta>0$, define $I_j\vcentcolon=(t_j-\frac{\delta}{3C},t_j+\frac{\delta}{3C})$ for each $j\in\mathbb{N}$.
Then $t_j\in I_j$ and
\begin{equation}\label{stempAA1}
\sup\limits_{k\in\mathbb{N}}\bnorm{f_{i_k}(\cdot,t)-f_{i_k}(\cdot,t_{j})}_{\tilde{H}^{-1}( D_L, r^m \dif r )} \le C \snorm{t-t_{j}} \le \dfrac{\delta}{3} \qquad\ \text{for all $t\in I_j$}.
\end{equation}
Since $\{t_j\}_{j\in\mathbb{N}}\subset I$ is dense,
$\{I_j\}_{j\in\mathbb{N}}$ is an open covering of $I=[a,b]$.
It follows that there exists a finite subcovering $\{I_j\}_{j=1}^{M}$ for some $M\in\mathbb{N}$.
Now, for each $j=1,\dotsc,M$, it follows from \eqref{stempCauchy} that
there exists $N_j\in\mathbb{N}$ such that
\begin{equation}\label{stempAA2}
\sbnorm{f_{i_p}(\cdot,t_j)-f_{i_q}(\cdot,t_{j})}_{\tilde{H}^{-1}(D_L, r^m \dif r)}
\le \dfrac{\delta}{3} \qquad\ \text{for all $p, q \ge N_j$}.
\end{equation}
Set $N\vcentcolon=\max\{N_j \vcentcolon j=1,\dotsc,M\}$. Then, for all $t\in I$,
there exists $j\in\{1,\dotsc,M\}$ such that $t\in I_j$ so that,
by \eqref{stempAA1}--\eqref{stempAA2}, for all $p, q \ge N$,
\begin{align*}
\sbnorm{f_{i_p}(\cdot,t)-f_{i_q}(\cdot,t)}_{\tilde{H}^{-1}(D_L,r^m \dif r )}
&\le \sbnorm{f_{i_p}(\cdot,t)-f_{i_p}(\cdot,t_j)}_{\tilde{H}^{-1}(D_L,r^m \dif r)} + \sbnorm{f_{i_p}(\cdot,t_j)-f_{i_q}(\cdot,t_j)}_{\tilde{H}^{-1}(D_L,r^m \dif r)}\\
&\quad +\sbnorm{f_{i_q}(\cdot,t_j)-f_{i_q}(\cdot,t)}_{\tilde{H}^{-1}(D_L,r^m \dif r)}
\le
\delta.
\end{align*}
Since this is true for any $t\in I$,
\begin{equation*}
\sup\limits_{t\in I} \sbnorm{f_{i_p}(\cdot,t)-f_{i_q}(\cdot,t)}_{\tilde{H}^{-1}(D_L,r^m \dif r )}
\le \delta \qquad\ \text{for all $p, q\ge N$}.
\end{equation*}
Therefore, $\{ f_{i_k} \}_{k\in\mathbb{N}}$ is a Cauchy sequence
in $L^{\infty}(I;\tilde{H}^{-1}(D_L,r^m \dif r))$, which is a Banach space.
Then there exists $g_L \in L^{\infty}(I;\tilde{H}^{-1}(D_L,r^m \dif r ))$ such that
\begin{equation}\label{stemp2}
\lim\limits_{k\to\infty} \sup\limits_{t\in I}
\sbnorm{f_{i_k}(\cdot,t)-g_L(\cdot,t)}_{\tilde{H}^{-1}(D_L,r^m \dif r)} = 0.
\end{equation}
It also follows from the Lipschitz continuity assumption
and \eqref{stemp2} that
\begin{equation*}
\sbnorm{g_L(\cdot,t_1)-g_L(\cdot,t_2)}_{\tilde{H}^{-1}(D_L,r^m \dif r)}
\le C \snorm{t_1-t_2} \qquad \text{for all $t_1, t_2 \in I$},
\end{equation*}
which implies that $g_L\in C^{0}(I; \tilde{H}^{-1}( D_L, r^m \dif r ))$.

\smallskip
4. {\it The independence of domain size}.
We claim there exists $f-1\in L^{\infty}(0,T;L^2(D,r^m \dif r))$, independent of $L\in\mathbb{N}$,
such that $f=g_L$ in $L^{\infty}(I,\tilde{H}^{-1}(D_L,r^m \dif r))$.
	
It follows from \eqref{jiashe} that there exist both a function $f-1\in L^{\infty}(I ; L^2( D, r^m \dif r))$
and a subsequence (still denoted) $\{f_i\}_{i\in\mathbb{N}}$ such that
\begin{equation}\label{stemp1}
f_i-1 \overset{\ast}{\rightharpoonup}  f-1
\qquad \text{in $L^{\infty}\big(I;L^2(D,r^m\dif r)\big)$ \, as $i\to\infty$}.
\end{equation}
Let $\varphi\in L^1(I,\tilde{H}^{1}_0(D_L,r^m \dif r))$,
and let $\{i_k\}_{k\in\mathbb{N}}$ be the subsequence in \eqref{stemp2}.
Then, using \eqref{stemp2}--\eqref{stemp1}, it follows that
\begin{align*}
&\Bignorm{\int_{I}\langle f-g_L,\varphi\rangle_{\tilde{H}^{-1},\,\tilde{H}_{0}^1} \dif t} \le \Bignorm{\int_{I}\langle f-f_{i_k},\varphi\rangle_{\tilde{H}^{-1},\,\tilde{H}_{0}^1} \dif t} + \Bignorm{\int_{I}\langle f_{i_k}-g_L,\varphi\rangle_{\tilde{H}^{-1},\,\tilde{H}_{0}^1} \dif t}\\
&\le  \Bignorm{\int_{I}\int_{D_L} (f-f_{i_k})(r,t) \varphi(r,t)\,r^m \dif r \dif t}
+ \int_{I}\sbnorm{f_{i_k} - g_L }_{\tilde{H}^{-1}}
\sbnorm{\varphi}_{\tilde{H}^{1}_0} \dif t \to  0 \qquad\mbox{as $k\to\infty$},
\end{align*}
which implies that $f=g_L$ in $L^{\infty}(I,\tilde{H}^{-1}(D_L,r^m \dif r))$.
Since $f$ is independent of $L\in\mathbb{N}$, we conclude from \eqref{stemp2}
that \eqref{temp:scHm0} holds for all $L\in\mathbb{N}$.
\end{proof}

\section{Some Properties of Convex Functions}\label{append:convex}
This appendix is devoted to disclosing some properties of convex functions,
which  have been used in \S\ref{subsec:klimEnt}, \S\ref{subsec:meas}, and \S\ref{subsec:alim}.

\begin{proposition}\label{prop:omega}
Let $G(\cdot)$, $\psi(\cdot)$, and $H(\cdot)$ be the convex functions defined in \eqref{eqs:GpsiH}.
If $G^{-1}_{+}$, $\psi_{+}^{-1}$, and $H_{+}^{-1}$ are their corresponding right branch inverses,
then they are concave, strictly increasing continuous functions defined on the domain{\rm :}
\begin{equation*}
G_{+}^{-1} \vcentcolon [0,\infty) \to [1,\infty),\quad
\ \psi_{+}^{-1} \vcentcolon [0,\infty) \to [1,\infty),
\quad \ H_{+}^{-1} \vcentcolon [-e^{-1},\infty) \to [e^{-1},\infty).
\end{equation*}
Moreover, $(f_1,f_2,f_3)(y,z)$ given in \eqref{eqs:fis} are well defined for $(y,z)\in(0,\infty)^2$
such that, for each fixed $z>0$, $y\mapsto f_i(y;z)$ is positive monotone increasing
for $i=1,2,3$, and
\begin{align*}
	\lim\limits_{y\searrow 0} f_i(y;z) = 0 \qquad\,\, \text{for $i=1,2,3$}.
\end{align*}
\end{proposition}

\begin{proof}
First, $G(\zeta)$ and $\psi(\zeta)$ are uniformly convex and achieve their minimum
at $\zeta=1$ with $G(1)=0$.
Hence, $G(\zeta)$ and $\psi(\zeta)$ are strictly increasing in  $[1,\infty)$.
Thus, the inverse functions $\psi_{+}^{-1}, G_{+}^{-1}\vcentcolon [0,\infty)\to [1,\infty)$
exist and satisfy $G(G_{+}^{-1}(\xi))=\psi(\psi_{+}^{-1}(\xi))=\xi$
for all $\xi\in[0,\infty)$.
In addition, $ H(\zeta)$ is uniformly convex and  achieves its minimum
at $\zeta=e^{-1}$ with $H(e^{-1})=-e^{-1}$.
Hence, $H(\zeta)$ is strictly increasing in  $[e^{-1},\infty)$.
Then the inverse function $H_{+}^{-1}\vcentcolon [e^{-1},\infty)\to [-e^{-1},\infty)$
exists and satisfies $H(H_{+}^{-1}(\xi))=\xi$ for all $\xi\in[-e^{-1},\infty)$.
Since $H(1)=0$, the restricted function $H_{+}^{-1} \vcentcolon [0,\infty) \to [1,\infty)$
is monotone increasing with $H_{+}^{-1}(0)=1$.
	
Next, by the inverse function theorem,
\begin{align*}
\dfrac{\dif G_{+}^{-1}}{\dif \xi}(\xi)
= \Big( \dfrac{\dif G}{\dif \zeta}\big( G_{+}^{-1}(\xi) \big) \Big)^{-1}
= \dfrac{1}{\log\big( G_{+}^{-1}(\xi) \big)} \qquad \ \text{for $\xi\in (0,\infty)$}.
\end{align*}
Since $\xi\mapsto G_{+}^{-1}(\xi)$ is a monotone increasing bijective map
and $G_+^{-1}(\xi)\to \infty$ as $\xi\to \infty$,
\begin{align*}
\lim\limits_{\xi\to\infty}  \dfrac{\dif G_{+}^{-1}}{\dif \xi}(\xi)= \lim\limits_{\xi\to\infty} \dfrac{1}{\log \big( G_{+}^{-1}(\xi) \big)} = 0,
\end{align*}
which, along with the L'H\^opital rule, yields that
\begin{align*}
\lim\limits_{y\searrow 0}  f_1(y;z)\vcentcolon
=\lim\limits_{y\searrow 0} y G_{+}^{-1}(\dfrac{z}{y})
\overset{\text{H}}{=} z \lim\limits_{y\searrow 0} \dfrac{\dif G_{+}^{-1}}{\dif \xi}
(\dfrac{z}{y})=0.
\end{align*}
In addition,
\begin{align*}
\dfrac{\partial f_1}{\partial y} (y;z)
= G_{+}^{-1}(zy^{-1}) - \dfrac{zy^{-1}}{\log \big( G_{+}^{-1}(zy^{-1})\big)}
\qquad\ \text{for $(y,z)\in(0,\infty)^2$},
\end{align*}
where $G_{+}^{-1}(z y^{-1})\in (1,\infty)$ for each $(y,z)\in (0,\infty)^2$.
It follows from  $G(\zeta)=1-\zeta+\zeta\log \zeta$ and the identity $G_{+}^{-1}(G(\zeta))=\zeta$
that
\begin{align*}
\dfrac{\partial f_1}{\partial y}(y;z) = \dfrac{G_{+}^{-1}(zy^{-1})-1}{\log\big(G_{+}^{-1}(zy^{-1})\big)}>0
\qquad\ \text{for all $(y,z)\in(0,\infty)^2$},
\end{align*}
which implies that, for fixing $z>0$, $f_1 (y;z)$ is monotone increasing for $y\in(0,\infty)$.
	
\smallskip
Next, by the inverse function theorem, it follows that
\begin{align*}
\dfrac{\dif \psi_+^{-1}}{\dif \xi} (\xi) = \Big( \dfrac{\dif \psi}{\dif \zeta}\big( \psi_{+}^{-1}(\xi) \big) \Big)^{-1} = \Big(1-\dfrac{1}{\psi_{+}^{-1}(zy^{-1})}\Big)^{-1}
\qquad\ \text{for $\xi\in(0,\infty)$}.
\end{align*}
Since $\psi_{+}^{-1}(\xi)\to \infty$ as $\xi\to \infty$, for each fixed $z>0$, by the L'H\^opital rule,
we have
\begin{align*}
\lim\limits_{y\searrow 0} y \psi_{+}^{-1}(\dfrac{z}{y}) \overset{\text{H}}{=}
z \lim\limits_{y\searrow 0} \Big( 1 - \dfrac{1}{\psi_{+}^{-1}(zy^{-1})} \Big)^{-1} =z
\quad \ \text{for $z>0$}.
\end{align*}
Thus, it follows that
\begin{align*}
\lim\limits_{y\searrow 0} f_2(y;z) = -z + \lim\limits_{y\searrow 0}  y \psi_{+}^{-1}(\dfrac{z}{y})
=0.
\end{align*}
Moreover, fixing $z>0$ and taking the derivative of $y\mapsto f_2(y;z)$, one has
\begin{align*}
\dfrac{\partial f_2}{\partial y}(y;z)
= \psi_{+}^{-1}(\dfrac{z}{y})
- \dfrac{z}{y^2} \dfrac{\dif \psi_{+}^{-1}}{\dif \xi}(\dfrac{z}{y})
= \dfrac{\psi_{+}^{-1}(zy^{-1}) \log ( \psi_{+}^{-1}(zy^{-1}) )}{\psi_{+}^{-1}(zy^{-1})-1}>0
\quad \text{ for $y\in (0,\infty)$},
\end{align*}
which implies that $y\to f_2(y;z)$ is increasing for each fixed $z>0$.
	
Finally, by the inverse function theorem, it follows that
\begin{align*}
\dfrac{\dif H_{+}^{-1}}{\dif \xi}(\xi) = \Big( \dfrac{\dif H}{\dif \zeta}( H_{+}^{-1}(\xi)) \Big)^{-1}
= \dfrac{1}{1 + \log ( H_{+}^{-1}(\xi))}   \qquad \ \text{for $\xi\in[0,\infty)$}.
\end{align*}
Since $H_+^{-1}(\xi)\to \infty$ as $\xi\to \infty$,
\begin{equation*}
\lim\limits_{\xi\to\infty}  \dfrac{\dif H_{+}^{-1}}{\dif \xi} (\xi)
= \lim\limits_{\xi\to\infty}\dfrac{1}{1 + \log ( H_{+}^{-1}(\xi))} = 0,
\end{equation*}
which, along with the  L'H\^opital rule, implies that
\begin{align*}
\lim\limits_{y\searrow 0} f_3(y;z) \overset{\text{H}}{=}
\lim\limits_{y\searrow 0} \dfrac{-z y^{-2}}{-y^{-2}} \dfrac{\dif H_{+}^{-1}}{\dif \xi}(z y^{-1})
= \lim\limits_{y\searrow 0}\dfrac{\dif H_{+}^{-1}}{\dif \xi}(z y^{-1}) = 0.
\end{align*}
Furthermore, for fixing $z>0$, we see that, for all $y\in(0,\infty)$,
\begin{align*}
\dfrac{\partial f_3}{\partial y}(y;z)
= H_{+}^{-1}(zy^{-1}) -\dfrac{z}{y} \dfrac{\dif H_{+}^{-1}}{\dif \xi}(z y^{-1})
= \dfrac{H_{+}^{-1}(z y^{-1})}{1+ \log H_{+}^{-1}( zy^{-1} )}>0.
\end{align*}
Therefore, $y\mapsto f_3(y;z)$ is strictly increasing and positive for each fixed $z>0$.
\end{proof}

Before proving the next proposition, we state the so-called Mazur's Lemma:
\begin{theorem}[Mazur's Lemma]\label{thm:mazur}
Let $(B,\bnorm{\cdot}_{B})$ be a Banach space, and let $S\subseteq B$ be a convex subset.
Then $S$ is closed in the strong topology if and only if it is closed in the weak topology.
\end{theorem}
For further detail, we refer the reader to Theorem 3.7 in Chapter 3.3 in \cite{brezis}.

\begin{proposition}\label{prop:mazur}
Let $d>0$, $T>0$, and $G(\zeta)\vcentcolon= 1-\zeta + \zeta\log \zeta$.
Assume there exist constants $\tilde{C}_1,\tilde{C}_2>0$, $\tilde{\delta}\ge0$,
and a sequence $\{ \theta_i \}_{i\in\mathbb{N}}\subset L^2(0,T;L^2([d,\infty),r^m \dif r))$
satisfying{\rm :}
\begin{equation}\label{eqs:B1}
\begin{dcases}
\theta_i -1 \rightharpoonup \theta - 1 \quad
   \text{ weakly in $L^2\big(0,T;L^2([d,\infty),r^m \dif r)\big)$}, \\
\sup_{i\in\mathbb{N}}\esssup_{t\in[0,T]}\int_d^{\infty} G(\theta_i)(r,t)\,r^m \dif r \le \tilde{C}_1,\\
  \tilde{\delta} \le \theta_i \le \tilde{C}_2 \quad \,\,
  \text{for {\it a.e.} $(r,t)\in [d,\infty)\times[0,T]$ and for each $i\in\mathbb{N}$}.
\end{dcases}
\end{equation}
Then the limit function $\theta$ satisfies
\begin{equation*}
\esssup_{t\in[0,T]} \int_{d}^{\infty} G(\theta)(r,t)\, r^m \dif r \le \tilde{C}_1.
\end{equation*}
\end{proposition}

\begin{proof}
Fix an integer $N\in \mathbb{N}$.
Denote $B_N \vcentcolon= L^2\big(0,T;L^2([d,N],r^m\dif r)\big)$.
Then $B_N$ is a Banach space.
Set $S_N$ to be the subset:
{\small
\begin{align*}
S_N\vcentcolon= \Big\{ w \in B_N \vcentcolon \
\esssup_{t\in[0,T]} \int_{d}^{N} G(w)(r,t)\,r^m \dif r \le \tilde{C}_1,
\,\,\tilde{\delta}\le w(r,t) \le \tilde{C}_2 \ \text{{\it a.e.} $(r,t)\in[d,N]\times[0,T]$} \Big\}.
\end{align*}
}
	
First, we claim that $S_N$ is a convex subset of $B_N$.
Since $G(z)=1-z+z\log z$ is convex, it follows that, for each $w_1, w_2\in S_N$ and $0<\eta<1$,
\begin{align*}
&\esssup_{t\in[0,T]} \int_{d}^{N} G(\eta w_1 + (1-\eta) w_2)(r,t)\,r^m \dif r\\
&\quad \le \eta \esssup_{t\in[0,T]} \int_{d}^{N} G(w_1)(r,t)\,r^m \dif r
 + (1-\eta) \esssup_{t\in[0,T]} \int_{d}^{N} G(w_2)(r,t)\,r^m \dif r \le \tilde{C}_1,\\
& \tilde{\delta} \le \big(\eta w_1 + (1-\eta) w_2\big)(r,t) \le \tilde{C}_2 \qquad \
 \text{for {\it a.e.} $ (r,t)\in[d,N]\times[0,T]$},
\end{align*}
which implies that $S_N$ is indeed a convex subset of $B_N$.
	
We next show that, if $\{w_i\}_{i\in\mathbb{N}} \subset S_N$ is a sequence such that,
for some $w\in B_N$,
\begin{align*}
\lim\limits_{i\to\infty}\sbnorm{w_i-w}_{B_N}
=\lim\limits_{i\to\infty}\int_{0}^{T}\int_{d}^{N} \snorm{w_i-w}^2(r,t)\,r^m \dif r \dif t =0,
\end{align*}
then $w\in S_N$.

Since $w_i\to w$ strongly in $L^2(0,T;L^2([d,N],r^m\dif r))$,
there exists a subsequence (still denoted) $\{w_i\}_{i\in\mathbb{N}}$ such that
\begin{equation}\label{tempB1}
w_i(r,t) \to w(r,t) \qquad \text{for {\it a.e.} $(r,t)\in[d,N]\times[0,T]$}.
\end{equation}
Since $\{w_i\}_i\subset S_N$, the convergence in \eqref{tempB1} implies that
\begin{equation}\label{tempB2}
\tilde{\delta} \le w(r,t) \le \tilde{C}_2 \qquad \text{for {\it a.e.} $(r,t)\in[d,N]\times[0,T]$}.
\end{equation}
Since $G(\cdot)$ is convex and achieves its minimum at $G(\zeta=1)=0$,
and $G(\zeta=0)=1$,
it follows that, for all $i\in\mathbb{N}$ with $\tilde{\delta} \le w_i\le \tilde{C}_2$,
$G(w_i(r,t)) \le \max\{1,G(\tilde{\delta}),\,G(\tilde{C}_2)\}$.
Moreover, since $G(\cdot)$ is continuous, it follows from \eqref{tempB1} that
$G(w_i(r,t)) \to G(w(r,t))$ as $i\to\infty$ for {\it a.e.} $(r,t)\in[d,N]\times[0,T]$.
By the dominated convergence theorem, we see that, for each $0\le t_1<t_2\le T$,
\begin{align*}
\dfrac{1}{t_2-t_1}\int_{t_1}^{t_2} \int_{d}^{N} G(w)(r,t)\,r^m \dif r \dif t
= \lim\limits_{i\to \infty} \dfrac{1}{t_2-t_1}\int_{t_1}^{t_2}
\int_{d}^{N} G(w_i)(r,t)\,r^m \dif r \dif t \le \tilde{C}_1.
\end{align*}
Since this is true for any  interval $[t_1,t_2]\subseteq[0,T]$,
then, by the Lebesgue differentiation theorem,
\begin{equation}\label{tempB3}
\esssup_{t\in[0,T]}\int_{d}^{N} G(w)(r,t)\,r^m \dif r \le \tilde{C}_1.
\end{equation}
It follows from \eqref{tempB2}--\eqref{tempB3}
that $w\in S_N$, which, along with Theorem \ref{thm:mazur}, implies that
\begin{equation}\label{tempB4}
S_N \ \text{ is closed in the weak topology of $B_N\vcentcolon=L^2\big(0,T;L^2([d,N],r^m \dif r)\big)$}.
\end{equation}
	
By $\eqref{tempB1}$, it follows that $\{\theta_i\}_{i\in\mathbb{N}}\subset S_N$
and $\theta_i \rightharpoonup \theta$ in $B_N=L^2(0,T;L^2([d,N],r^m \dif r))$.
Assertion \eqref{tempB4} then implies that $\theta \in S_N$ and, in particular,
\begin{equation*}
\esssup_{t\in[0,T]}\int_{d}^{N} G(\theta)(r,t)\,r^m \dif r \le \tilde{C}_1.
\end{equation*}
Now, since this is true for all $N\in\mathbb{N}$, and $G(\cdot)$ is non-negative,
it follows from the monotone convergence theorem that
\begin{align*}
\esssup_{t\in[0,T]}\int_{d}^{\infty} G(\theta)(r,t) r^m \dif r=\lim\limits_{N\to\infty}\esssup_{t\in[0,T]}\int_{d}^{N} G(\theta)(r,t) r^m \dif r \le \tilde{C}_1.
\end{align*}
\end{proof}

\bigskip
\noindent{\bf Acknowledgements:}
The research of Gui-Qiang G. Chen was supported in part by the UK Engineering and Physical Sciences Research
Council Award EP/L015811/1, EP/V008854, and EP/V051121/1.
The research of Yucong Huang was supported in part by the UK Engineering and Physical Sciences Research
Council Award EP/L015811/1.
The research of Shengguo Zhu was also supported in part by
the National Natural Science Foundation of China under the Grant 12101395,
the Royal Society–Newton International
Fellowships Alumni AL/201021 and AL/211005.

\bigskip
\noindent{\bf Conflict of Interest:} The authors declare  that they have no conflict of
interest.
The authors also  declare that this manuscript has not been previously  published,
and will not be submitted elsewhere before your decision.

\bigskip
\noindent{\bf Data availability:} Data sharing is not applicable to this article as no datasets were generated or analysed during the current study

\end{document}